\documentclass[12pt, reqno]{amsart}
\usepackage{amssymb,latexsym, amsmath, amsfonts, amsxtra, amsthm,xcolor,mathrsfs,enumerate,array, color, url}
\usepackage{graphicx}
\usepackage{enumitem}
\usepackage{subfig}

\usepackage[margin=0.75in]{geometry}

\usepackage{mathtools}

\usepackage{xurl}
\usepackage[hidelinks]{hyperref}

\allowdisplaybreaks

\newcommand{\bburl}[1]{\textcolor{blue}{\url{#1}}}

\usepackage{soul}

\usepackage{import}

\usepackage[T1]{fontenc}
\usepackage{mathtools}

\theoremstyle{definition}
\newtheorem{thm}{Theorem}[section]
\newtheorem*{HyPi}{Hypothesis $(\mathbf{\Pi}^4)$}
\newtheorem*{Testfunc}{Test functions}

\newtheorem*{ConLam}{Condition $(\mathbf{\Lambda})$}

\newtheorem{cor}[thm]{Corollary}
\newtheorem{lem}[thm]{Lemma}

\newtheorem{prop}[thm]{Proposition}

\newtheorem{rem}[thm]{Remark}

\theoremstyle{definition}

\theoremstyle{definition}

\newtheorem*{definition*}{Definition}
\theoremstyle{remark}
\newtheorem{cla}[thm]{Claim}

\newcommand\be{\begin{equation}}
\newcommand\ee{\end{equation}}
\newcommand\bee{\begin{equation*}}
\newcommand\eee{\end{equation*}}
\newcommand\ben{\begin{enumerate}}
\newcommand\een{\end{enumerate}}

\newcommand{\A}{\ensuremath{\mathbb{A}}}
\newcommand{\R}{\ensuremath{\mathbb{R}}}
\newcommand{\C}{\ensuremath{\mathbb{C}}}
\newcommand{\Z}{\ensuremath{\mathbb{Z}}}
\newcommand{\Q}{\mathbb{Q}}
\newcommand{\N}{\mathbb{N}}
\numberwithin{equation}{section}

\DeclareMathOperator{\PGL}{PGL}

\DeclareMathOperator{\im}{Im}
\DeclareMathOperator{\re}{Re}

\newcommand{\sumflat}[2]{\sideset{}{^\flat}\sum_{{#1}\,(\!\bmod\,{#2})}\, }

\newcommand{\ma}{\mathfrak{a}}
\newcommand{\mh}{\mathfrak{h}}
\newcommand{\mk}{\mathfrak{k}}

\newcommand{\md}{\mathfrak{d}}
\newcommand{\mz}{\mathfrak{z}}

\newcommand{\ml}{\mathnormal{l}c}
\newcommand{\mnl}{\mathnormal{l}}
\newcommand{\mfr}{\mathfrak{r}}
\newcommand{\mfC}{\mathbf{c}}

\newcommand{\mbu}{\mathbf{u}}
\newcommand{\mbv}{\mathbf{v}}
\newcommand{\mbs}{\mathbf{s}}
\newcommand{\mbw}{\mathbf{w}}
\newcommand{\mfq}{\mathfrak{q}}
\newcommand{\mcX}{\mathcal{X}}
\newcommand{\mcY}{\mathcal{Y}}
\newcommand{\mcP}{\mathcal{P}}
\newcommand{\mcQ}{\mathcal{Q}}

\usepackage{mathabx,epsfig}

\usepackage{pst-node}
\usepackage{tikz-cd}

\usepackage{verbatim}  

\usepackage[backend=biber,
style=alphabetic,]{biblatex}
\renewbibmacro{in:}{}
\DeclareFieldFormat*{title}{#1}
\usepackage{listings}
\newcommand{\ctb}[1]{{\color{teal} (CTB: #1)}}

\begin{document}

 \thanks{This research was supported by NSF DMS-1854398 FRG. The second author was supported by the EPSRC grant: EP/W009838/1, and the Harish-Chandra Fund from the Institute for Advanced Study. The fourth author is partially supported by NSF CAREER DMS-2239681.}

\title[Critical Zeros and Mean Value Theorems for $\hbox{PGL}(2)$ and $\hbox{PGL}(3)$]
{Critical Zeros and Unconditional Mean Value Theorems for twisted $\hbox{PGL}(2)$ and $\hbox{PGL}(3)$  $\mathrm{L}$-functions}

\author{J.~B. Conrey}
\address{American Institute of Mathematics \\ 1200 E California Blvd, Pasadena CA 91125}
\email{\href{mailto:conrey@aimath.org}{conrey@aimath.org}}

\author{C.-H. Kwan}
\address{Institute for Advanced Study\\   1 Einstein Drive, Princeton, New Jersey 08540,  USA}
\email{\href{mailto:chkwan@ias.edu}{chkwan@ias.edu}}

\author{Y. Lin}
\address{Data Science Institute, Shandong University, Jinan 250100, China; State Key Laboratory of Cryptography and Digital Economy Security, Shandong University, Jinan 250100, China}
\email{\href{mailto:yongxiao.lin@sdu.edu.cn}{yongxiao.lin@sdu.edu.cn}}

\author{C.~L. Turnage-Butterbaugh}
\address{Carleton College \\ 1 North College Street Northfield, MN 57707, USA}
\email{\href{mailto:cturnageb@carleton.edu}{cturnageb@carleton.edu}}


\maketitle

\begin{abstract}
Let $\Pi_{0}$ be a cuspidal automorphic representation of $\PGL_{3}(\A_{\Q})$. In this paper, we use Levinson's method to prove that, as $Q\to \infty$, at least $1/9$  of the zeros of the $L$-functions $L(s, \Pi_{0}\,\times\, \chi)$ lie on the critical line, where $\chi$ ranges over the family of primitive Dirichlet characters of conductor up to $Q$. This result is unconditional when $\Pi_{0}$ is self-dual, and otherwise holds under a mild condition.

The key technical input is a new asymptotic formula with a power-saving error term for the mean square of the product of $L(s, \Pi_{0}\times \chi)$ and a Dirichlet polynomial with arbitrary coefficients in both the $T$- and $Q$-aspects for the range $Q^{\epsilon}\le T \le Q^{1/3-\epsilon}$. When $T=Q^{\epsilon}$, our asymptotic formula allows Dirichlet polynomials of length $\theta <1/2-\epsilon$; when $\theta=0$, it gives a strong error term of size $O_{\epsilon}(Q^{7/4+\epsilon})$. Furthermore, our result provides evidence for the CFKRS conjectures for large twists and large vertical shifts. 

We also obtain corresponding results for $\PGL_{2}(\A_{\Q})$, which are fully unconditional, quantitatively stronger, and also appear to be new. 

 This work develops a refined, flexible, and uniform version of the Asymptotic Large Sieve for $L$-functions that \emph{does not} require any unproven progress toward the Generalized Ramanujan Conjecture. The arithmetic of $\Pi_{0}$ plays a crucial and delicate role in our argument. This work also makes extensive use of \texttt{Mathematica} to handle various elaborate Hecke algebra computations. Our mean value theorem is readily applicable to many other problems in analytic number theory.

 \end{abstract}

\setcounter{tocdepth}{1}


\section{Introduction}

\subsection{Zeros of $L$-functions}\label{sect: introzeros}

Since Riemann's seminal memoir in 1859, the zeros of $L$-functions have occupied a central place in number theory as they encode deep information about the statistics of a wide range of arithmetic objects. For instance, the truth of the (Generalized) Riemann Hypothesis would have profound consequences for the distribution of prime numbers and class groups.


While the Riemann Hypothesis remains far out of reach, it is a fundamental problem in analytic number theory to establish \emph{unconditional} and \emph{quantitative}, albeit weaker, results on the zeros of the Riemann zeta function $\zeta(s)$ inside the critical strip $0\le \sigma\le 1$, where $s=\sigma+it$. In 1914, Hardy established the infinitude of zeros of $\zeta(s)$ on the critical line $\sigma=1/2$ by detecting the sign changes of a normalized zeta function. His method was subsequently refined in joint work with Littlewood \cite{HL21}. A major breakthrough came in 1942, when Selberg \cite{Selberg42} introduced a \emph{mollifier} into the method of Hardy and Littlewood, allowing him to prove that a \emph{positive proportion} of the zeros of $\zeta(s)$ lie on the critical line. Three decades later, Levinson \cite{Lev74} introduced a new zero-detection method, based instead on detecting changes in argument of the derivative $\zeta'(s)$, together with a new class of mollifiers. He proved the landmark result that at least \emph{one-third} of the zeros of $\zeta(s)$ are on the critical line. In 1989, Conrey \cite{Con89} made substantial refinements to Levinson's method and improved this proportion to at least \emph{two-fifths}. This laid the foundation for a continuing sequence of improvements over the following three decades. Currently, the best-known proportion is at least \emph{five-twelfths} \cite{Wu19, Pratt20}. 

For some time, it was not entirely clear whether Levinson’s method could yield a positive proportion of zeros on the critical line for a degree-two $L$-function (see \cite{Far94}, \cite{Haf89}, \cite[Section 36]{Con16}). The main obstacle was that the method requires strong input from the spectral theory of automorphic forms, namely bounds for shifted convolution sums, which were not available in sufficient strength. This was compounded by the technical intricacy of Levinson’s method itself: Levinson’s original paper runs to more than fifty pages and consists of long, delicate calculations, giving the method a lasting reputation in the literature for being difficult to implement. For degree-two $L$-functions, the first positive proportion result obtained by Levinson’s method was achieved only in 2015 by Bernard \cite{Ber15}, who established an explicit but modest proportion of at least 2.97\% (subsequently improved to $6.9\%$ in \cite{AT21}), building on improved spectral ingredients of \cite{Blo04, BH08}, and a streamlined, seven-page treatment of Levinson's method by Young \cite{Yo10}. Prior to Bernard’s work, Hafner \cite{Haf83, Haf87} had obtained a positive proportion result by Selberg’s method instead. However, he did not provide an explicit proportion; extracting one is far from straightforward and yields a very small value; see \cite{Rez16} and \cite{Dou25}.

Unfortunately, this line of investigation does not extend to $L$-functions of degree greater than two. In fact, implementing Levinson's method for a degree-three $L$-function is at least as difficult as obtaining an asymptotic formula for the sixth moment of $\zeta(1/2+it)$, which is widely regarded as inaccessible by current spectral, automorphic, Fourier, or algebro-geometric methods.

There is, however, a natural variation on the theme above: one may study the critical zeros of $L$-functions in a \emph{family}. Indeed, analytic number theory has achieved many of its most striking unconditional results by exploiting additional cancellation through averaging over the harmonics in families, which exhibit some forms of \emph{orthogonality}. In the 1970-80s, there was particularly active research on zeros of $L$-functions \emph{off} the critical line, leading to a substantial body of results known as the zero-density estimates, which demonstrate the rarity of such zeros; see \cite[Chapter 10]{IK04}. For zeros \emph{on} the critical line, the theme was taken up quite recently by Conrey, Iwaniec and Soundararajan \cite{CIS13, CIS19}, whose result is striking: at least $56\%$ of the critical zeros of the family of Dirichlet $L$-functions lie on the critical line. As the authors comment, ``one may say that the [Generalized] Riemann Hypothesis is more likely to be true than not!'' 

This work continues the themes initiated in \cite{CIS13, CIS19}, now for automorphic $L$-functions. Let $\Pi$ be any \emph{self-dual} cuspidal automorphic representation of $\PGL_{3}(\A_{\Q})$ with its arithmetic conductor denoted by $\mfq_{\Pi}\in \mathbb{N}$. Using Levinson's method, we prove the following \emph{unconditional} result.

\begin{cor}\label{maincor: pergl3}
 Let $\epsilon>0$ be given.\footnote{ In this article, $\epsilon>0$ is an arbitrarily small quantity, and $O(\epsilon)$ denotes a (large) absolute constant multiple of $\epsilon$.} As $Q\to \infty$,  at least $1/9-O(\epsilon)$ of the zeros $\rho_{\chi}=\beta_{\chi}+i\gamma_{\chi}$ with $0\le \beta_{\chi}\le 1$ and $|\gamma_{\chi}|\le Q^{\epsilon}$ of the $L$-functions $L(s, \Pi\times \chi)$ lie on the critical line, where $\chi$ ranges over the primitive Dirichlet characters $\chi\, (\bmod\, q)$ with $(q, \mfq_{\Pi})=1$ and $q\le Q$.
\end{cor}

For a slightly more precise formulation of this result, see Corollary \ref{cor: fullresultpgl3per}. We have not attempted to optimize the numerical proportion above. In fact, we use the simplest mollifier from Levinson's original work \cite{Lev74}. Moreover, $1/9$ is a convenient round-down of $0.114663\ldots$, as computed in \cite{ConShortMollif}. Corollary \ref{maincor: pergl3} continues to hold if the range $|\gamma_{\chi}|\le Q^{\epsilon}$ replaced by $|\gamma_{\chi}|\le Q^{\varpi}$, for any $\varpi<1/3$, albeit with a smaller resulting proportion.

\begin{rem}
    As far as we are aware, this is the first \emph{unconditional} result that establishes a positive proportion of zeros on the critical line for a family of $\PGL_{3}$ $L$-functions. 
    
    In \cite{CIS13}, it was remarked that ``Our method also applies to twists of $\mathrm{GL}_{2}$ and $\mathrm{GL}_{3}$ $L$-functions. These cases are easier because the off-diagonal analysis is not necessary''. While it is true that only the diagonal terms of the relevant moments of $L$-functions enter the computation of the numerical proportion $1/9$ in Corollary \ref{maincor: pergl3}, the required \emph{estimations} for the off-diagonal parts are quite subtle. To our knowledge, the details of establishing the asymptotic formulae for these moments have not previously appeared in the literature. 
    
      This issue is particularly relevant in view of the limited progress toward the Generalized Ramanujan Conjecture (\textbf{GRC}). Indeed, all previous works based on the method of \cite{CIS13} (namely, the \emph{Asymptotic Large Sieve}) rely in important ways on the bound $O_{\epsilon}(n^{\epsilon})$ for the Dirichlet coefficients of the relevant $L$-functions. A direct adaptation of the method, with this bound replaced by the currently available bound toward \textbf{GRC}, is insufficient. It is therefore necessary to revisit the argument carefully and exploit additional arithmetic structures to obtain the required estimates unconditionally. We find that the more recent work \cite{CIS19}  illuminates the delicate nature of moments twisted by Dirichlet polynomials, already in the case of $\mathrm{GL}_{1}$ $L$-functions.  In the present automorphic setting, further subtleties arise; we discuss these in Sections \ref{sect: large sieve} and \ref{sect: sketUrsum}.
 \end{rem}

If $\Pi$ is non-self-dual, the same result holds under a mild assumption on the Dirichlet coefficients $\lambda_{\Pi}(n)$. In the statement below, denote by  $\mathbf{c}(\Pi)$ the analytic conductor of $\Pi$; see (\ref{IS: analyticond}).
\begin{HyPi}\label{L4normassu}
    For $X\ge 1$, we have
    \begin{align}\label{fourthpowerestHec}
    \sum_{n\le X} \, \frac{|\lambda_{\Pi}(n)|^4}{n} \, \ll_{\epsilon}\, (\mfC(\Pi)X)^{\epsilon}.
\end{align}
\end{HyPi}
\noindent In light of recent advances on the adjoint lifting for $\mathrm{GL}_{3}$ by \cite{Gan26}, together with the forthcoming work announced therein, it is reasonable to expect that (\ref{L4normassu}) can be established unconditionally. In any event, this point lies outside the scope of this article.


We also prove, unconditionally, a result parallel to Corollary \ref{maincor: pergl3} for \emph{all} cuspidal automorphic representations $\Pi$ of $\PGL_{2}(\A_{\Q})$, which also appears to be new: at least $1/3$ of the critical zeros for the $L$-functions in the family of primitive character twists of $\Pi$ lie on the critical line, matching the proportion obtained by Levinson for $\zeta(s)$.\footnote{ With Conrey's modification \cite{Con83}, the proportion improves slightly to $36.58\%$.}  In fact, more is true; see Corollary \ref{cor: simplpergl2}.


\subsection{Mean value theorems}\label{sect: MVTtheoemre}

The principal number-theoretic input in Levinson's method is an asymptotic formula for the mean square of the product of $\zeta(1/2+it)$ with a Dirichlet polynomial:
\begin{align}\label{BCHBtypemoment}
    \int_{0}^{T} \, |\zeta(1/2+it)|^2 \,\Bigg|\,\sum_{h\le T^{\theta}} \, \frac{\lambda_{h}}{h^{1/2+it}}\Bigg|^2 \ dt \hspace{30pt} (T \,\to \, \infty),
\end{align}
for some $\theta>0$ and coefficients $(\lambda_h)$. In \cite{Lev74}, the coefficients $\lambda_h$ are chosen to approximate the M\"obius function (see (\ref{zetamoll})), and the resulting Dirichlet polynomial is a mollifier. Levinson obtained an asymptotic formula for (\ref{BCHBtypemoment}) for any $\theta<1/2$. His result was generalized in \cite{BCHB85}, where the same range $\theta<1/2$ was established for \emph{arbitrary} coefficients $(\lambda_h)$ satisfying $\lambda_h\ll_\epsilon h^\epsilon$.

Expanding the square of the Dirichlet polynomial in (\ref{BCHBtypemoment}), one is led to the twisted mean square 
 \begin{align}\label{twistmoexa}
       \int_{0}^{T}  |\zeta(1/2+it)|^2 \Big(\frac{h}{k}\Big)^{-it} \, dt
    \end{align}
for integers $h,k\ge 1$. We note that the study of twisted mean squares has its roots in Selberg’s work from the 1940s \cite{SelbRie, Sel46}. For \emph{fixed} $h, k$, the asymptotic formula for (\ref{twistmoexa}) takes the form
\begin{align}\label{fixedexpldep}
       \int_{0}^{T}  |\zeta(1/2+it)|^2 \Big(\frac{h}{k}\Big)^{-it} \, dt \, = \,  T\, \frac{(h,k)}{\sqrt{hk}} \Big(\log \frac{T(h,k)^2}{2\pi hk}+ 2\gamma-1\Big) \, + \, O((hk)^{\delta_{1}}T^{1-\delta_{2}})
\end{align}
for some $\delta_{1}$, $\delta_{2}>0$. The aforementioned result of \cite{BCHB85} is obtained by summing (\ref{fixedexpldep}) uniformly over $h,k\le T^{\theta}$ \,($\theta<1/2$) against arbitrary coefficients. 

The exponent $\theta$ in (\ref{BCHBtypemoment}) is referred to as the \emph{``length''} of the Dirichlet polynomial or mollifier. A central objective is to evaluate (\ref{BCHBtypemoment}) asymptotically with $\theta$ as large as possible, since this typically leads to stronger quantitative results in applications. To go beyond the range $\theta<1/2$, pointwise estimates such as (\ref{fixedexpldep}) are insufficient; one must instead exploit cancellation in the error term of (\ref{twistmoexa}) after summing bilinearly over $h$ and $k$. For Levinson's mollifier, this was first achieved in \cite{BCHB85} for any $\theta<1/2+1/34$. Conrey \cite{Con89} later improved this range to $\theta<4/7$ using spectral theory of automorphic forms, and this remains the best-known result to date.

However, breaking the $1/2$-barrier for \emph{arbitrary} Dirichlet polynomials in (\ref{BCHBtypemoment}) proved more delicate. It was achieved only relatively recently in \cite{BCR17}, where the range $\theta<1/2+1/66$ was obtained using bounds for trilinear forms involving Kloosterman fractions. This advance is important for a variety of applications, e.g.  fractional moments and value distributions of $\zeta(1/2+it)$. 

In \cite{BCHB85}, it was conjectured that the same asymptotic formula continues to hold for every $\theta<1$ and for \emph{arbitrary} Dirichlet polynomial. In fact, this conjecture would imply the Lindel\"of Hypothesis. Nevertheless, the work \cite{CIS19} provides strong support for this conjecture in an average sense. More precisely, for any $\theta<1$, they proved an asymptotic formula for the following mean square of a Dirichlet $L$-function twisted by an arbitrary Dirichlet polynomial:
\begin{align}\label{CIS19moq-ana}
  \sum_{q\le  Q} \  \sideset{}{^*}{\sum}_{\chi\, (\bmod\, q)} \, |L(1/2, \chi)|^2 \Bigg|\,\sum_{h\le Q^{\theta}} \, \frac{\lambda_{h}\chi(h)}{h^{1/2}}\Bigg|^2,
\end{align}
where the average is taken over the family of primitive characters of all conductors up to $Q$, as $Q\to \infty$. Without the $q$-average, and as $q\to \infty$ over primes, the best known range is only $\theta<1/2+1/202$, due to \cite{BPR+20}.



The present work is motivated by \cite{CIS19}. We study the analogous problem for automorphic $L$-functions and its applications to critical zeros. We now describe the set-up of this article.  Let  $\mathrm{G}_{r}:= \hbox{PGL}_r$ ($r\in \{2,3\}$) and $\A:=\A_{\Q}$ denote the ring of adeles over $\Q$. Let $\mathcal{A}_{0}(\mathrm{G}_{r})$ be the set of isomorphism classes of irreducible cuspidal automorphic representations $\Pi \subset  L^{2}(\mathrm{G}_{r}(\Q)\setminus \mathrm{G}_{r}(\A))$. Denote by $L(s, \Pi)$ the standard $L$-function of $\Pi$. On $\re s\gg 1$, it is given by the Dirichlet series
\begin{align}
    L(s, \Pi) \, := \, \sum_{n=1}^{\infty} \, \frac{\lambda_{\Pi}(n)}{n^{s}},
\end{align}
 and it admits an entire continuation. Furthermore, it satisfies a functional equation of the form
\begin{align}\label{jacstdfunc}
    \Lambda(s, \Pi) \, := \, (\mfq_{\Pi})^{s/2} L_{\infty}(s, \Pi)L(s, \Pi) \ = \  \epsilon(\Pi) \Lambda(1-s, \widetilde{\Pi})   
\end{align}
for any $s\in \C$; see \cite{Jac79}. In (\ref{jacstdfunc}), $\mfq_{\Pi}\in \N$ denotes the arithmetic conductor of $\Pi$ as defined in \cite{JPS81};  $L_{\infty}(s, \Pi)$ is of the form $\prod_{1\le i\le r} \, \Gamma_{\R}(s+\mu_{i}(\Pi))$ with $\Gamma_{\R}(s):= \pi^{-s/2}\Gamma(s/2)$ and $(\mu_{i}(\Pi))_{1\le i \le r}$ being the spectral parameters of $\Pi$;  $\epsilon(\Pi)$ is the root number which satisfies $|\epsilon(\Pi)|=1$; and  $\widetilde{\Pi}$ is the contragredient of $\Pi$. It is often desirable to establish estimates with explicit dependence on $\Pi$. To this end, we employ the \emph{analytic conductor} of Iwaniec--Sarnak, which is defined as
 \begin{align}\label{IS: analyticond}
      \mfC(\Pi) \ := \ \mfq_{\Pi}\,\mfC_{\infty}(\Pi) \ := \ \mfq_{\Pi} \, \prod_{1\le i \le r}  \, (1+|\mu_{i}(\Pi)|).
 \end{align}

Let $\chi$ be a primitive Dirichlet character $(\bmod\, q)$ with $(q, \mfq_{\Pi})=1$. For $ \re s \gg  1$, the $L$-function of $\Pi$ twisted by $\chi$ can be defined by
 \begin{align}
     L(s, \Pi \times \chi) \, := \, \sum_{n=1}^{\infty} \, \frac{\lambda_{\Pi}(n)\chi(n)}{n^{s}}.
 \end{align}
It likewise admits an entire continuation and functional equation; see Section \ref{sect: twistdata}.

 In practice, it suffices and is technically convenient to work with \emph{smooth} averages of $L$-functions. For this, we introduce the following pairs of test functions:

\begin{Testfunc}\label{regtest}
The functions $W, \,\eta: \R \rightarrow[0,\infty)$ are smooth and satisfy the conditions that
\begin{enumerate}
\item \label{dyadic} $W$ is supported on $[1,2]$ and $W^{(j)} \ll_{j} 1$ for any $j\ge 0$;

 \item \label{smooth} 
 $\eta=\eta_{T,\Delta}$ is supported on $[T/2-\Delta, \, T+\Delta]$, and satisfies $1_{[T/2, T]} \le \eta\le 1 $,  $\eta^{(j)}  \ll_{j}   \Delta^{-j}$. The parameter $\Delta$ satisfies $(TQ)^{\epsilon} \ll_{\epsilon} \Delta \,\le\, T$.  
\end{enumerate}
\end{Testfunc}

We have implicitly assumed that $T\ge Q^{\epsilon}$, and this is necessary for our argument. This condition is also natural as it ensures an adequate supply of critical zeros for each $L$-function in the family (see \cite[Theorem 5.8]{IK04}).  For Levinson's method, one takes $\Delta=T/\log T$ as in \cite{Yo10}. 

Let  $h,k\ge 1$ be integers, and $\alpha, \beta$ be the ``shifts'' satisfying $|\alpha|,  |\beta| \, \ll \, (\log TQ)^{-1} \ll |\alpha+\beta|$. The object of this article is the following (shifted) twisted mean square of $L$-functions: 
\begin{align}\label{eq: newmainob}
\mathcal{M}_{\alpha, \, \beta}(h,k; \Pi) \, := \, \sum_{\substack{(q, \, \mfq_{\Pi})=1}} \,  & W\big(\frac{q}{Q}\big) \,
    \sumflat{\chi}{q} \,   \chi(h)\overline{\chi}(k) \nonumber\\
    &\hspace{-10pt}\cdot \, \int_{\mathbb{R}}    \eta(t)L(1/2  + \alpha+it, \, \Pi \times \chi)  L(1/2+ \beta-it, \, \widetilde{\Pi}
    \times \overline{\chi}) \big(\frac{h}{k}\big)^{-it} \, dt.
 \end{align}
 Here, we average over the full set (archimedean and non-archimedean) of $\mathrm{GL}_{1}$ harmonics over $\Q$. The symbol ``$\flat$'' indicates that the sum over $\chi$ is restricted to \emph{even} primitive characters. The argument for odd primitive characters is similar.

We now describe the main terms of $\mathcal{M}_{\alpha, \, \beta}(h,k; \Pi)$, for which we need a couple of notation. Let $\phi^{*}(q)$ be the number of primitive characters $(\bmod\, q)$.  For $\delta\in \{0,1\}$ and $\re s \gg 1$,  we define 
\begin{align}
    \mathcal{L}_{\alpha, \beta; \, h,k}^{(q)}(s, \Pi) &= \mathcal{L}_{\alpha, \beta; \, h,k}^{(q)}(s, \Pi,\widetilde{\Pi}) \, := \, \sum_{\substack{m,n\ge 1\\ mh=nk\\ (mn, \, q)=1}} \frac{\lambda_{\Pi}(m) \lambda_{\widetilde{\Pi}}(n)}{m^{s+\frac{1}{2}+\alpha} n^{s+\frac{1}{2}+\beta}}, \label{MTdirhcl}\\
    \mathcal{W}_{\alpha, \beta; \, h,k}^{\delta}(s,\Pi) \ &:= \ \frac{1}{2} \, \sum_{\substack{ (q, \, hk\mfq_{\Pi})=1}}  W\Big(\frac{q}{Q}\Big) \phi^{*}(q)(\mfq_{\Pi}q^r)^{s+\delta(\alpha+\beta)}  \mathcal{L}_{\alpha, \beta;\, h,k}^{(q)}(s, \Pi). \label{beforesumq}
\end{align}
  The series (\ref{MTdirhcl}) admits a holomorphic continuation to $\re s > -1/14$ for $r\in \{2,3\}$, except for a simple pole at $s=-\frac{\alpha+\beta}{2}$. We also let 
\begin{align}
 \gamma_{(\alpha, \beta); \, it}(s;\, \Pi, \widetilde{\Pi}) \, &:= \, L_{\infty}(s+\alpha+it,\, \Pi)L_{\infty}(s+\beta-it, \, \widetilde{\Pi}).
 \end{align}
Given any integers $h,k\ge 1$. The \emph{``recipe''} of \cite{CFKRS} predicts an asymptotic formula for (\ref{eq: newmainob}):
\begin{align}\label{CFKconj}
   \hspace{-7pt} \mathcal{M}_{\alpha, \, \beta}(h,k; \Pi)  \ &\sim \   \, \int_{\R} \, \eta(t) \, \Big(  \mathcal{W}_{\alpha, \beta;\, h,k}^{0}(0, \Pi) +\frac{\gamma_{(-\beta, -\alpha); \, it}(1/2, \Pi, \widetilde{\Pi}) }{\gamma_{(\alpha, \beta); \, it}(1/2,  \Pi, \widetilde{\Pi}) } \,  \mathcal{W}^{1}_{-\beta, -\alpha;\, h,k}(0, \Pi)  \Big) \  dt.
\end{align}

In (\ref{CFKconj}) and (\ref{beforesumq}), we retain the sum over $q$ since its evaluation is not needed for the implementation of Levinson's method in Section \ref{sect: CSLevGL3}. Moreover, this gives the natural form of the main terms, as the asymptotic formula is expected to hold for each individual $q$. Nevertheless, the sum over $q$ in (\ref{beforesumq}) can be evaluated quite easily, and we record the result here for the reader's convenience. We introduce the following shorthand
\begin{align}
     \mathcal{L}_{\alpha, \beta; \,h,k}(s,\Pi) \ &:= \  \mathcal{L}_{\alpha, \beta; \, h,k}^{(1)}(s, \Pi), \hspace{20pt} \mathcal{L}^{(q)}(1+2s+\alpha+\beta,  \Pi\otimes \widetilde{\Pi}) \ := \ \mathcal{L}^{(q)}_{\alpha, \beta;\,  1,1}(s, \Pi),\label{shorthand}\\[5pt]
     &\hspace{20pt} \hspace{15pt} \mathcal{L}^{(q)}(w,\Pi\otimes \widetilde{\Pi})\ := \ \prod_{p\nmid q}\, \mathcal{L}_p(w, \Pi\otimes \widetilde{\Pi}),\label{shorthand2}
\end{align}
and
\begin{align}\label{MTeulerpord}
   \mathcal{B}_{p}(w, \Pi\otimes \widetilde{\Pi}) \ := \  (1-p^{-2})^2 \, + \, (\mathcal{L}_{p}(w, \Pi\otimes \widetilde{\Pi})^{-1}-1)p^{-2}(p-2+p^{-2}). 
\end{align}
Then from  (\ref{beforesumq}) we deduce that
\begin{align}\label{arithform}
 \mathcal{W}_{\alpha, \beta;\, h,k}^{0}(0,\Pi) \, \sim \     \frac{ Q^2}{2}\, \Big(\int_{\R} \, W(x)x \, dx\Big) \,\frac{\phi(hk\mfq_{\Pi})}{hk\mfq_{\Pi}}  \mathcal{L}_{\alpha, \beta; \, h,k}(0, \Pi) \prod_{p\nmid hk\mfq_{\Pi}} \, \mathcal{B}_{p}(1+\alpha+\beta, \,  \Pi\otimes \widetilde{\Pi}).
\end{align}
The Euler product in (\ref{arithform}) is interpreted as follows. On $\re w > 10$ (say), we have $\mathcal{L}_{p}(w, \Pi\otimes \widetilde{\Pi}) \neq 0$ for any $p$, and (\ref{MTeulerpord}) admits the expansion:
\begin{align}\label{arithexp}
  \mathcal{B}_{p}(w,\, \Pi\otimes \widetilde{\Pi}) \ =\    1+O(p^{-1-\re w}|\lambda_{\Pi}(p)|^2) +O(p^{-2}).
\end{align}
Thus, the Euler product converges absolutely for $\re w > 1-2/7$. The calculation for $\delta=1$ is similar. For more details, see Section \ref{sect: diagterm}. 

For practical applications (see, e.g., \cite{IS97, Yo10, Ber15}), it is more useful to work wth a smooth version of (\ref{CFKconj}). We let
\begin{align}\label{diagwithcutof}
	\mathcal{D}_{\alpha, \,\beta}^{0}(h,k;\, \Pi) \, &:= \,  \frac{1}{2\pi i}  \int_{\mathbb{R}} \, \eta(t)  \int_{(\epsilon)} \, G(s)\frac{\gamma_{(\alpha,\, \beta); \, it}(1/2+s,  \Pi, \widetilde{\Pi})}{\gamma_{(\alpha, \,\beta); \, it}(1/2,  \Pi, \widetilde{\Pi})} \mathcal{W}_{\alpha, \beta;\, h,k}^{0}(s, \Pi) \, \frac{ds}{s} \, dt, \nonumber\\
    \mathcal{D}_{-\beta, -\alpha}^{1}(h,k;\, \Pi) \, &:= \,  \frac{1}{2\pi i}  \int_{\mathbb{R}} \, \eta(t)  \int_{(\epsilon)} \, G(s)\,  \frac{\gamma_{(- \beta, \, -\alpha); \, it}(1/2+s,  \Pi, \widetilde{\Pi})}{\gamma_{(\alpha, \,\beta); \, it}(1/2,  \Pi, \widetilde{\Pi})} \mathcal{W}_{-\beta, -\alpha;\, h,k}^{1}(s, \Pi) \, \frac{ds}{s} \, dt,
\end{align}
where
\begin{align*}
    G(s) \, := \,  e^{s^2} \, \frac{(\alpha+\beta)^2-(2s)^2}{(\alpha+\beta)^2}. 
\end{align*}

In this article, we prove an asymptotic formula for (\ref{eq: newmainob}) in the range
\begin{align}\label{range: gl3hybrid}
    Q^{\epsilon} \ \le \  T \ \le  \ Q^{1/3-\epsilon} \hspace{20pt} (Q \, \to \, \infty),
\end{align}
and we are able to average over $h,k$ non-trivially with arbitrary coefficients. Throughout, the coefficients $(\lambda_{h})$ are always supported on $h\le (TQ)^{\theta}$, where $\theta\ge 0$. We are ready to state a simplified version of our main result.

\begin{thm}\label{thm: gl3, dual, strongspec}
Let $\Pi \in \mathcal{A}_{0}(\mathrm{G}_{3})$. Suppose that the coefficients satisfy $\lambda_{h}\ll_{\epsilon} h^{\epsilon}$. Then 
\begin{align}\label{mainthm: momentexp}
    \mathcal{M}_{\alpha,\, \beta}(h,k;\, \Pi) \, = \, \mathcal{D}^{0}_{\alpha,\, \beta}(h,k; \, \Pi) \, + \, \mathcal{D}^{1}_{-\beta, \, -\alpha}(h,k; \, \Pi) \, + \, \mathcal{E}_{\alpha, \beta}(h,k; \, \Pi),
\end{align}
where in the range $ Q^{\epsilon} \le  T \le   Q^{1/3-\epsilon}$, the error term $\mathcal{E}_{\alpha,\, \beta}(h,k; \, \Pi)$ satisfies
\begin{align}\label{bilerr}
\Bigg|\, \mathop{\sum \sum}_{h, k \,\le \, (TQ)^{\theta}  } \frac{\lambda_{h}\overline{\lambda_{k}}}{\sqrt{hk}} \,\mathcal{E}_{\alpha,\, \beta}(h,k; \, \Pi)  \, \Bigg|\ \ll_{\epsilon} \ \mfC(\Pi)^{4}(TQ)^{7/4+\theta/2} (\mfC(\Pi)TQ)^{O(\epsilon)}.
\end{align}
\end{thm}

\begin{rem}
In particular, Theorem \ref{thm: gl3, dual, strongspec} gives an asymptotic formula for the twisted mean square 
 \begin{align}\label{exa: twistMSq}
     \sum_{\substack{(q, \, \mfq_{\Pi})=1}} \,    W\big(\frac{q}{Q}\big)\sumflat{\chi}{q} \,    \int_{\mathbb{R}}    \eta(t)|L(1/2+it, \, \Pi \times \chi)|^2  \, \bigg| \sum_{h\le (TQ)^{\theta}} \ \frac{\lambda_{h}\chi(h)}{h^{1/2+it}}\bigg|^2 \, dt.
 \end{align}
 The size of the main terms is typically $\asymp TQ^2$. If $ T\asymp Q^{\varpi}$, Theorem \ref{thm: gl3, dual, strongspec} is nontrivial when 
     \begin{align}
         0 \, \le \, \theta \, < \, \frac{1-3\varpi}{2(1-\varpi)} - O(\epsilon).  
     \end{align}
 When $T\asymp Q^{\epsilon}$, the length $\theta$ of the Dirichlet polynomial goes up to $1/2-O(\epsilon)$.  
\end{rem}

When $h=k=1$ and $\theta=0$, we have an asymptotic formula for the untwisted moment. While this is not the main goal of this paper, our result improves the error term obtained in \cite{CIS12}, and applies to a larger range (\ref{range: gl3hybrid}). Furthermore, this offers some evidence towards the validity of the CFKRS conjecture for large vertical shifts (see \cite{Bet10, Kov25} for moments of the $\zeta$-function).

\begin{cor}\label{cor: untwismoSpect}
Suppose that $ Q^{\epsilon} \le  T \le   Q^{1/3-\epsilon}$. Then
    \begin{align}
        \mathcal{M}_{\alpha, \, \beta}(1,1; \,\Pi) \ = \  \frac{1}{2} \, &\Big(\int_{\R} \, \eta(t) \, dt \Big)\, \sum_{\substack{ (q, \, \mfq_{\Pi})=1}}  \,  W\Big(\frac{q}{Q}\Big) \phi^{*}(q)\mathcal{L}^{(q)}(1+\alpha+\beta, \, \Pi\otimes \widetilde{\Pi})  \nonumber\\
        &+ \frac{1}{2} \,  \Big(\int_{\R} \  \eta(t)\frac{\gamma_{(-\beta, -\alpha); \, it}(1/2, \Pi, \widetilde{\Pi}) }{\gamma_{(\alpha, \beta); \, it}(1/2,  \Pi, \widetilde{\Pi}) } \, dt \Big) \sum_{\substack{ (q, \, \mfq_{\Pi})=1}}  \,  W\Big(\frac{q}{Q}\Big)\phi^{*}(q) (\mfq_{\Pi}q^3)^{-\beta-\alpha} \nonumber\\
        &\hspace{40pt}\cdot \mathcal{L}^{(q)}(1-\beta-\alpha, \, \Pi\otimes \widetilde{\Pi})  \, + \, O_{\epsilon}\left(\mfC(\Pi)^{4}(TQ)^{7/4}(\mfC(\Pi)TQ)^{O(\epsilon)}\right).
    \end{align}
\end{cor}

\begin{rem}
One can show that
\begin{align}\label{intro: naiveRs}
   \mathcal{L}^{(q)}(w, \Pi\otimes \widetilde{\Pi}) \ = \  \prod\nolimits_{p\nmid q} \, L_{p}(w, \Pi\otimes \widetilde{\Pi})H_{p}(w,\,  \Pi,\, \widetilde{\Pi}),
\end{align}
where $L_{p}(w, \Pi\otimes \widetilde{\Pi})$ is the Rankin--Selberg local factor of \cite{JPSS_RS}, and the factor $H_{p}(w, \Pi, \widetilde{\Pi})$ admits a simple explicit expression (Lemma \ref{lem: naiveRSexpl}). The infinite product $\prod_{p\nmid q} \, H_{p}(w, \Pi, \widetilde{\Pi})$ is holomorphic on $\re w > 1-1/7$, and by the unitarity of $\Pi$, it does not vanish at $w=1$ (see Corollary \ref{Hlocalnonvan}).
\end{rem}

Although Theorem \ref{thm: gl3, dual, strongspec} is by itself unconditional, a careful reader will notice that the condition $\lambda_{h}\ll_{\epsilon} h^{\epsilon}$ imposed there is too strong in practice. Indeed, the coefficients $(\lambda_{h})$ are chosen in terms of the Dirichlet coefficients of $L(s, \Pi)$ in many applications; see, for instance, (\ref{ourexamollif})--(\ref{Ex: gl3mollif}). For such choices, the bound $\lambda_{h}\ll_{\epsilon} h^{\epsilon}$ is not known, as it would require the full \textbf{GRC}, which is far out of reach. This motivates us to work with a more general condition on $(\lambda_{h})$. This issue appears to have received little attention in the literature, and will lead to some delicate difficulties; see Section \ref{sect: large sieve}.

\begin{ConLam}\label{assu: coeff}
The coefficients $(\lambda_{h})$ satisfy the following bounds:
    \begin{align}\label{bdd: 24normabstract}
    \lambda_{h} \, \ll \,  h^{1/2} \hspace{20pt} \text{and} \hspace{20pt} \sum_{h\le (TQ)^{\theta}} \,  \frac{|\lambda_{h}|^{2m}}{h} \ll_{\epsilon} (TQ)^{\epsilon} \hspace{15pt} ( m \in \{1,2\}).
\end{align}    
\end{ConLam}

Corollary \ref{maincor: pergl3} is deduced from the following result.

\begin{thm}\label{generalL4version them}
Suppose that $\Pi \in \mathcal{A}_{0}(\mathrm{G}_{3})$ satisfies \textbf{Hypothesis $(\mathbf{\Pi}^4)$}, and that \textbf{Condition $(\mathbf{\Lambda})$} holds. Then the same asymptotic formula of Theorem \ref{thm: gl3, dual, strongspec} holds in the same range (\ref{range: gl3hybrid}).    
\end{thm}

\begin{rem}\label{selfdualassu}
If\, $\Pi\in \mathcal{A}_{0}(\mathrm{G}_{3})$ is \textbf{self-dual}, then it is known from Jiang--L\"{u} \cite[Theorem 1.2]{JL17} that Hypothesis $(\mathbf{\Pi}^4)$, and hence Theorem \ref{generalL4version them}, holds \textbf{unconditionally}.    In any case, $(\mathbf{\Pi}^4)$ is a much weaker assumption than \textbf{GRC}. 
\end{rem}

For applications to simultaneous large values of $L$-functions (see, e.g., \cite{HL23, BELP25}), it is essential to have a mean value theorem for two automorphic representations $\Pi_{1}$ and $ \Pi_{2}$ that are not contragredient to one another. Let
\begin{align}\label{mom: twodistform}
\mathcal{M}_{\alpha, \, \beta}(h,k; \Pi_{1}, \Pi_{2}) \, := \, \sum_{\substack{(q, \, \mfq_{\Pi_{1}}\mfq_{\Pi_{2}})=1}} \,  & W\big(\frac{q}{Q}\big) \,
    \sumflat{\chi}{q} \,   \chi(h)\overline{\chi}(k) \nonumber\\
    &\hspace{-47pt}\cdot \, \int_{\mathbb{R}}    \eta(t)L(1/2  + \alpha+it, \, \Pi_{1} \times \chi)  L(1/2+ \beta-it, \, \Pi_{2}
    \times \overline{\chi}) \big(\frac{h}{k}\big)^{-it} \, dt, 
 \end{align}
 Define $\mathcal{D}^{\delta}_{\alpha,\, \beta}(h,k;\, \Pi_{1}, \Pi_{2})$ analogously to (\ref{diagwithcutof}) (see Section \ref{sect: diagterm}). We arrive at the most general form of our theorem.

\begin{thm}\label{thm: gl3, dualcase}
Suppose that  $\Pi_{1},  \Pi_{2} \in \mathcal{A}_{0}(\mathrm{G}_{3})$ satisfy $\mfq_{\Pi_{1}}=\mfq_{\Pi_{2}}$ and  \textbf{Hypothesis $(\mathbf{\Pi}^4)$}, and the coefficients $(\lambda_{h})$ satisfy \textbf{Condition $(\mathbf{\Lambda})$}. Then we have
\begin{align}\label{twoformasym}
    \mathcal{M}_{\alpha,\, \beta}(h,k;\, \Pi_{1}, \Pi_{2}) \, = \, \mathcal{D}^{0}_{\alpha,\, \beta}(h,k;\, \Pi_{1}, \Pi_{2}) \, + \, \epsilon(\Pi_{1})\epsilon(\Pi_{2}) \mathcal{D}_{-\beta, \, -\alpha}^{1}(h,k; \,\widetilde{\Pi_{2}}, \widetilde{\Pi_{1}}) \, + \, \mathcal{E}_{\alpha, \beta}(h,k; \, \Pi_{1}, \Pi_{2}),
\end{align}
where in the range (\ref{range: gl3hybrid}), the error term $\mathcal{E}_{\alpha, \beta}(h,k; \, \Pi_{1}, \Pi_{2})$ satisfies
\begin{align}
\Bigg|\, \mathop{\sum \sum}_{h, k \,\le \, (TQ)^{\theta}  } \frac{\lambda_{h}\overline{\lambda_{k}}}{\sqrt{hk}} \, \mathcal{E}_{\alpha, \beta}(h,k; \, \Pi_{1}, \Pi_{2})  \, \Bigg|\, \ll_{\epsilon} \, (\mfC(\Pi_{1})\mfC(\Pi_{2}))^{2}(TQ)^{7/4+\theta/2}\left(\mfC(\Pi_{1})\mfC(\Pi_{2})TQ\right)^{O(\epsilon)}.
\end{align}
\end{thm}

\begin{rem}
    If $\Pi_{1}, \Pi_{2}\in \mathcal{A}_{0}(\mathrm{G}_{3})$ are self-dual, they are of orthogonal type and $\epsilon(\Pi_{1})=\epsilon(\Pi_{2})=1$. 
\end{rem}       
 
Let us conclude this section by stating the corresponding result for the case of $\mathrm{PGL}_{2}$.

\begin{thm}\label{thm: GL2general}
    Suppose that $\Pi_{1},  \Pi_{2} \in \mathcal{A}_{0}(\mathrm{G}_{2})$ satisfy $\mfq_{\Pi_{1}}=\mfq_{\Pi_{2}}$,  and the coefficients $(\lambda_{h})$ satisfy \textbf{Condition $(\mathbf{\Lambda})$}. Then  (\ref{twoformasym}) holds  as $Q\to \infty$ and for $Q^{\epsilon} \le  T \le  Q^{1-\epsilon},$ and the error term satisfies 
\begin{align}
\Bigg|\, \mathop{\sum \sum}_{h, k \,\le \, (TQ)^{\theta}  } \frac{\lambda_{h}\overline{\lambda_{k}}}{\sqrt{hk}} \, \mathcal{E}_{\alpha, \beta}(h,k; \, \Pi_{1}, \Pi_{2}) \, \Bigg|\, \ll_{\epsilon} \, (\mfC(\Pi_{1})\mfC(\Pi_{2}))^{2}(TQ)^{3/2+\theta/2}\left(\mfC(\Pi_{1})\mfC(\Pi_{2})TQ\right)^{O(\epsilon)}.
\end{align}
\end{thm}

\begin{rem}
 In Theorem \ref{thm: GL2general}, we do not impose Hypothesis $(\mathbf{\Pi}^4)$ because it is a theorem of Moreno and Shahidi \cite{MS83}. Also,  $\Pi\in \mathcal{A}_{0}(\mathrm{G}_{2})$ must be self-dual and hence $\epsilon(\Pi) = \pm 1$. 
\end{rem}


\subsection{Levinson's method}\label{sect: introLevin}

Let $\mcP(x)$, $\mcQ(y)$ be two real polynomials such that
\begin{align}
    \mcP(0) \, = \, 0,   \hspace{15pt} \mcP(1) \, = \, 1,  \hspace{15pt}  \mcQ(0) \, = \, 1 \hspace{15pt} \text{and} \hspace{15pt} \mcQ'(y) \, =\,  \mcQ'(1-y). \nonumber
\end{align}
Let $R>0$ be a constant, $L:=\log T$, $X:=T^{\theta}$, $ a  := 1/2-R/L$, and $\mu$ be the M\"obius function. The Conrey--Levinson mollifier refers to 
\begin{align}\label{zetamoll}
  M(s) \, = \,  M(s,\mcP) \, := \, \sum_{h\le X}\, \frac{\mu(h)}{h^{s+1/2-a}} \mcP\Big( \frac{\log X/h}{\log X}\Big).
\end{align}
 According to \cite[eq. (39)]{Con89}, if $\theta>0$ is a constant such that the asymptotic formula
\begin{align}\label{Conmollif}
    \frac{1}{T} \int_{0}^{T} \, \Big|\mcQ\Big(-\frac{1}{L} \frac{d}{ds}\Big)(\zeta M)(a+it) \Big|^2 \, dt \, \sim \,  c_{\theta}(\mcP,\mcQ,R) \hspace{20pt} (T \, \to\, \infty)
\end{align}
holds, and if $c_{\theta}(\mcP,\mcQ, R)>0$, then the proportion of zeros of $\zeta(s)$ on $\sigma=1/2$ \emph{is at least}
\begin{align}\label{kappaprop}
    1- \ \frac{\log c_{\theta}(\mcP,\mcQ,R)}{R}.
\end{align}
 By \cite[Theorem 2]{Con89} (see also \cite{Yo10}), the constant $c_{\theta}(\mcP,\mcQ,R)$ is determined to be
\begin{align}\label{Conreyformconst}
 1+\frac1 \theta \int_0^1\int_0^1 \big(w(y)\mcP'(x)+\theta w'(y) \mcP(x)\big)^2 ~dx ~dy  \hspace{20pt} (w(y) \, := \, e^{Ry}\mcQ(y)).
\end{align}
In the literature, it is customary to write $\kappa(\theta)$ as the \emph{supremum} of (\ref{kappaprop}) over $\mcP, \mcQ, R$.

The appendix of \cite{CIS13} explains how Levinson's method may be implemented for a very general class of $L$-functions and extended to families of $L$-functions. For our applications, let $\Pi\in \mathcal{A}_{0}(\mathrm{G}_{r})$, $L := \log TQ$, $X:=(TQ)^{\theta}$, and $a=1/2-R/L$. Our mollifier is chosen as
\begin{align}\label{ourexamollif}
   M(s, \, \Pi\times \chi)  \, := \ \sum_{\substack{h\le X}}\frac{\mu_{\Pi}(h)\chi(h)}{h^{s+1/2-a}} \mcP\Big( \frac{\log X/h}{\log X}\Big), 
\end{align}
where the coefficients $\mu_{\Pi}$ are defined by
\begin{align}
    \frac{1}{L(s, \Pi)} \ = \  \sum_{h=1}^{\infty} \, \frac{\mu_{\Pi}(h)}{h^s}.\nonumber
\end{align}
In particular, for $\Pi\in \mathcal{A}_{0}(\mathrm{G}_{3})$, the values of $\mu_{\Pi}$ on prime powers (with $p\nmid \mfq_{\Pi}$) are given by: 
\begin{align}\label{Ex: gl3mollif}
    \mu_{\Pi}(1) =  1, \hspace{5pt}  \mu_{\Pi}(p) \, = \,   -\lambda_{\Pi}(p), \hspace{5pt} \mu_{\Pi}(p^2)\, = \, \overline{\lambda_{\Pi}}(p), \hspace{5pt} \mu_{\Pi}(p^3) \, = \,  -1, \hspace{5pt} \mu_{\Pi}(p^m) \, = \, 0    \hspace{5pt} (m\ge  4).
\end{align}
The linear combination of derivatives of $L$-functions takes the form
\begin{align}
    \mcQ\Big(-\frac{1}{rL} \frac{d}{ds}\Big)L(a+it, \Pi \times \chi).\nonumber
\end{align}

Denote by $N(T;\, \Pi\times \chi)$ the number of zeros $\rho=\beta+i\gamma$ of $L(s, \Pi \times \chi)$ with $0\le \beta\le 1$ and $|\gamma|\le T$ (counted with multiplicities), and $N_{0}(T;\, \Pi\times \chi)$ be the number of such zeros with $\beta=1/2$. We now describe a more precise and general form of Corollary \ref{maincor: pergl3}, which is an application of our mean value theorems described in Section \ref{sect: MVTtheoemre}.

\begin{cor}\label{cor: fullresultpgl3per}
    Let $\Pi \in \mathcal{A}_{0}(\mathrm{G}_{r})$ for $r\in \{2,3\}$. If $r=3$, we also assume that $\Pi$ is self-dual. Let $\kappa(\cdot)$ be the constant defined via (\ref{kappaprop}). Let $\varpi\in (0,1)$ if $r=2$, and $\varpi\in (0,1/3)$ if $r=3$. Define
    \begin{align}
        \varkappa_{r}(\varpi) \ := \  \begin{cases}
        \kappa\big(\frac{1-\varpi}{2(1+\varpi)}\big) \hspace{15pt} \text{if} \hspace{10pt}  r\, = \, 2, \\[5pt]
            \kappa\big(\frac{1-3\varpi}{6(1-\varpi)}\big)  \hspace{15pt} \text{if} \hspace{10pt}  r\, = \, 3.
        \end{cases}
    \end{align} 
    The following holds \emph{unconditionally} for $Q$ sufficiently large: 
    \begin{align}\label{avergproporfami}
    \hspace{5pt} \sum_{\substack{Q<q<2Q\\\substack{(q, \, \mfq_{\Pi})=1}}} \     \frac{2}{\phi^{*}(q)} \  \sumflat{\chi}{q} N_{0}(Q^{\varpi};\,  \Pi \times \chi) \, > \,   \varkappa_{r}(\varpi)  \sum_{\substack{Q<q<2Q\\\substack{(q, \, \mfq_{\Pi})=1}}} \ \frac{2}{\phi^{*}(q)} \  \sumflat{\chi}{q} \, N(Q^{\varpi};\,  \Pi \times \chi).
    \end{align}
     Under Hypothesis $(\mathbf{\Pi}^4)$, the same result holds for all $\Pi \in \mathcal{A}_{0}(\mathrm{G}_{3})$.
\end{cor}

In \cite{ConShortMollif}, we showed that $\kappa(\theta)>0$ regardless how small $\theta$ is. More quantitatively, we have
  \begin{align}\label{ourproportionlev}
        \varkappa_{r}(\varpi) \ > \  \begin{cases}
        \frac{1-\varpi}{3(1+\varpi)} \hspace{15pt} \text{if} \hspace{10pt}  r\, = \, 2, \\[5pt]
            \frac{1-3\varpi}{9(1-\varpi)}  \hspace{15pt} \text{if} \hspace{10pt}  r\, = \, 3.
        \end{cases}
    \end{align}
 Taking $\varpi=\epsilon$, we obtain the proportions $1/9-O(\epsilon)$ and $1/3-O(\epsilon)$ for $r=3$ and $r=2$ described in Section \ref{sect: introzeros}.  Moreover, for $\varpi<1/3$ when $r=3$, and for $\varpi< 1$ when $r=2$, the proportion
\begin{align*}
    N_{0}(Q^{\varpi};\,  \Pi \times \chi)/N(Q^{\varpi};\,  \Pi \times \chi)
\end{align*}
is \emph{positive} on average for the family in \eqref{avergproporfami}. 

In fact, our mean value theorems also yield unconditional positive proportions of \emph{simple zeros} on the critical line for the same family of $L$-functions. 

\begin{cor}\label{cor: simplpergl2}
 Let $\Pi \in \mathcal{A}_{0}(\mathrm{G}_{2})$. As $Q\to \infty$,  at least $1/3-O(\epsilon)$ of the zeros $\rho_{\chi}=\beta_{\chi}+i\gamma_{\chi}$ with $0\le \beta_{\chi}\le 1$ and $|\gamma_{\chi}|\le Q^{\epsilon}$ of the $L$-functions $L(s, \Pi\times \chi)$ lie on the critical line and are simple, where $\chi$ ranges over the primitive Dirichlet characters $\chi\, (\bmod\, q)$ with $(q, \mfq_{\Pi})=1$ and $q\le Q$. 
 
 Moreover, for the same family of $L$-functions, a positive proportion of their zeros in the rectangle $0\le \beta\le 1$ and $|\gamma|\le Q^{1/2-\epsilon}$ lie on the critical line and are simple. 
\end{cor}

\begin{rem}\label{gl3simplezeors}
Let $\Pi\in \mathcal{A}_{0}(\mathrm{G}_{3})$. The proportion of simple zeros on the critical line for the family $L(s, \Pi\otimes \chi)$ is substantially smaller than $1/9$ in Corollary \ref{maincor: pergl3}: here the proportion is only $1/200$.    
\end{rem}

\begin{rem}
  There are several subtleties involved in the implementation of Levinson’s method for the $L$-functions of the group $\mathrm{PGL}_{3}$; see Section \ref{sect: CSLevGL3}.
\end{rem}


\subsection{The large sieves: sketch of our argument}\label{sect: large sieve}

Since Linnik’s pioneering work \cite{Lin41} in the 1940s, large sieve inequalities have become fundamental tools to exploit the \emph{orthogonality} of various families of harmonics. A particularly relevant  example for us is Gallagher’s Hybrid Large Sieve (\textbf{HLS}) \cite{Gal67, Gal70}:
\begin{align}\label{intro: HLS}
	\sum_{Q<q<2Q} \ \sideset{}{^*}{\sum}_{\chi \, (\bmod\, q)} \, \int_{0}^{T} \, \Big| \sum_{n\le N}  \, a_{n}\chi(n)n^{-it}\Big|^2 \, dt \, \ll \, (TQ^2+N)  \sum_{n\le N} \, |a_{n}|^2
\end{align}
for $Q,\, T \ge 1$,  where  $(a_{n})_{n\le N}$ is \emph{any} sequence of complex numbers. \textbf{HLS} has found important applications to primes in arithmetic progressions (notably the Bombieri--Vinogradov theorem), and to zero-density estimates for Dirichlet $L$-functions (see \cite[Chapters 10, 17, 18]{IK04}). In many applications, it provides powerful unconditional substitutes for the Generalized Riemann Hypothesis.

In our context, let us take $\Pi \in \mathcal{A}_{0}(\mathrm{G}_{3})$ and $\lambda_{h}=1$ in our object of study (\ref{exa: twistMSq}). By the functional equation, we may replace $L(1/2+it,\Pi\times\chi)$ by a Dirichlet polynomial of length $\ll (TQ)^{3/2}$. A direct application of (\ref{intro: HLS}) then gives
 \begin{align}\label{HLSheyrbdd}
 	\sum_{Q<q<2Q} \ \sideset{}{^*}{\sum}_{\chi \, (\bmod\, q)} \, \int_{0}^{T} \Big|L\Big(\frac{1}{2}+it, \Pi \times \chi\Big)\sum_{h\le (TQ)^{\theta}} \frac{\chi(h)}{h^{\frac{1}{2}+it}}\Big|^2 \,  dt  \, \ll_{\epsilon} \,  (TQ^2)^{1+\epsilon}\Bigg(1+ (TQ)^{\theta}\sqrt{\frac{T}{Q}} \, \Bigg).
 \end{align}
 The \emph{upper bound} (\ref{HLSheyrbdd}) is Lindel\"of on average when $1\le T \ll Q^{\frac{1-2\theta}{1+2\theta}}$, and in particular, $\theta<1/2$. This estimate uses only the lengths of the Dirichlet polynomials and the $\ell^{2}$ bound of $(\lambda_{\Pi}(n))$. Obtaining an asymptotic formula requires additional \emph{arithmetic} information about the $L$-functions themselves.

 The Asymptotic Large Sieve \textbf{(ALS)} of Conrey--Iwaniec--Soundararajan \cite{CIS12, CIS19} is designed to extract precisely such arithmetic input. It is a method for deriving \emph{asymptotic formulae} for moments of $L$-functions (or related objects) through adroit use of the large sieve. In our case, we apply \textbf{HLS} \emph{twice}, the orthogonality of characters \emph{three times}, together with a judicious switching of divisors. We illustrate these ideas with a sketch, intended not as a fully rigorous argument but as a conceptual road map for the proof. Various technicalities, e.g., auxiliary summations, coprimality, weight functions, etc., are suppressed, and the equalities below should be taken formally. 

Let $\Pi\in \mathcal{A}_{0}(\mathrm{G}_{3})$. Given $1\le h, k\le (TQ)^{\theta}$, $T\ge Q^{\epsilon}$, and $Q\gg 1$.  By the functional equation,
\begin{align}\label{DPafe}
    \int_{t\sim T} \, |L(1/2+it, \Pi\times \chi)|^2(\frac{k}{h})^{it} \, dt \ \sim \  T \mathop{\sum  \sum}_{\substack{mn\ll (TQ)^3 \\ mh \asymp nk}} \, \frac{\lambda_{\Pi}(m)\lambda_{\widetilde{\Pi}}(n)\chi(m)\overline{\chi}(n)}{\sqrt{mn}}.
\end{align}
We then make the first use of the orthogonality relation:
\begin{align}
\sideset{}{^*}{\sum}_{\chi\, (q)} \, \chi(mh)\overline{\chi(nk)} \, &= \ \, \mathbf{1}_{(mnhk,q)=1}\sum_{\substack{cd=q\\mh\equiv  nk \, (d)}}\, \mu(c)\phi(d).
\end{align}
It follows that
\begin{align}\label{sketch: objmoment}
  \sum_{q\sim Q} \ \sideset{}{^*}{\sum}_{\chi \,(q)} \, 
    \int_{t\sim T} \,  |L(1/2+it, \Pi\times \chi)|^2(\frac{k}{h})^{it} \ dt\sim \ T  \,  \mathop{\sum \sum}_{cd\sim Q} \,   \mu(c) \phi(d)  \mathop{\sum\sum}_{\substack{mn \ll (TQ)^3 \\ mh\asymp nk \\ mh \, \equiv \,  nk \, (d)  }  } \, \frac{\lambda_{\Pi}(m)\lambda_{\widetilde{\Pi}}(n) }{\sqrt{mn}}. 
\end{align}

Let $C\in [1,Q]$ be a parameter to be optimized; it plays a key role in obtaining the strong mean value theorems described in Section \ref{sect: MVTtheoemre}. The mean square (\ref{sketch: objmoment}) is split into three sums:
\begin{align}
  \mathcal{L}(h,k) \, + \, \mathcal{D}(h,k) \, + \, \mathcal{U}(h,k),
\end{align}
according to (i). $c>C$; (ii). $c\le C$ with $mh=nk$; and (iii). $c\le C$  with $mh\neq nk$. Clearly, the sum $\mathcal{D}(h,k)$ contributes to the diagonal of (\ref{sketch: objmoment})

For the sum $\mathcal{L}(h,k)$ (see \textbf{Section \ref{LsumHLSbdd}} for details),  observe that $d\ll Q/C$, and we make the second use of orthogonality modulo $d$ :
\begin{align}
      \mathbf 1_{\substack{m h \equiv   n k \, (d) \\ (mnhk, d)=1}} \, &= \,   \frac{1}{\phi(d)}\ 
  \sum_{\psi\, (d)} \, \psi(m h) \overline{\psi}( n k).
\end{align}
We write $ \mathcal{L}(h,k) =  \mathcal{L}^{(0)}(h,k) \, + \, \mathcal{L}^{(r)}(h,k)$ according to  $\psi= \psi_{0} \, (d)$ and  $\psi\neq \psi_{0} \, (d)$. Here we focus on the second piece. Let $(\lambda_{h})$ be a sequence of complex numbers. Then
\begin{align}\label{sketch: LsumdualHLS}
   \Big|\, \mathop{\sum \sum}_{h, k\le (TQ)^{\theta}} \, \frac{\lambda_{h}\overline{\lambda_{k}}}{\sqrt{hk}}\mathcal{L}^{(r)}(h,k)\, \Big| \ll \sum_{c>C} \, \sum_{d\ll \frac{Q}{c}} \, \sum_{\psi \neq \psi_{0} \,(d)} \, \int_{0}^{T} \, \Big|L\Big(\frac{1}{2}+it, \Pi \times \psi\Big)\Big|^2  \, \bigg| \sum_{h\le (TQ)^{\theta}} \ \frac{\lambda_{h}\psi(h)}{h^{\frac{1}{2}+it}}\bigg|^2 \, dt.
\end{align}

When $\theta=0$ (and $h=k=1$) as in \cite{CIS12}, we have a second moment of $L$-functions without twist. One applies \textbf{HLS} with the size of the family being $\ll T(Q/c)^2$ and the length of the Dirichlet polynomials being $\ll (TQ/c)^{3/2}$. A small enough $C$ gives the bound $\ll TQ^2/C$ for (\ref{sketch: LsumdualHLS}). 

In \cite{CIS19}, one encounters the case $\theta>0$ and $\Pi=1$. There, Cauchy’s inequality is applied  to (\ref{sketch: LsumdualHLS}), reducing the problem to bounding two fourth moments of Dirichlet polynomials for which \textbf{HLS} yields sharp upper bounds. In our case, however, the same approach no longer suffices for $\Pi\in\mathcal{A}_{0}(\mathrm{G}_{3})$: one is led to a fourth moment that is comparable to a twelfth moment of Dirichlet $L$-functions, for which no sharp upper bound is known. We therefore take a slightly different route, observing that the right-hand side of (\ref{sketch: LsumdualHLS}) can be written in the form
\begin{align}\label{sketch: rewrite}
    \sum_{c>C} \, \sum_{d\ll Q/C} \, \sum_{\psi \neq \psi_{0} (d)} \, \int_{0}^{T} \,  \Big| \sum_{n\le N } \ \frac{b_{n}\psi(n)}{n^{1/2+it}}\Big|^2 \, dt
\end{align}
for some $N\ge 1$ and coefficients $(b_{n})$. By \textbf{HLS}, the expression (\ref{sketch: rewrite}) is bounded by:
\begin{align}\label{sketch: HLSonce}
     \frac{TQ^2}{C}  \sum_{n\le N} \, \frac{|b_{n}|^2}{n}.
\end{align}
One would expect that the $\ell^2$-type bounds suffice, which yields the bound $(TQ^2)^{1+\epsilon}/C$ for (\ref{sketch: LsumdualHLS}). 

However, a more careful thought reveals that
\begin{align}\label{sketch: gcdarise}
    \sum_{n\le N} \, \frac{|b_{n}|^2}{n} \ \approx \  \mathop{\sum \sum}_{m,h \,\le \, N} \, \frac{|\lambda_{\Pi}(m)|^2 |\lambda_{h}|^2}{\sqrt{mh}} \, \frac{(m,h)}{\sqrt{mh}}.
\end{align}
Were the bound $\lambda_{h}\ll_{\epsilon} h^{\epsilon}$ available, the sum (\ref{sketch: gcdarise}) could be treated easily and unconditionally by an application of M\"obius inversion. Without this bound, and given that the coefficients $(\lambda_h)$ are \emph{arbitrary} (and, in particular, need not be multiplicative), we require an additional bilinear estimate, which asserts that the spectral norm of the GCD matrix $((m,n)/\sqrt{mn})_{1\le m, n\le N}$ is tiny:
\begin{align}\label{sketch: Galsum}
   \Big| \,\mathop{\sum \sum}_{m,n\le N} \, a_{m}b_{n} \frac{(m,n)}{\sqrt{mn}} \, \Big| \, \ll_{\epsilon} \, N^{\epsilon}\, \Big( \sum_{n\le N} \, |a_{n}|^2\Big)^{1/2} \Big( \sum_{n\le N} \, |b_{n}|^2\Big)^{1/2}.  
\end{align}
This elementary-looking estimate was first proved by Dyer--Harman \cite{DH86}, and is in fact quite deep. Its sharp form has been a subject of interest due to its applications to large values of $\zeta(s)$, Diophantine approximation, additive combinatorics, and harmonic analysis (see \cite{LR17, BS17, BW20}). Returning to our case, from (\ref{sketch: HLSonce}), (\ref{sketch: gcdarise}), and (\ref{sketch: Galsum}), we obtain the bound
\begin{align}\label{sketch: Lsumfinalbdd}
     \frac{(TQ^2)^{1+\epsilon}}{C} \, \Big(\sum_{m\le N}\, \frac{|\lambda_{\Pi}(m)|^4 }{m}\Big)^{1/2}\Big(\sum_{h\le N}\, \frac{|\lambda_{h}|^4 }{h}\Big)^{1/2}
\end{align}
for (\ref{sketch: LsumdualHLS}). This explains the need to impose \textbf{Hypothesis $(\mathbf{\Pi}^4)$} and \textbf{Condition $(\mathbf{\Lambda)}$}. This also gives the simplest illustration of the role of GCD sums in our argument; they also enter, less visibly but essentially, at several other points in this article.

We turn to the sum\, $\mathcal{U}(h,k)$. This time, as $c\le C$, we have $d\gg Q/C$. From $mh\neq nk$ and $mh\asymp nk$, it follows that $mh,\, nk \ll (TQ)^{3/2+\theta}$. Writing the congruence $mh\equiv nk\, (\bmod\, d)$ as $mh=nk+d\ell$, we replace the sum over $d$ by a sum over the complementary moduli $\ell$, where
\begin{align}\label{sketch: duallength}
    \ell \, \asymp \, \frac{|mh-nk|}{d} \, \ll \, \frac{C}{Q} (TQ)^{\frac{3}{2}+\theta} \, =: \, L. 
\end{align}
We are now in a position to apply the orthogonality relation for the third and final time:
\begin{align*}
    \mathbf{1}_{\substack{mh \equiv nk \ (\ell)\\ (mnhk, \ell)=1}} \, &= \,  \frac{1}{\phi(\ell)} \ \sum_{\psi\, (\ell)} \ \psi(mh)\overline{\psi}(nk).
\end{align*}
As with the sum $\mathcal{L}(h,k)$, we decompose $ \mathcal{U}(h,k)$ as the sum of $\mathcal{U}^{(0)}(h,k)$ and $\mathcal{U}^{(r)}(h,k)$, according to  $\psi= \psi_{0} \, (\bmod\, \ell)$ or not. For full details, see \textbf{Sections \ref{absDiv} and \ref{sect: dualengthcond}.}

Next, we focus on $\mathcal{U}^{(r)}(h,k)$. The exact expression from which we begin is described in \textbf{Section \ref{sepmovar}}. Very roughly speaking, we have
\begin{align}\label{sketch: afterdivswitUsum}
     \mathop{\sum \sum}_{h, k\le (TQ)^{\theta}} \, \frac{\lambda_{h}\overline{\lambda_{k}}}{\sqrt{hk}}\, \mathcal{U}^{(r)}(h,k) \, \approx \,  \frac{TQ}{L}  \sum_{\substack{c\le C}} \, \frac{\mu(c)}{c} \sum_{\substack{ \ell \ll L }} \,  \sum_{\substack{\psi\, (\ell)\\ \psi \neq \psi_{0}}} \, \mathop{\sum\sum \sum \sum}_{\substack{mh,\, nk \,\ll\, (TQ)^{3/2+\theta}\\ mh\neq nk} } \, \frac{\lambda_{\Pi}(m)\lambda_{\widetilde{\Pi}}(n) \lambda_{h}\overline{\lambda_{k}}\psi(mh) \overline{\psi}(nk)}{\sqrt{mnhk}}. 
\end{align}

A key feature of \textbf{ALS} is the \emph{avoidance} of shifted convolution sums, unlike treatments on moments in the $t$-aspect such as those for $\zeta(1/2+it)$; see \cite{HY10, BBLR20}. In our setting,  the large $q$-average of (\ref{sketch: objmoment}) allows us to \emph{remove} the condition $mh \neq nk$ at the cost of a power of $T$; the saving in the $Q$-aspect is nevertheless sufficient to yield an asymptotic formula in the hybrid range (\ref{range: gl3hybrid}) where  $T$ is not too large. See \textbf{Section \ref{sect: dropoffdiagcon}} for the details. Beyond (\ref{range: gl3hybrid}), one would require strong and uniform shifted convolution estimates (see \cite{Ng21, NSW25}), but they are currently far out of reach. 

At this point, multiplicative number theory enters the argument, which, \emph{morally speaking}, gives
\begin{align}\label{sketch: UrsumaboutoappHLS}
    \Big| \,  \mathop{\sum \sum}_{h, k\le (TQ)^{\theta}} \, \frac{\lambda_{h}\overline{\lambda_{k}}}{\sqrt{hk}}\, \mathcal{U}^{(r)}(h,k) \, \Big| \, \ll \,  \frac{Q}{L} \  \sum_{\substack{ \ell \ll L }} \  \sum_{\substack{\psi\, (\ell)\\ \psi \neq \psi_{0}}} \, \int_{0}^{T} \, \Big|L\Big(\frac{1}{2}+it, \Pi \times \psi\Big) \sum_{h\le (TQ)^{\theta}} \ \frac{\lambda_{h}\psi(h)}{h^{\frac{1}{2}+it}}\Big|^2 \, dt.
\end{align}
We will have a finer discussion on (\ref{sketch: afterdivswitUsum})--(\ref{sketch: UrsumaboutoappHLS}) in \textbf{Section \ref{sect: sketUrsum}}. For the exact form of (\ref{sketch: UrsumaboutoappHLS}), we refer the reader to \textbf{Proposition \ref{afteriniclearnUr}}. Now, a second application of \textbf{HLS} \emph{should} give:
\begin{align}\label{sketch: finalUbdd}
     \Big| \,  \mathop{\sum \sum}_{h, k\le (TQ)^{\theta}} \, \frac{\lambda_{h}\overline{\lambda_{k}}}{\sqrt{hk}}\, \mathcal{U}^{(r)}(h,k) \, \Big| \, \ll \, \frac{Q}{L} TL^2 \ = \  CT(TQ)^{\frac{3}{2}+\theta}. 
\end{align}
Comparing the bounds (\ref{sketch: finalUbdd}) and (\ref{sketch: Lsumfinalbdd}), we are led to the choice of $C$ such that
\begin{align}\label{sketch: cutC}
 TQ^2/C  \ = \  CT(TQ)^{\frac{3}{2}+\theta} \ \iff \ C= Q^{1/4-\theta/2}/T^{3/4+\theta/2}.
\end{align}
We then arrive at the desired estimate (\ref{bilerr}), and the range (\ref{range: gl3hybrid}) (as $C\ge 1$). This provides a high-level summary of \textbf{Section \ref{Sec: finauseLS}}. We point out that our use of \textbf{HLS} here differs from, and is more streamlined than, that in \cite[Proposition 1.1]{CIS19}. Indeed, their argument would require a sharp bound for the second integral moment of $L(1/2+it,\Pi)$, which is presently unavailable. Our argument above requires \textbf{Hypothesis $(\mathbf{\Pi}^{4})$} and \textbf{Condition $(\mathbf{\Lambda})$}.


\subsection{Further discussions on arithmetic: in the absence of GRC and positivity}\label{sect: sketUrsum}

The sketch in the preceding section is oversimplified, and somewhat optimistic, especially for the sum $\mathcal{U}(h,k)$. As seen above, divisor-switching is effective in reducing the size of the conductor and length of summation in $\mathcal{U}(h,k)$ (see (\ref{sketch: duallength})). What the sketch conceals, however, is the role of the arithmetic of the twists and the automorphic representation $\Pi$. This is the focus of the present section.

In fact, the relevant issues are already visible in the simpler case $\theta=0$. Let us first discuss this case, and explain in more detail how (\ref{sketch: UrsumaboutoappHLS}) is derived from (\ref{sketch: afterdivswitUsum}) following \cite[Section 8]{CIS12}. In (\ref{sketch: afterdivswitUsum}), one inserts a smooth weight to remember the sizes of $m,n$. After Mellin inversion and the removal of the condition $m\neq n$, one is led to a triple Dirichlet series
\begin{align}\label{CIS: tripleDS}
    \sum_{g=1}^{\infty} \, \mathop{\sum_{M=1}^{\infty} \,\sum_{N=1}^{\infty}}_{(M,N)=1} \, \frac{\lambda_{\Pi}(gM)\lambda_{\widetilde{\Pi}}(gN) \psi(M)\overline{\psi}(N)}{g^{1+z} M^{1/2+u_{1}}N^{1/2+u_{2}}}, 
\end{align}
which converges absolutely for $\re z, \, \re u_{i} \gg 1$ ($i=1,2$). Notice that divisor-switching introduces additional variables and entanglements of variables. This issue was not made explicit in Section \ref{sect: large sieve}. One then proceeds by an Euler product calculation, leading to
\begin{align}
    (\ref{CIS: tripleDS})  \ =  \ & L(1/2+u_{1}, \Pi\times \psi)L(1/2+u_{2}, \widetilde{\Pi}\times \overline{\psi}) \,  \sum_{g=1}^{\infty} \, \frac{\mathcal{F}(u_{1}, u_{2}; g; \Pi, \psi)}{g^{1+z}}, \label{CIS: extractEufac}
    \end{align}
where $\mathcal{F}=\prod_{p}\, \mathcal{F}_{p}$ is an arithmetic factor which, in fact, does not admit a simple formula. However, were \textbf{GRC} available, for $\re z, \, \re u_{i}>\epsilon$, one could apply leading-order approximations to the local factors $\mathcal{F}_{p}$'s to deduce the sharp bound
\begin{align}\label{RCboundeas}
    \mathcal{F}(u_{1}, u_{2}; g; \Pi, \psi) \  \ll_{\epsilon} \  g^{\epsilon},
\end{align}
and the $g$-sum of (\ref{CIS: extractEufac}) would trivially be $O_{\epsilon}(1)$. With the holomorphy of the $L$-functions in (\ref{CIS: extractEufac}), we could therefore shift the lines of integration to $\re z=\re u_{i} =2\epsilon$, and arrive at (\ref{sketch: UrsumaboutoappHLS}). 


In short, the full \textbf{GRC} is so strong that it suppresses all the arithmetic introduced by divisor-switching. Without the full \textbf{GRC}, the above argument no longer works.\footnote{ This bottleneck is not overcome merely by using higher-order asymptotics to $\mathcal{F}_{p}$'s with the available bounds toward \textbf{GRC}. The same issue appears in the simpler and more familiar problem of proving, \emph{unconditionally}, the holomorphy of $H(w, \Pi, \widetilde{\Pi})$ defined in (\ref{intro: naiveRs}) on $\re w > 1-\delta$. Even there, it requires subtle uses of the properties of the Hecke algebra of $\Pi_{p}$; see \cite{Jia25} and \cite{KR14}.  } In particular, we cannot even shift the contours to the correct positions. Furthermore, pointwise bounds for $\lambda_{\Pi}(n)$ would incur losses by significant powers of $Q$ at many parts of \cite{CIS12}---the bound of \cite{KS03} saves only an exponent $1/7$ over the trivial bound $O(n^{1/2})$. Although \textbf{ALS} has undergone considerable development over the past decade (\cite{CIS12, CL14,  chandee2023eighth, fourauthor6}), the question of whether it can be applied \emph{unconditionally} (without \textbf{GRC}) to \emph{automorphic $L$-functions} does not seem to have been addressed. The main focus has been on higher moments of Dirichlet $L$-functions for which the coefficients are the divisor function $\tau_{k}(n)$, where the sharp bound $\tau_{k}(n)\ll_{\epsilon} n^{\epsilon}$ comes for free. 

More broadly speaking, for large sieves in the automorphic setting, one repeatedly encounters the fundamental question of whether, and to what extent, they depend on progress toward \textbf{GRC}. This dependence has significant consequences for applications. Unconditional results that rely only on existing progress toward \textbf{GRC} are already of considerable interest; ideally, one would like to remove this dependence altogether and thereby obtain an optimal form of large sieve. For example, there is a long and active line of research on the large sieve for Hecke eigenvalues of automorphic representations, beginning with \cite{DK00} and continuing through the recent progress in \cite{ST19, TZ21, BTZ22, HT24, Jia25, PasTh25}, that is concerned precisely with this matter. As another example, removing such dependence is also a non-trivial problem for the spectral large sieve, which has important consequences for moments of twisted $\mathrm{GL}_{2}$ $L$-functions \cite{BM15, BFK+17}. In these works, \emph{positivity} is both crucial and intrinsic, since large sieves concern averages of \emph{absolute squares} of Dirichlet polynomials. For example, positivity is exploited efficiently in \cite{TZ21} via the inequality\footnote{ See also \cite{Li10} for a closely related theme. } 
\begin{align*}
    |\lambda_{\Pi}(n)|^2 \le  \lambda_{\Pi \otimes \widetilde{\Pi}}(n), \hspace{20pt} \text{for} \hspace{10pt} (n,\mfq_{\Pi}) \, = \, 1.
\end{align*}
One may therefore work entirely with the Rankin--Selberg $L$-function $L(s, \Pi \otimes \widetilde{\Pi})$, which naturally enjoys the desired analytic properties.

 The difficulty with \textbf{ALS} is that it \emph{does not} come with any inherent positivity, as is apparent from (\ref{CIS: tripleDS})--(\ref{CIS: extractEufac}). Yet, to apply \textbf{HLS}, one ultimately needs to \emph{recover} a form of positivity (see (\ref{sketch: UrsumaboutoappHLS})), while also circumventing the use of \textbf{GRC} when establishing the needed holomorphy (cf. (\ref{RCboundeas})). This is further complicated by the auxiliary conditions and summation variables introduced by \emph{divisor-switching} and \emph{twisting}, which are intertwined in a rather intricate way. It is necessary to analyze the arithmetic structures of these sums fine enough to make \textbf{HLS} applicable. In fact, these difficulties are already substantial when $\Pi$ comes from a Maass cusp form on $\mathrm{GL}_{2}$.

We return to our setting with twists and arbitrary coefficients. A more precise form of (\ref{sketch: afterdivswitUsum}) is:
 \begin{align}\label{arith: morerealistic}
     \mathop{\sum \sum}_{h, k\le (TQ)^{\theta}} \, \frac{\lambda_{h}\overline{\lambda_{k}}}{\sqrt{hk}}\, \mathcal{U}^{(r)}(h,k) \ \approx \  & \frac{TQ}{L} \  \sum_{\substack{c\le C}} \, \frac{\mu(c)}{c}  \sum_{\substack{ \ell \asymp L }} \  \sum_{\substack{\psi\, (\ell)\\ \psi \neq \psi_{0}}}  \iiint \  (\cdots)  \sideset{}{^*}{\sum}_{\substack{g_{1}, g_{2}, g_{3}, g_{4}\\ M, N,H,K}} \, \frac{1}{(g_{1}g_{2}g_{3}g_{4})^{1+z}}\nonumber\\
     \, & \cdot  \, \frac{\lambda_{g_{1}g_{4}H}\overline{\lambda_{g_{1}g_{3}K}}\psi(H)\overline{\psi}(K)}{\sqrt{HK}} \frac{\lambda_{\Pi}(g_{2}g_{3}M)\lambda_{\widetilde{\Pi}}(g_{2}g_{4}N) \psi(M) \overline{\psi}(N)}{ M^{1/2+u_{1}} N^{1/2+u_{2}} }, 
\end{align}
where $(\cdots)$ denotes some weight functions, and the asterisk on the summation sign indicates that several restrictions are suppressed; see \textbf{Section \ref{sepmovar}} (eqs. (\ref{dyadicaelsum}), (\ref{sepsum}), and (\ref{melsepU})).

To obtain an expression of the form (\ref{sketch: UrsumaboutoappHLS}), one must separate the sums over $M,N,H,K$ in (\ref{arith: morerealistic}). This entails arithmetic complications from \emph{twisting}, which introduce additional variables  $H, K, g_{1}, g_{3}, g_{4}$ that are absent from (\ref{CIS: tripleDS}). The separation must be carried out with great care: in applications of \textbf{HLS}, the coefficients of the Dirichlet polynomial must not ``see'' the harmonics in the family. Thus, the separation procedure must be arranged judiciously, ensuring that the new variables, factors, and conditions it introduces do not interfere the average of the family. The procedure is unfortunately rather lengthy; see \textbf{Sections \ref{sect: MobinTwist}--\ref{simplfffying}}. In what follows, we give only an indication of what one should expect.

Among the many coprimality, the most essential ones turn out to be $(M,K)=(N,H)=1$. These are handled by M\"obius inversion, which introduces two variables $r_{1}, r_{2}$ (see (\ref{grouptripDS})--(\ref{triplecontDS}) for the precise expression). In this sketch, let us \emph{pretend} that the Dirichlet coefficients are completely multiplicative. One should \emph{morally} rearrange and break up the eight-fold summation of (\ref{arith: morerealistic}) as
 \begin{align}\label{sketch: moral}
  & \approx \ \sum_{g_{1}} \, \frac{1}{g_{1}^{1+z}}\mathop{\sum\sum}_{\substack{g_{3}, g_{4}\\ (g_{3}, g_{4})=1}}  \ \mathop{\sum_{(r_{1}, \, g_{4})=1} \, \sum_{(r_{2}, \, g_{3})=1}}_{(r_{1}, r_{2})=1}  \frac{ \lambda_{\Pi}(g_{3}r_{1})\lambda_{\widetilde{\Pi}}(g_{4}r_{2}) \mu(r_{1})\mu(r_{2}) }{  (g_{3}g_{4})^{1+z} r_{1}^{1+u_{1}} r_{2}^{1+u_{2}}} \sum_{g_{2}}\, \frac{\lambda_{\Pi}(g_{2}) \lambda_{\widetilde{\Pi}}(g_{2})}{g_{2}^{1+z}} \nonumber\\
    &\cdot \mathop{\sum_{(M, \, r_{2}g_{4}\ell)=1} \sum_{(N, \, r_{1}g_{3}\ell)=1}}_{(M, N)=1}\, \frac{\lambda_{\Pi}(M) \psi(M)}{M^{1/2+u_{1}}}  \frac{\lambda_{\widetilde{\Pi}}(N) \overline{\psi}(N)  }{N^{1/2+u_{2}}} \sum_{(H, \,g_{3})=1} \, \frac{\lambda_{g_{1}  g_{4}r_{2}H} \psi(H)}{\sqrt{H}}\sum_{(K,\, g_{4})=1}\, \frac{\overline{\lambda}_{g_{1}g_{3}r_{1}K} \overline{\psi}(K)}{\sqrt{K}}.
\end{align}

Heuristically, the $g_{2}, M,N$-sums should be $ \approx L(1+z, \Pi\otimes \widetilde{\Pi}) L(1/2+u_{1}, \Pi \times\psi)L(1/2+u_{2}, \widetilde{\Pi} \times \overline{\psi})$. In reality, however, the situation is considerably more involved: the sums over $g_{2}$, $M$, $N$ \emph{do not} decouple in the simple manner suggested above, and there is an additional arithmetic factor (as an Euler product) whose unramified local factor has $158$ terms upon simplification and admits no apparent factorization. It is an explicit computation with the Hecke algebra, and is implemented by \texttt{Mathematica}; see \textbf{Section \ref{colHectwistRS}}. Now, using only the bounds of Kim--Sarnak \cite{KS03} and Li \cite{Li10}, we establish the holomorphy of the arithmetic factor in the desired region $\re u_{i},\, \re z>\epsilon$.

We may now shift the lines of integration to $\re u_{i}=\re z= 2\epsilon$, after which we insert absolute values. Making the changes of variables $k_{1}=g_{3}r_{1}$ and $k_{2}=g_{4}r_{2}$, we obtain
\begin{align}
 \mathop{\sum \sum}_{h, k\le (TQ)^{\theta}} \, \frac{\lambda_{h}\overline{\lambda_{k}}}{\sqrt{hk}}\, \mathcal{U}^{(r)}(h,k) \, \ll \,   &\frac{CTQ}{L} \    \sum_{\substack{ \ell \ll L }} \  \sum_{\substack{\psi\, (\ell)\\ \psi \neq \psi_{0}}}  \iiint \  |L(1/2+u_{1}, \Pi\times \psi)| | L(1/2+u_{2}, \widetilde{\Pi}\times \overline{\psi})| \nonumber\\
& \cdot \, \bigg(\, \max_{\substack{g_{1}, g_{3}, g_{4}, \\ k_{1}, k_{2}\le (TQ)^{\theta}}}  \, \Big| \sum_{(H, \,g_{3})=1} \, \frac{\lambda_{g_{1}k_{2}H} \psi(H)}{\sqrt{H}}\Big| \Big|\sum_{(K,\, g_{4})=1}\, \frac{\overline{\lambda}_{g_{1}k_{1}K} \overline{\psi}(K)}{\sqrt{K}}\Big|\, \bigg)   \nonumber\\[5pt]
    & \cdot  \  |L(1+z, \Pi\otimes \widetilde{\Pi})| \mathop{\sum\sum\sum}_{g_{1}, k_{1}, k_{2}\le (TQ)^{\theta}} \, \frac{1}{g_{1}}\frac{|\lambda_{\Pi}(k_{1})|\tau(k_{1})}{k_{1}} \frac{|\lambda_{\widetilde{\Pi}}(k_{2})|\tau(k_{2})}{k_{2}} .
\end{align}
The sums over $k_{1}$, $k_{2}$ and the Rankin--Selberg $L$-function $L(1+z, \Pi\otimes \widetilde{\Pi})$ can be treated by the results of \cite{Li10}. Then by Cauchy's inequality, we arrive at the desired (\ref{sketch: UrsumaboutoappHLS}). 

Unfortunately, the treatment of the $g_{3}, g_{4}, r_{1}, r_{2}$-sums sketched above is, once again, not quite accurate. What is not explained above is that the sums over $M$ (resp. $N$) and $g_{2}$ in (\ref{arith: morerealistic}) actually enter into the formation of a \emph{Dirichlet convolution} of $\lambda_{\Pi}$ (resp. $\lambda_{\widetilde{\Pi}}$) with a rather complicated multiplicative function, which is computed by \texttt{Mathematica} in \textbf{Section \ref{colHec1}}. Nevertheless, another point of this (oversimplified) sketch is that the auxiliary sums over $g_{3}, g_{4}, r_{1}, r_{2}$ \emph{cannot} be treated trivially, unlike in earlier works such as \cite{CIS19}.

The arithmetic of $\Pi$ also enters through $\mathcal{U}^{(0)}(h,k)$, the contribution of $\psi=\psi_{0}$ to $\mathcal{U}(h,k)$; see \textbf{Section \ref{sect: NoODU2}}. Very briefly, divisor-switching produces a somewhat exotic Euler product (see \textbf{Lemma \ref{lem: U0anapar}}), and one must analyze a non-standard convolution involving this Euler product and the Dirichlet coefficients of $\Pi$ as a Dirichlet series; see (\ref{keydoubDS}). We must show, in the absence of \textbf{GRC}, that this convolution admits a holomorphic continuation, and that the contribution of $\mathcal{U}^{(0)}(h,k)$ is negligible (\textbf{Proposition \ref{tersm}}). This entails another substantial computation with \texttt{Mathematica} (Proposition \ref{lem: ACKit}), but fortunately the argument remains completely unconditional.

Putting all components together (see \textbf{Section \ref{sect: DLUsumdef}} for a summary), our estimates for the sums $\mathcal{L}(h,k)$ and $\mathcal{U}(h,k)$ are independent of bounds towards $\textbf{GRC}$.


\subsection{Odds and ends}
Our work follows the framework of \textbf{ALS} developed in \cite{CIS19, CIS12}, which works directly with $L$-functions and is robust in handling the arithmetic of the coefficients and twists. It is also well-suited to treating the Eisenstein case of Theorem \ref{thm: gl3, dualcase}, where off-diagonal main terms are present. We will address this case in a separate paper. By contrast, the earlier version of \textbf{ALS} developed in \cite{CIS12-zeros, CIS11preprint} was designed to handle moments of \emph{Dirichlet polynomials}, for which the off-diagonal analysis requires less arithmetic treatment for the coefficients. One example is the application to small gaps between zeros \cite{CIS12-zeros}, where the Dirichlet polynomial is a sharp truncation of $L(s, \Pi\times \chi)^{-1}$, and the off-diagonal estimates follow directly from the GRH assumption.  

Suppose that $T=Q^{\epsilon}$ and $\theta=0$. The argument of \cite{CIS12} leads to the choice $C=Q^{1/10}$ (from  $Q^2/C=C^{3/2}Q^{7/4}$). The technical point concerning Mellin integrals in \cite[Section 2]{chandee2023eighth} is also relevant to us; see \textbf{Sections \ref{sect: attachsmwe} and \ref{sect: eleintrans}}. Together with the many considerations discussed in Sections \ref{sect: large sieve}--\ref{sect: sketUrsum}, we obtain an improved choice of $C$ described in (\ref{sketch: cutC}).


\subsection{Outline and reading suggestions}
Sections \ref{sect: notconvts}--\ref{sect: collectprelim} collect the necessary notation, tools, and background for this article. It is included to keep the exposition self-contained and may be consulted as needed. The same applies to
Sections \ref{sect: offdiagsepsmo} and \ref{sect: eleintrans}, which collect the more technical analytic lemmata; on a first reading, only the statements of these lemmata are needed.

Readers may begin with Section \ref{sect: DLUsumdef}, which defines the main sums to be estimated and gives a road map of the paper.  Section \ref{sect: diagterm} elaborates on the shapes of the main terms given in Section \ref{sect: MVTtheoemre}. Section \ref{absDiv} describes a useful and general form of divisor-switching.  We summarize the essential materials of the following sections:
\begin{itemize}
    \item Section \ref{LsumHLSbdd}: Proposition \ref{LsuHLSter};

    \item Section \ref{U0sumsect}: Lemma \ref{lem: U0anapar} and Proposition \ref{prop: auxcanU1L0};

    \item Section \ref{compdoubsumGCD}: Propositions \ref{U2dropeq}--\ref{L0ddropterms};

    \item Section \ref{sect: NoODU2}: Proposition \ref{lem: ACKit} (\texttt{Mathematica} notebook: \texttt{U2pnmid.nb}; Appendix \ref{sect:U2Lfuncmatchgl3}) and Proposition \ref{tersm};

    \item Section \ref{sect: UrHLS}: Lemma \ref{Urtrunc} and Proposition \ref{Umindneg}. The discussion begins with (\ref{sepsum});

    \item Section \ref{sect: boundUrsum}: Our work in Sections \ref{sect: MobinTwist}--\ref{sect: contshclean} can be summarized by Lemmata \ref{sum: MobcoprimEupr} and \ref{Trianclean}. Sections  \ref{colHec1}--\ref{colHectwistRS} can be checked using our \texttt{Mathematica} notebook \texttt{UrFinExact.nb, UrInf.nb} (Appendix \ref{sect: UrHeckeExact} \& \ref{sect: UnrInfExact}). Section \ref{Sec: finauseLS} describes the choices of parameters when applying \textbf{HLS}, and finishes the proof of our mean value theorems in Section \ref{sect: MVTtheoemre}.

    \item Section \ref{sect: CSLevGL3}: Proposition  \ref{compratioprop}, which can be checked with our notebook \texttt{GL3Ratio.nb} (Appendix \ref{sect: ratiosMath}). This section also includes the proof of Corollaries \ref{maincor: pergl3}, \ref{cor: fullresultpgl3per} and \ref{cor: simplpergl2}.
\end{itemize}
All of the \texttt{Mathematica} notebooks are available at \url{https://github.com/Chung-Hang-Kwan/GL3-Asymptotic-Large-Sieve}.


\subsection{Acknowledgments}

We sincerely thank David Farmer, Henryk Iwaniec, Matt Young for their support, encouragement and comments throughout this project. This work was initiated at the AIM workshop \textit{Delta Symbols and the Subconvexity Problem}, held online during the pandemic from 2-6 November 2020. It is a pleasure to thank the AIM staff and organizers for their excellent support.


\section{Notation and conventions}\label{sect: notconvts}

In this article, $\epsilon>0$ is an arbitrarily small quantity, and its usage differs from context to context. We adopt the ``$\epsilon$-convention'', e.g., $X^{\epsilon} X^{\epsilon} \ll X^{\epsilon}$ for $\epsilon$ small, $X^{100}X^{-A}\ll X^{-A}$ for $A$ large.   We use the standard asymptotic notation, e.g.,  $\ll$, $\gg$, $O$, $\asymp$, and the following notation:
\begin{itemize}
    \item Let $\mathrm{G}_{\mathfrak{r}}:= \hbox{PGL}(\mathfrak{r})$ \footnote{ From now on, we switch from ``$r$'' to ``$\mfr$'' for the rank of the automorphic representations, as $r$ will be commonly used as summation variables. } and  $\mathcal{A}_{0}(\mathrm{G}_{\mathfrak{r}})$ is the set of isomorphism classes of irreducible cuspidal automorphic representations in $L^{2}(\mathrm{G}_{\mathfrak{r}}(\Q)\setminus \mathrm{G}_{\mathfrak{r}}(\A_{\Q}))$.


   \item  Let $\Pi_{i} \in \mathcal{A}_{0}(\mathrm{G}_{\mfr})$, $i=1,2$, such that their arithmetic conductors satisfy $\mfq_{\Pi_{1}}=\mfq_{\Pi_{2}}=\mfq$. Denote by $\{\mu_{j}(\Pi_{i})\}_{j=1}^{\mfr}$ the spectral parameters of $\Pi_{i}$. Let $\mfC(\Pi_{i})$ be the analytic conductor defined in (\ref{IS: analyticond}), and let
   \begin{align}
       \mathfrak{C} \, := \, (\mfC(\Pi_{1})\mfC(\Pi_{2}))^{1/4} \hspace{20pt} \text{and} \hspace{20pt}  \mathfrak{C}_{\infty} \, := \, (\mfC_{\infty}(\Pi_{1})\mfC_{\infty}(\Pi_{2}))^{1/4}.
   \end{align}

   \item Let $\omega$ be the trivial character $(\bmod\, \mfq)$. 

   \item We write $\mathcal{A}\preceq \mathcal{B}$ if $ \mathcal{A} \,  \ll_{\epsilon} \,  (TQ \mathfrak{C})^{\epsilon} \mathcal{B}.$
    
    \item A quantity $\mathcal{A}$ is said to be 
   \begin{enumerate}
       \item  ``arbitrarily small'' if $\mathcal{A}\preceq \mathfrak{C}(TQ)^{-A}$ for any $A>0$;

       \item ``negligibly small'' if  $\mathcal{A}\preceq \mathfrak{C}(TQ)^{1+\theta}$;

       \item ``of acceptable size'' if $\mathcal{A}\preceq \mathfrak{C}^{B}(TQ)^{7/4+\theta/2}$ for $\mfr=3$, and  $\mathcal{A}\preceq \mathfrak{C}^{B}(TQ)^{3/2+\theta/2}$ for $\mfr=2$, for some absolute constant $B$.
   \end{enumerate}

    \item Denote by $\vartheta=\vartheta_{\mfr}$ the approximation towards the Generalized Ramanujan/Selberg Conjecture for all $\Pi\in \mathcal{A}_{0}(\mathrm{G}_{\mfr})$. 

    \item Denote by $\lambda_{\Pi}(n)$ the Dirichlet coefficients of $L(s, \Pi)$.

    \item  Let $\lambda_{h}$ be the coefficients of the Dirichlet polynomial. They are complex numbers supported on $h\le (TQ)^{\theta}$.


   \item We shall occasionally use the shorthand $(\psi_{1}, \psi_{2}):=(\psi, \overline{\psi})$. 

   \item The sums $\textstyle\sum^+_{\chi \, (\bmod\,q)}$,    $\textstyle\sum^\flat_{\chi \, (\bmod\,q)}$,   and $\textstyle\sum^*_{\chi \, (\bmod\,q)}$ are taken over even characters, even primitive characters, and all primitive characters modulo $q$, respectively.

    \item Let $q\ge 1$.  We write $ \mathcal{A}_{q}  :=   \prod_{p\mid q} \mathcal{A}_{p}$ and $\mathcal{A}^{(q)}  :=  \prod_{p\nmid q}   \mathcal{A}_{p}$.  So, $\mathcal{A}=\mathcal{A}_{q}\mathcal{A}^{(q)} = \prod_{p} \mathcal{A}_{p}$.

    \item  We write $a\mid b^{\infty}$ if $a\mid b^n$ for some $n\gg 1$. Equivalently, every prime factor of $a$ divides $b$.

      \item  We denote by $o_p(h)$ the order of $h \in \mathbb{N}$ at the prime $p$.

      \item We write $\sum_{S}^d$ for $\sum_{S=2^k, \, k\in \Z}$, and $s\sim S$ for $S\le |s|\le 2S$. For $B\in [1,\infty)$ and $B' \in [1, \infty]$,  we note the bounds:
    \begin{align}\label{geodyad}
        \sideset{}{^d}{\sum}_{1\le S\le B} \, S^b \, \ll_{\epsilon} \, 
            B^{b+\epsilon}  \hspace{10pt}  \text{if} \hspace{10pt}   b \, \ge \, 0;
        \hspace{30pt} 
         \sideset{}{^d}{\sum}_{B\le S\le B'} \, \frac{1}{S^b} \, \asymp \, \frac{1}{B^{b}} \hspace{15pt} \text{if} \hspace{10pt}   b \, > \, 0. 
    \end{align}

    \item Let $f*g$ be the Dirichlet convolution of two multiplicative functions $f$ and $g$. 

    \item Denote by $\mathbf{1}_{A}$ the indicator function of the set or condition $A$. 

\end{itemize}
A list of the variables used in this article is provided in \textbf{Appendix \ref{sect: gloss}}.



\section{Preliminaries}\label{sect: collectprelim}

In this section, we collect the necessary analytic and automorphic preliminaries. We begin with the elementary bounds: for any $a\ge 1$, we have
\begin{align}\label{elefinitedivbdd}
    \sum_{p\mid a} \, p^{-\sigma} \, \ll \,  (\log a)^{1-\sigma}, \hspace{15pt} \prod_{p\mid a} \, c  \, \ll_{\epsilon} \, a^{\epsilon}, \hspace{15pt} \sum_{m\mid a^{\infty}} m^{-c} \, \ll_{\epsilon} \,  a^{\epsilon},
\end{align}
 where $\sigma\in [0,1]$ and $c>0$ are absolute constants. The following is elementary analysis.

\begin{lem}\label{absunifin}
Let $\mathcal{S}$ be a set of primes and $F(s) := \prod_{p\in \mathcal{S}} \, (1+f_{p}(s))$ for $\sigma\ge \sigma_{0}$.  Suppose there is a sequence of real numbers $(M_{p})_{p\in \mathcal{S}}$ such that $\sum_{p\in \mathcal{S}} \, M_{p}<\infty$, and $|f_{p}(s)| \le M_{p}$ for any $p\in \mathcal{S}$ and $\sigma \ge \sigma_{0}$.   Then $F(s)$ converges absolutely pointwise and uniformly on $\sigma\ge \sigma_{0}$. In addition, if $M_{p}\in (0,1)$ for any $p\in \mathcal{S}$, then $|F(s)|\asymp 1$ on $\sigma\ge \sigma_{0}$.
\end{lem}

Gallagher's Hybrid Large Sieve (\cite{Gal70}) is our main tool to handle the off-diagonal contribution. Here and throughout the paper, let $||a||_{2}:= (\sum_{n} |a_{n}|^2)^{1/2}$ for the sequence $a=(a_{n})$.

\begin{lem}\label{Gallager}
For any $Q, T \ge 1$ and any sequence $(a_{n})_{n\le N}$ of complex numbers, we have 
\begin{align}
	\sum_{Q<q<2Q} \ \sideset{}{^*}{\sum}_{\chi \, (\bmod\, q)} \, \int_{0}^{T} \, \Bigg| \sum_{n\le N}  \, a_{n}\chi(n)n^{-it}\Bigg|^2 \, dt \, \ll \, (TQ^2+N)  ||a||_{2}^2. 
\end{align}
\end{lem}

\begin{lem}\label{lem: simpleGCD}
    \begin{align}
        \sum_{n\le N} \, \frac{(a,n)}{n} \, \ll_{\epsilon} \, (Na)^{\epsilon}. 
    \end{align}
\end{lem}

\begin{proof}
    Follows directly from $ (a,n) =  \sum_{d\mid a, \, d\mid n} \, \phi(d). $
\end{proof}

The following bound is a nontrivial result of \cite{DH86}; it does not follow from elementary arguments.

\begin{lem}\label{lem: Galsum}
Given any two sequences $(a_{n})_{n\le N}$ and $(b_{n})_{n\le N}$ of complex numbers, we have
\begin{align}\label{Galsum}
   \Bigg| \, \mathop{\sum \sum}_{m,n\le N} \, a_{m}b_{n} \frac{(m,n)}{\sqrt{mn}} \, \Bigg| \, \ll_{\epsilon} \, N^{\epsilon}\, ||a||_{2} \, ||b||_{2}.  
\end{align}
\end{lem}

\subsection{Generalities on Euler products}

We recall some useful facts for obtaining Euler product expansions, which will be useful in Sections \ref{sect: U2sumloc} and 
\ref{sect: TDSandEP}. The relevant identities hold whenever either side converges absolutely.

\begin{lem}\label{genEPexp}
For any prime $p$, suppose $f_{p}: \N\rightarrow \C$ is a function such that $f_{p}(n)  =   f_{p}(p^{o_{p}(n)})$ for $n\ge 1$. Then we have
\begin{align}
	\sum_{n=1}^{\infty} \,  \prod_{p} f_{p}(n)  =   \prod_{p}  \sum_{n=0}^{\infty}  f_{p}(p^n). 
\end{align}
\end{lem}

This admits an obvious multivariable generalization. Also, when $f_{1}$ and $f_{2}$ are multiplicative functions, we have, for any positive integers  $g_{1}, g_{2}$,  we have
\begin{align}\label{simplefixmul}
     \sum_{n=1}^{\infty} f_{1}(g_{1}n)f_{2}(g_{2}n) = \prod_{p} \sum_{n=0}^{\infty} f_{1}( p^{n+o_{p}(g_{1})})f_{2}(p^{n+o_{p}(g_{2})}). 
\end{align}
For any function $f:\mathbb{P}\rightarrow \C$, we have
\begin{align}\label{mudivEP}
    \sum_{d=1}^{\infty} \, \mu(d) \prod_{p\mid d} \, f(p) \, = \, \prod_{p} \, (1-f(p)). 
\end{align}

\begin{lem}\label{lem: coprimEPdoub}
    Suppose $f_{i}$, $h_{i}$ are multiplicative functions ($i=1,2$). For any integers $a_{1}, a_{2}, g_{1}, g_{2}\ge 1$, we have
\begin{align}\label{coprimEPdoub}
  \mathop{\sum_{(m_{1}, a_{1})=1} \, \sum_{(m_{2},a_{2})=1}}_{(m_{1}, m_{2})=1} \, & f_{1}(g_{1}m_{1})h_{1}(m_{1}) f_{2}(g_{2}m_{2}) h_{2}(m_{2}) \, =  \, \prod_{\substack{p\,\mid\, a_{1} \\ p\,\nmid\, a_{2}}}  (F_{2})_{p}(g_{2})  \prod_{\substack{p\,\mid\, a_{2} \\ p\,\nmid\, a_{1}}}  (F_{1})_{p}(g_{1}) \prod_{p\,\nmid \,a_{1}a_{2}}    (F_{3})_{p}(g_{1}, g_{2}),
\end{align}
where $(F_{i})_{p}(g_{i}) := \sum_{m_{i}=0}^{\infty} \, f_{i}(p^{m_{i}+o_{p}(g_{i})})h_{i}(p^{m_{i}})$ for $i=1,2$, and
\begin{align*}
    (F_{3})_{p}(g_{1}, g_{2}) \, := \, f_{1}(p^{o_{p}(g_{1})})(F_{2})_{p}(g_{2})+f_{2}(p^{o_{p}(g_{2})})(F_{1})_{p}(g_{1})-f_{1}(p^{o_{p}(g_{1})})f_{2}(p^{o_{p}(g_{2})}). 
\end{align*}
\end{lem}

\begin{proof}
By M\"obius inversion and (\ref{simplefixmul}), we have
\begin{align}
     \mathcal{T} &:= \mathop{\sum_{(m_{1}, a_{1})=1} \, \sum_{(m_{2}, a_{2})=1}}_{(m_{1}, m_{2})=1} \, f_{1}(g_{1}m_{1}) h_{1}(m_{1}) f_{2}(g_{2}m_{2}) h_{2}(m_{2}) = \sum_{(d,\, a_{1}a_{2})=1} \, \mu(d) \, \prod_{i=1}^{2} \, \sum_{(m_{i}, a_{i})=1} f_{i}(dg_{i}m_{i}) h_{i}(dm_{i}) \nonumber\\
     &= \sum_{(d, \,a_{1}a_{2})=1}\, \mu(d) \, \prod_{i=1}^{2} \, \prod_{p\, \nmid\, a_{i}} \, \sum_{m_{i}=0}^{\infty} \, f_{i}(p^{m_{i}+o_{p}(g_{i})+o_{p}(d)}) h_{i}(p^{m_{i}+o_{p}(d)}).
\end{align}
The $d$-sum is supported on square-free integers. Thus, $o_{p}(d)\in \{0,1\}$ for any prime $p$. Also,  since $(d, a_{1}a_{2})=1$, we have $o_{p}(d)=1$ $\implies$ $p\nmid a_{1}a_{2}$. Thus, 
\begin{align}
     \mathcal{T} = \sum_{(d, \,a_{1}a_{2})=1}\, \mu(d) \, \prod_{i=1}^{2} \, \Big(\prod_{p\, \nmid\, da_{i}} \,  (F_{i})_{p}(g_{i})\Big) \prod_{p\mid d} \, \big((F_{i})_{p}(g_{i})- f_{i}(p^{o_{p}(g_{i})})\big).
\end{align}

Note that for any quantities $L_{p}$'s,  we have $ \prod_{p\,\nmid\, da_{i}} L_{p} \, = \,  \prod_{p\, \nmid \, a_{i}} L_{p} \cdot  \prod_{p\mid d} L_{p}^{-1}.$ By (\ref{mudivEP}), we have
\begin{align}
     \mathcal{T} &= \,  \Big(\prod_{i=1}^{2} \, \prod_{p \,\nmid \,  a_{i}} (F_{i})_{p}(g_{i}) \Big)\sum_{(d, \,a_{1}a_{2})=1}\, \mu(d) \,   \prod_{p\mid d} \, \prod_{i=1}^{2} \, \frac{(F_{i})_{p}(g_{i})- f_{i}(p^{o_{p}(g_{i})})}{(F_{i})_{p}(g_{i})}\nonumber\\
     \, &= \, \Big(\prod_{i=1}^{2} \, \prod_{p \,\nmid \,  a_{i}}  (F_{i})_{p}(g_{i}) \Big)   \prod_{p \, \nmid \,  a_{1}a_{2}}  \bigg\{ 1 \, - \,  \prod_{i=1}^{2} \, \frac{(F_{i})_{p}(g_{i})- f_{i}(p^{o_{p}(g_{i})})}{(F_{i})_{p}(g_{i})} \bigg\} \nonumber\\
     \, &= \,  \prod_{\substack{p\,\mid\, a_{1} \\ p\,\nmid\, a_{2}}} \, (F_{2})_{p}(g_{2})  \prod_{\substack{p\,\mid\, a_{2} \\ p\,\nmid\, a_{1}}} \, (F_{1})_{p}(g_{1}) \prod_{p\,\nmid \,a_{1}a_{2}} \, \Big(\prod_{i=1}^{2} \, (F_{i})_{p}(g_{i}) - \prod_{i=1}^{2} \, ((F_{i})_{p}(g_{i})- f_{i}(p^{o_{p}(g_{i})})) \Big).\nonumber
\end{align}
\end{proof}


\subsection{Hecke combinatorics}\label{Hprelim}
Let  $\mathrm{G}_{\mathfrak{r}}:= \hbox{PGL}(\mathfrak{r})$, $\Pi = \otimes_{p\le \infty}' \, \Pi_{p} \in \mathcal{A}_{0}(\mathrm{G}_{\mfr})$, and $\mfq =\mfq_{\Pi}$ be the arithmetic conductor. Let $p$ be any prime. In \cite{Jac79}, Jacquet defines the local $L$-factor $L(u, \Pi_{p})$ of $\Pi_{p}$, and shows that its inverse is a polynomial in $p^{-u}$ of degree at most $\mfr$. In other words,
\begin{align}\label{defineHecLfac}
   L(u, \Pi_{p})^{-1} \ = \  \prod\nolimits_{1\le i\le \mfr} \ (1-\alpha_{\Pi,\, i}(p)p^{-u}) \hspace{15pt} (u \,\in \,   \C)
\end{align}
for some $\alpha_{\Pi, \, i}(p) \in \C$ ($1\le i\le \mfr$). It is also customary to write $  L(u, \Pi_{p}) =L_{p}(u, \Pi)$. Let $\widetilde{\Pi}$ be the contragredient of $\Pi$, and $\omega$ the trivial character $(\bmod\, \mfq)$. We have $\mfq_{\widetilde{\Pi}}=\mfq$, and
\begin{align}
\prod\nolimits_{1\le i\le \mfr} \, \alpha_{\Pi, \, i}(p) \, &= \, \omega(p) \hspace{69pt}  \text{for any prime } p,\label{trivcen}\\
    \hspace{20pt}  \{\alpha_{\widetilde{\Pi}, \,i}(p)\}_{1\le i\le \mfr} \  &= \  \{\,\overline{\alpha_{\Pi,\, i}(p)}\, \}_{1\le i\le \mfr} \hspace{15pt}  \text{for any prime } p, \label{Satacontra} \\
    \{ \alpha_{\Pi, i}(p)^{-1}\}_{1\le i \le \mfr} \, &= \, \{\, \overline{\alpha_{\Pi, i}(p)}\, \}_{1 \le i\le \mfr} \hspace{18pt}  \text{for any } \ p\nmid \mfq. \label{eqn: unrecip}
\end{align}

The global $L$-function of $\Pi$ is defined, for $\re u\gg 1$, by
\begin{align}
     L(u, \Pi) \ := \  \prod\nolimits_{p} \, L_{p}(u, \Pi).
\end{align}
We define a multiplicative function $\lambda_{\Pi}(m)$ by specifying its values on prime powers through 
\begin{align}
    \sum\nolimits_{r\ge 0} \, \lambda_{\Pi}(p^r)p^{-ru} \ = \  L_{p}(u, \Pi).
\end{align}
That is, $\lambda_{\Pi}(m)$ are the Dirichlet coefficients for  $L(u, \Pi)$:
\begin{align}
    L(u, \Pi) \ = \ \sum\nolimits_{m\ge 1} \, \lambda_{\Pi}(m)m^{-u} \hspace{20pt} (\re\, u \, \gg 1).
\end{align}
By definition, we have
\begin{align}\label{taylorsatake}
    \lambda_{\Pi}(p^n)  \, = \, \sum\nolimits_{\substack{k_{1}, \ldots, k_{\mfr}\ge 0 \\ k_{1}+\cdots +k_{\mfr}=n}} \, \prod\nolimits_{1\le i \le \mfr} \, \alpha_{\Pi,\, i}(p)^{k_{i}},
\end{align}
which coincides with the Schur polynomial $\mathrm{S}_{(0,\ldots,0, n)}(\underline{\alpha})$, where $\underline{\alpha}:=(\alpha_{\Pi, \, i}(p))_{1\le i \le \mfr}$ is as in (\ref{defineHecLfac}). 


In this work, we are primarily interested in the cases $\Pi\in \mathcal{A}_{0}(\mathrm{G}_{2})$ and $\mathcal{A}_{0}(\mathrm{G}_{3})$, for which the Schur polynomials are explicitly given as follows: \footnote{ We have suppressed the dependence on $\Pi$ and $p$ in $\alpha_{i}$'s for ease of typesetting.  }
\begin{align}\label{schurgl2}
    \mathrm{S}_{(n)}(\underline{\alpha}) \, = \,  \frac{\alpha_{1}^{n+1}-\alpha_{2}^{n+1}}{\alpha_{1}-\alpha_{2}}; 
\end{align}
\begin{align}\label{schurgl3}
\mathrm{S}_{(n_{1}, n_{2})}(\underline{\alpha})  =   \scalebox{1.2}{ $ \frac{\alpha_{1}^{n_{1}+n_{2}+2}(\alpha_{2}^{n_{1}+1}-\alpha_{3}^{n_{1}+1})+ \alpha_{2}^{n_{1}+n_{2}+2} (\alpha_{3}^{n_{1}+1}-\alpha_{1}^{n_{1}+1})+ \alpha_{3}^{n_{1}+n_{2}+2} (\alpha_{1}^{n_{1}+1}-\alpha_{2}^{n_{1}+1})}{\prod_{1\le i<j\le 3} \, (\alpha_{i}-\alpha_{j})}$, }
\end{align}
where $n, n_{1},n_{2}\in \mathbb{N} \cup \{0\}$. See \cite[p. 233]{Go15} for a more general treatment. We define another function $\lambda_{\Pi}(m_{1},\ldots, m_{\mfr-1})$ multiplicative in all variables such that
\begin{align}\label{def: Schur}
    \lambda_{\Pi}(p^{n_{1}},\ldots, p^{n_{\mfr-1}}) \, = \, \mathrm{S}_{(n_{1}, \, \ldots, \,  n_{\mfr-1})}(\underline{\alpha}),
\end{align}
for any prime $p$ and integers $n_{i}\in \N\cup \{0\}$. Then for any $m\in \N$, we have 
\begin{align}
    \lambda_{\Pi}(m) \, = \, \lambda_{\Pi}(1,\ldots, 1, m).
\end{align}

Denote by $\vartheta=\vartheta_{\mfr}$ the constant such that, for any $\Pi\in \mathcal{A}_{0}(\mathrm{G}_{\mfr})$ and prime $p$, the following holds:
\begin{align}\label{satabdd}
    |\alpha_{\Pi,\, i}(p)| \, \le \, p^{\vartheta} \hspace{20pt} (1 \, \le \,  i \, \le \,  \mfr). 
\end{align}
From \cite{LRS95}, it is known that $\vartheta_{\mfr}<1/2$ for any $\mfr\ge 2$, and from \cite{KS03}, we have
\begin{align}
    \vartheta_{2} \, \le \, 7/64  \hspace{20pt} \text{and} \hspace{20pt}  \vartheta_{3} \, \le \, 5/14 \, = \, 1/2-1/7. \label{KimSarbdd}
 \end{align}
It follows, for any $c\ge 1$ and $\re u> \vartheta$, that $c^{-\epsilon}  \ll_{\epsilon}  L_{c}(u, \Pi)^{-1}  := \prod_{p\mid c} \, L_{p}(u, \Pi)^{-1}  \ll_{\epsilon}  c^{\epsilon}.$ Moreover,
\begin{align}\label{uniformRC}
     |\lambda_{\Pi}(p^n)|  \  \le \   \Big(\sum\nolimits_{\substack{k_{1}, \ldots, k_{\mfr}\ge 0 \\ k_{1}+\cdots +k_{\mfr}=n}}  1\Big) p^{n\vartheta}.
\end{align}



Let $\Pi\in \mathcal{A}_{0}(\mathrm{G}_{3})$. From (\ref{taylorsatake}) we have
\begin{align}
      \mathfrak{e}_{\Pi}(p) \, &:= \,  \sum\nolimits_{1\le i<j\le 3 } \, \alpha_{\Pi,\, i}(p)\alpha_{\Pi,\, j}(p)   \, = \, \lambda_{\Pi}(p)^2 \, - \, \lambda_{\Pi}(p^2). \label{GL3secelem}
\end{align}
It follows from (\ref{trivcen}), (\ref{eqn: unrecip}) and (\ref{Satacontra}) that for any  $p\nmid \mfq$, it holds that 
\begin{align}
    \mathfrak{e}_{\Pi}(p)  \, &= \,\sum\nolimits_{1\le j\le 3}\, \alpha_{\Pi,\,j}(p)^{-1} \, = \, \sum\nolimits_{1\le j\le 3} \, \overline{\alpha_{\Pi,\, j}(p)} \, = \, \overline{\lambda_{\Pi}(p)} \, = \, \lambda_{\widetilde{\Pi}}(p).  \label{unrcontrae(p)}
\end{align}
From this, (\ref{trivcen}) and (\ref{satabdd}), the following bounds can readily be observed:
\begin{align}\label{epiboundgl3}
|\mathfrak{e}_{\Pi}(p)| \, \le \, 3p^{\vartheta_{3}} \hspace{10pt} \text{if} \hspace{15pt}   p\nmid \mfq \hspace{15pt} \text{and} \hspace{15pt}  |\mathfrak{e}_{\Pi}(p)| \, \le \, p^{2\vartheta_{3}} \hspace{10pt} \, \text{if}   \hspace{15pt}   p\mid \mfq.
\end{align}
Let  $\Pi \in \mathcal{A}_{0}(\mathrm{G}_{3})$ and $n_{1}, n_{2}\in \N \cup\{0\}$. By (\ref{taylorsatake}) and the Jacobi--Trudi identity, we have
\begin{align}\label{JacobiTrud}
      \lambda_{\Pi}(p^{n_{1}}, p^{n_{2}}) \ = \  \lambda_{\Pi}(p^{n_{1}+n_{2}})\lambda_{\Pi}(p^{n_{1}})- \lambda_{\Pi}(p^{n_{1}+n_{2}+1})\lambda_{\Pi}(p^{n_{1}-1})
\end{align}
for any prime $p$. If we restrict to $p\nmid \mfq$, then from \cite[Theorem 6.4.11]{Go15}, we have
\begin{align}\label{FourtoHeckgl3}
    \lambda_{\Pi}(p^{n_{1}}, p^{n_{2}}) \ = \ \lambda_{\widetilde{\Pi}}(p^{n_{1}})\lambda_{\Pi}(p^{n_{2}}) \, - \, \lambda_{\widetilde{\Pi}}(p^{n_{1}-1})\lambda_{\Pi}(p^{n_{2}-1}),
\end{align}
and 
\begin{align}
    \lambda_{\widetilde{\Pi}}(p^{n_{1}},p^{n_{2}}) \ = \ \overline{\lambda_{\Pi}}(p^{n_{1}},p^{n_{2}}) \, = \, \lambda_{\Pi}(p^{n_{2}}, p^{n_{1}}). 
\end{align}
It follows that
\begin{align}
   \hspace{10pt} |\lambda_{\Pi}(p^{n_{1}}, p^{n_{2}})|  \ll_{\epsilon} (p^{n_{1}+n_{2}})^{\vartheta_{3}+\epsilon} \hspace{15pt} \text{if} \hspace{15pt} p \, \nmid \mfq; \hspace{15pt}   |\lambda_{\Pi}(p^{n_{1}}, p^{n_{2}})|  \ll_{\epsilon}  (p^{2n_{1}+n_{2}})^{\vartheta_{3}+\epsilon} \hspace{15pt} \text{if}  \hspace{15pt} p \mid \mfq.
\end{align}

For any prime $p$, it follows from (\ref{taylorsatake}) the following Hecke recurrence relations:
\begin{align}
&\hspace{-10pt} \lambda_{\Pi}(p^{n+1}) \, =\,  \lambda_{\Pi}(p)\lambda_{\Pi}(p^n) - \omega(p)\lambda_{\Pi}(p^{n-1}) \hspace{70pt} (n\ge 1), \hspace{10pt}\text{if} \hspace{10pt}  \Pi\in \mathcal{A}_{0}(\mathrm{G}_{2}), \label{Hecrecur} \\
  \lambda_{\Pi}(p^{n+1}) &=\,   \lambda_{\Pi}(p)\lambda_{\Pi}(p^n) - \mathfrak{e}_{\Pi}(p)\lambda_{\Pi}(p^{n-1}) + \omega(p) \lambda_{\Pi}(p^{n-2}) \hspace{22pt} (n\ge 2), \hspace{10pt}\text{if} \hspace{10pt}  \Pi\in \mathcal{A}_{0}(\mathrm{G}_{3}). \label{gl3: Hecrecurrel}
\end{align}
From these, we may deduce that
\begin{lem}\label{Heckasymp}
For $m\ge 1$, $\re u\gg 1$, and $X:=p^{-u}$, the polynomial 
\begin{align}
    P_{m}(X) \ := \  L_{p}(u, \Pi)^{-1}\sum_{r=0}^{\infty} \, \lambda_{\Pi}(p^{m+r})p^{-ru}.
\end{align}
can be evaluated as
\begin{align}
   \begin{cases}
          \hspace{75pt} \lambda_{\Pi}(p^m) - \omega(p) \lambda_{\Pi}(p^{m-1})X, \hspace{108pt}   \Pi \, \in \, \mathcal{A}_{0}(\mathrm{G}_{2}); \\[5pt]  
          \lambda_{\Pi}(p^m) + (\lambda_{\Pi}(p^{m+1})-\lambda_{\Pi}(p)\lambda_{\Pi}(p^m))X +\omega(p)\lambda_{\Pi}(p^{m-1})X^2, \hspace{16pt}   \Pi \, \in \, \mathcal{A}_{0}(\mathrm{G}_{3}).
      \end{cases} \nonumber
\end{align}
\end{lem}

\begin{proof}
Suppose $\Pi\in \mathcal{A}_{0}(\mathrm{G}_{2})$. From (\ref{Hecrecur}), we have $ P_{m}(X)=  \lambda P_{m-1}(X) - \omega P_{m-2}(X)$ for $m\ge 2$, where $\lambda:= \lambda_{\Pi}(p)$ and $\omega:=\omega(p)$. Clearly, $P_{0}(X)=1$ and $ P_{1}(X)= \lambda-\omega X$. The result follows from
\begin{align*}
 \sum_{m=0}^{\infty} P_{m}(X) t^m \, = \,   \frac{1-\omega Xt}{1-\lambda t+ \omega t^2}  =    \sum_{m=0}^{\infty} \, \lambda_{\Pi}(p^m) t^m   -  \omega X  \sum_{m=0}^{\infty} \, \lambda_{\Pi}(p^m) t^{m+1}.
\end{align*}

Suppose $\Pi\in \mathcal{A}_{0}(\mathrm{G}_{3})$. A similar argument \footnote{ or an easy \texttt{Mathematica} computation using (\ref{schurgl3}). } shows that $P_{m}(X)$ is a quadratic polynomial with the coefficients for $X^{0}$, $X^{1}$, $X^2$ being, respectively, $\mathrm{S}_{(0,m)}(\underline{\alpha})$, $-\lambda \cdot \mathrm{S}_{(0,m)}(\underline{\alpha})+ \mathrm{S}_{(0,m+1)}(\underline{\alpha})$, $\mathrm{S}_{(0,m-1)}(\underline{\alpha})$. The desired result follows with $\alpha_{1}+\alpha_{2}+\alpha_{3}= \lambda$ and $\alpha_{1}\alpha_{2}\alpha_{3}=\omega$.
\end{proof}

\begin{rem}\label{conv: Hecke}
The following sets of conventions are convenient and are recorded in the \texttt{Mathematica} package used in this article: for $\Pi \in  \mathcal{A}_{0}(\mathrm{G}_{2})$, we set
\begin{align*}
    \lambda_{\Pi}(p^{-1}) \, := \, 0, \hspace{15pt}  \lambda_{\Pi}(p^{-2}) \, := \, -1, \hspace{15pt}\lambda_{\Pi}(p^{-3}) \, := \, -\lambda_{\Pi}(p);
\end{align*}
 and  for $\Pi \in \mathcal{A}_{0}(\mathrm{G}_{3})$, we set
 \begin{align*}
     &\lambda_{\Pi}(p^{-1}) \ = \ \lambda_{\Pi}(p^{-2}) \  =  \ \lambda_{\Pi}(p, p^{-1}) \ = \  \lambda_{\Pi}(p^2, p^{-1}) \ := \ 0, \nonumber\\ 
    -&\lambda_{\Pi}(p, p^{-2})  \ = \ \lambda_{\Pi}(p^{-3}) \ := \ 1, \hspace{20pt}\lambda_{\Pi}(p^2, p^{-2}) \ = \ -\mathfrak{e}_{\Pi}(p). 
 \end{align*}

\end{rem}


 \subsection{Twisting}\label{sect: twistdata}

Let $\chi \, (\bmod\, q)$ be an even primitive Dirichlet character and $\Pi \in \mathcal{A}_{0}(\mathrm{G}_{\mfr})$. In this article, we assume that $(q, \mfq)=1$. Then \emph{for any} prime $p$, we have
\begin{align}\label{twistEulerfacallp}
    L_{p}(u, \, \Pi\times \chi) \, = \,   \prod_{1\le i \le \mfr} \, (1-\alpha_{\Pi, \,i}(p)\chi(p) p^{-u})^{-1}.
\end{align}
As a result, the local factor $ L_{p}(u, \,\Pi\times \chi)^{-1}$ can be computed as
\begin{align}
 \hspace{-13pt}1- \lambda_{\Pi}(p)\chi(p)p^{-u} + \omega(p)\chi^{2}(p)p^{-2u}\hspace{62pt}\text{if} \hspace{7pt}  \Pi\in \mathcal{A}_{0}(\mathrm{G}_{2}), \label{GL2EP}\\
 1- \lambda_{\Pi}(p)\chi(p)p^{-u} + \mathfrak{e}_{\Pi}(p)\chi^{2}(p)p^{-2u} - \omega(p)\chi^{3}(p)p^{-3u} \hspace{12pt}\text{if} \hspace{8pt}  \Pi \in \mathcal{A}_{0}(\mathrm{G}_{3}).  \label{GL3EP}
\end{align}
On $\re u > 1+\vartheta$, the (global) standard $L$-functions of $\Pi$ and its twist by $\chi$ are defined by 
\begin{align}\label{stdL: seriesdef}
    L(u, \Pi) \, := \, \prod_{p} \, L_{p}(u, \Pi) \hspace{15pt} \text{and} \hspace{15pt}  L(u, \, \Pi \times \chi) \, := \, \prod_{p} \,   L_{p}(u, \, \Pi \times \chi)
\end{align}
respectively. From (\ref{taylorsatake}) and (\ref{twistEulerfacallp}), they can be expressed as Dirichlet series of the form:
\begin{align}
   L(u, \, \Pi \times \chi) \ = \ \sum_{m=1}^{\infty} \, \frac{\lambda_{\Pi}(m)\chi(m)}{m^u}. 
\end{align}

Let $\mu(\Pi):=\{\mu_{i}(\Pi)\}_{1\le i\le \mfr}$ be the spectral parameters, $\mfq=\mfq_{\Pi}$ the arithmetic conductor of $\Pi$, and $\mfC(\Pi)$ the analytic conductor (see  (\ref{IS: analyticond})) of $\Pi$. We have $\sum_{1\le i\le \mfr} \, \mu_{i}(\Pi)=0$, and  $|\mu_{i}(\Pi)|\le \vartheta_{\mfr}$ for any $1\le i\le \mfr$, with $\vartheta_{\mfr}$ given by (\ref{KimSarbdd}). By \cite[eqs. (21)--(25) \& Proposition 4.1]{BR94},  we have
\begin{align}\label{twistcond}
    \mfq_{\Pi\times \chi} \, = \, \mfq q^\mfr \hspace{15pt} \text{and} \hspace{15pt} \mu_{i}(\Pi\times \chi)  = \mu_{i}(\Pi).
\end{align}
The analytic conductor for  $ L(s, \Pi\times \chi)$ is given by 
\begin{align}\label{analycondtwis}
    \mathbf{c}(s, \Pi\times\chi) \, = \, \mfq_{\Pi\times\chi} \, \mathbf{c}_{\infty}(s, \Pi\times \chi) \, := \, \mfq_{\Pi\times\chi}\cdot  \prod_{i=1}^{\mfr} \, (1+|s+\mu_{i}(\Pi)|) \, \le \mathbf{c}(\Pi) (q(1+|s|))^\mfr.
\end{align}
Let $\Gamma_{\R}(s):= \pi^{-s/2}\Gamma(s/2)$, and $ L_{\infty}(s,\Pi):= \prod_{i=1}^{\mfr}\Gamma_{\R}(s+\mu_{i}(\Pi)).$ The functional equation for $L(s, \Pi)$ was stated in (\ref{jacstdfunc}), and that for its twists is given by
\begin{align}\label{twisfuncJS}
  \Lambda(s, \Pi \times \chi) \, := \, (\mfq q^\mfr)^{s/2}L_{\infty}(s, \Pi)L(s, \Pi\times\chi) \, = \,  \epsilon(\Pi\times \chi)  \Lambda(1-s, \,\widetilde{\Pi} \times \overline{\chi}).
\end{align}
The  root numbers satisfy the relation
\begin{align}\label{eq: rootnum}
    \epsilon(\Pi\times \chi) \,  = \,  \epsilon(\Pi)\epsilon(\chi)^{\mfr}.
\end{align}


\subsection{Rankin--Selberg convolution}\label{RSprelim}

Let $\Pi_{1}, \Pi_{2} \in \mathcal{A}_{0}(\mathrm{G}_{\mfr})$ with $\mfq_{\Pi_{1}}=\mfq_{\Pi_{2}}=\mfq$. The local Rankin--Selberg factor at $p\nmid \mathfrak{q}$ is defined as 
\begin{align}\label{trueRSconv}
  L_{p}(s, \Pi_{1}\otimes \Pi_{2})  \, &:=  \, \prod_{1\le i,j\le \mfr} \  (1- \alpha_{\Pi_{1},\, i}(p)\alpha_{\Pi_{2}, \,j}(p)p^{-s})^{-1}. 
\end{align}
For any prime $p$, we write $L_{p}^{\mathrm{ur}}(s, \Pi_{1}\otimes \Pi_{2})$ for the right-hand side of (\ref{trueRSconv}), since when $p\mid \mfq$, this needs not coincide with the Rankin--Selberg local factor defined in \cite{JPSS_RS}. In general,  the factor $L_{p}(s, \Pi_{1}\otimes \Pi_{2})^{-1}$ of \cite{JPSS_RS} is a polynomial in $p^{-s}$ of degree at most $\mfr^2$, i.e., there are  $\alpha_{\Pi_{1}\otimes \Pi_{2}}(p; i,j)\in \C$ such that
\begin{align}\label{formalRS}
    L_{p}(s, \Pi_{1}\otimes \Pi_{2}) \,  = \, \prod_{1\le i,j\le \mfr} \, (1-\alpha_{\Pi_{1}\otimes \Pi_{2}}(p; i,j)p^{-s})^{-1}.
\end{align}
It is shown in \cite{JPSS_RS} that the global Rankin--Selberg $L$-function 
\begin{align}
    L(s, \Pi_{1}\otimes \Pi_{2}) \ := \  \prod_{p} \,  L_{p}(s, \Pi_{1}\otimes \Pi_{2}) \hspace{20pt} (\sigma \, \gg \, 1)
\end{align}
admits a meromorphic continuation to $\C$ and satisfies a functional equation. From these, Jacquet and Shalika (\cite[Theorem 5.3]{JS81}) proved, via Landau's lemma (see \cite[Lemma 5.56]{IK04}), that the Euler product $ L^{(\mfq)}(s, \Pi_{1}\otimes \Pi_{2})$ is $O_{\sigma,\, \Pi_{1},\, \Pi_{2}}(1)$ for $\sigma> 1$.

 It is less obvious (see \cite[(2.6) \& Appendix]{ST19}) that \emph{for any} prime $p$ and $1\le i, j\le \mfr$, we have
\begin{align}\label{genRSbdd}
    |\alpha_{\Pi_{1}\otimes \Pi_{2}}(p; i,j)| \, \le \,  p^{2\vartheta}.
\end{align}
Thus, for any prime $p$, the local factor of the ``naive'' convolution of $\Pi_{1}$ and $\Pi_{2}$ at $p$, i.e., 
\begin{align}\label{naivelocalRSfac}
    \mathcal{L}_{p}(s, \Pi_{1}\otimes \Pi_{2}) \, := \, \sum_{n=0}^{\infty}\,  \frac{\lambda_{\Pi_{1}}(p^{n})\lambda_{\Pi_{2}}(p^n)}{p^{ns}},
\end{align}
converges absolutely for $\sigma> 2\vartheta$, and by Lemma \ref{uniformRC}, we have
\begin{align}\label{naivgl2exp}
   \mathcal{L}_{p}(s, \Pi_{1}\otimes \Pi_{2}) \ = \  1 \ + \ O_{\sigma}(p^{-(\sigma-2\vartheta)}).
\end{align}
Moreover, let $n_{2}=1$ and $n_{3}=3$. For $\sigma\ge 1$, it follows from (\ref{KimSarbdd}) that
\begin{align}
     \mathcal{L}_{p}(s, \Pi_{1}\otimes \Pi_{2})  \ = \  \sum_{n=0}^{n_{\mfr}} \, \frac{\lambda_{\Pi_{1}}(p^{n})\lambda_{\Pi_{2}}(p^n)}{p^{ns}} \, + \, O(p^{-1-\epsilon}). \label{rawnaiveexpand}
\end{align}

We set
\begin{align}\label{naRScomp}
  H_{p}(s, \Pi_{1}, \Pi_{2}) \, := \,   \mathcal{L}_{p}(s, \Pi_{1}\otimes \Pi_{2}) L_{p}(s, \Pi_{1}\otimes \Pi_{2})^{-1} \hspace{20pt} (\sigma> 2\vartheta).  
\end{align}
On the region of absolute convergence, we define
\begin{align}\label{naiveRSser}
    \mathcal{L}(s, \Pi_{1}\otimes \Pi_{2}) \ :=  \   \prod_{p} \,  \mathcal{L}_{p}(s, \Pi_{1}\otimes \Pi_{2}) \hspace{15pt} \text{and} \hspace{15pt}  H(s, \Pi_{1}, \Pi_{2}) \, :=  \, \prod_{p} \, H_{p}(s, \Pi_{1}, \Pi_{2}).
\end{align}

Recall \cite[Theorem 2]{Li10}: for $\re s\ge 1$, we have
\begin{align}\label{Lisbound}
    (s-1)L(s, \Pi\otimes \widetilde{\Pi})   \ \ll_{\epsilon} \    (\mfC(\Pi)|s|)^{\epsilon}.
\end{align}
Denote by $\lambda_{\Pi\otimes \widetilde{\Pi}}(m)$ and $\Lambda_{\Pi\otimes \widetilde{\Pi}}(m)$, respectively, the Dirichlet coefficient and von Mangoldt function of the $L$-function $L(s, \Pi\otimes \widetilde{\Pi})$. From \cite[Chapter 12]{Go15}, we have
    \begin{align}\label{heckevonmag}
\lambda_{\Pi\otimes \widetilde{\Pi}}(m) \, = \,     \sum_{\substack{m_{0}, m_{1},\ldots, m_{\mfr-1}\ge 1 \\ m_{0}^{\mfr} m_{1}^{\mfr-1}\cdots m_{\mfr-1}^{1}=m}} \, |\lambda_{\Pi}(m_{1}, \ldots, m_{\mfr-1})|^2 \hspace{15pt} \text{and} \hspace{15pt}  \frac{\Lambda_{\Pi\otimes \widetilde{\Pi}}(p^k)}{\log p}  = \Big|\frac{\Lambda_{\Pi}(p^k)}{\log p}\Big|^2
\end{align}
for  $(m, \mfq) = 1$ and $p\nmid \mathfrak{q}$. In particular, we have the inequalities
 \begin{align}\label{trivrankcoefbdd}
     |\lambda_{\Pi}(m)|^2 \,   \le \, \lambda_{\Pi\otimes \widetilde{\Pi}}(m) \hspace{20pt} \text{and} \hspace{20pt} \sum_{p\nmid \mfq} \, \frac{|\lambda_{\Pi}(p)|^2}{p^\sigma} \, \le \, \log \, L(\sigma, \Pi\otimes \widetilde{\Pi}) \hspace{15pt} (\sigma > 1).
 \end{align}

\begin{lem}\label{naivconvRSDS}
    The Dirichlet series and Euler product of $\mathcal{L}(s, \Pi_{1}\otimes \Pi_{2})$ converge absolutely for $\re s> 1$. Moreover, for $\re s\ge 1+\epsilon$, we have  $\mathcal{L}(s, \Pi_{1}\otimes \Pi_{2})\ll_{\epsilon} (\mathbf{c}(\Pi_{1})\mathbf{c}(\Pi_{2}))^{\epsilon}$. 
\end{lem}

\begin{proof}
By Cauchy's inequality, it suffices to consider the case $\Pi_{1}\simeq \widetilde{\Pi_{2}}=:\Pi$.  From Lemma \ref{uniformRC} and (\ref{genRSbdd}) (with $\vartheta\le 1/2$), the products $\mathcal{L}_{\mfq}(\sigma, \Pi\otimes\widetilde{\Pi})$ and $L_{\mfq}(\sigma, \Pi\otimes \widetilde{\Pi})^{-1}$ are $O_{\epsilon}(\mfq^{\epsilon})$ for $\sigma\ge 1$. By (\ref{trivrankcoefbdd}), we have $|\mathcal{L}(\sigma, \Pi\otimes \widetilde{\Pi})|\ll_{\epsilon}  \mfq^{\epsilon} |L(\sigma, \Pi\otimes \widetilde{\Pi})|$ for $\sigma>1$. The results follow from (\ref{Lisbound}).
\end{proof}

By Lemma \ref{naivconvRSDS} and the fact that $\mathbf{1}_{n\le X}\le (X/n)^{\epsilon}$ for any $\epsilon>0$, we have

\begin{lem}[Li \cite{Li10}]\label{lem: ROA}
For any $\Pi\in \mathcal{A}_{0}(\mathrm{G}_{\mfr})$, we have
\begin{align}\label{ROA}
\sum_{n\le X} \, \frac{|\lambda_\Pi(n)|^2}{n} \, \ll_{\epsilon} \,  (\mfC(\Pi)X)^{\epsilon}.   
\end{align}
\end{lem}

\begin{lem}\label{lem: diffEUproLi}
  For any $ \Pi_{1}, \Pi_{2} \in \mathcal{A}_{0}(\mathrm{G}_{\mfr})$,   the Euler product $\prod_{p} \, (1\pm  \lambda_{\Pi_{1}}(p)\lambda_{\Pi_{2}}(p)p^{-s})$ converges absolutely on $\sigma>1$, and is $O_{\epsilon}((\mathbf{c}(\Pi_{1})\mathbf{c}(\Pi_{2}))^{\epsilon})$ on $\sigma\ge 1+\epsilon$.
\end{lem}

\begin{proof}
This is obvious from (\ref{trivrankcoefbdd}) and (\ref{Lisbound}).
\end{proof}

\begin{rem}\label{Gl2Hfuncrem}
When $\Pi_{1}, \Pi_{2}\in \mathcal{A}_{0}(\mathrm{G}_{2})$, we have $ H_{p}(s, \Pi_{1}, \Pi_{2})=  \zeta_{p}(2s)^{-1}$
for $p\nmid \mfq$, and $H(s, \Pi_{1}, \Pi_{2})$ is holomorphic and non-vanishing on $\re s \ge 1/2$.    
\end{rem}

 When $\Pi_{1}, \Pi_{2}\in \mathcal{A}_{0}(\mathrm{G}_{3})$, we have the following evaluation.

\begin{lem}\label{lem: naiveRSexpl}
Let $\Pi_{1}, \Pi_{2}\in \mathcal{A}_{0}(\mathrm{G}_{3})$.
\begin{enumerate}
\item If $p\nmid \mfq$, then 
\begin{align}
    \hspace{30pt} H_{p}(s, \Pi_{1},  \Pi_{2}) 
    \, &= \, (1-p^{-3s})^2 - p^{-2s}(\lambda_{\widetilde{\Pi_{1}}}(p)-\lambda_{\Pi_{2}}(p)p^{-s})(\lambda_{\widetilde{\Pi_{2}}}(p)-\lambda_{\Pi_{1}}(p)p^{-s});   \label{unramHp} 
\end{align}
and if in addition, $\Pi_{1}\simeq \widetilde{\Pi_{2}} \simeq \Pi$, then  (\ref{unramHp}) simplifies into 
\begin{align}\label{eq: unrHpdualto}
    H_{p}(s, \Pi, \,  \widetilde{\Pi}) 
    \, &= \, (1-p^{-3s})^2 - |\lambda_{\Pi}(p)|^2 \big(p^{-s}(1-p^{-s})\big)^2. 
\end{align}

\item If $p\mid \mfq$, then
\begin{align}\label{ramH-func}
     \hspace{30pt} H_{p}(s, \Pi_{1}, \Pi_{2}) \, = \, \frac{L_{p}^{\mathrm{ur}}(s,\, \Pi_{1} \otimes \Pi_{2})}{L_{p}(s,\, \Pi_{1} \otimes \Pi_{2})}\, ( 1  -  (\lambda_{\Pi_{1}}(p)^2 \, - \, \lambda_{\Pi_{1}}(p^2) )(\lambda_{\Pi_{2}}(p)^2 \, - \, \lambda_{\Pi_{2}}(p^2) )p^{-2s}).
\end{align}
\end{enumerate}
\end{lem}

 \begin{proof}
    See \texttt{ShiftNaRS.nb} (or Appendix \ref{sect:shiftedRS}) for the \texttt{Mathematica} implementation. When $p\nmid \mfq$, the result can be simplified by (\ref{unrcontrae(p)}) and $\overline{ab}-(|a|^2+|b|^2)x+ abx^2= (\overline{a}-bx)(\overline{b}-ax)$. 
\end{proof}

\begin{cor}\label{gl3holoprop}
Suppose $\Pi_{1}, \Pi_{2}\in \mathcal{A}_{0}(\mathrm{G}_{3})$.  Then $H(s, \Pi_{1}, \Pi_{2})$ is holomorphic on $\sigma> 1/2+\vartheta$, and $\mathcal{L}(s, \Pi_{1}\otimes \Pi_{2})$ is holomorphic on $\sigma> 1/2+\vartheta$, except for a simple pole at $s=1$ when $\Pi_{1}\simeq \widetilde{\Pi_{2}}$. 
\end{cor}

\begin{proof}
Suppose $\sigma >1/2+\vartheta$. From (\ref{unramHp}),  we have $  H_{p}(s, \Pi_{1},   \Pi_{2})   = 1+O(p^{-2(\sigma-\vartheta)})$ for $p\nmid \mfq$. By Lemma \ref{absunifin}, the infinite product $ H^{(\mfq)}(s, \Pi_{1},  \Pi_{2})$ converges absolutely and uniformly on compacta (and hence holomorphic) on $\sigma > 1/2+\vartheta$. The holomorphy for $\mathcal{L}(s, \Pi_{1}\otimes \Pi_{2})$ follows from (\ref{naRScomp}). 
\end{proof}

\begin{cor}\label{Hlocalnonvan}
 For any $\Pi\in \mathcal{A}_{0}(\mathrm{G}_{3})$, we have $H(1, \Pi, \widetilde{\Pi})\neq 0$.  For any $\Pi_{1}, \Pi_{2}\in \mathcal{A}_{0}(\mathrm{G}_{3})$ with conductor $\mfq$ divisible by $30$, on the half-plane $\sigma \ge 1-10^{-3}$, we have $ |H(s, \Pi_{1}, \Pi_{2})| \, \gg_{\epsilon} \, \mfq^{-\epsilon}$.
\end{cor}

\begin{proof}
The finite product $ H_{\mfq}(s, \, \Pi_{1}, \Pi_{2})$ is bounded above and below by $\mfq^{\pm \epsilon}$, which follows from (\ref{ramH-func}), (\ref{formalRS}) and (\ref{genRSbdd}). 

Suppose that $p\nmid \mfq$. Putting $s=1$ in (\ref{eq: unrHpdualto}), we have
\begin{align*}
     H_{p}(1, \Pi, \,  \widetilde{\Pi}) 
    \, &= \, p^{-2}(1-p^{-1})^2\prod_{\pm} \, (p+1+p^{-1}\pm |\lambda_{\Pi}(p)|).
\end{align*}
Let $\{\alpha_{1},\alpha_{2}, \alpha_{3}\}$ be the Satake parameters of $\Pi_{p}$. We may assume that $|\alpha_{1}|\ge |\alpha_{2}|\ge |\alpha_{3}|$. By (\ref{eqn: unrecip}) and (\ref{trivcen}), we have $|\alpha_{2}|=1$ and $|\alpha_{3}|^{-1}=|\alpha_{1}|\ge 1$. It follows that
\begin{align*}
    p+1+p^{-1} - |\lambda_{\Pi}(p)| \ \ge \ p + p^{-1} -|\alpha_{1}|-|\alpha_{1}|^{-1} \ \ge \ p + p^{-1} -p^{1/2}-p^{-1/2} \, > \, 0
\end{align*}
Hence, we conclude that $H(1, \Pi, \widetilde{\Pi})\neq 0$.

Suppose that $30\mid \mfq$ and $p\nmid \mfq$. We must have $p\ge 7$. Then by (\ref{uniformRC}) and $\vartheta\le 1/2-1/7$,
\begin{align}
     |H_{p}(s,\, \Pi_{1}, \Pi_{2})| \, \ge \, 1- p^{6\cdot 10^{-3}}\big( 2p^{-3} +p^{-6} + 9p^{-2(1-\vartheta)} (1+p^{-1})^2\big) \, > \, 0.\nonumber
\end{align}
The second result now follows from $  \,  |H^{(\mfq)}(s,\, \Pi_{1}, \Pi_{2})| \, \gg \, 1.$
\end{proof}

We will frequently encounter the following local factor. For any $d_{1}, d_{2}\ge 0$, we let
\begin{align}\label{dshiftnaiveDsdef}
    \mathcal{L}_{p}(s;\,  \Pi_{1},   \Pi_{2}; \, d_{1}, d_{2}) \, := \, \sum_{n= 0}^{\infty} \, \frac{\lambda_{\Pi_{1}}(p^{n+ d_{1}})\lambda_{\Pi_{2}}(p^{n+d_{2}})}{p^{ns}}. 
\end{align}
In particular, we have $ \mathcal{L}_{p}(s;\,  \Pi_{1},   \Pi_{2}; \, 0,0)=\mathcal{L}_{p}(s, \Pi_{1}\otimes \Pi_{2})$. We shall need an explicit evaluation of (\ref{dshiftnaiveDsdef}) as in Lemma \ref{lem: naiveRSexpl}. Although the formulae appear somewhat involved, it is quite easy to apply them in practice with the help of \texttt{Mathematica}.

\begin{lem}\label{lem: diffby1RSn}
Suppose $\Pi_{1}, \Pi_{2}\in \mathcal{A}_{0}(\mathrm{G}_{\mfr})$ ($\mfr\in \{2,3\}$). Let $d\in \N$, $\nu\in \{0,1\}$, $p$ prime, and $\sigma>2\vartheta$. There is an explicit polynomial $\mathcal{P}$ in $p^{-s}$ such that
\begin{align}
     \mathcal{L}_{p}(s;\,  \Pi_{1},   \Pi_{2}; \, d,\nu)  \ = \   L_{p}^{\mathrm{ur}}(s, \Pi_{1}\otimes \Pi_{2})\mathcal{P}(p^{-s}; \Pi_{1}, \Pi_{2}; d, \nu).
\end{align}
The degree of $\mathcal{P}$ is at most $2$, $3$, $6$, $8$ for $(\mfr,\nu)=$ $(2,0)$, $(2,1)$, $(3,0)$, $(3,1)$, respectively. 

Moreover, there is an explicit multiplicative function $\mathfrak{f}_{s,\Pi_{1}, \Pi_{2}}: \N\rightarrow \C$ such that 
 \begin{align}\label{convshiftRS}
          \mathcal{L}_{p}(s;\,  \Pi_{1},   \Pi_{2}; \, d,0) \ = \ L_{p}^{\mathrm{ur}}(s, \Pi_{1}\otimes \Pi_{2})(\lambda_{\Pi_{1}}*\mathfrak{f}_{s,\Pi_{1}, \Pi_{2}})(p^d).
    \end{align}
  For $\re s\ge 1$, $n=n_{0}n'$ with $(n',\mfq)=1$ and $n_{0}\mid \mfq^{\infty}$, and $p$ prime,  the following bounds hold:
    \begin{align}
        |\, \mathfrak{f}_{s,\Pi_{1}, \Pi_{2}}(n)| \ &\ll_{\epsilon} \  n^{\epsilon} \, \begin{cases}
          \hspace{10pt}  n^{-1+\vartheta} \hspace{32pt} \text{if} \hspace{5pt} \mfr \, = \, 2\\
    n^{-1+2\vartheta}n_{0}^{\vartheta} \hspace{25pt} \text{if} \hspace{5pt} \mfr \, = \, 3,
     \end{cases}\label{convshbddarith}\\[5pt]
        &\hspace{-20pt} |\mathcal{L}_{p}(s;\,  \Pi_{1},   \Pi_{2}; \, 1,1)| \ \ll \ p^{2\vartheta}.\label{powerRStartfrom1}
    \end{align}
\end{lem}

 \begin{proof}
The \texttt{Mathematica} code used to compute the polynomials $\mathcal{P}$ is included in  \texttt{ShiftNaRS.nb} (or Appendix \ref{sect:shiftedRS}). When $\Pi_{1}, \Pi_{2}\in \mathcal{A}_{0}(\mathrm{G}_{2})$, we have
    \begin{align}
\mathcal{P}(X;\, \Pi_{1}, \Pi_{2}; d, 0) \ &= \    \lambda_{\Pi_{1}}(p^{d})-X \omega(p) \lambda_{\Pi_{1}}(p^{d-1})\lambda_{\Pi_{2}}(p) + X^2\omega(p)\lambda_{\Pi_{1}}(p^{d-2}),\label{gl2naiRSshif0}\\[5pt]
\mathcal{P}(X;\, \Pi_{1}, \Pi_{2}; d, 1) \ &= \ \lambda_{\Pi_{1}}(p^d)\lambda_{\Pi_{2}}(p)  - 
 X \omega(p)\big(\lambda_{\Pi_{1}}(p^{d-1})\lambda_{\Pi_{2}}(p)^2  +\lambda_{\Pi_{1}}(p^{d+1})\big)\nonumber\\
 &\hspace{45pt} + X^2 \omega(p) (\lambda_{\Pi_{1}}(p^{d-2}) + \lambda_{\Pi_{1}}(p^{d}) )\lambda_{\Pi_{2}}(p)-X^3\omega(p)\lambda_{\Pi_{1}}(p^{d-1}). \label{gl2naivsh1DS}
\end{align}
Let $\mathfrak{e}_{\Pi_{2}}(p)$ be defined in (\ref{GL3secelem}). When $\Pi_{1}, \Pi_{2}\in \mathcal{A}_{0}(\mathrm{G}_{3})$, we have
\begin{align}
      \mathcal{P}(X;\, \Pi_{1}, \Pi_{2}; d, 0)  \ = \    &\lambda_{\Pi_{1}}(p^d) \, - \, X\,\lambda_{\Pi_{1}}(p,p^{d-1})\lambda_{\Pi_{2}}(p)     \nonumber\\
    \, &\, + \, X^2\big\{ \, \omega(p)\lambda_{\Pi_{1}}(p^{d-1})\lambda_{\Pi_{2}}(p)^2 + \lambda_{\Pi_{1}}(p^2, p^{d-2})\mathfrak{e}_{\Pi_{2}}(p)\big\} \nonumber\\
      \, &\, + \, X^3\omega(p)\big\{\lambda_{\Pi_{1}}(p, p^{d-2}) (1-\lambda_{\Pi_{2}}(p)\lambda_{\widetilde{\Pi_{2}}}(p)) - \lambda_{\Pi_{1}}(p^3, p^{d-3}) \big\} \nonumber\\
       \, &\, + \,X^4\omega(p)\big\{ \lambda_{\Pi_{1}}(p^{d-2}) \lambda_{\widetilde{\Pi_{2}}}(p)^2+ (-\lambda_{\Pi_{1}}(p^{d-2}) + \lambda_{\Pi_{1}}(p^2, p^{d-3}) \lambda_{\Pi_{2}}(p)\big)\big\} \nonumber\\
       \, &\, - \, X^5\omega(p) \lambda_{\Pi_{1}}(p, p^{d-3})\lambda_{\widetilde{\Pi_{2}}}(p) \, + \, X^6\omega(p)\lambda_{\Pi_{1}}(p^{d-3}). 
      \label{eq: diffbygendRSn}
\end{align}
The expression of the polynomial $ \mathcal{P}(X;\, \Pi_{1}, \Pi_{2}; d, 1)$ is too lengthy to reproduce here.

We write $X=p^{-s}$ and $\mathfrak{f}=\mathfrak{f}_{s,\Pi_{1}, \Pi_{2}}$ in the following. Suppose that $\Pi_{1}, \Pi_{2}\in \mathcal{A}_{0}(\mathrm{G}_{2})$.  From (\ref{gl2naiRSshif0}), observe that (\ref{convshiftRS}) holds if $\mathfrak{f}$ is defined, on prime powers, by
\begin{align}
    \mathfrak{f}(1) \ := \ 1, \hspace{10pt}  \mathfrak{f}(p) \, := \,  -X\omega(p)\lambda_{\Pi_{2}}(p), \hspace{10pt} \mathfrak{f}(p^2) \, := \, X^2\omega(p) , \hspace{10pt} \mathfrak{f}(p^j) \, = \, 0 \hspace{10pt} (j\, \ge \, 3).
\end{align}
It is clear that $|\mathfrak{f}(p^j)| \ll (p^{j})^{-1+\vartheta}$ and (\ref{convshbddarith}) follows from multiplicativity. The bound (\ref{powerRStartfrom1}) is apparent from (\ref{gl2naivsh1DS}) and (\ref{uniformRC}). 

Suppose that $\Pi_{1}, \Pi_{2}\in \mathcal{A}_{0}(\mathrm{G}_{3})$. We first treat the unramified primes $p\nmid \mfq$. By \texttt{Mathematica} (or using (\ref{eq: diffbygendRSn}) with (\ref{FourtoHeckgl3})), we have (\ref{convshiftRS}) with $\mathfrak{f}$ chosen as follows:
\begin{align}\label{dthnaiveLbound}
  \mathfrak{f}(p^j) \, := \, \begin{cases}
       \hspace{90pt} 1\hspace{215pt} \text{ if } j\, = \, 0\\
      - X \lambda_{\Pi_{2}}(p) \lambda_{\widetilde{\Pi_{1}}}(p)+ X^2 (\lambda_{\Pi_{2}}(p)^2 - \lambda_{\widetilde{\Pi_{2}}}(p) )\hspace{102pt} \text{ if } j\, = \, 1\\[5pt]
       X \lambda_{\Pi_{2}}(p)+
  X^2 \lambda_{\widetilde{\Pi_{2}}}(p) \lambda_{\widetilde{\Pi_{1}}}(p^2) - X^4 (\lambda_{\Pi_{2}}(p) -  \lambda_{\widetilde{\Pi_{2}}}(p)^2 )\\
   \hspace{20pt}+X^3 (\lambda_{\widetilde{\Pi_{1}}}(p) - |\lambda_{\Pi_{2}}(p) |^2\lambda_{\widetilde{\Pi_{1}}}(p)) \hspace{135pt}\text{ if } j\, = \, 2\\[5pt]
  - X^2 \lambda_{\widetilde{\Pi_{2}}}(p) \lambda_{\widetilde{\Pi_{1}}}(p)
  -X^3 (1-|\lambda_{\Pi_{2}}(p)|^2+\lambda_{\widetilde{\Pi_{1}}}(p^3))\\
  \hspace{20pt }+ 
 X^4 \lambda_{\Pi_{2}}(p) \lambda_{\widetilde{\Pi_{1}}}(p^2)  - X^5\lambda_{\widetilde{\Pi_{2}}}(p) \lambda_{\widetilde{\Pi_{1}}}(p)+X^6   \hspace{73pt}\text{ if } j\, = \, 3\\[5pt]
\hspace{10pt} X^3 \lambda_{\widetilde{\Pi_{1}}}(p^2)- X^4 \lambda_{\Pi_{2}}(p) \lambda_{\widetilde{\Pi_{1}}}(p) +
 X^5 \lambda_{\widetilde{\Pi_{2}}}(p) \hspace{92pt}\text{ if } j\, = \, 4\\
  \hspace{90pt} 0\hspace{218pt} \text{ if } j\, \ge \, 5. 
  \end{cases}
\end{align}
By inspection (or using our code), we have $|\mathfrak{f}(p^j)|\ll (p^j)^{-1+2\vartheta}$. Next, suppose that $p\mid \mfq$. Then  \begin{align}
     \mathcal{P}(X;\, \Pi_{1}, \Pi_{2}; d, 0)  \ = \    &\lambda_{\Pi_{1}}(p^d) \, - \, X\,\lambda_{\Pi_{1}}(p,p^{d-1})\lambda_{\Pi_{2}}(p)      + \, X^2\lambda_{\Pi_{1}}(p^2, p^{d-2})\mathfrak{e}_{\Pi_{2}}(p)
\end{align}
from (\ref{eq: diffbygendRSn}).
Applying (\ref{JacobiTrud}) followed by (\ref{gl3: Hecrecurrel}), we once again have (\ref{convshiftRS}), but this time, $|\mathfrak{f}(p^j)|\ll (p^j)^{-1+3\vartheta}$. The bound (\ref{powerRStartfrom1}) can likewise be verified directly from the code.
\end{proof}


\subsection{The approximate functional equations}\label{sect: AFE}

Let $\chi\,(\bmod\, q)$ be an \emph{even} primitive Dirichlet character. Let $\Pi_{1},\Pi_{2} \in \mathcal{A}_{0}(\mathrm{G}_{\mfr})$ such that $\mfq_{\Pi_{1}}=\mfq_{\Pi_{2}}=\mfq$ and $(q,  \mfq)=1$. We take $ G(s):=e^{s^2}p(s)$, where $p(s)$ is any even polynomial satisfying $p(0)=1$. 

From (\ref{twisfuncJS}) and \cite[Theorem 5.3]{IK04}, on $\re s =1/2 +O(\epsilon)$, we have 
\begin{align}\label{singAFE}
    |L(s, \Pi_{i}\times \chi)| \, \ll_{\epsilon} \, (TQ)^{\epsilon}\Big|\sum_{n\ge 1} \, \frac{\lambda_{\Pi_{i}}(n) \chi(n)}{n^{s}} \,  V_{s}\Big( \frac{n}{\sqrt{\mfq q^\mfr}}; \Pi_{i} \Big) \Big|,
\end{align}
where the cut-off function $V_{s}(y; \Pi_{i})$ is explicitly given by 
\begin{align}\label{def: singlecutoff}
    V_{s}(y; \Pi_{i}) \, := \, \frac{1}{2\pi i} \int_{(\epsilon)} \, G(u) \frac{L_{\infty}(s+u, \Pi_{i})}{L_{\infty}(s,\Pi_{i})} y^{-u} \, \frac{du}{u}.
\end{align}
On vertical lines with $\re s > \vartheta$ and $\re u> -\vartheta$, it follows from Stirling's formula that
\begin{align}
    \frac{L_{\infty}(s+u, \Pi_{i})}{L_{\infty}(s,\Pi_{i})} \, \ll \, \mathbf{c}_{\infty}(s, \Pi_{i})^{(\re u)/2} e^{\frac{\pi\mfr}{4}|u|} \hspace{20pt} \text{and}  \hspace{20pt} V_{s}(y; \Pi_{i}) \, \ll_{A} \, \big(1+ y/  \sqrt{\mfC_{\infty}(s, \Pi_{i})}\, \big)^{-A}
\end{align}
for any $A>0$. By Lemma \ref{naivconvRSDS}, the contribution from $n> (TQ)^{\epsilon} \sqrt{\mfC(s, \Pi_{i} \times \chi)}$ in (\ref{singAFE}) is 
\begin{align}\label{singAFEtrunctail}
    \, \ll_{A, \,\epsilon} \, (TQ)^{-A} \sum_{n=1}^{\infty} \, \frac{|\lambda_{\Pi_{i}}(n)|}{\sqrt{n}}\Big(\frac{n}{\sqrt{\mfC(s, \Pi_{i} \times \chi)}}\Big)^{-1/2-\epsilon} \,  \preceq_{A} \, \mfC(\Pi_{i})^{1/4}(TQ)^{-A}.
\end{align}

Let $\eta=\eta_{T, \Delta}$ be the test function chosen in Section \ref{sect: MVTtheoemre}. To state the following lemma, we also require the following cut-off functions:
\begin{align}
	V_{\alpha,\beta}(y; it; \, \Pi_{1}, \Pi_{2}) 
	  \, &:= \,  \int_{(\epsilon)} \, \widetilde{V}_{\alpha,\beta}(s;it;\,  \Pi_{1}, \Pi_{2}) y^{-s} \, \frac{ds}{2\pi i},\label{newcutoff}\\
	H_{\alpha, \beta}(v,y;\,  \Pi_{1}, \Pi_{2}) \ &:= \  \int_{\mathbb{R}} \, \eta(t) V_{\alpha, \beta}\left( y; it;\,  \Pi_{1}, \Pi_{2}\right) v^{it} \, dt, \label{doublecut}
\end{align}
where
\begin{align}
 \gamma_{(\alpha, \beta); \, it}(s;\, \Pi_{1}, \Pi_{2}) \, &:= \, L_{\infty}(s+\alpha+it,\, \Pi_{1})L_{\infty}(s+\beta-it, \,\Pi_{2}), \\[5pt]
 \widetilde{V}_{\alpha,\beta}(s;it;\,  \Pi_{1}, \Pi_{2}) \ &:= \ \frac{G(s)}{s} \, \frac{\gamma_{(\alpha,\, \beta); \, it}(1/2+s,  \Pi_{1}, \Pi_{2})}{\gamma_{(\alpha, \,\beta); \, it}(1/2,  \Pi_{1}, \Pi_{2})}, \label{Mellincutprod} \\
 \mathfrak{X}_{(\alpha,\beta);\, it}(\Pi_{1}, \Pi_{2}) \, &:= \, \frac{\gamma_{(-\beta, -\alpha); it}(1/2, \widetilde{\Pi_{2}}, \widetilde{\Pi_{1}})}{\gamma_{(\alpha,\beta); it}(1/2, \Pi_{1}, \Pi_{2})}.
 \end{align}

\begin{lem}\label{AFE}
Let $\Pi_{1}, \Pi_{2}\in \mathcal{A}_{0}(\mathrm{G}_{\mfr})$ such that $\mfq_{\Pi_{1}}=\mfq_{\Pi_{2}}=\mfq$, and  $\chi\, (\bmod\, q)$ be an even primitive character such that $(q, \mfq)=1$.  Then we have
\begin{align}\label{compleAFE2}
    & L(1/2+\alpha+it ,\Pi_{1} \times \chi )  L(1/2+\beta-it, \Pi_{2} \times \overline{\chi} )   \nonumber\\
    \, &\hspace{20pt}= \,  S_{(\alpha,\, \beta)}(\chi, it; \,\Pi_{1}, \Pi_{2}) + \epsilon(\Pi_{1})\epsilon(\Pi_{2}) (\mfq q^{\mfr})^{-\alpha-\beta}  \mathfrak{X}_{(\alpha,\beta);\, it}(\Pi_{1}, \Pi_{2})S_{(-\beta,\, -\alpha)}(\chi, it; \, \widetilde{\Pi_{2}}, \widetilde{\Pi_{1}}),
\end{align}
where
\begin{align}
   S_{(\alpha,\, \beta)}(\chi, it; \,\Pi_{1}, \Pi_{2}) \, &:= \,   
    \sum_{m=1}^{\infty} \, \sum_{n=1}^{\infty} \, 
    \frac{\lambda_{\Pi_{1}}(m)\lambda_{\Pi_{2}}(n)\chi(m)\overline{\chi}(n)}
        {m^{\frac12 + \alpha+it} n^{\frac12 +\beta-it}} \,   V_{\alpha,\beta}\Big(\frac{m n}{\mfq q^\mfr}; it;\,  \Pi_{1}, \Pi_{2}\Big).
\end{align}
\end{lem}

\begin{proof}
See \cite[Theorem 5.3]{IK04}. Notice that $\epsilon(\Pi_{1}\times \chi)\epsilon(\Pi_{2}\times \overline{\chi}) = \epsilon(\Pi_{1})\epsilon(\Pi_{2})$ from (\ref{eq: rootnum}).
\end{proof}

Henceforth, we focus on the piece $S_{(\alpha,\, \beta)}(\chi, it; \,\Pi_{1}, \Pi_{2})$ in (\ref{compleAFE2}). To lighten the notation, we often suppress the dependence on $\Pi_{1}, \Pi_{2}$ from the associated sums to be defined in Section \ref{sect: DLUsumdef}, and also the cut-off functions defined above. Let $\mathfrak{C}_{\infty}:=(\mfC_{\infty}(\Pi_{1})\mfC_{\infty}(\Pi_{2}))^{1/4}$. For $\re s=\sigma  >-1/2+\vartheta$ and $t\asymp T$, and for any $A>0$, we have 
\begin{align}\label{stdStirAFEcut}
        \widetilde{V}_{\alpha,\beta}(s;it)
        \, &\ll_{A, \,\sigma} \, (T^{\mfr}(\mathfrak{C}_{\infty})^2)^{\sigma} \,  (1+ |\im s|)^{-A}.
    \end{align}
   For any $A>0$, the following bound holds uniformly in $y>0$ and $v>0$:
\begin{align}\label{trivweigh}
	 H_{\alpha, \beta}(v,y) \, \ll_{A} \,  \, T \Big(1 + \frac{y}{T^{\mfr}(\mathfrak{C}_{\infty})^2\, }\Big)^{-A}.
\end{align}
Suppose that $y\le (TQ)^{\epsilon} T^{\mfr}(\mathfrak{C}_{\infty})^2$. For any $j\ge 0$, we have
\begin{align}\label{IKcutbdd}
t^j\frac{\partial^j}{\partial t^j}	V_{\alpha,  \beta}(y; it) \, \ll_{j,\, \epsilon} \, (TQ)^{\epsilon}.  
\end{align}

\begin{lem}\label{effectrunc}
Suppose that $y\le  T^{\,\mfr}(TQ)^{\epsilon}(\mathfrak{C}_{\infty})^2$ and $|v-1|>(TQ)^{-\epsilon}$. For any $A\ge 1$, we have
\begin{align}\label{est: effectrunc}
    H_{\alpha, \beta}(v, y) \, \ll_{A, \, \epsilon} \, (TQ)^{-A}. 
\end{align}
\end{lem}

\begin{proof}
Suppose $y\le (TQ)^{\epsilon} T^{\mfr}(\mathfrak{C}_{\infty})^2$ and $v\neq 1$. Integrating by parts and using (\ref{IKcutbdd}), we have
\begin{align}
	H_{\alpha, \beta}(v,y) \ &\ll_{A} \  |\log v|^{-A}\, \int_{\mathbb{R}} \ \Big| \, \frac{d^{A}}{dt^{A}} \, \eta_{T,\Delta}(t) V_{\alpha,\, \beta}( y; it) \Big| \, dt \nonumber\\
  \ &\ll_{A,\, \epsilon} \ |\log v|^{-A}  \ \sum_{\substack{j+i=A \\ j,i\ge 0}} \ \binom{A}{j} \, \Delta^{-j} T^{1-i}(TQ)^{\epsilon}  
   \ \ll \ T(TQ)^{\epsilon}(\Delta|\log v|)^{-A}\nonumber
\end{align}
 for any integer $A\ge 1$. It suffices to consider the case when $v> 1+(TQ)^{-\epsilon/10}$. By the elementary inequality $\log (1+x) \ge x/2$ for $x\in (0,1)$, we arrive at $H_{\alpha, \beta}(v, y) \ll_{A, \, \epsilon}  T(TQ)^{\epsilon}(\Delta (TQ)^{-\epsilon/10})^{-A}$. Since $ \Delta\gg _{\epsilon}(TQ)^{\epsilon}$, by taking $A$ sufficiently large, the desired bound follows.
\end{proof}



\section{Definitions of the $\mathcal{D}$, $\mathcal{L}$ and $\mathcal{U}$-sums}\label{sect: DLUsumdef}
We begin with a form of the orthogonality relations for Dirichlet characters. 
\begin{lem}\label{lem: charsum}(\cite[Lem. 2]{CIS19})
We have
\begin{align}
\sumflat{\chi}{q}\chi(m)\overline{\chi(n)} \ = \ \frac{1}{2} \, \Big( \sum_{\pm} \, \sum_{\substack{d|q\\d|(m\pm n)}}\phi(d)\mu\left(\frac{q}{d}\right)\Big) \cdot \mathbf{1}_{(mn,q)=1}.    
\end{align}
Here, $\sum_{\pm}$ denotes the sum of two terms, one with $d\mid (m+n)$ and the other with $d\mid (m-n)$. 
\end{lem}

Applying Lemma \ref{AFE} and Lemma \ref{lem: charsum}  to $\mathcal{M}_{\alpha, \, \beta}(h,k)=\mathcal{M}_{\alpha, \, \beta}(h,k; \Pi_{1}, \Pi_{2})$ defined in (\ref{mom: twodistform}), observe that
\begin{align*}
    \mathcal{M}_{\alpha, \beta}(h,k) \, = \,  \mathcal{S}_{\alpha, \beta}(h,k) \, + \, \epsilon(\Pi_{1})\epsilon(\Pi_{2}) (\mfq q^{\mfr})^{-\alpha-\beta}  \widehat{\mathcal{S}}_{-\beta, -\alpha}(h,k),
\end{align*}
where
\begin{align}\label{plugAFES}
     \mathcal{S}_{\alpha, \beta}(h,k) \, :=\,  \frac{1}{2} \, \sum_{\pm} 
       \, \sum_{(q, \mfq)=1} \, W\Big(\frac{q}{Q}\Big) \,  \sum_{\substack{d | q } } \, 
    \phi(d)\mu\Big(\frac{q}{d}\Big)  \mathop{\sum \sum}_{\substack{n,m\ge 1 \\ mh \equiv \pm nk \, (d) \\ (mnhk, q)=1}} \frac{\lambda_{\Pi_{1}}(m)\lambda_{\Pi_{2}}(n)}
        {m^{\frac12 + \alpha} n^{\frac12 + \beta}} H_{\alpha, \beta}\Big(\frac{nk}{mh}, \frac{mn}{\mfq q^\mfr}\Big),
\end{align}
the cut-off function $H_{\alpha,\beta}$ is defined in (\ref{doublecut}), and $\widehat{\mathcal{S}}_{-\beta, -\alpha}(h,k)$ can be obtained by $(\Pi_{1}, \Pi_{2})\to (\widetilde{\Pi_{2}}, \widetilde{\Pi_{1}})$, $(\alpha, \beta)\to (-\beta, -\alpha)$, and $H_{\alpha,\beta}\to \widehat{H}_{-\beta, -\alpha}$, where
\begin{align*}
     \widehat{H}_{-\beta, -\alpha}(v,y) \ &:= \  \int_{\mathbb{R}} \, \eta(t) \mathfrak{X}_{(\alpha,\beta);\, it}(\Pi_{1}, \Pi_{2})\widetilde{V}_{-\beta, -\alpha}(s;it;\,  \widetilde{\Pi_{2}}, \widetilde{\Pi_{1}}) v^{it} \, dt. 
\end{align*}
Notice that $|\mathfrak{X}_{(0,0);\, it}(\Pi_{1}, \Pi_{2})|=1$. Clearly, it is enough to consider (\ref{plugAFES}).

We record the following permissible truncations of sums, which will be used in later sections. Recall that $\mathfrak{C}:=(\mfC(\Pi_{1})\mfC(\Pi_{2}))^{1/4}$. By (\ref{trivweigh}), the sums over $m,n$  of $\mathcal{S}_{\alpha, \beta}(h,k)$  can be truncated at 
\begin{align}\label{eq: alwaysefftruncran}
    mn \le (TQ)^{\mfr+\epsilon} \mathfrak{C}^2,
\end{align}
which only incurs an arbitrarily small error term (i.e., $ \preceq_{A}  \, \mathfrak{C}(TQ)^{-A}$ for any $A>0$).


We further truncate the sums of $ \mathcal{S}_{\alpha, \beta}(h,k)$ by the oscillation of the $t$-integral in (\ref{doublecut}). The \emph{off-diagonal} of $\mathcal{S}_{\alpha,\beta}(h,k)$  (i.e., the terms with $mh\neq nk$ in (\ref{plugAFES})) can be effectively truncated at
\begin{align}\label{baltrunc}
	hm \,  \le \, (TQ)^{\frac{\mfr}{2}+\theta+\epsilon}\mathfrak{C}  \hspace{10pt}  \text{ and }  \hspace{10pt} nk  \, \le \, 2(TQ)^{\frac{\mfr}{2}+\theta+\epsilon}\mathfrak{C}.
\end{align}
Indeed, we take $v=nk/mh\neq 1$ and $y=mn/(\mfq q^{\mfr})$ in Lemma \ref{effectrunc}. Due to (\ref{trivweigh}), we are in the desired range of $y$ in Lemma \ref{effectrunc}. It remains to consider $|v-1| \le (TQ)^{-\epsilon}$, and we may assume that $v>1$ by symmetry. Now, the off-diagonal of $\mathcal{S}_{\alpha, \beta}(h,k)$ is seen to truncate to the range (\ref{baltrunc}), upon noting that $ hm \le \sqrt{hmkn} $ and $nk< (1+(O(TQ)^{-\epsilon})) hm$.

We write $q=cd$ in (\ref{plugAFES}).  Let $C\ge 1$ be a parameter to be optimized. In the end, we will choose
 \begin{align}\label{parameterC}
     C \ = \  
     \begin{cases}
     Q^{(1-\theta)/2}/ T^{(1+\theta)/2}  \hspace{15pt} \text{ if } \hspace{15pt}  \mfr \, = \, 2 \\
         Q^{1/4-\theta/2}/T^{3/4+\theta/2} \hspace{13pt} \text{ if } \hspace{15pt}  \mfr \, = \, 3. 
     \end{cases}
 \end{align}
For the time being, let's assume that $C\ll Q^{1-\epsilon}$ unless otherwise specified. We divide the sum (\ref{plugAFES}) into three parts: first the contribution from $c>C$ :
    \begin{align}\label{LsumIC}
   \mathcal{L}_{\alpha, \beta}(h,k) \, := \,   &\frac{1}{2} \, \sum_{\pm} 
        \mathop{\sum\,\sum}_{\substack{c, d\ge 1 \\ c>C \\ (cd, \, \mfq hk)=1}} \, W\Big(\frac{cd}{Q}\Big) 
    \phi(d)\mu(c) \mathop{\sum\,\sum}_{\substack{n,m\ge 1 \\ mh \equiv \pm nk \, (d) \\ (mn, \,cd)=1}} \frac{\lambda_{\Pi_{1}}(m)\lambda_{\Pi_{2}}(n)}
        {m^{\frac12 + \alpha} n^{\frac12 + \beta}} H_{\alpha, \beta}\Big(\frac{nk}{mh}, \frac{mn}{\mfq(cd)^\mfr}\Big). 
\end{align}
We then decompose the remaining contribution into the sum of the following two components:\footnote{ The contribution from $mh=nk$ and $mh\equiv-nk \,(d)$ is negligibly small, as the $d$-sum has at most two terms.}
\begin{align}
  \mathcal{D}_{\alpha, \beta}(h,k) \, &:= \,    \frac{1}{2} \,
   \mathop{\sum \sum}_{\substack{c,d \ge 1 \\c\le C\\ (cd,\,\mfq hk)=1}}  W\Big(\frac{cd}{Q}\Big) \, 
    \phi(d)\mu(c) \, \mathop{\sum \sum}_{\substack{n,m\ge 1 \\ mh = nk  \\ (mn, \, cd)=1}} \, \frac{\lambda_{\Pi_{1}}(m)\lambda_{\Pi_{2}}(n)}
        {m^{\frac12 + \alpha} n^{\frac12 + \beta}} H_{\alpha, \beta}\Big(\frac{nk}{mh}, \frac{ mn}{\mfq(cd)^\mfr}\Big), \label{DsumIC} \\ 
      \mathcal{U}_{\alpha, \beta}(h,k) \, &:= \,    \frac{1}{2} \ \sum_{\pm} 
        \mathop{\sum \sum}_{\substack{c, d\ge 1 \\ c\le C \\ (cd,\,\mfq hk)=1}}  W\Big(\frac{cd}{Q}\Big) 
    \phi(d)\mu(c) \mathop{\sum \sum}_{\substack{n,m\ge 1 \\ mh \equiv \pm nk \, (d) \\ (mn, \,cd)=1 \\ mh\neq nk}} \frac{\lambda_{\Pi_{1}}(m)\lambda_{\Pi_{2}}(n)}
        {m^{\frac12 + \alpha} n^{\frac12 + \beta}}H_{\alpha, \beta}\Big(\frac{nk}{mh}, \frac{mn}{\mfq(cd)^\mfr}\Big). \label{UsumIC}
 \end{align}

As a result, the sum (\ref{plugAFES}) can be written as
\begin{align}\label{momendecom}
    \mathcal{S}_{\alpha, \beta}(h,k) \ = \   \mathcal{D}_{\alpha, \beta}(h,k) \, +\, \mathcal{L}_{\alpha, \beta}(h,k) \, + \, \mathcal{U}_{\alpha, \beta}(h,k). 
\end{align}
 We suppress the dependence on $\alpha, \beta$ in the notation whenever no ambiguity arises. In the forthcoming subsections, we will further decompose (\ref{momendecom}) as:
\begin{align}
    \, \, &\mathcal{D}(h,k) \, +\,  \sum_{* \,\in \, \{\mathbf{0},\, \mathbf{r}\}} \, \mathcal{L}^{(*)}(h,k)  \, + \,  \sum_{* \,\in \, \{\mathbf{0},\, \mathbf{r}\}} \ \mathcal{U}^{(*)}(h,k) \label{part: UL0r} \\
     \, = \, &\mathcal{D}(h,k) \, +\, \sum_{* \,\in \, \{\mathbf{0},\, \mathbf{r}\}} \, \mathcal{L}^{(*)}(h,k) \, + \, \sum_{\bullet \in \{\mathbf{1}, \mathbf{2}\}}  \mathcal{U}^{\bullet}(h, k)  \, + \, \mathcal{U}^{(\mathbf{r})}(h, k) \label{part: U12}\\
     \, =\, &\mathcal{D}(h,k) \, + \, \sum_{\bullet\in \{\mathbf{d}, \, \mathbf{od}\}}  \mathcal{L}^{(\mathbf{0})}(h,k)_{\bullet} \, +\, \mathcal{L}^{(\mathbf{r})}(h,k) \, + \, \mathcal{U}^{\mathbf{1}}(h,k) \, + \,  \mathcal{U}^{\mathbf{2}}(h,k)_{\circ} \, - \, \mathcal{U}^{\mathbf{2}}(h,k)_{\mathbf{d}} \, + \, \mathcal{U}^{(\mathbf{r})}(h,k). \label{part: LUdod}
\end{align}
The partition in (\ref{part: UL0r}) will be carried out in Sections \ref{LsumHLSbdd}--\ref{absDiv}. The partitions in (\ref{part: U12}) and (\ref{part: LUdod}) will be carried out in Section \ref{U0sumsect}; see (\ref{split}), (\ref{split-dec}), (\ref{dec: U2}), and (\ref{dec: L0}). The bilinear forms of 
\begin{itemize}

 \item $\mathcal{L}^{(\mathbf{r})}(h,k)$ is of acceptable size (Section \ref{LsumHLSbdd}).

  \item $\mathcal{L}^{(\mathbf{0})}(h,k)_{\mathbf{od}}+\mathcal{U}^{\mathbf{1}}(h,k)$ is of acceptable size (Section \ref{U0sumsect});

    \item $\mathcal{L}^{(\mathbf{0})}(h,k)_{\mathbf{d}}$ and  $ \mathcal{U}^{(\mathbf{0})}(h,k)_{\mathbf{d}}$ are negligibly small (Section \ref{compdoubsumGCD});

     \item $\mathcal{U}^{\mathbf{2}}(h,k)_{\circ}$ is negligibly small (Section \ref{sect: NoODU2});

    \item $\mathcal{U}^{\mathbf{2}}(h,k)_{\mathbf{d}}$ is of acceptable size (Section \ref{sect: dropoffdiagcon})

    \item $\mathcal{U}^{(\mathbf{r})}(h,k)$ is of acceptable size (Section \ref{sect: boundUrsum}).
\end{itemize}


\section{The diagonal}\label{sect: diagterm}


Since $W$ has compact support and  $V_{\alpha, \beta}$ has rapid decay, we may extend the $c$-sum of $\mathcal{D}_{\alpha, \beta}(h,k)$ to all $c\ge 1$, with an arbitrarily small error term, which we omit in the typesetting of (\ref{diag}) below. Grouping the terms with $cd=q$ and applying (\ref{newcutoff})--(\ref{doublecut}), we have
\begin{align}\label{diag}
\mathcal{D}_{\alpha, \beta}(h,k) \, = \,  \frac{1}{2}\int_{\R} \, \int_{(1)} \, \eta(t) \widetilde{V}_{\alpha, \beta}(s; it)\mfq ^s\sum_{\substack{q\ge 1 \\ (q, \, \mfq hk)=1}} q^{\mfr s}  \phi^{*}(q) W\Big(\frac{q}{Q}\Big)   \mathcal{L}_{\alpha, \beta; \, h,k}^{(q)}(s;\, \Pi_{1}, \Pi_{2})  \, \frac{ds}{2\pi i} \, dt,
 \end{align}
 where $\phi^*(q) := \sum_{cd=q}\phi(d)\mu(c)$ is the number of primitive Dirichlet characters $(\bmod\, q)$, and 
 \begin{align}\label{CFKRSdiagDSshape}
      \mathcal{L}_{\alpha, \beta; \, h,k}^{(q)}(s;\, \Pi_{1}, \Pi_{2}) \ := \   \sum_{\substack{m,n\ge 1\\ mh=nk\\ (mn, \, q)=1}} \frac{\lambda_{\Pi_{1}}(m) \lambda_{\Pi_{2}}(n)}{m^{s+\frac{1}{2}+\alpha} n^{s+\frac{1}{2}+\beta}}.
 \end{align}
 In the following, let 
 \begin{align}\label{conveshiftdiag}
     \mbs \, := \, 1+2s+\alpha+\beta.
 \end{align}

 \begin{lem}\label{holoCFKRSDS}
     The Dirichlet series $ \mathcal{L}_{\alpha, \beta; \, h,k}^{(q)}(s;\, \Pi_{1}, \Pi_{2})$ admits a holomorphic continuation to the half-plane $\re s >-(1/4-\vartheta/2)$, except for a simple pole at $s=-(\alpha+\beta)/2$ when $\Pi_{1}\simeq \widetilde{\Pi_{2}}$.
 \end{lem}

\begin{proof}
Let $\mathcal{L}^{(qhk)}(\mbs,\, \Pi_{1}\otimes \Pi_{2})$ be defined in Section \ref{RSprelim}. By multiplicativity, we have 
\begin{align}
   \mathcal{L}_{\alpha, \beta; \, h,k}^{(q)}(s;\, \Pi_{1}, \Pi_{2}) \ = \   \mathcal{L}^{(qhk)}(\mbs,\, \Pi_{1}\otimes \Pi_{2})\,  \cdot \, \prod_{\substack{ p \mid hk}}  \ \sum_{\substack{m,n\ge 0 \\ m+o_{p}(h)=n+o_{p}(k)}} \,  \frac{\lambda_{\Pi_{1}}(p^{m}) \lambda_{\Pi_{2}}(p^{n})}{(p^{m})^{s+\frac{1}{2}+\alpha} (p^{n})^{s+\frac{1}{2}+\beta}}. \label{EPdiag}
\end{align}
 By Remark \ref{Gl2Hfuncrem} and Corollary \ref{gl3holoprop}, the first factor on the right-hand side of (\ref{EPdiag}) admits a holomorphic continuation to $\re \mbs> 1/2+\vartheta$, except for a simple pole at $\mbs=1$ when $\Pi_{1}\simeq \widetilde{\Pi_{2}}$. Let $H:=h/(h,k)$ and $K:=k/(h,k)$. The second multiplicand of (\ref{EPdiag}) is
\begin{align}
   \frac{1}{K^{s+\frac{1}{2}+\alpha} H^{s+\frac{1}{2}+\beta}} \ \prod_{p\mid hk} \ \sum_{\ell=0}^{\infty} \  \frac{\lambda_{\Pi_{1}}(p^{\ell+o_{p}(K)}) \lambda_{\Pi_{2}}(p^{\ell+o_{p}(H)})}{ p^{\ell\mbs} }.
\end{align}
Since the $\ell$-sum above is clearly holomorphic on $\re \mbs > 2 \vartheta$, Lemma \ref{holoCFKRSDS} now follows. 
\end{proof} 

We remark that (\ref{diag})--(\ref{CFKRSdiagDSshape}) gives the form of the diagonal terms to be used in \textbf{Section \ref{sect: CSLevGL3}} for applications to
critical zeros via Levinson's method. Indeed, practical applications require special (arithmetic) choices of the coefficients $(\lambda_h)$ of the Dirichlet
polynomial. Then the sums over $h,k$ in
\begin{align}
    \sum_{h} \,  \sum_{k} \, \frac{\lambda_{h}\overline{\lambda_{k}}}{\sqrt{hk}}\, \mathcal{D}_{\alpha, \beta}(h,k)
\end{align}
must be executed non-trivially, for instance, in terms of ratio-type expressions as in \cite{CS07}.

For \emph{fixed} $h,k\ge 1$, one may further rewrite $\mathcal{D}_{\alpha, \beta}(h,k)$ as follows.  On $\re \mbs=100$, it is clear that $\mathcal{L}_{p}(\mbs, \, \Pi_{1}\otimes \Pi_{2})\neq 0$ for any prime $p$. By Mellin inversion and \eqref{EPdiag}, we have 
 \begin{align}
   \mathcal{D}_{\alpha, \beta}(h,k)\ = \   \frac{1}{2}  \int_{\R} \, \int_{(100)} \, \int_{(100)} \ \widetilde{W}(u)\eta(t) &\widetilde{V}_{\alpha, \beta}(s; it)\mfq^s Q^u   \mathcal{L}_{\alpha, \beta; \, h,k}^{(1)}(s;\, \Pi_{1}, \Pi_{2})\nonumber\\
       &\cdot \sum_{\substack{q\ge 1 \\(q, \, hk\mfq)=1}}   \frac{\phi^{*}(q)}{q^{u-\mfr s}}    \prod_{p\mid q} \,  \mathcal{L}_{p}(\mbs,\, \Pi_{1}\otimes \Pi_{2})^{-1} \  \frac{du}{2\pi i} \ \frac{ds}{2\pi i} \,  dt. \label{eqn: DsumMellin} 
  \end{align}
 Let $\mbw:=u-\mfr s$. For $\re \mbw \gg 1$, we have, from Lemma \ref{genEPexp}, that
 \begin{equation}\label{Arithmaintermprep}
 \sum_{\substack{q\ge 1 \\(q, \, hk\mfq)=1}}   \frac{\phi^{*}(q)}{q^\mbw}    \prod_{p\mid q} \,  \mathcal{L}_{p}(\mbs,\, \Pi_{1}\otimes \Pi_{2})^{-1} = \zeta^{(hk\mfq)}(\mbw-1)\cdot \mathcal{B}^{(hk\mfq)}(\mbs, \mbw, \, \Pi_{1}\otimes \Pi_{2}),
 \end{equation}
 where
 \begin{align}
        \mathcal{B}_{p}(\mbs, \mbw, \, \Pi_{1}\otimes \Pi_{2}) \, := \, 1-p^{1-\mbw} \, + \, \mathcal{L}_{p}(\mbs, \Pi_{1}\otimes \Pi_{2})^{-1}\big((1-p^{-\mbw})^2-(1-p^{1-\mbw})\big).
 \end{align}
For 
\begin{align}\label{dom: MTeurpord}
    \re \mbw>1, \hspace{20pt} \re \mbs> \vartheta, \hspace{20pt} \re \,(\mbw+\mbs)>2(1+\vartheta),
\end{align}
we obtain the expansion
\begin{align}\label{MaintermEuprodexp}
    \mathcal{B}_{p}(\mbs, \mbw, \, \Pi_{1}\otimes \Pi_{2}) \, &= \, 1+O(p^{1-\re (\mbw+\mbs)}|\lambda_{\Pi_{1}}(p)\lambda_{\Pi_{2}}(p)|) +O(p^{-\re \mbw}) \nonumber\\
    \ &= \  1+O(p^{1+2\vartheta-\re (\mbw+\mbs)}) +O(p^{-\re \mbw}), 
\end{align}
which can be obtained by Remark \ref{Gl2Hfuncrem} and Lemma \ref{lem: naiveRSexpl} (or a \texttt{Mathematica} computation; see \texttt{MainEuExp.nb}). Hence, the infinite product $\mathcal{B}^{(hk\mfq)}$ admits a holomorphic continuation to (\ref{dom: MTeurpord}).

Putting (\ref{Arithmaintermprep}) into \eqref{eqn: DsumMellin},  and shifting the line of integration for the $s$-integral
to $\re s = -\epsilon$ for some small $\epsilon>0$, we pick up the contribution of a pole at $s=0$ (from $\widetilde{V}_{\alpha, \beta}(s; it)$). Thus, $ \mathcal{D}_{\alpha, \beta}(h,k)$  is a sum of the following two expressions:
\begin{align}
   \hspace{10pt}  \frac{1}{2}  \, \Big(\int_{\R} \eta\Big) \mathcal{L}_{\alpha, \beta; \, h,k}^{(1)}(0;\, \Pi_{1}, \Pi_{2})\int_{(100)}  \widetilde{W}(u) Q^u   \zeta^{(hk\mfq)}(u-1)\mathcal{B}^{(hk\mfq)}(1+\alpha+\beta,u , \, \Pi_{1}\otimes \Pi_{2})\  \frac{du}{2\pi i}, \label{polars=0}
\end{align}
and
\begin{align}
     & \frac{1}{2}\, \int_{\R} \, \int_{(-\epsilon)} \, \int_{(100)} \   \widetilde{W}(u)\eta(t) \widetilde{V}_{\alpha, \beta}(s; it)\mfq^s Q^u \mathcal{L}_{\alpha, \beta; \, h,k}^{(1)}(s;\, \Pi_{1}, \Pi_{2})\zeta^{(hk\mfq)}(\mbw-1) \nonumber\\
        &\hspace{110pt}\cdot  \,  \mathcal{B}^{(hk\mfq)}(\mbs, \mbw, \, \Pi_{1}\otimes \Pi_{2})\, \,  \frac{du}{2\pi i}\,  \frac{ds}{2\pi i} \, dt.\label{shiftedpiece}
\end{align}
For both (\ref{polars=0}) and  (\ref{shiftedpiece}), we shift the lines of integration of the $u$-integrals to $\re u = 1+2\vartheta+\epsilon$. Hence, apart from an error of size $\preceq TQ^{1+2\vartheta}$, we have
\begin{align}
   \hspace{5pt} \mathcal{D}_{\alpha, \beta}(h,k) \, = \,   \frac{ 1}{2}\, \widetilde{W}(2) Q^2\Big(\int_{\R}\, \eta\Big)\mathcal{L}_{\alpha, \beta; \, h,k}^{(1)}(0;\, \Pi_{1}, \Pi_{2})  \,\frac{\phi(hk\mfq)}{hk\mfq}   \mathcal{B}^{(hk\mfq)}(1+\alpha+\beta,\, 2, \,  \Pi_{1}\otimes \Pi_{2}).
\end{align}


\section{The $\mathcal{L}$-sum: a first application of the large sieve}\label{LsumHLSbdd}

We apply the orthogonality relation
\begin{equation}
\frac{1}{2} \, \phi(d) \ \sum_{\pm }   \  \mathbf 1_{\substack{m h \equiv \pm  n k \, (\bmod\, d) \\ (mnhk, d)=1}} \, = \,   
    \sideset{}{^+}{\sum}_{\psi\, (\bmod\, d)} \, \psi(m h) \overline{\psi}( n k),
\end{equation}
where $+$ indicates the sum is taken over the even characters $(\bmod\, d)$. Then from (\ref{LsumIC}), 
\begin{align}
 \mathcal{L}(h,k)\ = \   \mathop{\sum\sum }_{\substack{c>C\\ (c, \, \mfq hk)=1\\(d, \mfq)=1}}  \mu(c) W\Big( \frac{cd}{Q}\Big) \ \sideset{}{^+}{\sum}_{\psi\, (d)}
  \mathop{\sum\sum}_{\substack{{n,m\ge 1}\\ 
    (m n , c ) = 1}}
    \frac{\lambda_{\Pi_{1}}(m)\lambda_{\Pi_{2}}(n)\psi(m h) \overline{\psi}( n k)}
        {m^{\frac12 + \alpha} n^{\frac12 + \beta}} H_{\alpha, \beta}\Big(\frac{nk}{mh}, \frac{mn}{\mfq(cd)^{\mfr}}\Big)     \label{foldbacL}.
\end{align}
We write $ \mathcal{L}(h,k)= \mathcal{L}^{(\mathbf{0})}(h, k) \, + \, \mathcal{L}^{(\mathbf{r})}(h, k)$, where  $\mathcal{L}^{(\mathbf{0})}(h, k)$ and $ \mathcal{L}^{(\mathbf{r})}(h, k)  $ denote the contributions from $\psi=\psi_{0}$ and $\psi\neq \psi_{0}$, respectively,  in (\ref{foldbacL}).  The analysis of $\mathcal{L}^{(\mathbf{0})}(h, k)$  is postponed to Section \ref{Fauxcan}.  We address the sum $ \mathcal{L}^{(\mathbf{r})}(h, k)$ in this section.

\begin{prop}\label{LsuHLSter}
Suppose that one of the following holds:
\begin{enumerate}
    \item $\lambda_{h}\ll_{\epsilon} h^{\epsilon}$; or

    \item both Condition $(\mathbf{\Lambda})$ (see (\ref{bdd: 24normabstract}))  and Hypothesis $(\mathbf{\Pi}^4)$ (see (\ref{fourthpowerestHec})) hold.
\end{enumerate}
Then we have the estimate:
\begin{align}\label{avLr}
    \mathop{\sum \sum}_{h, k \,\le\, (TQ)^{\theta}} \,  \frac{\lambda_{h}\overline{\lambda_{k}}}{\sqrt{hk}} \,  \mathcal{L}^{(\mathbf{r})}(h,k) \,\preceq \, \mathfrak{C}^3\, \Big\{ \frac{TQ^2}{C} \, + \, \frac{(TQ)^{\mfr/2+\theta}}{C^{\mfr/2-1}}\Big\}.
\end{align}
\end{prop}

\begin{proof}
Recall (\ref{newcutoff})--(\ref{doublecut}). Let $(\psi_{1}, \psi_{2}):=(\psi, \overline{\psi})$. Interchanging sums and integrals, we have
\begin{align}\label{openupLs}
    \hspace{-10pt} \mathcal{L}^{(\mathbf{r})}(h,k) \, = \, &     \sum_{\substack{c>C\\ (c, \, \mfq hk)=1}}  \sum_{(d, \, \mfq)=1} \ \mu(c)W\Big(\frac{cd}{Q}\Big) \ \sideset{}{^+}{\sum}_{\substack{\psi \, (d) \\ \psi\neq \psi_{0} }} \ \psi(h) \overline{\psi}(k)  \, \nonumber\\
    & \hspace{40pt}\cdot  \, \int_{\R} \, \int_{(1)} \, \eta(t)  \widetilde{V}_{\alpha, \beta}(s; it) (\mfq (cd)^\mfr)^{s} \Big(\frac{k}{h}\Big)^{it} \,  \prod_{i=1}^{2} \, \sum_{\substack{m\ge 1 \\ (m,c)=1}} \,  \frac{\lambda_{\Pi_{i}}(m)  \psi_{i}(m)}{m^{\mbw_{i}}}  \, \frac{ds}{2\pi i} \, dt,
\end{align}
where $\mbw_{1} := s+it+\frac{1}{2}+\alpha$ and $ \mbw_{2} :=  s-it+\frac{1}{2}+\beta. $
The $m$-sum of (\ref{openupLs}) is precisely $L^{(c)}\left(\mbw_{i}, \Pi_{i} \times \psi_{i}\right) $, which admits a holomorphic continuation to $\C$.

We shift the line of integration of the $s$-integral to $\re s =\epsilon$ using the rapid decay and holomorphy of the integrand. On $\re s=\epsilon$, we have $\re \mbw_{i} > 1/2$, and $L_{c}(\mbw_{i}, \, \Pi_{i}\times \psi_{i})^{-1}= O_{\epsilon}(c^{\epsilon})$. Hence, 
\begin{align}
    \bigg| \mathop{\sum \sum}_{h, k \,\le\, (TQ)^{\theta}}\ \frac{\lambda_{h}\overline{\lambda_{k}}}{\sqrt{hk}}\,  \mathcal{L}^{(r)}(h,k) \bigg| \ &\preceq \  \sum_{\substack{c>C}}  \sum_{d\ge 1} \ W\Big(\frac{cd}{Q}\Big)  \, \sideset{}{^+}{\sum}_{\substack{\psi \, (d) \\ \psi\neq \psi_{0} }} \, \int_{\R} \, \int_{(\epsilon)} \, \eta(t)  |\widetilde{V}_{\alpha, \beta}(s; it) | \nonumber\\
    & \hspace{65pt}  \cdot \ \bigg| \, \sum_{\substack{ (h,c)=1}} \, \frac{\lambda_{h}\psi(h)}{h^{1/2+it}} \bigg|^2 \prod_{i=1}^{2} \, \left|L\left(\mbw_{i},  \Pi_{i} \times \psi_{i}\right) \right|   \, |ds| \, dt. \label{LrsumTriandrop}
\end{align}

In the following, we write $\Pi\in \{\Pi_{1}, \Pi_{2}\}$ and $\mbw\in \{\mbw_{1}, \mbw_{2}\}$. We apply the convexity bound (see \cite[(5.21)]{IK04} and (\ref{analycondtwis}))
\begin{align}
  L\left(\mbw,  \Pi\times \psi\right) \ = \   O_{\epsilon}(\mfC(\mbw, \Pi\times \psi)^{1/4+\epsilon}),
\end{align}
and the bound $ \widetilde{V}_{\alpha, \beta}(s; it) \preceq_{A} (1+|\im s|)^{-A}$ (see (\ref{stdStirAFEcut})), the contribution from $|\im s| > (TQ)^{\epsilon}$ on the right-hand side of (\ref{LrsumTriandrop}) is arbitrarily small.   By $|ab| \ll |a|^2+|b|^2$, it suffices to estimate
\begin{align}\label{larsie}
    \max_{\substack{\re s=\epsilon\\|\im s| \le (TQ)^{\epsilon}}} \ \sum_{c>C}  \sum_{d\ge 1} \ W\Big(\frac{cd}{Q}\Big) \, \sideset{}{^+}{\sum}_{\substack{\psi \, (d) \\ \psi\neq \psi_{0} }} \, \int_{t\asymp T} \, \Big| \, \sum_{\substack{(h,c)=1}} \, \frac{\lambda_{h}\psi(h)}{h^{1/2+it}} \Big|^2 |L\left(\mbw,  \Pi \times \psi\right) |^2  \, dt.
\end{align}

We rewrite (\ref{larsie}) by $\psi = \psi_{0} \psi^{*}$, where $\psi^{*} \, (\bmod\, u)$ is primitive, $u\mid d$,  and $\psi_{0} \, (\bmod\, d)$ is the trivial character. We have $ L(\mbw, \Pi \times \psi) =  L(\mbw, \Pi \times \psi^{*}) L_{d}(\mbw, \, \Pi\times \psi^{*})^{-1}$.  As a result,   
\begin{align}\label{wholeLr}
 (\ref{larsie}) \, \preceq  \, \max_{\substack{\re s=\epsilon\\|\im s| \le (TQ)^{\epsilon}}} \, \sum_{C < c \ll Q} \,\sum_{r\le Q/c} \, \sum_{1< u\le Q/cr} \,  \sideset{}{^\flat}{\sum}_{\substack{\psi\,  (u)  }}  \, \int_{t\asymp T} \, \Big|\sum_{\substack{ (h,c)=1}} \, \frac{\lambda_{h}\psi(h)}{h^{1/2+it}} \Big|^2 |L\left(\mbw,  \Pi \times \psi\right) |^2  \, dt. 
\end{align}

Applying (\ref{singAFE}) to the factor $|L\left(\mbw,  \Pi \times \psi\right)|$, and notice that the sum obtained is truncated at 
\begin{align*}
    n \, \le \,  M \ = \  M_{c,r} \ := \ (TQ)^{\epsilon} \sqrt{\mfC(\Pi)(TQ/cr)^\mfr}
\end{align*}
by (\ref{singAFEtrunctail}). Opening the cut-off $V_{\mbw}$ with its definition (\ref{def: singlecutoff}), it follows that
\begin{align}
      \bigg|\sum_{n\le M} &\frac{\lambda_{\Pi_{1}}(n) \psi(n)}{n^{\mbw}} V_{\mbw}\Big(\frac{n}{\sqrt{\mfq_{\Pi_{1}\times \psi}}}\Big) \bigg|
       \, \ll_{\epsilon} \,   \mathbf{c}(\mbw, \Pi\times \psi)^{\epsilon} \max_{|\tau|\ll(TQ)^{\epsilon}} \, \Big|\sum_{n\le M} \frac{\lambda_{\Pi}(n) \psi(n)}{n^{\mbw+\epsilon+i\tau}}\Big|.
\end{align}
From this, (\ref{analycondtwis}) and (\ref{wholeLr}), observe that
\begin{align}\label{modsquareDS}
(\ref{larsie})\,  \preceq \, \mfC(\Pi)^{1/4}\max_{\substack{\epsilon/2 \le \re s \le 4\epsilon\\ |\im s| \ll (TQ)^{\epsilon}}} \, \sum_{C < c \ll Q} \,\sum_{r\le Q/c} \,  \sum_{1< u\le Q/cr} \,  \sideset{}{^\flat}{\sum}_{\substack{\psi \, (u)  }} \, \int_{t\asymp T} \,  \big| \,  \sum_{n \le M(TQ)^{\theta}} \ \frac{\psi(n) b(n)}{n^{1/2+it}}\big|^2   \,  dt,
\end{align}
where 
\begin{align}
    b(n) \ = \ b_{c,r,s}(n) \, := \,  \sum_{\substack{n=hm \\ m\le M }} \lambda_{h} \mathbf{1}_{(h,c)=1} \, \frac{\lambda_{\Pi}(m)}{m^{s}}.
\end{align}
Applying the Hybrid Large Sieve (Lemma \ref{Gallager}), we have   
\begin{align}\label{afterHLSLsum}
    \sum_{1< u\le Q/cr} \,  \sideset{}{^\flat}{\sum}_{\substack{\psi \, (u)  }} \, \int_{t\asymp T} \,  \big| \,  \sum_{n \le M(TQ)^{\theta}} \ \frac{\psi(n) b(n)}{n^{1/2+it}}\big|^2 \, \preceq \, \sqrt{\mfC(\Pi)}\Big( T\big(\frac{Q}{cr}\big)^{2} + (TQ)^{\theta} \big( \frac{TQ}{cr}\big)^{\frac{\mfr}{2}}\Big) \sum_{n\le M(TQ)^{\theta}} \, \frac{|b(n)|^2}{n}.
\end{align}

It remains to bound the $\ell^2$-norm of $(b(n)/\sqrt{n})$. By positivity and Cauchy's inequality, we obtain 
\begin{align}
     \, |b(n)|^2 \, \le  \, \Big(\sum_{\substack{m\mid n }} \, |\lambda_{\Pi}(m)|^2\Big)\Big( \sum_{\substack{h\mid n }} \, |\lambda_{h}|^2\Big).\nonumber
\end{align}
 Then
\begin{align}
    \sum_{n\le M(TQ)^{\theta}} \, \frac{|b(n)|^2}{n}   \, &= \, \mathop{\sum\sum}_{m,h \le M(TQ)^{\theta}} \,|\lambda_{\Pi}(m)|^2 |\lambda_{h}|^2 \sum_{\substack{n\le M(TQ)^{\theta}\\ [m,h]\mid n }} \, \frac{1}{n} \nonumber\\[7pt]
     \, &\ll_{\epsilon} \,   (TQ)^{\epsilon}  \mathop{\sum\sum}_{m,h \le M(TQ)^{\theta}} \, \frac{|\lambda_{\Pi}(m)|^2 |\lambda_{h}|^2}{\sqrt{mh}} \frac{(m,h)}{\sqrt{mh}}\nonumber
\end{align}
Firstly, if we have $\lambda_{h}\ll_{\epsilon} h^{\epsilon}$, then by Lemma \ref{lem: simpleGCD} and Lemma \ref{lem: ROA} we immediately have
\begin{align}
     \sum_{n\le M(TQ)^{\theta}} \, \frac{|b(n)|^2}{n} \, \preceq \, \sum_{m\le M(TQ)^{\theta}} \, \frac{|\lambda_{\Pi}(m)|^2}{m} \, \preceq 1. 
\end{align}
Secondly, suppose Condition $(\mathbf{\Lambda})$  and Hypothesis $(\mathbf{\Pi}^4)$ hold. Then it follows from Lemma \ref{lem: Galsum} that
\begin{align}
   \sum_{n\le M(TQ)^{\theta}} \, \frac{|b(n)|^2}{n} \,  \preceq \,  \Big(\sum_{m\le M(TQ)^{\theta}} \, \frac{|\lambda_{\Pi}(m)|^4}{m}\Big)^{1/2}\Big(\sum_{h\le M(TQ)^{\theta}} \, \frac{|\lambda_{h}|^4}{h}\Big)^{1/2} \, \preceq \, 1. \nonumber
\end{align}
In either case, Proposition \ref{LsuHLSter} follows from (\ref{modsquareDS}), (\ref{afterHLSLsum}), and summing over $r$ and $c$:
\begin{align}
\hspace{40pt} (\ref{larsie}) \,  \preceq \, \mfC(\Pi)^{3/4} \sum_{c>C} \,   \Big\{T\Big(\frac{Q}{c}\Big)^2  +  (TQ)^{\theta} \Big(\frac{TQ}{c}\Big)^{\mfr/2}\Big\} \, \preceq \, \mfC(\Pi)^{3/4} \, \Big\{ \frac{TQ^2}{C}  + \frac{(TQ)^{\mfr/2+\theta}}{C^{\mfr/2-1}}\Big\}. \nonumber  \hspace{120pt} \qedhere
\end{align}
\end{proof}


\section{An general form of divisor switching}\label{absDiv}

Let $F: [1, \infty) \rightarrow \C$ be a function, and let
\begin{align}\label{divswiauxpara}
    g \ :=  \ (mh, nk), \hspace{20pt}  M:= \frac{mh}{g}, \hspace{20pt} N:= \frac{nk}{g}.
\end{align}
In (\ref{UsumIC}), we have $M\neq N$, and we encounter sums of the following form: 
\begin{align}\label{abssum}
 \sum_{\substack{d \mid (mh\pm nk)\\ (d, \,\mfq mnhk)=1}}  \phi(d) F(d) \ = \    \sum_{\substack{d \mid (M\pm N)\\ (d,\, g\mfq)=1}}  \, \phi(d) F(d). 
\end{align}
Such sums also arise in earlier works on the Asymptotic Large Sieve.  We now explain the use of ``\textit{divisor switching}'' to (\ref{abssum}). By M\"obius inversions (\cite[Section 5]{CIS19} and \cite[Section 3]{BTB22}), 
\begin{align}
\sum_{\substack{d \mid (M\pm N)\\ (d,\, g\mfq)=1}}  \, \phi(d) F(d)\ = \  \sum_{ (e,\,\mfq g)=1} \frac{\mu(e)}{e} \sum_{\substack{a\mid g\\ (a,\,  \mfq)=1}} \mu(a) \sum_{\substack{b(ea\ell)= |M\pm N|\\ (b,\, \mfq)=1}} \ \frac{|M\pm N|}{\ell} \, F\Big( \frac{|M\pm N|}{\ell}\Big),\nonumber
\end{align}
where $d\ell =  |M\pm N|$, and $ d  = eab$. The variable $\ell$ is the ``complementary divisor''.  To eliminate the dependence on $b$ in the $\ell$-sum, we apply another M\"obius inversion, leading to
        \begin{align}\label{ddivswi}
          \hspace{10pt}\sum_{\substack{d \mid (M\pm N)\\ (d,\, g\mfq)=1}}  \, \phi(d) F(d) \, = \ \sum_{r\mid \mfq} \sum_{ (e, \mfq g)=1} \sum_{\substack{a\mid g\\ (a, \, \mfq)=1}}   \frac{\mu(r)\mu(e)\mu(a)}{e}  \,  \sum_{M \equiv \mp N \, (\bmod\, \mnl)} \, \frac{|M\pm N|}{\ell} \, F\Big( \frac{|M\pm N|}{\ell}\Big).
        \end{align}
For later convenience, we introduced the following notation
\begin{align}\label{key: conduct}
  \mnl \, := \, rea\ell.
\end{align}
 By definition, we have $(M, N)=1$. The congruence in (\ref{ddivswi}) implies that $(N, \, \mnl)=(M, \, \mnl)$. These imply  $(M, \mnl)=(N, \mnl)=1$. The orthogonality relation $(\bmod\, \mnl)$, i.e.,
        \begin{align*}
            \mathbf{1}_{M \equiv \mp N \ (\bmod\, \mnl)} \ = \  \frac{1}{\phi(\mnl)} \ \sum_{\psi \, (\bmod\, \mnl)} \ \psi(M)\overline{\psi}(\mp N),
        \end{align*}
        allows us to recast the right-hand side of  (\ref{ddivswi}) in terms of the multiplicative characters $(\bmod\, \mnl)$:
        \begin{align}\label{orthodivsw}
          \frac{Q}{c} \,  \sum_{r\mid \mfq}& \sum_{ (e,\,\mfq g)=1} \sum_{\substack{a\mid g\\ (a, \mfq)=1}}   \frac{\mu(r)\mu(e)\mu(a)}{e} \sum_{\ell=1}^{\infty} \  \mathcal{W}\Big( \frac{c|M\pm N|}{Q\ell}\Big)  \frac{1}{\phi(\mnl)} \ \sum_{\psi \, (\bmod\, \mnl)} \ \psi(M)\overline{\psi}(\mp N),
        \end{align}
        where $\mathcal{W}(x):= x F(Qx/c)$. 
        
        Applying (\ref{orthodivsw}) to  (\ref{UsumIC}), we obtain
\begin{align}
    \mathcal{U}_{\alpha, \beta}(h,k) \, &= \,  \sum_{* \,\in \, \{\mathbf{0},\, \mathbf{r}\}} \ \mathcal{U}^{(*)}(h,k) 
    \, := \, \sum_{* \, \in \, \{\mathbf{0},\, \mathbf{r}\}} \,
     \mathop{\sum \sum}_{\substack{m,n\ge 1 \\ mh\neq nk}} \,  \frac{\lambda_{\Pi_{1}}(m)\lambda_{\Pi_{2}}(n)}{m^{\frac12 + \alpha} n^{\frac12 + \beta}} \, \mathrm{U}^{(*)}(m,n; h,k), \label{USummm}
\end{align}
where 
\begin{align}\label{absU0sum}
\mathrm{U}^{(*)}(m,n; h,k) \, := \, \frac{Q}{2}\,  \sum_{\substack{c\le C \\(c,\, \mfq mnhk)=1 }} \, &\frac{\mu(c)}{c}\, \sum_{r\mid \mfq} \,    \,  \sum_{ (e, \,\mfq g)=1} \, \sum_{\substack{a\mid g\\(a, \mfq)=1}}\, \frac{\mu(r)\mu(e)\mu(a)}{e} \nonumber\\
&\hspace{-10pt}\cdot \, \sum_{\ell=1}^{\infty} \, \sum_{\pm}  \  \frac{1}{\phi(\mnl)} \mathcal{W}\big( \frac{c|mh\pm nk|}{Qg\ell}\big)\ \sideset{}{^{(*)}}\sum \psi\Big(\frac{mh}{g}\Big)\overline{\psi}\Big(\mp\frac{nk}{g}\Big),
\end{align}
with  $\sideset{}{^{(\mathbf{0})}}{\textstyle \sum} = \sideset{}{}{\textstyle \sum}_{\psi=\psi_{0}}$, $\sideset{}{^{(\mathbf{r})}}{\textstyle \sum} = \sideset{}{}{\textstyle \sum}_{\psi\neq \psi_{0}}$, and
\begin{align}
\mathcal{W}(x) \, &= \,  x     W(x) H_{\alpha,\beta}\Big(\frac{nk}{mh}, \frac{mn}{\mfq (Qx)^{\mfr}}\Big),\label{Ffunc}\\[5pt]
\sum_\pm \  \widetilde{\mathcal{W}}^{\pm}(w) \, :&= \,  \sum_\pm \,  \int_{0}^{\infty} \, \mathcal{W}\Big(\frac{c}{Qg}\frac{|mh \pm nk|}{x}\Big)   x^{w} \, d^{\times} x. \label{simMell}
\end{align}

We note that divisor switching itself (see (\ref{orthodivsw})) is rather general, and is insensitive to the arithmetic coefficients $\lambda_{\Pi_{i}}(m)$'s, the specific shape of $F$, or the choice of cut-off functions. However, the arithmetic of the coefficients enters the scene from Section \ref{sect: NoODU2} onward. As a summary ahead:
\begin{itemize}
    \item For $\mathcal{U}^{(\mathbf{0})}(h,k)$, we first derive an exact evaluation of the sums over $\ell$, $r$, $e,a$ in (\ref{USummm})--(\ref{absU0sum}) (see Lemma \ref{lem: U0anapar}). This is followed by the summation over $m,n$ (see Section \ref{sect: NoODU2}). The sums over $c, h,k$ are handled trivially. 

    \item  For $\mathcal{U}^{(\mathbf{r})}(h,k)$, we evaluate and bound the sums over $m,n,h,k$ non-trivially using the hybrid large sieve, but we apply the triangle inequality through the sums over $c,r,e,a, \ell$ (with $\mnl$ being the size of the conductor); see Section \ref{sect: boundUrsum}. 
\end{itemize}


\section{The sum $\mathcal{U}^{(\mathbf{0})}(h,k)$}\label{U0sumsect}

\subsection{A quadruple sum}\label{sect: divswitcEuprod}

In  $\mathcal{U}^{(\mathbf{0})}(h,k)$ (see (\ref{absU0sum})), we have  $\psi=\psi_{0}\, (\mnl)$ (with $\mnl:=rea\ell$), and hence the coprimality condition $(MN, \,\mnl)=1$, for $M,N$ defined in (\ref{divswiauxpara}). To facilitate the upcoming inclusion-exclusion process, we write
\begin{align}\label{splitmnaccortoq}
   \hspace{5pt}  m \, = \, m_{0}m', \hspace{5pt} n \, = \, n_{0}n', \hspace{5pt} h \, = \, h_{0}h', \hspace{5pt} k \, = \, k_{0}k', \hspace{5pt} \text{ where} \hspace{10pt}  m_{0}n_{0}h_{0}k_{0} \mid \mfq^{\infty}, \hspace{5pt} (m'n'h'k', \mfq) \, = \, 1.
\end{align}
 Accordingly, we have $g=g_{0}g'$, $M= M_{0}M'$ and $N= N_{0}N'$, where
\begin{align}\label{MNdecompo0pr}
  \hspace{5pt} g_{0}:=  (m_{0}h_{0}, n_{0}k_{0}), \hspace{5pt}  M_{0}  :=  \frac{m_{0}h_{0}}{g_{0}}, \hspace{5pt}  N_{0}  :=  \frac{n_{0}k_{0}}{g_{0}}, \hspace{5pt} g' :=  (m'h', n'k'), \hspace{5pt} M'  :=  \frac{m'h'}{g'},  \hspace{5pt} N'  =   \frac{n'k'}{g'}.
\end{align}

\begin{lem}\label{lem: U0anapar}
    The sum $ \mathrm{U}^{(\mathbf{0})}(m,n; h,k)$ of (\ref{absU0sum}) is equal to 
 \begin{align}\label{U0anapar}
    \frac{Q}{2}\, \sum_{\substack{c\le C \\(c, \mfq mnhk)=1 }} \, \frac{\mu(c)}{c}  \,   \int_{(\sigma_{w})} \,  \sum_\pm \ 
 \widetilde{\mathcal{W}}^{\pm}(w)  \, \zeta(1+w) \, R_{1}(w, g', M'N')R_{M_{0}N_{0}}(w)R_{\mfq}^{(M_{0}N_{0})}(w) \, \frac{dw}{2\pi i}, 
\end{align}
where the integral converges absolutely whenever $\sigma_{w}> 0$, and
\begin{align}
R_{1}(w; u, v)   \, &:= \,  \prod_{p\mid v} \, \big(1-\frac{1}{p^{1+w}}\big) \, \prod_{\substack{p\mid  u\\  p \, \nmid\, v}} \, \Big(1+ \frac{1}{(p-1) p^{1+w}}- \frac{1}{p-1}\Big) \,  \prod_{p \,\nmid\, uv} \, \Big(1+ \frac{p^{-w}-1}{p(p-1)}\Big), \label{resultan}\\
R_{M_{0}N_{0}}(w) \ &:= \  \prod_{p\mid M_{0}N_{0}} \, \Big(1-\frac{1}{p^{1+w}}\Big) \Big(1+ \frac{p^{-w}-1}{p(p-1)}\Big)^{-1}, \\
  R_{\mfq}^{(M_{0}N_{0})}(w) \ &:= \    \prod_{\substack{p\mid \mfq\\ p\nmid M_{0}N_{0}}} \,  \frac{p^{2+w}-2p^{1+w}+1}{p^{2+w}-p^{1+w}-p^w+1}.
\end{align}
\end{lem}

\begin{proof}
    By Mellin inversion, the $\ell$-sum of $ \mathrm{U}^{(\mathbf{0})}(m,n; h,k)$ is equal to
\begin{align}\label{ellsum}
      \frac{1}{\phi(rea)}\,  \int_{(\sigma_{w})}  \sum_\pm 
 \widetilde{\mathcal{W}}^{\pm}(w)  \zeta(1+w)  R(w; rea,  MN) \, \frac{dw}{2\pi i},
\end{align}
where
\begin{align}
    R(w; u,v) \, := \,   \prod_{p \mid v} \,  \Big(1-\frac{1}{p^{1+w}}\Big)  \, \prod_{p\,\nmid\, uv} \, \Big(1+ \frac{1}{p^{w+1}(p-1)}\Big).
\end{align}

Next, we evaluate the triple sum over $r$, $e$  and $a$:
\begin{align}\label{collapseeal}
   \mathfrak{R}(w; g, MN, \mfq) \, := \, \sum_{\substack{r\mid \mfq\\ (r, \, MN)=1}} \, \mu(r)\sum_{\substack{ (e, \, \mfq g MN) =1}} \frac{\mu(e)}{e} \sum_{\substack{a\mid g\\  (a, \mfq MN) =1}}  \, \frac{\mu(a)}{\phi(rea)} R(w;  rea,   MN).
\end{align}
We consider $m,n,h,k$ (hence $g$) and $\mfq$ as fixed for the moment. Observe that the quantities $r,e,a, M,N$ are pairwise coprime. Hence, 
\begin{align}
    R(w;  rea,   MN) \ =  \ \prod_{p\mid r} \, \Big(1+ \frac{1}{p^{w+1}(p-1)}\Big)^{-1} R(w; \, ea, MN), \nonumber
\end{align}
and that
\begin{align}
    R(w; ea, MN) \ = \ R(w; ea, \mfq M'N')  \prod_{\substack{p\mid \mfq \\p\nmid M_{0}N_{0}}}  \Big(1-\frac{1}{p^{1+w}}\Big)^{-1} \, \Big(1+ \frac{1}{p^{w+1}(p-1)}\Big).
\end{align}
We rewrite (\ref{collapseeal}) with (\ref{MNdecompo0pr}). It follows that
\begin{align}
     \mathfrak{R}(w; g, MN, \mfq) \ = \      & \prod_{\substack{p\mid \mfq \\p\nmid M_{0}N_{0}}} \, \Big(1-\frac{1}{p^{1+w}}\Big)^{-1}\Big(1+ \frac{1}{p^{w+1}(p-1)}\Big)\sum_{\substack{a\mid g'\\  (a, \,\mfq M'N') =1}}  \, \sum_{\substack{ (e, \, \mfq g' M'N') =1}} \,  \frac{\mu(e)\mu(a)}{\phi(e)\phi(a)e} \nonumber\\
     & \hspace{20pt}\cdot \ R(w, ea, \mfq M'N')   \sum_{\substack{r\mid \mfq \\ (r,\,  M_{0}N_{0})=1}} \, \frac{\mu(r)}{\phi(r)} \prod_{p\mid r} \, \Big(1+ \frac{1}{p^{w+1}(p-1)}\Big)^{-1} \nonumber\\
     \ = \ &  \prod_{\substack{p\mid \mfq\\ p\nmid M_{0}N_{0}}} \,  \frac{p^{2+w}-2p^{1+w}+1}{(p-1)(p^{1+w}-1)}   \sum_{\substack{a\mid g'\\  (a, \,\mfq M'N') =1}}  \, \sum_{\substack{ (e, \, \mfq g' M'N') =1}} \,  \frac{\mu(e)\mu(a)}{\phi(e)\phi(a)e}R(w, ea, \mfq M'N'). \nonumber
\end{align}
A simple Euler product calculation with Lemma \ref{genEPexp} (or see \cite[(7.7)--(7.8)]{CIS19}) yields
\begin{align}
  \mathfrak{R}(w; g, MN, \mfq) \ = \  &  R_{1}(w, g', M'N') \prod_{\substack{p\mid \mfq\\ p\nmid M_{0}N_{0}}} \,  \frac{p^{2+w}-2p^{1+w}+1}{(p-1)(p^{1+w}-1)}  \prod_{p\mid \mfq}\, \Big(1-\frac{1}{p^{1+w}}\Big) \Big(1+ \frac{p^{-w}-1}{p(p-1)}\Big)^{-1}  \nonumber\\
\ = \  &R_{1}(w, g', M'N')R_{M_{0}N_{0}}(w)R_{\mfq}^{(M_{0}N_{0})}(w).
\end{align}
The desired result follows from (\ref{absU0sum}).
\end{proof}


\subsection{Further decomposition}
Suppose that $mh\neq nk$.  We shift the line of integration in (\ref{U0anapar}) to $\re w=-\epsilon$, where $ \epsilon\in (0,1)$. We obtain the following expression for $\mathrm{U}^{(\mathbf{0})}(m,n; h,k)$:
\begin{align}\label{split}
       & \frac{Q}{2} \, \sum_{\substack{c\le C \\(c,\, \mfq mnhk)=1 }} \, \frac{\mu(c)}{c} \, \sum_\pm \,  \widetilde{\mathcal{W^{\pm}}}(0) \,R_{1}(0, g', M'N')R_{M_{0}N_{0}}(0)R_{\mfq}^{(M_{0}N_{0})}(0) \nonumber\\
      \ & \hspace{15pt} \, + \,   \frac{Q}{2}\,
      \sum_{\substack{c\le C \\(c, \,\mfq mnhk)=1 }} \, \frac{\mu(c)}{c} \ \sum_\pm   \, \int_{(-\epsilon)} \, \widetilde{\mathcal{W}^{\pm}}(w)  \, \zeta(1+w) \, R_{1}(w, g', M'N')R_{M_{0}N_{0}}(w)R_{\mfq}^{(M_{0}N_{0})}(w) \, \frac{dw}{2\pi i} \nonumber\\
       &\hspace{40pt} =: \,   \mathrm{U}^{\mathbf{1}}(m,n; h,k) \ + \  \mathrm{U}^{\mathbf{2}}(m,n; h,k). 
\end{align}
The sums $\mathcal{U}^{\mathbf{1}}(h,k)$ and $\mathcal{U}^{\mathbf{2}}(h,k)$ defined analogously to $\mathcal{U}^{(*)}(h,k)$ in (\ref{USummm}). Therefore, 
\begin{align}\label{split-dec}
    \mathcal{U}(h,k) \, = \, \mathcal{U}^{\mathbf{1}}(h, k) \, + \, \mathcal{U}^{\mathbf{2}}(h, k) \, + \, \mathcal{U}^{(\mathbf{r})}(h, k).  
\end{align}

Denote by $\mathcal{U}^{\mathbf{2}}(h,k)_{\mathbf{d}}$ the complementary contribution to $\mathcal{U}^{2}(h,k)$, obtained by replacing the condition $mh\neq nk$ in the definition of $\mathcal{U}^{\mathbf{2}}(h,k)$ with $mh=nk$. Define $\mathcal{U}^{\mathbf{2}}(h,k)_{\circ}$ to be the same sum as $\mathcal{U}^{\mathbf{2}}(h,k)$, but with the condition $mh\neq nk$ removed. In other words, we have 
\begin{align}\label{dec: U2}
    \mathcal{U}^{\mathbf{2}}(h,k) \, =\, \mathcal{U}^{\mathbf{2}}(h,k)_{\circ} \, - \, \mathcal{U}^{\mathbf{2}}(h,k)_{\mathbf{d}}.
\end{align}

        
\subsection{Cancellation: a remnant of divisor switching}\label{Fauxcan}

In this section,  we estimate the size of the bilinear form of $\mathcal{U}^{\mathbf{1}}(h,k)+\mathcal{L}^{(\mathbf{0})}(h,k)_{\mathbf{od}}$,  where
\begin{align}\label{dec: L0}
     \mathcal{L}^{(\mathbf{0})}(h,k) \, = \, \mathcal{L}^{(\mathbf{0})}(h,k)_{\mathbf{od}} \, + \,  \mathcal{L}^{(\mathbf{0})}(h,k)_{\mathbf{d}}.
\end{align} 
The sum $\mathcal{L}^{(\mathbf{0})}(h,k)$ was defined in (\ref{foldbacL}); $\mathcal{L}^{(\mathbf{0})}(h,k)_{\mathbf{od}}$ and $\mathcal{L}^{(\mathbf{0})}(h,k)_{\mathbf{d}}$ denote the contributions from  $mh\neq nk$ and $mh=nk$, respectively. We will handle $\mathcal{L}^{(\mathbf{0})}(h,k)_{\mathbf{d}}$ and $\mathcal{U}^{\mathbf{2}}(h,k)_{\mathbf{d}}$ in Section \ref{compdoubsumGCD}. 

\begin{lem}
The sum $\mathcal{U}^{\mathbf{1}}(h,k)$ can be rewritten as
\begin{align}
    Q\, \mathop{\sum \sum}_{\substack{m,n\ge 1 \\ mh\neq nk}}  \ \frac{\lambda_{\Pi_{1}}(m)\lambda_{\Pi_{2}}(n)}{m^{\frac12 + \alpha} n^{\frac12 + \beta}}  \frac{\phi(\mfq mnhk)}{\mfq mnhk} \Big(\int_{0}^{\infty} \, W(x) H_{\alpha,\beta}\Big(\frac{nk}{mh}, \frac{mn}{(Qx)^\mfr}\Big) \, dx \Big) \sum_{\substack{c\le C \\(c,\, \mfq mnhk)=1 }} \, \frac{\mu(c)}{c}.\nonumber
\end{align}   
\end{lem}

\begin{proof}
   Recall Lemma \ref{lem: U0anapar}. Observe that
\begin{align}
    R_{M_{0}N_{0}}(0)R_{\mfq}^{(M_{0}N_{0})}(0) \, = \, \frac{\phi(\mfq)}{\mfq} \ = \  \frac{\phi(\mfq m_{0}n_{0}h_{0}k_{0})}{\mfq m_{0}n_{0}h_{0}k_{0}}\hspace{15pt} \text{and} \hspace{15pt} R_{1}(0;  g', M'N') \ = \ \frac{\phi(m'n'h'k')}{m'n'h'k'}.
\end{align}
Thus, we have
\begin{align}\label{innerU1}
    \mathrm{U}^{\mathbf{1}}(m,n;h,k) \, = \,  Q \, \frac{\phi(\mfq mnhk)}{\mfq mnhk} \Big(\int_{0}^{\infty} \, \mathcal{W}(x) \, d^{\times} x \Big) \, \sum_{\substack{c\le C \\(c,\, \mfq mnhk)=1 }} \, \frac{\mu(c)}{c}. 
\end{align}
The result readily follows from (\ref{split}) and (\ref{Ffunc}).
\end{proof}

\begin{lem}
We have
     \begin{align}\label{MobL0}  
  \mathcal{L}^{(\mathbf{0})}(h,k)_{\textbf{od}} \, &= \,  -\,   \mathop{\sum\sum}_{\substack{{m,n\ge 1}\\ mh\neq nk}}  \,  \frac{\lambda_{\Pi_{1}}(m)\lambda_{\Pi_{2}}(n)}
        {m^{\frac12 + \alpha} n^{\frac12 + \beta}} \mathrm{L}^{\mathbf{0}}(m,n;h,k), 
\end{align}
where
\begin{align}\label{innerL0}
    \mathrm{L}^{\mathbf{0}}(m,n;h,k) \, := \,  \sum_{\substack{c\le C\\ (c,\, \mfq mnhk)=1}} \, \mu(c) \sum_{\substack{d\ge 1 \\ (d, \, \mfq m nhk) = 1}} \,   \frac{\mathcal{W}(x)}{x} \Big|_{x=\frac{cd}{Q}}.
\end{align}
\end{lem}

\begin{proof}
    We have
\begin{align}
\mathcal{L}^{(\mathbf{0})}(h,k)_{\mathbf{od}} \, = \,  \mathop{\sum\sum}_{\substack{{m,n\ge 1}\\ mh\neq nk}} \,  \frac{\lambda_{\Pi_{1}}(m)\lambda_{\Pi_{2}}(n)}{m^{\frac12 + \alpha} n^{\frac12 + \beta}}  \sum_{\substack{c>C\\ (c, \, \mfq mnhk)=1}}  \sum_{(d, \, \mfq mnhk)=1} \mu(c) W\Big( \frac{cd}{Q}\Big) \   
   H_{\alpha, \beta}\Big(\frac{nk}{mh}, \frac{mn}{\mfq(cd)^{\mfr}}\Big).\nonumber
\end{align}
Grouping the terms of the sums over $c,d$ by $cd=q$, and by M\"obius inversion,
\begin{align*}
    \sum_{(q, \, \mfq mnhk)=1} \, W(\frac{q}{Q})H_{\alpha, \beta}\Big(\frac{nk}{mh}, \frac{mn}{\mfq q^{\mfr}}\Big) \sum_{\substack{c\mid q\\ c>C}} \, \mu(c) \, = \, -  \sum_{(q, \, \mfq mnhk)=1} \, W(\frac{q}{Q})H_{\alpha, \beta}\Big(\frac{nk}{mh}, \frac{mn}{\mfq q^{\mfr}}\Big) \sum_{\substack{c\mid q\\ c\le C}} \, \mu(c).
\end{align*}
Now, re-opening the $q$-sum with $q=cd$, the desired result follows. 
\end{proof}

\begin{prop}\label{prop: auxcanU1L0}
Suppose that  $\sum_{h} |\lambda_{h}|^2/h \ll_{\epsilon} (TQ)^{\epsilon}$. Then the following estimate holds:
    \begin{align}
        \Big|\mathop{\sum \sum}_{h,k \,\le \, (TQ)^{\theta}  }\, &\frac{\lambda_{h}\overline{\lambda_{k}}}{\sqrt{hk}} \, \big( \, \mathcal{U}^{\mathbf{1}}(h,k)+ \mathcal{L}^{(\mathbf{0})}(h,k)_{\mathbf{od}} \big) \Big| \, \preceq \, CT(TQ)^{\mfr/2+\theta}\mathfrak{C}.
    \end{align}
\end{prop}

\begin{proof}
 Partial summation with $\sum_{\substack{d\le x\\ (d, \,\mfq mnhk)=1}} 1 = \frac{\phi(\mfq mnhk)}{\mfq mnhk}x+O((\mfq mnhk)^{\epsilon})$, the $d$-sum of (\ref{MobL0}) equals
\begin{align}
  \frac{\phi(\mfq mnhk)}{\mfq mnhk}  \, \frac{Q}{c}  \,  \int_{\R} \, \mathcal{W}(x)  \, d^{\times}x \  + \   O\big((\mfq mnhk)^{\epsilon}  \max_{x\asymp 1} |\mathcal{W}(x)|\big)\, +\, O\Big( (\mfq mnhk)^{\epsilon} \int_{x\asymp 1} \, \big| \frac{d}{dx} \, \frac{\mathcal{W}(x)}{x}\big| \Big).\nonumber
\end{align}
The two error terms above are clearly of sizes $O(T(\mfq mnhk)^{\epsilon})$ because of (\ref{trivweigh}) and (\ref{Ffunc}). It follows from (\ref{innerU1}) and (\ref{innerL0}) that 
\begin{align}\label{auxcanc}
    \mathrm{U}^{\mathbf{1}}(m,n;h,k) \,- \, \mathrm{L}^{\mathbf{0}}(m,n;h,k) \, = \,  O(CT(\mfq mnhk)^{\epsilon}).
\end{align}
Recall the truncation obtained in (\ref{eq: alwaysefftruncran}). It follows that
\begin{align}
\Big|\mathop{\sum \sum}_{h,k \,\le \, (TQ)^{\theta}  }\, \frac{\lambda_{h}\overline{\lambda_{k}}}{\sqrt{hk}} \, &\big(\mathcal{U}^{\mathbf{1}}(h,k)+ \mathcal{L}^{(\mathbf{0})}(h,k)_{\mathbf{od}} \big) \Big|\nonumber\\
\, &\preceq \, CT\, 
   \mathop{\sum \sum}_{h,k \,\le \, (TQ)^{\theta}  }\, \frac{|\lambda_{h}\lambda_{k}|}{\sqrt{hk}} \,  \mathop{\sum \sum}_{\substack{ m, n\ge 1\\ mn\,\le\, \mathfrak{C}^2(TQ)^{\mfr} \\ mh\neq nk}}  \frac{|\lambda_{\Pi_{1}}(m)||\lambda_{\Pi_{2}}(n)|}{\sqrt{mn}} \nonumber\\
   \, &\preceq\,  CT(TQ)^{\theta}\Big( \sum_{h\le (TQ)^{\theta}  }\, \frac{|\lambda_{h}|^2}{h}\Big)  \sum_{i=1}^{2}\, \sum_{d\,\le\, \mathfrak{C}^2 (TQ)^{\mfr}} \, \frac{1}{\sqrt{d}}\  \sum_{m\mid d} \, |\lambda_{\Pi_{i}}(m)|^2\nonumber\\
   \, &\preceq \, CT (TQ)^{\mfr/2+\theta}\mathfrak{C}\Big( \sum_{h\le (TQ)^{\theta}  }\, \frac{|\lambda_{h}|^2}{h}\Big)  \sum_{i=1}^{2} \, \sum_{m\le \mathfrak{C}^2 (TQ)^\mfr} \, \frac{|\lambda_{\Pi_{i}}(m)|^2}{m}.\nonumber
\end{align}
The result follows from Lemma \ref{lem: ROA}. 
\end{proof}

\subsection{Off-diagonal separation of variables and smoothing}\label{sect: offdiagsepsmo}

\begin{lem}\label{smbiexp}(\cite[Proposition 2]{CIS19})
Let $\delta>0$, 
\begin{align}\label{CIStrickdel}
    g_{\delta}(z) \ :=  \ (e^{\delta z}-e^{-\delta z})/ 2\delta z,  \hspace{10pt} \text{ and } \hspace{10pt}  \mathcal{H}(z, -w)  \ := \ \sqrt{\pi} \, \frac{\Gamma\left(\frac{1+w}{2}\right) \Gamma\left(\frac{z}{2}\right) \Gamma\left(\frac{-w-z}{2}\right)}{\Gamma\left(-\frac{w}{2}\right) \Gamma\left(\frac{1-z}{2}\right) \Gamma\left(\frac{1+w+z}{2}\right)}.
 \end{align}
Then for $-1<\re w<0$, $r \in (0,\infty) -\{1\}$, and $0< c< \re \, (-w)$, we have
     \begin{align}\label{smbino}
   \frac{1}{2\delta} \int_{-\delta}^{\delta} \,  \sum_{\pm} \, |1\pm  e^{\xi} r|^{w} \, d\xi \ = \ \int_{(c)} \ \mathcal{H}(z, -w) g_{\delta}(z) r^{-z} \ \frac{dz}{2\pi i}.
\end{align}
 \end{lem}

Suppose $\re w =-\epsilon$.  We consider an approximation to (\ref{simMell}):
 \begin{align}\label{eq: smallsmthwei}
\sum_{\pm} \, \widetilde{\mathcal{W}_{\delta}^{^\pm}}(w) 
\ &:= \  \sum_{\pm} \, \int_{0}^{\infty} \,  \frac{1}{2\delta} \, \int_{-\delta}^{\delta} \,  \mathcal{W}\Big(\frac{c}{Qg} \frac{|mh \pm e^{\xi} nk|}{x} \Big) \, d\xi \, x^{w} \, d^{\times} x.
\end{align}
 Since $mh\neq nk$, we conclude from (\ref{Ffunc}) and  Lemma \ref{smbiexp} that
\begin{align}\label{locsmcut}
     \sum_{\pm} \, \widetilde{\mathcal{W}_{\delta}^{^\pm}}(w) \, = \,  \Big(\frac{cmh}{Qg}\Big)^{w} \int_{0}^{\infty} \, x W(x) H_{\alpha,\beta}\Big(\frac{nk}{mh}, \frac{mn}{\mfq (Qx)^\mfr}\Big) \int_{(\epsilon/2)} \, \mathcal{H}(z, -w) g_{\delta}(z) \Big(\frac{nk}{mh}\Big)^{-z} \, \frac{dz}{2\pi i} \, x^{-w} \, d^{\times} x.
\end{align}
 The extra weight $g_{\delta}(z)$ guarantees rapid decay of the $z$-integral. Plugging in (\ref{doublecut}), rearranging the order of the integrals, and applying (\ref{newcutoff}), we find that (\ref{locsmcut}) is equal to 
\begin{align}
 \Big(\frac{cmh}{Qg}\Big)^{w} \int_{\mathbb{R}} \, \eta(t) \,  \int_{(\epsilon/2)}  \mathcal{H}(z, -w) g_{\delta}(z) \Big(\frac{nk}{mh}\Big)^{-z+it} \frac{dz}{2\pi i} \int_{(1)} \widetilde{V}_{\alpha,\beta}(s;it) \widetilde{W}(1+\mfr s-w)  \Big(\frac{mn}{\mfq Q^{\mfr}}\Big)^{-s} \,  \frac{ds}{2\pi i}  \,  dt. \nonumber 
\end{align}
From this, and (\ref{USummm}) and (\ref{split}), we define:
\begin{align}\label{U2arrane}
 \mathcal{U}_{\delta}^{\mathbf{2}}(h,k) \, := \,  & \frac{Q}{2}   \  \sum_{\substack{c\le C \\(c,\, \mfq hk)=1 }} \ \frac{\mu(c)}{c} \, \sum_{\substack{m,n\ge 1 \\ mh\neq nk \\ (mn,\, c)=1}}  \frac{\lambda_{\Pi_{1}}(m) \lambda_{\Pi_{2}}(n)}{m^{\frac12 + \alpha} n^{\frac12 + \beta}} \,  \int_{\R} \, \eta(t) \, \int_{(1)}  \, \widetilde{V}_{\alpha, \,\beta}(s;it) \Big(\frac{mn}{\mfq Q^\mfr}\Big)^{-s} 	 \nonumber\\
  &\hspace{30pt} \cdot  \,   \int_{(-\epsilon)} \,    \int_{(\epsilon/2)}   \,   \widetilde{W}(1+\mfr s-w) \zeta(1+w) \mathcal{H}(z, -w) g_{\delta}(z) 	\Big(\frac{cmh}{Qg}\Big)^{w+z-it}  \Big(\frac{cnk}{Qg}\Big)^{-z+it}  \nonumber\\
  &\hspace{120pt} \cdot \,  R_{1}(w, g', M'N')R_{M_{0}N_{0}}(w)R_{\mfq}^{(M_{0}N_{0})}(w) \,   \frac{dz}{2\pi i}  \, \frac{dw}{2\pi i} \, \frac{ds}{2\pi i} \,  dt,
\end{align}
with $g', M', N', M_{0}, N_{0}$ defined in (\ref{MNdecompo0pr}). Taking $\delta = (TQ)^{-100}$ (say), it follows from elementary analysis (or \cite[Lemma 9.1]{BTB22}) that $ \mathcal{U}_{\delta}^{\mathbf{2}}(h,k) =  \mathcal{U}^{\mathbf{2}}(h,k)$ with a negligible error.

\begin{lem}\label{lem: doubintc}
For any $\delta> 0$ and $s\in \C$, we have 
    \begin{align}\label{doubintc}
    \int_{(-\epsilon)}\, \int_{(\epsilon/2)} \,  \big|\widetilde{W}(1+\mfr s-w)  \, \zeta(1+w) \mathcal{H}(z, -w) g_{\delta}(z)\big| \,  |dz|\, |dw| \, \ll_{\epsilon,\,  \re s } \,  (1+ |\im s|)^{1/2+\epsilon} \delta^{-\epsilon}.
\end{align}
\end{lem}

\begin{proof}
    Recall the following elementary inequalities:\footnote{ The implicit constants depend only on the real parts of the variables. The admissible vertical lines for the second and the fourth inequalities satisfy  $0 \le \re w < 1$,   $\re z, \, \re w,\,  \re (z+w) \not\in \mathbb{Z} $, respectively.}
    \begin{align}
       &\hspace{5pt} g_{\delta}(z) \ll  \min\left\{1, \, (\delta|z|)^{-1}\right\}, 
       \hspace{7pt}  \zeta(1+w)  \ll  \log \, (3+|\im w|), \hspace{7pt }\widetilde{W}(1+\mfr s-w)  \ll_{A}  \left(1+ |\im (\mfr s-w)|\right)^{-A}, \nonumber\\
        &\hspace{45pt}\mathcal{H}(z, -w) \ \asymp \ (1+ |\im w|)^{\frac{1}{2}+\re w} (1+ |\im z|)^{-\frac{1}{2}+\re z} (1+|\im (z+w)|)^{-\re (z+w)-\frac{1}{2}}. \nonumber
    \end{align}
In particular, on $\re w =-\epsilon$ and $\re z =\epsilon/2$, we have 
\begin{align}
     |\mathcal{H}(z, -w)  \zeta(1+w)| \, \ll_{\epsilon} \,    \big[(1+ |\im z|)(1+|\im (z+w)|)\big]^{-\frac{1-\epsilon}{2}} \big[\left( 1+ |\im(w-\mfr s)|\right)(1+ |\im s|)\big]^{\frac{1}{2}}. \nonumber
\end{align}
Let  $b:= - \im (z+\mfr s)$ and $y:= \im w$.  WLOG, assume that $b\ge 0$. Then the  $w$-integral is
\begin{align*}
  \ &\ll_{\epsilon}  \,  (1+ |\im z|)^{-\frac{1-\epsilon}{2}}(1+ |\im s|)^{1/2} \, \Big( \int_{-\infty}^{0} + \int_{0}^{b} + \int_{b}^{\infty}\Big) \,  (1+|y|)^{-A} (1+ |y-b|)^{-\frac{1-\epsilon}{2}} \, dy\nonumber\\
   \ &\ll_{\epsilon} \,  (1+ |\im z|)^{-\frac{1-\epsilon}{2}}(1+ |\im s|)^{1/2}(1+b)^{-\frac{1-\epsilon}{2}},
\end{align*}
provided that $A>2$.  As a result,  the integral  (\ref{doubintc}) is 
    \begin{align}\label{z-inttoest}
  \ll_{\epsilon} \,  (1+ |\im s|)^{1/2}\, \int_{(\epsilon/2)} \  \frac{\min\left\{1, \, (\delta|z|)^{-1}\right\}}{(1+ |\im z|)^{\frac{1-\epsilon}{2}} \, (1+|\im (z+ \mfr s)|)^{\frac{1-\epsilon}{2}}} \,  |dz|.
    \end{align}
The portion of (\ref{z-inttoest}) with $|\im z|>\delta^{-1}$ is bounded by
\begin{align}
       (1+ |\im s|)^{1/2} \int_{|y|>\delta^{-1}} \ (\delta |y|)^{-1} \cdot |y|^{-\frac{1-\epsilon}{2}} \cdot |y|^{-\frac{1-\epsilon}{2}} \ dy \, \ll \, \delta^{-\epsilon}(1+ |\im s|)^{1/2}. 
\end{align}
Let $c:= -\im (\mfr s)$. WLOG, assume that $c\ge 0$. The part of (\ref{z-inttoest}) with $|\im z|\le \delta^{-1}$ is bounded by
\begin{align}
     (1+ |\im s|)^{1/2} \, \Big(\int_{-\delta^{-1}}^{0}  +  \int_{0}^{c}  +  \int_{c}^{\delta^{-1}} \Big) \, (1+|y|)^{-\frac{1-\epsilon}{2}} (1+|y-c|)^{-\frac{1-\epsilon}{2}} \, dy 
    \, \ll \,  (1+ |\im s|)^{1/2+\epsilon} \delta^{-\epsilon}. \nonumber
\end{align}
 This completes the proof. 
\end{proof}


\section{Completion of the sums $\mathcal{U}^{\mathbf{2}}(h,k)$ and $\mathcal{L}^{(\mathbf{0})}(h,k)$}\label{compdoubsumGCD}

In this section, we handle the sums $\mathcal{U}^{\mathbf{2}}(h,k)_{\mathbf{d}}$ and  $\mathcal{L}^{\mathbf{(0)}}(h,k)_{\mathbf{d}}$ defined in (\ref{split-dec}) and (\ref{dec: L0}).

\begin{prop}\label{U2dropeq}
Suppose $\lambda_{h}\ll h^{1/2}$. Then the following estimate holds:
\begin{align}
 \mathop{\sum\, \sum}_{h, \, k \, \le\, (TQ)^{\theta}} \,  \frac{\lambda_{h} \overline{\lambda_{k}}}{\sqrt{hk}} \,  \mathcal{U}^{\mathbf{2}}(h,k)_{\mathbf{d}} \, \preceq \, \mathfrak{C} \,(TQ)^{1+\theta}. 
\end{align}
\end{prop}

\begin{proof}
Let $H:=h/(h,k)=H_{0}H'$ and $K:=k/(h,k)=K_{0}K'$, where $H_{0}K_{0}\mid \mfq^{\infty}$ and $(H'K', \mfq)=1$, and recall the notation described in (\ref{splitmnaccortoq})--(\ref{MNdecompo0pr}). Suppose that $mh=nk$. Then we have 
\begin{align}\label{diagspecasesplit}
    m_{0}= K_{0}\ell_{0}, \hspace{10pt} m' = K'\ell', \hspace{10pt} n_{0} = H_{0}\ell_{0}, \hspace{10pt} n' = H'\ell',
\end{align}
for some $\ell_{0}\mid \mfq^{\infty}$ and $(\ell', \mfq)=1$. Observe that $g_{0}=m_{0}h_{0}=n_{0}k_{0}$ and $g'=m'h'=n'k'$. From (\ref{U2arrane}),  
    \begin{align}
  \mathcal{U}_{\delta}^{\mathbf{2}}(h,k)_{\mathbf{d}} \ = \  & \frac{Q}{2}\frac{1}{H^{1/2+\beta}K^{1/2+\alpha}}   \,  \sum_{\substack{c\le C \\(c, \, \mfq hk)=1 }}  \frac{\mu(c)}{c} \, \sum_{\substack{\ell_{0}\mid \mfq^{\infty} \\ (\ell_{0},c)=1}} \, \frac{\lambda_{\Pi_{1}}(K_{0}\ell_{0})\lambda_{\Pi_{2}}(H_{0}\ell_{0})}{\ell_{0}^{1+\alpha+\beta}} \sum_{\substack{(\ell',\, c\mfq)=1}}  \frac{\lambda_{\Pi_{1}}(K'\ell') \lambda_{\Pi_{2}}(H'\ell')}{(\ell')^{1+\alpha+\beta}}\nonumber\\
   &\hspace{25pt}\cdot  \frac{1}{(2\pi i)^3} \iiiint_{\substack{t\in \R\\ \re s=1 \\  \re w=-\epsilon \\  \re z =\epsilon/2 }} \  \eta(t)\widetilde{V}_{\alpha,\beta}(s; it) \Big(\frac{\mfq Q^\mfr}{\pi^{\mfr}HK(\ell_{0}\ell')^2}\Big)^{s} \Big(\frac{c}{Q}\Big)^{w}  \,\widetilde{W}(1+\mfr s-w)	\nonumber\\
  & \hspace{120pt}\cdot  \,  \zeta(1+w) \mathcal{H}(z, -w) g_{\delta}(z)(w)R_{\mfq}^{(1)}(w)R_{1}\left(w; \ell' h'K', 1\right).  \label{U2eqcasefourint}
\end{align}

 Moving the line of integration of the $s$-integral to $\re s =B$, for $B\gg 1$. Using (\ref{stdStirAFEcut}), Lemma \ref{lem: doubintc}, and the bounds $R_{\mfq}^{(1)}(w) \ll_{\epsilon} \mfq^{\epsilon}$, $R_{1}\left(w; \ell' h'K', 1\right) \ll_{\epsilon} \, (\ell' h'k')^{\epsilon}$,  the quadruple integral above is
\begin{align}
    \ \preceq_{B} \  T \, \Big(\frac{HK\ell^2}{(TQ)^\mfr  \mathfrak{C}^2}\Big)^{-B}. \nonumber
\end{align}
Let $A>0$. Take $B> (A+1+ \mfr \vartheta)/\epsilon$ below. The contribution of $\ell > (TQ)^{\mfr/2+\epsilon}\,\mathfrak{C}$ to  $ \mathcal{U}^{\mathbf{2}}(h,k)_{\mathbf{d}}$ is:
\begin{align}
   &\preceq_{B}    \, \frac{Q}{\sqrt{HK}} \,   \sum_{\ell > (TQ)^{\mfr/2+\epsilon}\mathfrak{C}} \, \frac{(HK\ell^2)^{\vartheta}}{\ell}\,  T (HK(TQ)^{\epsilon})^{-B}\Big(\frac{HK\ell^2}{(TQ)^\mfr\mathfrak{C}^2}\Big)^{-\vartheta-\epsilon} \nonumber\\[5pt]
   &\preceq_{B} \,  (TQ)^{1-\epsilon B+ \mfr \vartheta} (HK)^{-1/2-B} \mathfrak{C}^{2\vartheta} \ \preceq_{A} \, (TQ)^{-A}\,  \mathfrak{C}. \nonumber
\end{align}

Next, we handle the terms with $\ell \le (TQ)^{\mfr/2+\epsilon}\mathfrak{C}$. We shift the line of integration of the $s$-integral to  $\re s=\epsilon$. The quadruple integral in (\ref{U2eqcasefourint}) is now $\preceq T$.  Hence, we obtain
\begin{align}
 \Big|  \, \mathop{\sum\, \sum}_{h, \, k \, \le\,  (TQ)^{\theta}}   \, \frac{\lambda_{h} \overline{\lambda_{k}}}{\sqrt{hk}}  \  \mathcal{U}^{\mathbf{2}}(h,k)_{\mathbf{d}} \, \Big| 
 &\preceq  \, TQ\mathfrak{C}\ \mathop{\sum \sum\sum}_{\substack{v, H, K\\ vH,\, vK\le (TQ)^{\theta} \\ (H, K)=1}} \, \frac{|\lambda_{vH}\lambda_{vK}|}{vHK} \, \sum_{\ell\le (TQ)^{\mfr/2}\,\mathfrak{C}} \, \frac{|\lambda_{\Pi_{1}}(K\ell)| |\lambda_{\Pi_{2}}(H\ell)|}{\ell} \nonumber\\
 &\hspace{-70pt}\preceq \, TQ \mathfrak{C}\,  \mathop{\sum \sum}_{m, n\le (TQ)^{\mfr/2+\theta}\, \mathfrak{C}} \, |\lambda_{\Pi_{1}}(m)| |\lambda_{\Pi_{2}}(n)| \mathop{\sum\sum\sum\sum}_{\substack{v, H, K, \ell\\ m=K\ell, \, n=H\ell\\ (H,K)=1\\  vH,\, vK\le (TQ)^{\theta},\   \ell \le (TQ)^{\mfr/2}\, \mathfrak{C} }} \, \frac{|\lambda_{vH}\lambda_{vK}|}{v\ell HK}.\label{complquadsum}
    \end{align} 
Let $m, n\ge 1$. Since $\lambda_{h}\ll h^{1/2}$, the quadruple sum over $v,H,K, \ell$ of (\ref{complquadsum}) is
\begin{align}
   \preceq  \mathop{\sum\sum\sum}_{\substack{H, K, \ell\\ m=K\ell, \, n=H\ell\\  H,\, K\le (TQ)^{\theta},\,  \ell \le (TQ)^{\mfr/2}\, \mathfrak{C} }} \, \frac{1}{\ell HK} \, \sum_{v\le (TQ)^{\theta}/\max\{H,K\}} \, \frac{(v^2HK)^{1/2}}{v}. \label{quadruextendsum}
\end{align}
We extend the range of summation in $H, K, \ell$ of (\ref{quadruextendsum}) at the cost of incurring a factor $(\ell HK)^{\epsilon}$:
\begin{align}
   (\ref{quadruextendsum}) \ &\preceq \   (TQ)^{\theta} \  \mathop{\sum\sum\sum}_{\substack{H, K, \ell\\ m=K\ell, \, n=H\ell }} \, \frac{1}{(\ell HK)^{1+\epsilon}}  \Big(\frac{HK}{\max\{H,K\}^2}\Big)^{1/2} \nonumber\\
    \ &\preceq \, (TQ)^{\theta} \, \prod_{p} \, \mathop{\sum \sum\sum}_{\substack{\ell,\, H, \, K \ge 0 \\ K+\ell =o_{p}(m)\\ H+\ell= o_{p}(n)}} \, \frac{1}{p^{(1+\epsilon)(\ell+H+K)}} \nonumber\\
    \, &\preceq \, (TQ)^{\theta} \, \prod_{p} \, \sum_{r=\max\{o_{p}(m), \, o_{p}(n)\}}^{o_{p}(mn)} \, \frac{1}{p^{r(1+\epsilon)}} \nonumber\\
    \, &\preceq \, (TQ)^{\theta} \, \prod_{p} \, \frac{1}{p^{(1+\epsilon)\max\{o_{p}(m), \, o_{p}(n)\}}} \sum_{r=0}^{\infty} \, \frac{1}{p^{r(1+\epsilon)}} 
    \, \preceq \, \frac{ (TQ)^{\theta}}{[m,n]}. \nonumber
\end{align}

As a result, by the above discussion and Lemma \ref{lem: Galsum}, it follows that
\begin{align}
     \Big|  \, \mathop{\sum\, \sum}_{h, \, k \, \le\,  (TQ)^{\theta}}    \, \frac{\lambda_{h} \overline{\lambda_{k}}}{\sqrt{hk}}  \  \mathcal{U}^{\mathbf{2}}(h,k)_{\mathbf{d}} \, \Big| 
     \, &\preceq \,  (TQ)^{1+\theta}\,  \mathfrak{C}\mathop{\sum \sum}_{m, n\le (TQ)^{\mfr/2+\theta}\mathfrak{C}} \, \frac{|\lambda_{\Pi_{1}}(m)| |\lambda_{\Pi_{2}}(n)|}{\sqrt{mn}} \frac{(m,n)}{\sqrt{mn}}\nonumber\\
      \, &\preceq \,  (TQ)^{1+\theta}\,  \mathfrak{C}\prod_{i=1}^{2} \, \Big(\sum_{m\le (TQ)^{\mfr/2+\theta}\mathfrak{C}} \, \frac{|\lambda_{\Pi_{i}}(m)|^2}{m}\Big)^{\frac{1}{2}}.\nonumber
\end{align}
Our desired result now follows from Lemma \ref{lem: ROA}.
\end{proof}

\begin{prop}\label{L0ddropterms}
Suppose that $\lambda_{h}\ll h^{1/2}$. Then we have
    \begin{align}
     \mathop{\sum \sum}_{h, k \le (TQ)^{\theta}} \,  \frac{\lambda_{h}\overline{\lambda_{k}}}{\sqrt{hk}} \, \mathcal{L}^{(\mathbf{0})}(h, k)_{\mathbf{d}} \, \preceq \,  \mathfrak{C}\, (TQ)^{1+\theta}. 
    \end{align}
\end{prop}

\begin{proof}
From (\ref{MobL0}), we have 
 \begin{align} 
  \mathcal{L}^{(\mathbf{0})}(h,k)_{\mathbf{d}} \, &= \,  -\,  \mathop{\sum \sum}_{\substack{m,n\ge 1\\mh=nk}} \,  \frac{\lambda_{\Pi_{1}}(m)\lambda_{\Pi_{2}}(n)}
        {m^{\frac12 + \alpha} n^{\frac12 + \beta}}  \, \mathop{\sum \sum}_{\substack{c\le C, \, d\ge 1\\ (cd,\, \mfq mnhk)=1}} \, \mu(c) W\Big(\frac{cd}{Q}\Big) H_{\alpha, \beta}\Big(1, \, \frac{mn}{(cd)^\mfr}\Big).
\end{align}
We use the same parametrization of variables as in the proof of Proposition \ref{U2dropeq}. Recall (\ref{eq: alwaysefftruncran}), i.e., the $m,n$-sums are effectively truncated at $mn\le (TQ)^{\mfr+\epsilon}\mathfrak{C}^2$. Therefore, 
\begin{align}
    \mathcal{L}^{(\mathbf{0})}(h,k)_{\mathbf{d}} \, \preceq  \, TQ \mathfrak{C}\mathop{ \sum}_{\substack{HK\ell^2 \le (TQ)^\mfr}\mathfrak{C}^2}   \frac{|\lambda_{\Pi_{1}}(K\ell)||\lambda_{\Pi_{2}}(H\ell)|}{\ell\sqrt{HK}}. 
\end{align}
Similar to before, we have
\begin{align}
     \Big|  \, \mathop{\sum\, \sum}_{h, \, k \, \le\,  (TQ)^{\theta}}   \, \frac{\lambda_{h} \overline{\lambda_{k}}}{\sqrt{hk}}  \  \mathcal{L}^{\mathbf{2}}(h,k)_{\mathbf{d}} \, \Big| \, 
 & \, \preceq  \, TQ \mathfrak{C}  \mathop{\sum \sum\sum \sum}_{\substack{v, H, K, \ell\\ vH,\, vK\le (TQ)^{\theta} \\ HK\ell^2 \le (TQ)^{\mfr/2}\mathfrak{C} \\ (H, K)=1}} \, \frac{|\lambda_{vH}\lambda_{vK}|}{vHK} \, \frac{|\lambda_{\Pi_{1}}(K\ell)||\lambda_{\Pi_{2}}(H\ell)|}{\ell}.  \nonumber
\end{align}
The result follows by the same argument used in the proof of Proposition \ref{U2dropeq}.
\end{proof}


\section{The sum $\mathcal{U}^{\mathbf{2}}(h,k)_{\circ}$: absence of off-diagonal contributions}\label{sect: NoODU2}

In (\ref{U2arrane}),  we apply the change of variables $w\to 1-w$ and (\ref{splitmnaccortoq})--(\ref{MNdecompo0pr}), we obtain 
\begin{align}
 \mathcal{U}_{\delta}^{\mathbf{2}}(h,k)_{\circ} \ = \  & \frac{Q}{2}   \,  \sum_{\substack{c\le C \\(c,\, \mfq hk)=1 }} \ \frac{\mu(c)}{c} \,   \int_{\R} \, \eta(t) \, \int_{(1)}  \,  \widetilde{V}_{\alpha,\beta}(s; it) (\mfq Q^{\mfr})^s \int_{(1+\epsilon)} \,    \int_{(\epsilon/2)}   \,   \widetilde{W}(w+\mfr s) \zeta(2-w) \mathcal{H}(z, w-1)  \nonumber\\
  &\hspace{50pt} \cdot  \,  g_{\delta}(z)	   \Big(\frac{h}{k}\Big)^{z-it}	\Big(\frac{Q}{ch}\Big)^{w-1}  \mathcal{E}_{\mfq}(s; z,w; it )\mathcal{K}^{(\mfq)}_{it}(s; z,w) \, \frac{dz}{2\pi i}  \, \frac{dw}{2\pi i} \, \frac{ds}{2\pi i} \,  dt. \label{eq: U2hksubdel}
\end{align}
In  (\ref{eq: U2hksubdel}), we have the double Dirichlet series
\begin{align}\label{keydoubDS}
\mathcal{K}^{(\mfq)}_{it}(s; z,w) \, := \,     \mathop{\sum \sum}_{\substack{m',n'\ge 1\\ (m'n',\,c\mfq)=1}}  \,  \frac{\lambda_{\Pi_{1}}(m')\lambda_{\Pi_{2}}(n')}{m'^{s+\frac12 + \alpha} n'^{s+\frac12 + \beta}} \Big(\frac{g'}{m'}\Big)^{w-1} \Big(\frac{m'}{n'}\Big)^{z-it} R_{1}\Big(1-w; \,g',\, \frac{m'h'}{g'}\frac{n'k'}{g'}\Big),
\end{align}
as well as the finite Euler product 
\begin{align}\label{ram: finiteEPforU0}
  \mathcal{E}_{\mfq}(s; z,w; it )  :=   \sum_{\substack{m_{0}, n_{0}\mid \mfq^{\infty}\\ (m_{0}n_{0},c)=1}} \,  \frac{\lambda_{\Pi_{1}}(m_{0}) \lambda_{\Pi_{2}}(n_{0})}{m_{0}^{s+\frac12 + \alpha} n_{0}^{s+\frac12 + \beta}}\Big(\frac{g_{0}}{m_{0}}\Big)^{w-1}\Big(\frac{m_{0}}{n_{0}}\Big)^{z-it} R_{\frac{m_{0}n_{0}}{g_{0}^2}}(1-w)R_{\mfq}^{(\frac{m_{0}n_{0}}{g_{0}^2})}(1-w).
\end{align}
The Euler products  $R_{1}$, $R_{m_{0}n_{0}/g_{0}^2}$, and $R_{\mfq}^{(m_{0}n_{0}/g_{0}^2)}$ are defined in Lemma \ref{lem: U0anapar}. In this section, we suppress the dependence on  $h,k,c,\alpha,\beta,\Pi_{1},\Pi_{2}$, which are considered fixed, in our notation. We also maintain the following lines of integration:
\begin{align}\label{wztlineseps}
    \re w \  = \ 1+\epsilon, \hspace{15pt} \re z \, = \,  \epsilon/2, \hspace{15pt} \text{and} \hspace{15pt} t \, \in \, \R.
\end{align}
The series (\ref{keydoubDS}) and (\ref{ram: finiteEPforU0}) converge absolutely for $\re s\gg 1$. The goal of this section is to prove 
\begin{prop}\label{tersm}
We have
\begin{align}
     \Big| \mathop{\sum\, \sum}_{h,  k \, \le\, (TQ)^{\theta}} \,  \frac{\lambda_{h} \overline{\lambda_{k}}}{\sqrt{hk}} \,  \mathcal{U}_{\delta}^{\mathbf{2}}(h,k)_{\circ}\Big| \  \preceq \ (TQ)^{1+\theta} \mathfrak{C}. 
\end{align}
\end{prop}

The main ingredient of the proof of Proposition \ref{tersm} is a delicate local analysis of (\ref{keydoubDS}).  In the following, we write $m,n,h,k,g$ instead of $m',n',h',k',g'$ for the ease of typesetting.


\subsection{Local analysis}\label{sect: U2sumloc}
By Lemma \ref{genEPexp}, we have an Euler product expansion for (\ref{keydoubDS}) of the form
\begin{align}\label{DS3pcsEPexp}
    \mathcal{K}^{(\mfq)}_{it}(s; z,w)  \, = \,    \prod_{p }   \, \mathop{\sum_{m=0}^{\infty}\, \sum_{n=0}^{\infty}} \, \mathcal{K}(p^m, p^n),
\end{align}
where the local constituents $ \mathcal{K}(p^m, p^n)$ can be determined using the following elementary facts:
\begin{align}
  \mathbf{1}_{(mn,\,c\mfq)}=\prod_{p} f_{p}(m,n) \hspace{15pt} &\text{with} \hspace{15pt} f_{p}(m,n) \, = \, 
    \begin{cases}
        \hspace{20pt} 1 \hspace{35pt} \text{if} \hspace{15pt} p \,\nmid\, c\mfq\\
        \mathbf{1}_{o_{p}(mn)=0}  \hspace{15pt} \text{if} \hspace{15pt} p \mid  c\mfq;
    \end{cases} \label{coprimindic}\\[5pt]
g \, &= \,  \prod_{p}  \, p^{\min\{o_{p}(mh), \, o_{p}(nk)\}};\\
    p\mid \frac{mhnk}{g^2} \, \iff \, |o_{p}(mh)-o_{p}(nk)| \, &= \, o_{p}(mh)+o_{p}(nk)-2 \cdot \min\{o_{p}(mh), \, o_{p}(nk)\} \, > \, 0;
\end{align}
\begin{align}
    (R_{1})_{p}(m,n) \, := \, 
    \begin{cases}
       \hspace{25pt} 1-\frac{1}{p^{2-w}} \hspace{45pt} \text{if} \hspace{15pt} o_{p}(mh) \, \neq \, o_{p}(nk)\\[7pt]
        1+ \frac{1}{(p-1) p^{2-w}}- \frac{1}{p-1}  \hspace{17pt} \text{if} \hspace{15pt} o_{p}(mh) \, = \, o_{p}(nk)\,  > \,0\\[10pt]
       \hspace{25pt} 1+ \frac{p^{w-1}-1}{p(p-1)}  \hspace{35pt} \text{if} \hspace{15pt} o_{p}(mh) \, = \, o_{p}(nk)\, = \,0, 
    \end{cases} \label{rewritedivresul}
\end{align}
and
\begin{align}
   R_{1}(1-w; g, mhnk/g^2) \ = \ \prod_{p}  (R_{1})_{p}(m,n).
\end{align}

(1). When $p\mid c\mfq$, we must have $m=n=0$ in (\ref{DS3pcsEPexp}); the double sum is indeed a singleton. Recall that $(c, \mfq)=(c,hk)=1$. If $p\mid c$, then  and $o_{p}(hk)=0$, and the Euler factor at such $p$ is 
\begin{align}\label{simplefactratc}
    1+ \frac{p^{w-1}-1}{p(p-1)}.
\end{align}
At $p\mid \mfq$, the Euler factor is:
\begin{align}\label{ramifiedprKDs}
    \begin{cases}
       \hspace{40pt}1+ \frac{p^{w-1}-1}{p(p-1)} \hspace{90pt} \text{if} \hspace{20pt} o_{p}(h)=o_{p}(k)=0 \\[10pt]
        (p^{o_{p}(h)})^{w-1}
       ( 1+ \frac{1}{(p-1) p^{2-w}}- \frac{1}{p-1}) \hspace{25pt} \text{if} \hspace{20pt} o_{p}(h)=o_{p}(k)>0 \\[10pt]
        \hspace{5pt}p^{\min\{o_{p}(h),\, o_{p}(k)\}})^{w-1}(1-\frac{1}{p^{2-w}})  \hspace{35pt} \text{if} \hspace{20pt} o_{p}(h)\neq o_{p}(k). 
    \end{cases}
\end{align}

(2). When $p\nmid c\mfq$, we obtain
\begin{align}
    \mathcal{K}(p^m, p^n) \, = \,  \frac{\lambda_{\Pi_{1}}(p^m)\lambda_{\Pi_{2}}(p^n)}{p^{m(s+\frac12 + \alpha)} p^{n(s+\frac12 + \beta)}} \Big(\frac{p^{\min\{ m+o_{p}(h), \, n+o_{p}(k)\}}}{p^m}\Big)^{w-1} p^{(z-it)(m-n)} (R_{1})_{p}(p^m, p^n).
\end{align}
(2).(i). If we have $p\nmid c\mfq hk$, then it follows from (\ref{rewritedivresul}) that $\mathop{\sum \sum_{m, n\ge 0}} \, \mathcal{K}(p^m, p^n) $ is given by
\begin{align}
   \Big(1-\frac{1}{p^{2-w}}\Big)&\sum_{\substack{m,n\ge 0\\ m\neq n}} \, \frac{\lambda_{\Pi_{1}}(p^m)\lambda_{\Pi_{2}}(p^n)}{p^{m(s+\frac12 + \alpha)} p^{n(s+\frac12 + \beta)}} \Big(\frac{p^{\min\{ m, \, n\}}}{p^m}\Big)^{w-1} p^{(z-it)(m-n)} \label{maincompnmidchk} \\
    & + \Big( 1+ \frac{1}{(p-1) p^{2-w}}- \frac{1}{p-1} \Big) \sum_{\substack{m,n\ge 1\\ m=n}} \, \frac{\lambda_{\Pi_{1}}(p^m)\lambda_{\Pi_{2}}(p^n)}{p^{m(s+\frac12 + \alpha)} p^{n(s+\frac12 + \beta)}} + \Big(1+ \frac{p^{w-1}-1}{p(p-1)}\Big). \label{nmidchk}
    \end{align}
(2).(ii). If $p\nmid c\mfq$ but $p\mid hk$, the third option of (\ref{rewritedivresul}) cannot occur, and $\mathop{\sum \sum_{m, n\ge 0}} \,  \mathcal{K}(p^m, p^n)$ is
\begin{align}
 &\Big(1-\frac{1}{p^{2-w}}\Big) \sum_{\substack{m,n\ge 0\\ m+o_{p}(h)\neq n+o_{p}(k) }}   \frac{\lambda_{\Pi_{1}}(p^m)\lambda_{\Pi_{2}}(p^n)}{p^{m(s+\frac12 + \alpha)} p^{n(s+\frac12 + \beta)}} \Big(\frac{p^{\min\{ m+o_{p}(h), \, n+o_{p}(k)\}}}{p^m}\Big)^{w-1} p^{(z-it)(m-n)} \label{uneq: midhk}\\
     &\hspace{7pt}+ \,\Big( 1+ \frac{1}{(p-1) p^{2-w}}- \frac{1}{p-1} \Big) \cdot  p^{o_{p}(h)(w-1)} \sum_{\substack{m,n\ge 0\\ m+o_{p}(h)=n+o_{p}(k)}} \frac{\lambda_{\Pi_{1}}(p^m)\lambda_{\Pi_{2}}(p^n)}{p^{m(s+\frac12 + \alpha)} p^{n(s+\frac12 + \beta)}p^{(z-it)(n-m)}}.   \label{eq: midhk}
\end{align}
For reasons that will soon become clear, the Euler factors above can be expressed in terms of
\begin{align}\label{U2varchoi}
        \mbu :=s+w-1/2+ \alpha-z+it, \hspace{20pt}
        \mbv_{0}:= s+1/2 + \beta+z-it, \hspace{20pt}
        \mbs:=1+2s+\alpha+\beta. 
\end{align}



\subsection{Analytic continuation: the primes $p\nmid \mfq chk$}\label{matEu}

The object is (\ref{nmidchk}). We will show that

\begin{prop}\label{lem: ACKit}
Suppose (\ref{wztlineseps}) holds. Then $\mathcal{K}^{(\mfq)}_{it}(s; z,w)$  admits a holomorphic continuation to $\re s> 1000\epsilon$. Moreover, on the region
\begin{align}\label{mbuv0mbsrange}
    \re \mbu \,  > \, 1/2, \hspace{15pt}\re \mbv_{0} \, > \, 1/2,  \hspace{15pt} \text{and} \hspace{15pt} \re \mbs \, > \, 1,
\end{align}
we have
\begin{align}\label{nmidchkexact}
 \prod_{p\,\nmid\, \mfq chk}  \,  ( L_{p}(\mbu, \Pi_{1})L_{p}(\mbv_{0}, \Pi_{2}))^{-1}\sum_{m=0}^{\infty} \, \sum_{n=0}^{\infty} \,  \mathcal{K}(p^m, p^n) \, = \,  O_{\epsilon}(\mathfrak{C}^{\epsilon}). 
\end{align}
\end{prop}

\begin{rem}
Were \textbf{GRC} holds in full, Proposition \ref{lem: ACKit} would follow immediately, since only $(n,m)\in \{(1,0), \, (0,1)\}$ contribute in (\ref{maincompnmidchk})--(\ref{nmidchk}). To obtain an \emph{unconditional} proof, however, one requires a good understanding of the arithmetic of the Dirichlet coefficients. This is already somewhat non-trivial for $\PGL_{2}$. For $\PGL_{3}$, we will use \texttt{Mathematica} for the computations.
\end{rem}

We will maintain that $p\nmid \mfq chk$ and (\ref{mbuv0mbsrange}). Display (\ref{nmidchk}) is $ \mathcal{L}_{p}(\mbs, \, \Pi_{1}\otimes \Pi_{2}) +  O_{\epsilon}(p^{-1-\epsilon}).$ Let
\begin{align}
     \mathbf{L}_{p}(\mbu, \mbs; \Pi_{1}, \Pi_{2} ) \ := \  L_{p}(\mbu, \Pi_{1}) \, - \, \mathcal{L}_{p}(\mbs, \Pi_{1}\otimes \Pi_{2})  \,+ \,  \sum_{n\ge 1} \, \frac{\lambda_{\Pi_{2}}(p^n)}{p^{n\mbs}}\sum_{r\ge 0} \, \frac{\lambda_{\Pi_{1}}(p^{n+r})}{ p^{r\mbu}} .  \label{mathbddLuspi}
\end{align}
This can be rewritten as follows:
\begin{align}
     \mathbf{L}_{p}(\mbu, \mbs; \Pi_{1}, \Pi_{2} )  \, &= \,  (L_{p}(\mbu, \Pi_{1})-1) \,+ \,  \sum_{n\ge 1} \, \frac{\lambda_{\Pi_{2}}(p^n)}{p^{n\mbs}} \Big\{\sum_{r\ge 0} \, \frac{\lambda_{\Pi_{1}}(p^{n+r})}{ p^{r\mbu}}  \, - \, \lambda_{\Pi_{1}}(p^n)\Big\} \nonumber\\
     \,  &= \, \sum_{r\ge 1} \, \frac{\lambda_{\Pi_{1}}(p^{r})}{ p^{r\mbu}} \,+ \,  \sum_{n\ge 1} \, \frac{\lambda_{\Pi_{2}}(p^n)}{p^{n\mbs}}\sum_{r\ge 1} \, \frac{\lambda_{\Pi_{1}}(p^{n+r})}{ p^{r\mbu}}  \nonumber\\
     \ &= \  \mathop{\sum \sum}_{\substack{m,n\ge 0\\ m> n}} \, \frac{\lambda_{\Pi_{1}}(p^m)\lambda_{\Pi_{2}}(p^n)}{p^{m(s+\frac12 + \alpha)} p^{n(s+\frac12 + \beta)}} \Big(\frac{p^{\min\{ m, \, n\}}}{p^m}\Big)^{w-1} p^{(z-it)(m-n)}. \label{U2nmiduneqcase}
\end{align}
The last quantity is precisely the double sum of display (\ref{maincompnmidchk}). Clearly, the double sum in (\ref{U2nmiduneqcase}) taken instead with  $m<n$ is  $\mathbf{L}_{p}(\mbv_{0}, \mbs; \Pi_{2}, \Pi_{1} )$. Hence, up to an error $O_{\epsilon}(p^{-1-\epsilon})$, we have
\begin{align}
      \mathop{\sum_{m=0}^{\infty}\sum_{n=0}^{\infty}}\,  \mathcal{K}(p^m, p^n) \ = \      \Big(1-\frac{1}{p^{2-w}}\Big)&( \mathbf{L}_{p}(\mbu, \mbs; \Pi_{1}, \Pi_{2} )  +  \mathbf{L}_{p}(\mbv_{0}, \mbs; \Pi_{2}, \Pi_{1} ) ) + \mathcal{L}_{p}(\mbs, \, \Pi_{1}\otimes \Pi_{2}). \label{new: nmidchk}
\end{align}
 In the following, we let
\begin{align}\label{U2Dfactdef}
 \mathbf{D}_{p}(\mbu, \mbv_{0}; \, \Pi_{1}, \Pi_{2}) \ := \    L_{p}(\mbu, \Pi_{1})^{-1}+L_{p}(\mbv_{0}, \Pi_{2})^{-1} - L_{p}(\mbu, \Pi_{1})^{-1}L_{p}(\mbv_{0}, \Pi_{2})^{-1}.
\end{align}
The computation is implemented by the \texttt{Mathematica} file \texttt{U2pnmid.nb} (see Appendix \ref{sect:U2Lfuncmatchgl3}).



\subsubsection{The case of $\mathrm{PGL}(2)$}\label{sect: U20withoutGRC}

 The $r$-sum of (\ref{mathbddLuspi}) can be computed by Lemma \ref{Heckasymp}. Hence, we have
\begin{align}
    \mathbf{L}_{p}(\mbu, \mbs; \Pi_{1}, \Pi_{2} ) 
     \, &= \, (L_{p}(\mbu, \Pi_{1})-1) \mathcal{L}_{p}(\mbs, \, \Pi_{1}\otimes \Pi_{2})\, - \, \frac{ L_{p}(\mbu, \Pi_{1})}{p^{\mbu+\mbs}} \sum_{n\ge 0} \, \frac{\lambda_{\Pi_{2}}(p^{n+1})\lambda_{\Pi_{1}}(p^{n})}{p^{n\mbs}}. \nonumber   
\end{align}
By Lemma \ref{lem: diffby1RSn} with $d=1$ and $\Pi_{1}\leftrightarrow \Pi_{2}$, we find that
\begin{align}
  \hspace{-5pt}\mathbf{L}_{p}(\mbu, \mbs; \Pi_{1}, \Pi_{2} ) \, = \,   & \mathcal{L}_{p}(\mbs, \Pi_{1}\otimes \Pi_{2})  (L_{p}(\mbu, \Pi_{1})-1) \\[5pt]
    &\hspace{20pt}-  \mathcal{L}_{p}(\mbs, \Pi_{1}\otimes \Pi_{2})L_{p}(\mbu, \Pi_{1})\zeta_{p}(2\mbs)p^{-\mbu-\mbs} \big( \lambda_{\Pi_{2}}(p)-  \lambda_{\Pi_{1}}(p)p^{-\mbs} \big)\label{Eulfactnmid} \\[5pt]
        \, = \, &\mathcal{L}_{p}(\mbs, \Pi_{1}\otimes \Pi_{2})(L_{p}(\mbu, \Pi_{1})-1) + O_{\epsilon}(p^{-1-\epsilon}).   \label{eq: tribddgl2gl1heck}
\end{align}
We have applied $\vartheta<1/2$ and the fact that $L_{p}(\mbu, \Pi_{1})$, $\zeta_{p}(2\mbs)$, $\mathcal{L}_{p}(\mbs, \Pi_{1}\otimes \Pi_{2})$ are all $O(1)$.  The \texttt{Mathematica} code provided verifies the truth of (\ref{eq: tribddgl2gl1heck}).

Putting Displays (\ref{new: nmidchk}), (\ref{eq: tribddgl2gl1heck}) and (\ref{nmidchk}) together, we have
\begin{align}
\mathop{\sum_{m=0}^{\infty}\sum_{n=0}^{\infty}}\,  \mathcal{K}(p^m, p^n)
 \, &= \,   \mathcal{L}_{p}(\mbs, \Pi_{1}\otimes \Pi_{2})   L_{p}(\mbu, \Pi_{1})L_{p}(\mbv_{0}, \Pi_{2}) \mathbf{D}_{p}(\mbu, \mbv_{0}; \, \Pi_{1}, \Pi_{2})  +O_{\epsilon}(p^{-1-\epsilon}),  \label{comparegl2pnmid}
 \end{align}
 where by (\ref{GL2EP}) the identity $A+B-AB=1-(1-A)(1-B)$, we have
 \begin{align}
     \mathbf{D}_{p}(\mbu, \mbv_{0}; \, \Pi_{1}, \Pi_{2}) \ = \ 1-\frac{\lambda_{\Pi_{1}}(p)\lambda_{\Pi_{2}}(p)}{p^{\mbu+\mbv_{0}}}.
 \end{align}
Since $ L_{p}(\mbu, \Pi_{1})^{-1}, \, L_{p}(\mbv_{0}, \Pi_{2})^{-1} \, \asymp \, 1$,  it follows that
\begin{align}
\hspace{15pt}\mathop{\sum_{m=0}^{\infty}\sum_{n=0}^{\infty}}\,  \mathcal{K}(p^m, p^n) = L_{p}(\mbu, \Pi_{1})L_{p}(\mbv_{0}, \Pi_{2}) \Big\{\mathcal{L}_{p}(\mbs, \Pi_{1}\otimes \Pi_{2}) \Big(1-\frac{\lambda_{\Pi_{1}}(p)\lambda_{\Pi_{2}}(p)}{p^{\mbu+\mbv_{0}}}\Big)+ O_{\epsilon}(p^{-1-\epsilon})\Big\}. \nonumber
\end{align}
Applying the asymptotic expansion (\ref{rawnaiveexpand}) (which requires $\vartheta<1/4$ \footnote{ This is not strictly necessary for proving Lemma \ref{lem: diffEUproLi}, but it is a classical result in analytic number theory, and much stronger bounds are known; see (\ref{KimSarbdd}).}), we obtain
\begin{align}
\mathop{\sum_{m=0}^{\infty}\sum_{n=0}^{\infty}}\,  \mathcal{K}(p^m, p^n) 
    \, &= \, L_{p}(\mbu, \Pi_{1})L_{p}(\mbv_{0}, \Pi_{2}) \Big\{ 1 \, + \, O_{\epsilon}\Big(\frac{|\lambda_{\Pi_{1}}(p) \lambda_{\Pi_{2}}(p)|}{p^{1+\epsilon}}\Big) + O_{\epsilon}(p^{-1-\epsilon}) \Big\}.\nonumber
\end{align}
Now, the desired result follows from Lemma \ref{lem: diffEUproLi}.


\subsubsection{The case of $\mathrm{PGL}(3)$}\label{twistedprinesund: GL3}
The computations are more intricate than the $\mathrm{PGL}(2)$ counterpart, and we rely heavily on \texttt{Mathematica}. Recall (\ref{new: nmidchk}), (\ref{mathbddLuspi}) and Lemma \ref{Heckasymp}. For $p\nmid \mfq chk$, we have
\begin{align}
 L_{p}(\mbu, \Pi_{1})^{-1}&L_{p}(\mbs, \Pi_{1}\otimes \Pi_{2})^{-1}  \mathbf{L}_{p}(\mbu, \mbs; \Pi_{1}, \Pi_{2} ) \nonumber\\[5pt]
& = \, L_{p}(\mbu, \Pi_{1})^{-1} L_{p}(\mbs, \Pi_{1}\otimes \Pi_{2})^{-1}  (L_{p}(\mbu, \Pi_{1})-\mathcal{L}_{p}(\mbs, \Pi_{1}\otimes \Pi_{2}))\nonumber \\[3pt]
  & \hspace{60pt} \,+ \,  L_{p}(\mbs, \Pi_{1}\otimes \Pi_{2})^{-1}  \sum_{n\ge 1} \, \frac{\lambda_{\Pi_{2}}(p^n)}{p^{n\mbs}} \, \cdot \,  L_{p}(\mbu, \Pi_{1})^{-1}\sum_{r\ge 0} \, \frac{\lambda_{\Pi_{1}}(p^{n+r})}{ p^{r\mbu_{1}}}, \label{gl3notdivU2}
  \end{align}
  where
\begin{align}
    & L_{p}(\mbu, \Pi_{1})^{-1}\sum_{r\ge 0} \, \frac{\lambda_{\Pi_{1}}(p^{n+r})}{ p^{r\mbu_{1}}}
     \, = \, \lambda_{\Pi_{1}}(p^{n+1})p^{-\mbu}+\lambda_{\Pi_{1}}(p^{n-1})p^{-2\mbu} +  \lambda_{\Pi_{1}}(p^n)\big\{1- \lambda_{\Pi_{1}}(p)p^{-\mbu}  \big\}. \label{gl3U2: pluginHec}
\end{align}

The \texttt{Mathematica} notebook \texttt{U2pnmid.nb} (Appendix \ref{sect:U2Lfuncmatchgl3}) makes use of (\ref{gl3U2: pluginHec}) and the bound $\vartheta\le 5/14$. The result is as follows:
\begin{align}
    L_{p}(\mbs, \Pi_{1}\otimes \Pi_{2})^{-1} p^{-\mbu} \sum_{n\ge 1}\, \frac{\lambda_{\Pi_{1}}(p^{n+1})\lambda_{\Pi_{2}}(p^n)}{p^{n\mbs}} \ =  \ \mathcal{A}_{p}(\mbu, \mbs; \Pi_{1}, \Pi_{2}) \, + \,  O_{\epsilon}(p^{-1-\epsilon}),
\end{align}
where $ \mathcal{A}_{p}(\mbu, \mbs; \Pi_{1}, \Pi_{2})$ is defined (and output as \texttt{NP1} in the code) as 
\begin{align}
   \frac{(\lambda_{\Pi_{1}}(p)^2-\lambda_{\widetilde{\Pi_{1}}}(p))\lambda_{\Pi_{2}}(p)}{p^{\mbu+\mbs}} \, + \, \frac{(3-|\lambda_{\Pi_{1}}(p)|^2)\lambda_{\Pi_{2}}(p)^2+\lambda_{\Pi_{1}}(p)(2\lambda_{\widetilde{\Pi_{1}}}(p)-\lambda_{\Pi_{1}}(p)^2)\lambda_{\widetilde{\Pi_{2}}}(p)}{p^{\mbu+2\mbs}}.  \label{remainbyKSgl3}
\end{align}

\begin{cla}\label{EPgl3sumclaimbound}
\begin{align}
    \sum_{p\nmid \mfq chk} \, |\mathcal{A}_{p}(\mbu, \mbs; \Pi_{1}, \Pi_{2})| \, \ll_{\epsilon} \,  1 + \log \mathfrak{C}^{\epsilon}.\nonumber
\end{align}    
\end{cla}

\noindent Indeed, by (\ref{satabdd}) (with $\vartheta<1/2$), the inequality $|XY|\ll |X|^2+|Y|^2$, the bounds  (\ref{trivrankcoefbdd}) and (\ref{Lisbound}), observe that the sum over $p$ of the first summand of (\ref{remainbyKSgl3}) is bounded by 
\begin{align}
   \sum_{p\nmid \mfq}\, \frac{|\lambda_{\Pi_{1}}(p)\lambda_{\Pi_{2}}(p)|}{p^{3/2-\vartheta}} \ \ll_{\epsilon} \ \sum_{i=1}^{2} \, \sum_{p\nmid \mfq} \, \frac{|\lambda_{\Pi_{i}}(p)|^2}{p^{1+\epsilon}} \ \ll_{\epsilon} \ \log \mathfrak{C}^{\epsilon},
\end{align}
whereas that for the second summand of (\ref{remainbyKSgl3}) is bounded by 
\begin{align}
  1+ \sum_{i=1}^{2}\,  \sum_{p\nmid \mfq} \, \frac{|\lambda_{\Pi_{i}}(p)|}{p^{5/2-3\vartheta}} \ \ll \   1+ \sum_{i=1}^{2}\,  \sum_{p\nmid \mfq} \, \frac{|\lambda_{\Pi_{i}}(p)|^2}{p^{1+\epsilon}} \ \ll_{\epsilon} \ 1+\log \mathfrak{C}^{\epsilon},
\end{align}
As a result, Claim \ref{EPgl3sumclaimbound} follows.

Our \texttt{Mathematica} code (output as \texttt{NP2}) also shows that 
\begin{align}
     L_{p}(\mbs, \Pi_{1}\otimes \Pi_{2})^{-1} p^{-2\mbu}  \sum_{n\ge 1} \, \frac{\lambda_{\Pi_{2}}(p^n)\lambda_{\Pi_{1}}(p^{n-1})}{p^{n\mbs}}  \, = \, O_{\epsilon}(p^{-1-\epsilon}). \nonumber
\end{align}
From (\ref{gl3U2: pluginHec}) and (\ref{gl3notdivU2}),  we conclude that
\begin{align}
     \mathbf{L}_{p}(\mbu, \mbs; \Pi_{1}, \Pi_{2} )
    \, &= \, \mathcal{L}_{p}(\mbs, \Pi_{1}\otimes \Pi_{2})(  L_{p}(\mbu, \Pi_{1}) -1) \, + \, \mathcal{A}_{p}(\mbu, \mbs; \Pi_{1}, \Pi_{2})L_{p}(\mbu, \Pi_{1})L_{p}(\mbs, \Pi_{1}\otimes \Pi_{2})  \nonumber\\
      &\hspace{30pt} \,- \, \lambda_{\Pi_{1}}(p)p^{-\mbu}  L_{p}(\mbu, \Pi_{1})(\mathcal{L}_{p}(\mbs, \Pi_{1}\otimes \Pi_{2})-1) \, + \, O_{\epsilon}(p^{-1-\epsilon}). \nonumber
\end{align}
Moreover, we have the estimate
\begin{align}
    \sum_{p} \, \Big|\frac{ \lambda_{\Pi_{1}}(p)}{p^{\mbu}}  L_{p}(\mbu, \Pi_{1})(\mathcal{L}_{p}(\mbs, \Pi_{1}\otimes \Pi_{2})-1)\Big| \, &\ll \, \sum_{p} \, \frac{|\lambda_{\Pi_{1}}(p)^2\lambda_{\Pi_{2}}(p)|}{p^{3/2}} \, \ll_{\epsilon} \, \log \mathfrak{C}^{\epsilon}.  \nonumber
\end{align}
 From this, (\ref{new: nmidchk}) and (\ref{comparegl2pnmid}), it follows that
\begin{align}
\mathop{\sum_{m=0}^{\infty}\sum_{n=0}^{\infty}}\,  \mathcal{K}(p^m, p^n) \ = \ &L_{p}(\mbu, \Pi_{1})L_{p}(\mbv_{0}, \Pi_{2}) \cdot \big\{  \mathcal{A}_{p}(\mbu, \mbs; \Pi_{1}, \Pi_{2}) \, + \, \mathcal{A}_{p}(\mbv_{0}, \mbs; \Pi_{2}, \Pi_{1})\nonumber\\
 &\hspace{60pt}  \ + \  \mathcal{L}_{p}(\mbs, \Pi_{1}\otimes \Pi_{2})\mathbf{D}_{p}(\mbu, \mbv_{0}; \, \Pi_{1}, \Pi_{2}) \big\}. \nonumber
\end{align}
By (\ref{rawnaiveexpand}) and (\ref{GL3EP}), observe that $\mathcal{L}_{p}(\mbs, \Pi_{1}\otimes \Pi_{2})\mathbf{D}_{p}(\mbu, \mbv_{0}; \, \Pi_{1}, \Pi_{2})$  is equal to
\begin{align}
   \Big(\sum_{k=0}^{3} \, \frac{\lambda_{\Pi_{1}}(p^k)\lambda_{\Pi_{2}}(p^k)}{p^{k\mbs}}\Big)\Big(1\, - \,\frac{\lambda_{\Pi_{1}}(p)\lambda_{\Pi_{2}}(p)}{p^{\mbu+\mbv_{0}}} +\frac{\lambda_{\widetilde{\Pi_{1}}}(p)\lambda_{\Pi_{2}}(p)}{p^{2\mbu+\mbv_{0}}}  +\frac{\lambda_{\Pi_{1}}(p)\lambda_{\widetilde{\Pi_{2}}}(p)}{p^{\mbu+2\mbv_{0}}}  \Big)\,+ \, O(p^{-1-\epsilon}).\nonumber
\end{align}
By the same argument as Claim \ref{EPgl3sumclaimbound}, the proof of Proposition \ref{lem: ACKit} is complete.


\subsection{The primes $p\mid \mfq chk$}\label{sect: twistedprimes}

Here, we collect the expressions (\ref{ram: finiteEPforU0}), (\ref{simplefactratc}), (\ref{ramifiedprKDs}), (\ref{uneq: midhk})--(\ref{eq: midhk}). For (\ref{ram: finiteEPforU0}), it admits the following Euler product expansion:
\begin{align}
  \mathcal{E}_{\mfq}(s; z,w; it ) \ := \  \prod_{p\mid \mfq} \, \sum_{m,n\ge 0}  \ & \frac{\lambda_{\Pi_{1}}(p^m) \lambda_{\Pi_{2}}(p^n)}{(p^m)^{s+\frac12 + \alpha} (p^n)^{s+\frac12 + \beta}}\Big(\frac{p^{\min\{m,n\}}}{p^m}\Big)^{w-1}(p^{m-n})^{z-it} \nonumber\\
  & \cdot \ \begin{cases}
     \big(1-\frac{1}{p^{1+w}}\big) \big(1+ \frac{p^{-w}-1}{p(p-1)}\big)^{-1} \hspace{10pt} \text{ if } \hspace{10pt} m\, \neq \, n \\[5pt]
\hspace{20pt} \frac{p^{2+w}-2p^{1+w}+1}{p^{2+w}-p^{1+w}-p^w+1} \hspace{40pt} \text{ if } \hspace{7pt} \, m  = \, n.
  \end{cases}
\end{align}
In total, the Euler product to be considered is given by
\begin{align}
    \prod_{p\mid c} \,  &\Big( 1+ \frac{p^{w-1}-1}{p(p-1)}\Big) \cdot \prod_{\substack{p\mid \mfq\\ p\nmid hk}} \, \mathcal{E}_{p}(s; z,w; it )\Big(1+ \frac{p^{w-1}-1}{p(p-1)}\Big) \label{ramcnhkEP}\\
    & \cdot \prod_{\substack{p\mid \mfq\\  o_{p}(h)=o_{p}(k)>0}} \,  \mathcal{E}_{p}(s; z,w; it )  (p^{o_{p}(h)})^{w-1}
       \Big( 1+ \frac{1}{(p-1) p^{2-w}}- \frac{1}{p-1}\Big) \label{ramhkeqEP}\\
       & \cdot \, \prod_{\substack{p\mid \mfq\\  o_{p}(h)\neq o_{p}(k)}} \,  \mathcal{E}_{p}(s; z,w; it )(p^{\min\{o_{p}(h),\, o_{p}(k)\}})^{w-1}\big(1-\frac{1}{p^{2-w}}\big) \cdot \, \prod_{\substack{p\nmid \mfq\\ p\mid hk}} \, \sum_{m\ge 0} \sum_{n\ge 0} \, \mathcal{K}(p^m, p^n).\label{ramnramhkdiv}
\end{align}

\begin{lem}\label{prop: remainplace}
Suppose that (\ref{wztlineseps}) and $\re s>1000\epsilon$ hold. Then the product of all of the Euler factors in (\ref{ramcnhkEP}), (\ref{ramhkeqEP}), and (\ref{ramnramhkdiv}) is $\preceq 1$.
\end{lem}

\begin{proof}
The Euler product over $p\mid c$ is clearly $\ll_{\epsilon} c^{\epsilon}\preceq 1$. By the triangle inequality and dropping the restrictions in the sums of (\ref{uneq: midhk})--(\ref{eq: midhk}), we have
    \begin{align}\label{trivbddtwsprime}
        \sum_{m=0}^{\infty} \, \sum_{n=0}^{\infty} \,  \mathcal{K}(p^m, p^n) \ \ll \  p^{o_{p}(hk)\epsilon} \sum_{\substack{m,n\ge 0 }}   \frac{|\lambda_{\Pi_{1}}(p^m)||\lambda_{\Pi_{2}}(p^n)|}{p^{(m+n)/2} } \ \ll \ p^{o_{p}(hk)\epsilon}
    \end{align}
    where the implied constants are absolute. We have used (\ref{uniformRC}) with $\vartheta<1/2$. Thus,
    \begin{align}\label{twisramprimes}
\prod_{\substack{p\nmid \mfq \\p \mid hk}}\  \sum_{m=0}^{\infty} \, \sum_{n=0}^{\infty} \,  \mathcal{K}(p^m, p^n) \, \ll_{\epsilon} \, (hk)^{\epsilon}  \, \preceq \, 1.
    \end{align}
    The bound $ \mathcal{E}_{\mfq}(s; z,w; it ) \, \preceq \, 1$
    follows from a similar argument.  The desired result follows from the fact that the product of  (\ref{ramcnhkEP}), (\ref{ramhkeqEP}), and (\ref{ramnramhkdiv}) is 
    \begin{align*}
        \preceq \,  \prod_{p\mid \mfq } p^{\min\{o_{p}(h),\, o_{p}(k)\}})^{\epsilon} \, \preceq (h,k)^{\epsilon} \, \preceq 1. 
    \end{align*}
\end{proof}



\subsection{Proof of Proposition \ref{tersm}: bounding $\mathcal{U}^{\mathbf{2}}(h,k)_{\circ}$}

From (\ref{DS3pcsEPexp}), Propositions \ref{lem: ACKit}, and Lemma \ref{prop: remainplace}, the Dirichlet series $\mathcal{K}^{(\mfq)}_{it}(s; z,w)$ admits a holomorphic continuation to $\re s>1000\epsilon$ and (\ref{wztlineseps}), on which the following bound holds: $ |\mathcal{K}^{(\mfq)}_{it}(s; z,w)| \preceq   |L(\mbu, \Pi_{1})L(\mbv_{0}, \Pi_{2})|.$
 Shifting the line of integration for the $s$-integral in (\ref{eq: U2hksubdel}) to $\re s=2000\epsilon$. It follows that
\begin{align}
| \, \mathcal{U}_{\delta}^{\mathbf{2}}(h,k)_{\circ} \, | \, \preceq \     Q   \,   \int_{\R} \, \eta(t) \, \int_{(2000\epsilon)}  \, \int_{(1+\epsilon)} \,    \int_{(\epsilon/2)} \, &|\widetilde{V}_{\alpha, \,\beta}(s;it)|    |\widetilde{W}(w+\mfr s) \zeta(2-w) \mathcal{H}(z, w-1)|  \nonumber\\
  & \hspace{-5pt}\cdot  \,  |g_{\delta}(z)|	   |L(\mbu, \Pi_{1})L(\mbv_{0}, \Pi_{2})| \,  |dz| \, |dw| \, |ds| \,  dt.
\end{align}
By Lemma \ref{lem: doubintc}, the convexity bounds for the $L$-functions and the rapid decay of $\widetilde{V}_{\alpha, \,\beta}(s;it)$ (see (\ref{stdStirAFEcut})), it follows that $\mathcal{U}_{\delta}^{\mathbf{2}}(h,k)_{\circ} \preceq TQ \mathfrak{C}$. We thus arrive at Proposition \ref{tersm}.


\section{The sum $\mathcal{U}^{(\mathbf{r})}(h,k)$: preparatory treatment}\label{sect: UrHLS}

In this and the upcoming sections, we study the following sum defined in (\ref{USummm}) and (\ref{absU0sum}):
\begin{align}
  \mathcal{U}^{(\mathbf{r})}(h,k) \, := \,  &\frac{Q}{2}\  \sum_\pm \, \sum_{r\mid \mfq}
 \sum_{\substack{c\le C \\(c,\, \mfq hk)=1}} \sum_{(e,\mfq)=1}\sum_{(a,\mfq)=1}\sum_{\ell \ge 1}  \ \frac{\mu(r)\mu(c) \mu(e) \mu(a)}{ce}  \frac{1}{\phi(\mnl)}   \sum_{\substack{\psi \, (\mnl) \\ \psi\neq \psi_{0}}}\nonumber\\
 & \cdot \, \mathop{\sum \sum}_{\substack{m,n\ge 1 \\ mh\neq nk \\(mn,c) =1 \\ (e,g)=1, \  a \mid g}} 
 \ \frac{\lambda_{\Pi_{1}}(m)\lambda_{\Pi_{2}}(n)}
  {m^{\frac12 + \alpha} n^{\frac12 + \beta}} \,  \psi\Big(\frac{mh}{g}\Big)\overline{\psi}\Big(\mp \frac{nk}{g}\Big) \, \mathcal{W}\Big(\frac{c}{Qg}\frac{|mh\pm nk|}{\ell}; \frac{nk}{mh}, \, \frac{mn}{\mfq Q^{\mfr}}\Big),  \label{Ursummod}
\end{align} 
 where $g:=(mh,nk)$, $\mnl:=rea\ell$, and $\mathcal{W}(x; v,u)  :=   x     W(x) H_{\alpha,\beta}(v, u/x^\mfr).$


\subsection{Truncating and conductor lowering}\label{sect: dualengthcond}

The following lemma shows that divisor switching reduces the size of the conductor by a factor of $Q$ in the sum\, $\mathcal{U}^{(\mathbf{r})}(h,k)$.

\begin{lem}\label{Urtrunc}
Suppose that $\lambda_{h}\ll h^{1/2}$, $\theta<1$, and $C\ll Q$.  Denote by $\mathcal{U}^{(\mathbf{r})}(h,k)_{>}$ the contribution from the terms with
\begin{align}\label{truncateUrsum}
    a\ell 
\, \gg \,  \frac{C}{Q}(TQ)^{\frac{\mfr}{2}+\theta+\epsilon} \mathfrak{C}  \hspace{20pt} \text{ or } \hspace{20pt}  rae\ell \, \gg  \,  (TQ)^{1000}\mathfrak{C}
\end{align}
to the sum $\mathcal{U}^{(\mathbf{r})}(h,k)$. Then we have
\begin{align*}
 \Big|\mathop{\sum\, \sum}_{h,  k \, \le\, (TQ)^{\theta}} \,  \frac{\lambda_{h} \overline{\lambda_{k}}}{\sqrt{hk}} \, \mathcal{U}^{(\mathbf{r})}(h,k)_{>} \Big| \ \preceq \ \mathfrak{C}(TQ)^{-100}. 
\end{align*}
\end{lem}

\begin{proof}
With $mh\neq nk$, recall from Lemma \ref{effectrunc} that the cut-off function $H_{\alpha, \beta}(\cdots)$ effectively truncates the sums in   $\mathcal{U}^{(\mathbf{r})}(h, k)$  to the ranges (\ref{baltrunc}). Since $a\mid g$ and $W$ is compactly supported, we can restrict the sums over $a$ and $\ell$, at the cost of an arbitrarily small error term, to
\begin{align}\label{supptrun}
	a	\ell \,  \asymp \, \frac{a}{g}\cdot \frac{c | nk \pm mh|}{Q}  \, \ll \,  \mathfrak{C}\,  \frac{C (TQ)^{\frac{\mfr}{2}+\theta+\epsilon}}{Q}.
\end{align}
We now show that terms satisfying (\ref{supptrun}) and $\mnl:=rae\ell \gg (TQ)^{1000}\mathfrak{C} $ contribute negligibly small.

\noindent \textbf{Case 1:} $mh/g \not \equiv  \mp nk/g \ (\bmod\ \mnl)$. Denote by  $\mathcal{U}^{(\mathbf{r})}(h,k)_{>,\, \circ}$ the sum $\mathcal{U}^{(\mathbf{r})}(h, k)$ but with this extra restriction. By orthogonality, the sum over $\psi$ is identical to $-1$. Thus, we have
	\begin{align}
	 \mathop{\sum \sum}_{h,k\le (TQ)^{\theta}} \, \frac{\lambda_{h} \overline{\lambda_{k}}}{\sqrt{hk}} \ \mathcal{U}^{(\mathbf{r})}(h,k)_{>, \, \circ}  \  \preceq \  &TQ \,  \mathop{\sum \sum}_{h,k\le (TQ)^{\theta}} \  \mathop{\sum \sum \sum\sum}_{\substack{r\mid \mfq\\ a\ell  \ll \,  \mathfrak{C}(TQ)^{20} \\ re\gg (TQ)^{400}}} \ \frac{1}{e} \frac{(rea\ell)^{\epsilon}}{rea\ell}  \Big(\sum_{m \ll (TQ)^{\mfr/2+\theta}\mathfrak{C}} \,  \frac{|\lambda_{\Pi}(m)|}{\sqrt{m}}\Big)^2\nonumber\\
        \ \preceq \ &\mathfrak{C}(TQ)^{-100}.\nonumber
		\end{align}

\noindent \textbf{Case 2:}  $mh/g  \equiv  \mp nk/g \ (\bmod\ \mnl)$. The sum over $\psi$ is $\phi(\mnl)- 1$. As $mh\neq nk$ and $g\ge 1$, observe that $|mh\pm nk|/g$ must be a \emph{positive} multiple of $\mnl$, and thus, we have  $\mnl  < 2 \cdot \max\{mh,nk\}$. There is no term with $\mnl \gg    (TQ)^{1000}\mathfrak{C}$  because of (\ref{baltrunc}). This proves Lemma \ref{Urtrunc}. 
\end{proof}

\begin{rem}
The factors $(AEL)^{\epsilon}$ will appear frequently. All such factors can be conveniently absorbed as $\preceq 1$ because of Lemma \ref{Urtrunc}.
\end{rem}


\subsection{Separating the sums over $m,n,h,k$}\label{sepmovar}
The sums over $m,n,h,k$ in  $\mathcal{U}^{(\mathbf{r})}(h,k)$ are entangled by the seemingly simple condition  $g  =  (mh, nk)$. To disentangle them, we introduce auxiliary variables $g_{1}, g_{2}, g_{3}, g_{4}$ define by
\begin{align}\label{auxvar}
g_{1} \, := \, (h,k), \hspace{10pt} g_{2} \, := \, (m,n),   \hspace{10pt} g_{3} \, := \,  (m/g_{2}, k/g_{1}), \hspace{10pt}  g_{4} \, := \,  (n/g_{2}, h/g_{1}),
\end{align}
along with several extra conditions as in  \cite[p. 162]{CIS19},  in addition to $r, a,e,c, \ell$ previously introduced in divisor switching. We make the substitutions\footnote{ In the rest of the paper, the variables $M$, $N$ should be understood independently of their usage in Sections \ref{absDiv}--\ref{U0sumsect}.}
\begin{align}\label{substi}
h \, = \,  g_{1}g_{4}H, \hspace{10pt} k \, = \,  g_{1}g_{3}K, \hspace{10pt} m \, =\,  g_{2}g_{3}M, \hspace{10pt} n \, = \,  g_{2}g_{4}N, \hspace{10pt} g \, =\, g_{1}g_{2}g_{3}g_{4},
\end{align}
under which the sums over $M,N,H,K$ are subject to the coprimality constraints:
\begin{align}\label{coprimcond}
	\begin{cases}
	(M, N) \, = \,  (H, K) \, = \,  (M, K) \, = \, (N, H) \, =  \, 1,\\
	 (H, g_{3}) \, =  \, (K, g_{4}) \, =  \,  (M, g_{4}) \, = \, (N, g_{3}) \, = \, 1, \\
	 \hspace{50pt} (MNHK, rcea\ell) \, = \, 1;
	 \end{cases}
	\end{align}
and
\begin{align}\label{rescond}
	MH  \, \neq \,  NK.
\end{align}
Moreover, the sums over $g_{1}, g_{2}, g_{3}, g_{4}$ are constrained by the conditions:
\begin{align}\label{g-coprdiv}
	\begin{cases}
	\hspace{20pt}  (g_{3}, g_{4}) \, =  \, 1, \\
	(ec, g_{1}g_{2}g_{3}g_{4}) \, =  \, 1, \\
	\hspace{20pt} a \mid g_{1}g_{2}g_{3}g_{4}. 
	\end{cases}
	\end{align}
The above conditions may be read off from those appearing in (\ref{Ursummod}) involving $m,n,h,k$.

We trivially bound the sums over $r$, $c$, $e$, $a$, $\ell$ of (\ref{Ursummod}).
We then localize the $a, e, \ell$-sums to dyadic ranges  $a\sim A, \, e\sim E, \, \ell\sim L$ (incurring a harmless factor $(AEL)^{\epsilon} \preceq 1$) according to Lemma \ref{Urtrunc}:
\begin{align}\label{dyadicaelsum}
     \bigg| \,  \mathop{\sum}_{h, k \ge 1} \ \frac{\lambda_{h}\overline{\lambda_{k}}}{\sqrt{hk}}  \ \mathcal{U}^{(\mathbf{r})}(h, k)\bigg| 
 \ \preceq \   &Q \,   \max_{\substack{AL 
	\, \ll \,  \frac{C}{Q}(TQ)^{\mfr/2+\theta}\mathfrak{C}\\ AEL \ll \mathfrak{C} (TQ)^{1000}  }} \, \frac{1}{AE^2L} \, \sum_{\pm}\,
 \sum_{c\le C }   \,  \frac{1}{c} \nonumber\\
 \, &\hspace{60pt}\cdot \ \sum_{ AEL\ll \mnl\ll \mfq AEL}  \mathop{\sum}_{\substack{rea\ell=\mnl \\ a\sim A \\ e\sim E\\ r\mid \mfq}} \,   \sum_{\substack{\psi \, (\mnl) \\ \psi\neq \psi_{0}}}  \, \mu^2(a)\left|\, \mathcal{U}^{\pm}(c; a, e, \ell; \psi)\right|, 
\end{align}
where
\begin{align}\label{sepsum}
\mathcal{U}^{\pm}(c; a, e, \ell; \psi) \, := \,  \sideset{}{^*}{\sum}_{\substack{g_{1}, g_{2}, g_{3}, g_{4} \\ M, N, H, K}} \  \frac{\lambda_{g_{1}g_{4}H} \overline{\lambda_{g_{1}g_{3} K}}}{\sqrt{g_{1}g_{3} g_{1}g_{4}HK}} &  \frac{\lambda_{\Pi_{1}}(g_{2}g_{3}M) \lambda_{\Pi_{2}}(g_{2}g_{4}N) \psi(MH) \overline{\psi}(\mp NK)}{(g_{2}g_{3}M)^{\frac{1}{2}+\alpha} (g_{2}g_{4}N)^{\frac{1}{2}+\beta}}  \nonumber\\
& \hspace{-20pt}\cdot \  \mathcal{W}\Big(\Big|\frac{cMH}{Q\ell}\pm  \frac{cNK}{Q\ell}\Big|; \, \frac{NK}{MH}, \, \frac{g_{2}g_{3} g_{2}g_{4}MN }{\mfq Q^\mfr}  \Big),
\end{align}
and the sums are constrained by  (\ref{coprimcond}),  (\ref{rescond}) and (\ref{g-coprdiv}).


\subsection{Attaching smooth weights}\label{sect: attachsmwe}

We shall apply Mellin inversion to all variables in $\mathcal{W}(\cdots)$ so as to separate the sums over $M, N, H, K$ (and a couple of others) in (\ref{melsepU}). To truncate the integrals in (\ref{mainredestUr}) using Lemma \ref{SOVmainlem}, it is essential to remember the sizes of variables after divisor switching. For this reason, we attach redundant smooth weights to (\ref{sepsum}) at this stage.  We let 
\begin{align}\label{keyXvalue}
    \mcX \, := \, \mcY\, \frac{C(TQ)^{\mfr/2+\theta}}{QL}\mathfrak{C}  \hspace{15pt} \text{with} \hspace{15pt} \mathcal{Y} \, \gg \, (TQ)^{10\epsilon},
\end{align}
 and let $\Psi_{\mcX}$ be a smooth function on $(0, \infty)$ such that 
\begin{align}\label{func: redund}
    0\le \Psi_{\mcX}\le 1; \hspace{10pt} \Psi_{\mcX} \equiv 1 \ \text{ on } (0, \mcX];  \hspace{10pt} \Psi_{\mcX} \, \equiv\,  0  \text{ on } [\mcX+U, \infty); \hspace{10pt} \Psi_{\mcX}^{(j)} \ll_{j} U^{-j} \ \text{ for any } j\ge 0.
\end{align}
  Here, $U$, $\mcY$ are parameters to be specified in Lemma \ref{SOVmainlem} ($U\asymp \mathcal{X}$) and Proposition \ref{cleanupUbd} ($\mathcal{Y}=(TQ)^{100\epsilon}$, say).  In particular, we have $\mcX \gg A\mcY \succeq 1$. Also,  
\begin{align}\label{newtrans}
\mathfrak{W}_{\mcX}^{\pm}(x,y; v,u) \, := \, 	|x\pm y| W(|x\pm y|)   \Psi_{\mcX}(x)\Psi_{\mcX}(y) H_{\alpha, \beta}(v,u|x\pm y|^{-\mfr}). 
\end{align}

The sums over $M, N$ in  (\ref{sepsum}) are effectively truncated to: 
\begin{align}\label{effecrangMNHK}
	HM \,  \le \, (TQ)^{\frac{\mfr}{2}+\theta} \mathfrak{C} (TQ)^{\epsilon}  \hspace{10pt}  \text{ and }  \hspace{10pt} NK  \, \le \, 2(TQ)^{\frac{\mfr}{2}+\theta}\mathfrak{C} (TQ)^{\epsilon}. 
\end{align}
In this range, observe that
\begin{align}
    \frac{cMH}{Q\ell}, \hspace{5pt}  \frac{cNK}{Q\ell} \,\le \, 2 (\mcX/\mcY)  (TQ)^{\epsilon} \, < \, \mcX,   \nonumber
\end{align}
and hence, 
\begin{align}\label{compweighfunc}
  \mathcal{W}\Big(\Big|\frac{cMH}{Q\ell}\pm  \frac{cNK}{Q\ell}\Big|;  \, \frac{NK}{MH}, \, \frac{g_{2}g_{3} g_{2}g_{4}MN }{\mfq Q^\mfr}   \Big) 
  \, = \,   \mathfrak{W}_{\mcX}^{\pm}\Big(\frac{cMH}{Q\ell}, \frac{cNK}{Q\ell}; \frac{NK}{MH}, \frac{g_{2}g_{3} g_{2}g_{4}MN }{\mfq Q^\mfr} \Big). 
\end{align}
Outside the range (\ref{effecrangMNHK}), both sides of (\ref{compweighfunc}) are arbitrarily small, i.e., $ \preceq_{A}  \mathfrak{C}(TQ)^{-A}$. As a result, up to an arbitrarily small error,  we have
\begin{align}\label{Usumcealaftersmooth}
\mathcal{U}^{\pm}(c; a, e, \ell; \psi) \, =  \,  \sideset{}{^*}{\sum}_{\substack{g_{1}, g_{2}, g_{3}, g_{4} \\ M, N, H, K}} \  \frac{\lambda_{g_{1}g_{4}H} \overline{\lambda_{g_{1}g_{3} K}}}{\sqrt{g_{1}g_{3} g_{1}g_{4}HK}} &  \frac{\lambda_{\Pi_{1}}(g_{2}g_{3}M) \lambda_{\Pi_{2}}(g_{2}g_{4}N) \psi(MH) \overline{\psi}(\mp NK)}{(g_{2}g_{3}M)^{\frac{1}{2}+\alpha} (g_{2}g_{4}N)^{\frac{1}{2}+\beta}}  \nonumber\\
& \hspace{-20pt}\cdot \, \mathfrak{W}_{\mcX}^{\pm}\Big(\frac{cMH}{Q\ell}, \frac{cNK}{Q\ell}; \frac{NK}{MH},  \frac{g_{2}g_{3} g_{2}g_{4}MN }{\mfq Q^\mfr}  \Big). 
\end{align}

Let $ \mathcal{U}^{\pm}(c; a, e, \ell; \psi)_{\circ}$ denote the sum of (\ref{Usumcealaftersmooth}) with the condition $MH\neq NK$ omitted, and let $ \mathcal{U}^{\pm}(c; a, e, \ell)_{\mathbf{d}}$ be the corresponding sum but with $MH=NK$. In particular,
\begin{align*}
    \mathcal{U}^{\pm}(c; a, e, \ell; \psi) \, = \,  \mathcal{U}^{\pm}(c; a, e, \ell; \psi)_{\circ} \, - \,  \mathcal{U}^{\pm}(c; a, e, \ell)_{\mathbf{d}}.
\end{align*}
The sum $\mathcal{U}^{\pm}(c; a, e, \ell)_{\mathbf{d}}$ is indeed independence of $\psi$. Suppose that $MH=NK$. The mutual coprimality of (\ref{coprimcond}) implies that $M=N=H=K=1$, and hence,
\begin{align}\label{Ur: diagcomple}
    \mathcal{U}^{\pm}(c; a, e, \ell)_{\mathbf{d}} =   \sum_{\substack{g_{1}, g_{2}, g_{3}, g_{4}: \\ (\ref{g-coprdiv})\  \text{holds}}}   \frac{\lambda_{g_{1}g_{4}} \overline{\lambda_{g_{1}g_{3} }}}{\sqrt{g_{1}g_{3} g_{1}g_{4}}} &  \frac{\lambda_{\Pi_{1}}(g_{2}g_{3}) \lambda_{\Pi_{2}}(g_{2}g_{4})}{(g_{2}g_{3})^{\frac{1}{2}+\alpha} (g_{2}g_{4})^{\frac{1}{2}+\beta}} \, \mathfrak{W}_{\mcX}^{\pm}\Big(\frac{c}{Q\ell}, \frac{c}{Q\ell}; 1,  \frac{g_{2}g_{3} g_{2}g_{4} }{\mfq Q^\mfr}  \Big).
\end{align}


\subsection{Triple Mellin inversion}\label{tripMellUsum}
To simplify the notation in the discussions below, we set:
\begin{align}\label{mamlu12}
 \mathfrak{a}:= rcea, \hspace{15pt} &\mnl \, := \, rea\ell, \hspace{15pt}  \mathfrak{w} \, := \, 1+s_{1}+s_{2}+z,\nonumber\\   u_{1}:= 1/2+\alpha+it+s_{1}+z,  \hspace{15pt} &u_{2} :=   1/2+\beta-it+s_{2}+z, \hspace{15pt} \mz := 1+ \alpha+\beta+2z.
\end{align}
Moreover, we define
\begin{align}\label{2-varMelint}
    \widetilde{\mathcal{V}}_{\mcX}^{^\pm}(s_{1},s_{2}; z)  \, := \,   \int_{0}^{\infty}  	\int_{0}^{\infty}  \,   |x\pm y|^{\mfr z+1} W(|x\pm y|)  \Psi_{\mcX}(x)\Psi_{\mcX}(y) \,   x^{s_{1}-1} y^{s_{2}-1}  \, dx \, dy.
 \end{align}
We will study this double integral in greater detail in Section \ref{sect: eleintrans}. In (\ref{newtrans}), we open up $H_{\alpha, \beta}(\cdots)$ using (\ref{doublecut}), apply a double Mellin inversion in the variables $x,y$, and make the change of variables $u\to u|x\pm y|^\mfr$, it follows, for  $\sigma_{1}, \sigma_{2}, \sigma_{z} \gg 1$, that $\mathcal{U}^{\pm}(c; a, e, \ell; \psi)_{\circ} $ is equal to
 \begin{align}
    & \psi(\mp 1)\int_{\mathbb{R}}  \int_{(\sigma_{1})} \int_{(\sigma_{2})} \int_{(\sigma_{z})}   \eta(t) \widetilde{V}_{\alpha,\beta}(z;it)\widetilde{\mathcal{V}}^{^\pm}_{\mcX}(s_{1},s_{2}; z) \Big( \frac{c}{Q\ell}\Big)^{-s_{1}-s_{2}} (\mfq Q^\mfr)^z\nonumber\\
	&\hspace{10pt}  \cdot \sum_{\substack{g_{1}, g_{2}, g_{3}, g_{4}: \\ (\ref{g-coprdiv})\  \text{holds}}  } \  \frac{1}{g_{1} g_{2} g_{3} g_{4}  (g_{2}g_{3})^{\alpha+z} (g_{2}g_{4})^{\beta+z}} \, 
    \mathop{\sum_{(H, \,g_{3}\ml)=1} \sum_{(K, \,g_{4}\ml)=1}}_{(H, K)=1} \frac{\lambda_{g_{1}g_{4}H} \overline{\lambda_{g_{1}g_{3}K}} \, \psi(H) \overline{\psi}(K)}{H^{\frac{1}{2}+s_{1}+it}K^{\frac{1}{2}+s_{2}-it}}\nonumber\\
	&\hspace{30pt} \cdot  \mathop{\sum_{(M, \, g_{4}K\ml)=1} \sum_{(N, \, g_{3}H\ml)=1}}_{(M,N)=1} \, \frac{\lambda_{\Pi_{1}}(g_{2}g_{3}M) \psi(M)}{M^{u_{1}}} \frac{\lambda_{\Pi_{2}}(g_{2}g_{4}N) \overline{\psi}(N)  }{N^{u_{2}}} \,  \frac{dz}{2\pi i} \, \frac{ds_{2}}{2\pi i} \,  \frac{ds_{1}}{2\pi i} \, dt.    \label{melsepU}
 \end{align}
We will use ``$\iiiint$'' for the integral signs when the choices of contours are obvious from the context.


\subsection{An elementary integral}\label{sect: eleintrans}

The following is a routine estimation of Mellin integrals, which refines \cite[Lemma 5]{CIS19} and \cite[Lemma 5]{CIS12}.

\begin{lem}\label{SOVmainlem}
Suppose that	$\re s_{1}= \re s_{2}= \re z= \epsilon$. For any $k_{1}, k_{2}\ge 1$, we have 
\begin{align}\label{desiretransbd}
 \widetilde{\mathcal{V}}_{\mcX}^{^\pm}(s_{1},s_{2}; z) \ \ll_{\epsilon,\, k_{1},\, k_{2}} \  \frac{\mcX^{k_{1}-1+2\epsilon} }{ \max\{ |s_{1}|, |s_{2}|\}^{k_{1}} |s_{1}+s_{2}|^{k_{2}}} |z|^{k_{1}+k_{2}}.
	\end{align}
\end{lem}

\begin{proof}
Let $W_{z}(x):= x^{\mfr z+1}W(x)$. Clearly, $\partial_{x}^{k}W_{z} \ll_{k} |z|^{k}$, the integrals (\ref{2-varMelint}) are supported on $|x\pm y|\asymp 1$, and we may assume $|s_{1}|\ge |s_{2}|$. Let $k_{1}, k_{2}$ be positive integers. In the following, all implied constants depend at most on $\epsilon, k_{1}, k_{2}$.

Integrating by parts $k_{1}$ times in the $x$-integral and apply  Leibniz's rule,  we have
\begin{align}
 |\widetilde{\mathcal{V}}_{\mcX}^{^\pm}(s_{1},s_{2}; z)| \, &\ll \, \frac{1}{|s_{1}|^{k_{1}}} \left|\int_{0}^{\infty}  \Psi_{\mcX}(y) y^{s_{2}-1}  \, \int_{0}^{\infty} x^{s_{1}-1} \, x^{k_{1}}\frac{d^{k_{1}}}{dx^{k_{1}}}  W_{z}(|x\pm y|)  \Psi_{\mcX}(x) \, dx \, dy\right| \nonumber\\
 \, &\ll \, \frac{1}{|s_{1}|^{k_{1}}} \sum_{\substack{b+c=k_{1} \\b,c\ge 0}} \, \binom{k_{1}}{b}   \,  \Big|\iint_{(0,\infty)^2}  \,  W_{z}^{(b)}(|x\pm y|) \Psi_{\mcX}^{(c)}(x) \Psi_{\mcX}(y)    x^{s_{1}+k_{1}-1} y^{s_{2}-1} \,     dx \, dy \Big|. \label{keydounint}
\end{align}

Making the changes of variables $x= \omega$ and $y=\omega \tau$, the double integral in (\ref{keydounint}) is equal to
\begin{align}
    \int_{0}^{\infty} \, \tau^{s_{2}-1}\int_{0}^{\infty} \,  W_{z}^{(b)}(\omega|1\pm \tau|) \Psi_{\mcX}^{(c)}(\omega) \Psi_{\mcX}(\omega \tau)    \omega^{s_{1}+s_{2}+k_{1}-1}  \, d\omega \, d\tau.
\end{align}
Now, integrating by parts $k_{2}$ times with respect to $\omega$, this double integral is 
\begin{align}\label{secINBdec}
  & \, \ll \,  \, |s_{1}+s_{2}|^{-k_{2}} \cdot \int_{0}^{\infty} \,  \tau^{\epsilon-1} \int_{0}^{\infty} \, \omega^{k_{1}+k_{2}+2\epsilon-1} \, \Big| \frac{d^{k_{2}}}{d\omega^{k_{2}}} \, \Psi_{\mcX}(\omega \tau)W_{z}^{(b)}(\omega|1\pm \tau|) \Psi_{\mcX}^{(c)}(\omega) \Big|   \, d\omega  \, d\tau. 
\end{align}
Again, by the Leibniz rule, we have
\begin{align}
    \frac{d^{k_{2}}}{d\omega^{k_{2}}} \, \Psi_{\mcX}(\omega \tau) W_{z}^{(b)}(\omega|1\pm \tau|) \Psi_{\mcX}^{(c)}(\omega) 
    \, &\ll \, \frac{|z|^{k_{1}+k_{2}}}{U^{c}} \big( \frac{1+\tau}{U}+ |1\pm \tau|\big)^{k_{2}}. \label{secLeib}
\end{align}
As a result, from (\ref{keydounint}), (\ref{secINBdec}) and (\ref{secLeib}), we have 
\begin{align}
 |\widetilde{\mathcal{V}}_{\mcX}^{^\pm}(s_{1},s_{2}; z)|  \ll  \  &  |s_{1}|^{-k_{1}}|s_{1}+s_{2}|^{-k_{2}} |z|^{k_{1}+k_{2}}\sum_{\substack{b+c=k_{1} \\ b,c\ge 0}} \, \binom{k_{1}}{b} \,  U^{-c}   \nonumber\\
    & \hspace{30pt}
    \cdot \iint_{\substack{\tau>0 \\ \omega < (\mcX+U) \min\{1, \tau^{-1}\} \\ 1\le \omega |1\pm \tau| \le 2}} \,  \Big( \frac{1+\tau}{U}+ |1\pm \tau|\Big)^{k_{2}} \, \frac{\omega^{k_{1}+k_{2}+2\epsilon-1} }{\tau^{1-\epsilon}}  \, d\omega \, d\tau\nonumber\\
    \, \ll \  &   |s_{1}|^{-k_{1}}|s_{1}+s_{2}|^{-k_{2}} |z|^{k_{1}+k_{2}}  \int_{\substack{\tau>0 \\ |1\pm \tau|^{-1} < (\mcX+U) \min\{1, \tau^{-1}\} }} \, \frac{ \big( \frac{1+\tau}{U}+ |1\pm \tau|\big)^{k_{2}}}{ \tau^{1-\epsilon} |1\pm \tau|^{k_{1}+k_{2}+2\epsilon-1}} \,    \, d\tau. \label{lastcritint}
\end{align}

Let's first handle the $(+)$-case. The domain of integration for (\ref{lastcritint}) is given by the union of $\tau>1$, $\tau(1+\tau)^{-1}< \mcX+U$;  and $0<\tau\le 1$, $(1+\tau)^{-1}<\mcX+U$. Since $\mcX\gg (TQ)^{\epsilon}$, this domain is simply $\tau>0$, and we readily have
\begin{align}
 |\widetilde{\mathcal{V}}_{\mcX}^{^+}(s_{1},s_{2}; z)|  
  \, \ll_{\epsilon,\, k_{1},\, k_{2}} \, & |s_{1}|^{-k_{1}}|s_{1}+s_{2}|^{-k_{2}} |z|^{k_{1}+k_{2}}. \nonumber
\end{align}

Let's consider the $(-)$-case. The part of the integration domain with  $\tau>1$ is
contained in $\tau > 1+1/(\mcX+U)$. Taking $U\asymp \mcX$ and using (\ref{geodyad}), the integral over such a domain is
\begin{align}
     &\ll \, \int_{\tau> 1/U} \,  (\tau+1)^{\epsilon-1} \tau^{-(k_{1}+2\epsilon-1)} \, d\tau + U^{-k_{2}} \int_{1/U> \tau> 1/(\mcX+U)}  (\tau+1)^{\epsilon-1} \tau^{-(k_{1}+k_{2}+2\epsilon-1)}\, d\tau \nonumber\\
      \ &\le \, \sideset{}{^d}{\sum}_{\mathcal{T}>1/U} \, \mathcal{T}^{-(k_{1}-1+\epsilon)} \, + \,  U^{-k_{2}}\sideset{}{^d}{\sum}_{1/U>\mathcal{T}>1/(\mcX+U)} \, \mathcal{T}^{-(k_{1}+k_{2}+\epsilon-1)} \ \ll \ \mcX^{k_{1}-1+\epsilon}.  \nonumber
\end{align}

For the part of the integral in Display (\ref{lastcritint}) with $0<\tau\le 1$, the integration domain is $0<\tau< 1-1/(\mcX+U)$. Such an integral is 
\begin{align}
   \,  &\ll \,  \int_{1/U<\tau<1} \, (1-\tau)^{\epsilon-1}\tau^{-(k_{1}+2\epsilon-1)}   \, d\tau \, + \, U^{-k_{2}}\int_{1/(\mcX+U)<\tau<1/U} \, (1-\tau)^{\epsilon-1}\tau^{-(k_{1}+k_{2}+2\epsilon-1)}   \, d\tau \nonumber\\
   \, &\ll \, U^{k_{1}-1+2\epsilon} \int_{0}^{1} \, \tau^{\epsilon-1} \, d\tau \, + \, U^{-k_{2}} (\mcX+U)^{k_{1}+k_{2}+2\epsilon-1} \int_{0}^{1} \, \tau^{\epsilon-1} \, d\tau \, \ll \, \mcX^{k_{1}-1+2\epsilon}. \nonumber
\end{align}
 This completes the proof of this lemma. 
\end{proof}


\subsection{Dropping the off-diagonal restriction}\label{sect: dropoffdiagcon}

\begin{prop}\label{+vedrop}
Suppose that $C \ll Q^{1-\epsilon}$. Then the sum   $\mathcal{U}^{+}(c; a, e, \ell)_{\mathbf{d}}$ is identical to $0$.
\end{prop}

\begin{proof}
Since $c/Q\ell \ll Q^{-\epsilon}$ and $W$ is supported on $[1,2]$, we have $ W(\frac{2c}{Q\ell})=0$. Thus, the weight function $ \mathfrak{W}_{\mcX}^{+}(\frac{c}{Q\ell}, \frac{c}{Q\ell}; 1,  (\frac{\pi}{Q})^\mfr g_{2}g_{3} g_{2}g_{4} )$ of (\ref{newtrans}) and (\ref{Ur: diagcomple}) are zero.
\end{proof}

Showing that the contribution of  $\mathcal{U}^{-}(c; a, e, \ell)_{\mathbf{d}}$ is of acceptable size is more involved. In the following, we will write $\sum_{S}^d := \sum_{S=2^k, \, k\in \Z}$ for the dyadic summation.

\begin{prop}\label{Umindneg}
Suppose that one of the following holds:
\begin{enumerate}
    \item $\lambda_{h}\ll_{\epsilon} h^{\epsilon}$; or

    \item both Condition $(\mathbf{\Lambda})$ (see (\ref{bdd: 24normabstract}))  and Hypothesis $(\mathbf{\Pi}^4)$ (see (\ref{fourthpowerestHec})) hold.
\end{enumerate}
Then the contribution of $\mathcal{U}^{-}(c; a, e, \ell)_{\mathbf{d}}$ to (\ref{dyadicaelsum}) is $ \preceq \ CT(TQ)^{\mfr/2+\theta}\mathfrak{C}.$
\end{prop}

\begin{proof}
Applying Mellin inversion to (\ref{Ur: diagcomple}), followed by shifting the lines of integration for the $s_{1}$, $s_{2}$-integral to $\re s_{1}=\re s_{2}=\epsilon$, we have
\begin{align}
      \mathcal{U}^{-}(c; a, e, \ell)_{\mathbf{d}} \, = \, & \int_{\mathbb{R}}  \int_{(\epsilon)} \int_{(\epsilon)} \int_{(\sigma_{z})}   \eta(t)\, \widetilde{V}_{\alpha,\beta}(z;it)\widetilde{\mathcal{V}}^{^\pm}_{\mcX}(s_{1},s_{2}; z)  \Big( \frac{c}{Q\ell}\Big)^{-s_{1}-s_{2}}(\mfq Q^\mfr)^z \sum_{(g_{1}, ec)=1} \, \frac{1}{g_{1}}\nonumber\\
	&\hspace{-10pt}  \cdot \sum_{\substack{g_{3}, g_{4} \\  (g_{3}, g_{4})  =  1 \\ ( g_{3}g_{4}, ec) =   1 }  } \,  \frac{\lambda_{g_{1}g_{4}} \overline{\lambda_{g_{1}g_{3}}} }{   (g_{3})^{1+\alpha+z} (g_{4})^{1+\beta+z}}  \mathcal{L}_{a, g_{1}}^{(ec)}(\mz; \, \Pi_{1}, \Pi_{2};\, g_{3}, g_{4}) \, 
      \frac{dz}{2\pi i} \, \frac{ds_{2}}{2\pi i} \,  \frac{ds_{1}}{2\pi i} \, dt.  \label{melsepU_diag}
 \end{align}
 where
\begin{align}\label{dropdiagg2sum}
  \mathcal{L}_{a, g_{1}}^{(ec)}(\mz; \, \Pi_{1}, \Pi_{2};\, g_{3}, g_{4}) \ := \  \sum_{\substack{(g_{2},\, ec)=1 \\ a| g_{1}g_{2}g_{3}g_{4}}} \ \frac{\lambda_{\Pi_{1}}(g_{2}g_{3})\lambda_{\Pi_{2}}(g_{2}g_{4})}{(g_{2})^{\mz}} \hspace{25pt} (\mz \, := \,  1+ \alpha+\beta+2z).
\end{align}

Let
\begin{align}\label{gp0: divindic}
    g_{p}(a)  \, := \,    \max\{0, o_{p}(a)-o_{p}(g_{1}g_{3}g_{4})\} \hspace{20pt} \text{and} \hspace{20pt}  a_{0} \, := \,  \frac{a}{(a, g_{1}g_{3}g_{4})}.
\end{align}
As $a$ is \emph{square-free} (see (\ref{dyadicaelsum})), the quantity $g_{p}(a)$ must be either $0$ or $1$. In fact, $g_{p}(a)=1$ $\iff$ $o_{p}(a)=1$ and $p\nmid g_{1}g_{3}g_{4}$ $\iff$ $p\mid a_{0}$. By Lemma \ref{genEPexp}, we have 
\begin{align}
    \mathcal{L}_{a, g_{1}}^{(ec)}(\mz; \, \Pi_{1}, \Pi_{2};\, g_{3}, g_{4}) \, = \,  \mathbf{1}_{(ec,a)=1} \prod_{p\nmid ec} \, \sum_{g_{2}= g_{p}(a)}^{\infty}\, \frac{\lambda_{\Pi_{1}}(p^{g_{2}+o_{p}(g_{3})})\lambda_{\Pi_{2}}(p^{g_{2}+o_{p}(g_{4})})}{(p^{g_{2}})^{\mz}}. \nonumber
\end{align}
Now, we consider $p\mid a_{0}$ or not. When $p\mid a_{0}$, we automatically have $p\nmid ec$ as $(ec,a)=1$, and $o_{p}(g_{3}g_{4})=0$. Using also $(g_{3}, g_{4})=1$ and the fact that $(a_{0}, g_{1}g_{3}g_{4})=1$,  we have
\begin{align}\label{g2eprod}
   \mathcal{L}_{a, g_{1}}^{(ec)}(\mz; \, \Pi_{1}, \Pi_{2};\, g_{3}, g_{4}) \ = \ &\mathcal{L}^{( a_{0}ecg_{3}g_{4})}(\mz, \Pi_{1}\otimes \Pi_{2}) \prod_{\substack{ p\mid g_{3}}} \,  \mathcal{L}_{p}(\mz;\,  \Pi_{1},   \Pi_{2}; \, o_{p}(g_{3}),  0 ) \nonumber\\
   &\hspace{30pt}\cdot \prod_{\substack{ p\mid g_{4}}} \, \mathcal{L}_{p}(\mz;\,  \Pi_{2},   \Pi_{1}; 0, o_{p}(g_{4}))\cdot  \frac{1}{(a_{0})^{\mz}} \prod_{p\mid a_{0}} \, \mathcal{L}_{p}(\mz;\,  \Pi_{1},   \Pi_{2}; 1, 1), 
\end{align}
where the Euler factors are defined in (\ref{naivelocalRSfac}) and (\ref{dshiftnaiveDsdef}).

The absolute convergence of (\ref{dropdiagg2sum}) on $\re \mz >1$ follows from Lemma \ref{naivconvRSDS}.  Notice that the sums over $g_{1}, g_{3}, g_{4}$ are finite. As a consequence, we can shift the line of integration of the $z$-integral in (\ref{melsepU_diag}) to $\re z=\epsilon$. By the triangle inequality and the bound (\ref{stdStirAFEcut}), we have 
\begin{align}
     | \, \mathcal{U}^{-}(c; a, e, \ell)_{\mathbf{d}}| \  \preceq \  & T   \max_{\substack{\re z =\epsilon\\ |\im z|\le (TQ)^{\epsilon}}}  \,  \mathop{\sum \sum\sum}_{\substack{(g_{3}, g_{4})=1\\ (g_{1}g_{3}g_{4}, ec)=1}} \,  \frac{|\lambda_{g_{1}g_{4}} \lambda_{g_{1}g_{3}}| }{   g_{1}g_{3}g_{4}} | \mathcal{L}_{a, g_{1}}^{(ec)}(\mz; \, \Pi_{1}, \Pi_{2};\, g_{3}, g_{4})| \,  \iint_{\substack{\re s_{1}=\epsilon\\ \re s_{2}=\epsilon}}  |\widetilde{\mathcal{V}}^{^\pm}_{\mcX}(s_{1},s_{2}; z)|. \nonumber
 \end{align}
Applying Lemma \ref{SOVmainlem} with $k_{2}$ sufficiently large, observe that the last double integral can be truncated to $|s_{1}+s_{2}|\le (TQ)^{\epsilon}$. With this, the double integral over $\max\{ 1+|s_{1}|, 1+|s_{2}|\} \sim S$ is $\preceq S$. By a dyadic decomposition (with (\ref{geodyad})), it follows that
\begin{align}
     \iint_{(\epsilon)}  |\widetilde{\mathcal{V}}^{^\pm}_{\mcX}(s_{1},s_{2}; z)| \, |ds_{1}||ds_{2}| \, \preceq \, \sideset{}{^d}{\sum}_{1\le S\le  \mcX}  \, 1 +  \mcX\sideset{}{^d}{\sum}_{S> \mcX} \frac{1 }{S } \, \preceq \, 1, 
\end{align}
and
\begin{align}
     | \, \mathcal{U}^{-}(c; a, e, \ell)_{\mathbf{d}}| \  \preceq \  & T   \max_{\substack{\re \mz =1+\epsilon\\ |\im \mz|\le (TQ)^{\epsilon}}}   \mathop{\sum \sum\sum}_{\substack{(g_{3}, g_{4})=1\\ (g_{1}g_{3}g_{4}, ec)=1}} \,  \frac{|\lambda_{g_{1}g_{4}} \lambda_{g_{1}g_{3}}| }{   g_{1}g_{3}g_{4}} | \mathcal{L}_{a, g_{1}}^{(ec)}(\mz; \, \Pi_{1}, \Pi_{2};\, g_{3}, g_{4})|.
 \end{align}
 
By the bounds (\ref{powerRStartfrom1}) and $\vartheta\le 1/2$, we have that
\begin{align}
     \frac{1}{(a_{0})^{\mz}} \prod_{p\mid a_{0}} \, \mathcal{L}_{p}(\mz;\,  \Pi_{1},   \Pi_{2}; 1, 1)   \ \ll \ (a_{0})^{-1+2\vartheta} \ \ll \ 1.
\end{align}
Let $\mathfrak{f}_{\mz; \,\Pi_{1}, \Pi_{2}}$ be the multiplicative function defined in (\ref{convshiftRS}). We now have
\begin{align}
     | \, \mathcal{U}^{-}(c; a, e, \ell)_{\mathbf{d}}| \  \preceq \  & T   \max_{\substack{\re \mz =1+\epsilon\\ |\im \mz|\le (TQ)^{\epsilon}}} \sum_{g_{1}} \, \sum_{g_{3}}  \, \sum_{g_{4}} \,  \frac{|\lambda_{g_{1}g_{4}} \lambda_{g_{1}g_{3}}| }{   g_{1}g_{3}g_{4}} | (\lambda_{\Pi_{1}}*\mathfrak{f}_{\mz; \,\Pi_{1}, \Pi_{2}})(g_{3})| | (\lambda_{\Pi_{2}}*\mathfrak{f}_{\mz; \,\Pi_{2}, \Pi_{1}})(g_{4})|.\nonumber
 \end{align}
 Let $r_{3}=g_{1}g_{3}$ and $r_{4}=g_{1}g_{4}$. By the triangle inequality and multiplicativity, it follows that
\begin{align}
   | \, \mathcal{U}^{-}(c; a, e, \ell)_{\mathbf{d}}| \preceq \  & T  \mathop{\sum \sum}_{r_{3}, r_{4}\le (TQ)^{\theta}} \, \frac{|\lambda_{r_{3}}\lambda_{r_{4}}|}{r_{3}r_{4}}\, \sum_{g_{1}\mid (r_{3}, r_{4})} \, g_{1} \, \sum_{g_{3}\mid r_{3}} \, |(\lambda_{\Pi_{1}}*\mathfrak{f}_{\mz; \,\Pi_{1}, \Pi_{2}})(g_{3})| \,  \sum_{g_{4}\mid r_{4}} \, |(\lambda_{\Pi_{2}}*\mathfrak{f}_{\mz; \,\Pi_{2}, \Pi_{1}})(g_{4})| \nonumber\\
    \ \preceq \ &   T \mathop{\sum \sum}_{r_{3}, r_{4}\le (TQ)^{\theta}}  \, \frac{|\lambda_{r_{3}}|(|\lambda_{\Pi_{1}}|*|\mathfrak{f}_{\mz; \,\Pi_{1}, \Pi_{2}}|*1)(r_{3})}{\sqrt{r_{3}}} \frac{|\lambda_{r_{4}}|(|\lambda_{\Pi_{2}}|*|\mathfrak{f}_{\mz; \,\Pi_{2}, \Pi_{1}}|*1)(r_{4})}{\sqrt{r_{4}}} \frac{(r_{3},r_{4})}{\sqrt{r_{3}r_{4}}}.\nonumber
\end{align}
The case $\lambda_{h}\ll_{\epsilon} h^{\epsilon}$ is much simpler. In the following, suppose that Condition $(\mathbf{\Lambda})$  and Hypothesis $(\mathbf{\Pi}^4)$ hold. By Lemma \ref{lem: Galsum}, we have the inequality
\begin{align}
      | \, \mathcal{U}^{-}(c; a, e, \ell)_{\mathbf{d}}| \, \preceq \,  T  \, \max_{\substack{\Pi\in \{\Pi_{1},\Pi_{2}\}\\\mathfrak{f} \, \in \,  \{\mathfrak{f}_{\mz; \,\Pi_{1}, \Pi_{2}}, \, \mathfrak{f}_{\mz; \,\Pi_{2},  \Pi_{1}}\}}}\sum_{r\le (TQ)^{\theta}} \, \frac{|\lambda_{r}|^2(|\lambda_{\Pi}|*|\mathfrak{f}|*1)(r)^2}{r}  \nonumber
\end{align}
 By Cauchy's inequality, observe that
\begin{align}
     \sum_{r\le (TQ)^{\theta}} \, \frac{|\lambda_{r}|^2(|\lambda_{\Pi}|*|\mathfrak{f}|*1)(r)^2}{r} \, \le \,  \Big( \sum_{r\le (TQ)^{\theta}} \, \frac{|\lambda_{r}|^4}{r}\Big)^{\frac{1}{2}} \Big( \sum_{r\le (TQ)^{\theta}} \, \frac{(|\lambda_{\Pi}|*|\mathfrak{f}|*1)(r)^4}{r}\Big)^{\frac{1}{2}}.\nonumber
\end{align}
The first factor on the right is $\preceq 1$. For the second factor, we use (\ref{convshbddarith}), $\tau(n)\ll_{\epsilon} n^{\epsilon}$, together with H\"older's inequality of exponents $(1/4,1/2, 1/2)$:\footnote{Note: with $\vartheta\le 5/14$, we have $1+4(1-3\vartheta)>0$.}
\begin{align}
    \sum_{r\le (TQ)^{\theta}} \, \frac{(|\lambda_{\Pi}|*|\mathfrak{f}|*1)(r)^4}{r} \ &\le \    \sum_{r\le (TQ)^{\theta}} \frac{1}{r}\Big(\sum_{a\mid r} |\lambda_{\Pi}(a)| \sum_{\substack{b'\mid (r/a)\\(b',\mfq)=1}} \, (b')^{-1+2\vartheta} \sum_{\substack{b_{0}\mid (r/a)\\b_{0}\mid \mfq^{\infty}}} b_{0}^{-1+3\vartheta}\Big)^4 \nonumber\\
    \, &\preceq \,   \sum_{r\le (TQ)^{\theta}} \frac{1}{r}\Big(\sum_{\substack{ab_{0}k=r\\ b_{0}\mid \mfq^{\infty}}} |\lambda_{\Pi}(a)|b_{0}^{-1+3\vartheta} \Big)^4 \nonumber\\
    &\preceq \ \sum_{\substack{a\le (TQ)^{\theta}}} \, \frac{|\lambda_{\Pi}(a)|^4}{a} \sum_{\substack{b_{0}\mid \mfq^{\infty}}} \, \frac{1}{b_{0}^{1+4(1-3\vartheta)}}\ \sum_{k\le (TQ)^{\theta}} \, \frac{1}{k} \, \preceq \, 1. \nonumber
\end{align}
 Therefore, we have established the bound $ | \, \mathcal{U}^{-}(c; a, e, \ell)_{\mathbf{d}}| \  \preceq \   T. $ Its contribution to (\ref{dyadicaelsum}) is
\begin{align*}
\preceq \    \max_{\substack{AL 
	\, \ll \,  \frac{C}{Q}(TQ)^{\mfr/2+\theta}\mathfrak{C} \\ AEL \ll \mathfrak{C} (TQ)^{1000}  }} \, \frac{Q}{AE^2L} \, (AEL)^2T \ \ll \ CT(TQ)^{\mfr/2+\theta}\mathfrak{C}. 
\end{align*}
This completes the proof of Proposition \ref{Umindneg}.
\end{proof}




\section{The $\mathcal{U}^{(\mathbf{r})}(h,k)$-sum: a second application of large sieve}\label{sect: boundUrsum}

This section contains several technicalities, but the final expression (\textbf{Lemma \ref{Trianclean}}) is quite clean.


The first obstacle is the coprimality in the sums over $H$ and $K$.  When applying any large sieve, one must ensure that the coefficients of the Dirichlet polynomial are independent of the harmonics (or the underlying family). In Lemma \ref{Gallager}, the $a_{n}$'s and $N$ must not depend of $q$, $\chi$ and $t$. In the current context of (\ref{melsepU}), no extra factor depending on $H$ (resp. $K$) and \emph{at least one of} $\psi$, $\ell$, $u_{1}$, $u_{2}$ should arise when removing the coprimality conditions: 
\begin{align*}
    (H,K) \, = \, 1, \hspace{20pt} (H, \ell) \, = \, (K,\ell)=1, \hspace{20pt} (M, K)\, = \, (N,H)\, = \, 1.
\end{align*}
These five coprimality conditions will be handled by M\"obius inversion (Section \ref{sect: MobinTwist}).  The final coprimality for the sums over $H, K$ are $(H, g_{3}\ma)=(K, g_{4}\ma)=1$ ($\ma:= rcea$), but there will be new sums, and added coprimality conditions for the sums over $M, N$. 

Moreover, unlike in previous works on the Asymptotic Large Sieve, the sums over $g_{2}, g_{3}, g_{4}$, together with several auxiliary sums, \emph{cannot} be estimated trivially and require more careful treatment.



\subsection{Successive M\"obius inversion}\label{sect: MobinTwist}
We apply M\"obius inversion 
\begin{align*}
    \mathbf{1}_{(H, K)} \ = \ \sum_{d_{1}\mid H, \, d_{1} \mid K} \mu(d_{1})
\end{align*}
in (\ref{melsepU}), followed by the changes of variables  $H\to d_{1}H$ and $K\to d_{1}K$. The new sums over $H$ and $K$ retain the same coprimality conditions, but the $d_{1}$-sum is now restricted to $(d_{1}, g_{3}g_{4}\ml)=1$, and the $M$- and $N$-sums acquire $(MN, d_{1})=1$.  Hence, we can rewrite $\mathcal{U}^{\pm}(c; a, e, \ell; \psi)_{\circ}$  as
\begin{align}\label{nowprepsecMob}
    \    &\psi(\mp 1)\, \iiiint \, \eta(t)\, \widetilde{V}_{\alpha,\beta}(z;it)\widetilde{\mathcal{V}}^{^\pm}_{\mcX}(s_{1},s_{2}; z)   \Big( \frac{c}{Q\ell}\Big)^{-s_{1}-s_{2}} (\mfq Q^\mfr)^z\nonumber\\
	& \hspace{5pt}  \cdot \sum_{\substack{g_{1}, g_{2}, g_{3}, g_{4}: \\ (\ref{g-coprdiv})\  \text{holds}}  } \, \sum_{(d_{1}, \, g_{3}g_{4}\ml)=1} \frac{\mu(d_{1})}{g_{1} g_{2} g_{3} g_{4}  (g_{2}g_{3})^{\alpha+z} (g_{2}g_{4})^{\beta+z} d_{1}^{1+s_{1}+s_{2}}}  \sum_{\substack{(H, \, g_{3}\ma)=1 \\ (H, \ell)=1}} \sum_{\substack{(K, \, g_{4}\ma)=1 \\ (K, \, \ell)=1}}\, \frac{\lambda_{g_{1}g_{4}d_{1}H}\overline{\lambda}_{g_{1}g_{3}d_{1}K} \psi(H) \overline{\psi}(K)}{H^{\frac{1}{2}+s_{1}+it}K^{\frac{1}{2}+s_{2}-it}}\nonumber\\
	&\hspace{20pt} \cdot  \mathop{\sum_{(M, \,g_{4}d_{1}K\ml)=1 } \sum_{(N, \,g_{3}d_{1}H\ml)=1}}_{(M, N)=1}\ \frac{\lambda_{\Pi_{1}}(g_{2}g_{3}M) \psi(M)}{M^{u_{1}}}   \  \frac{\lambda_{\Pi_{2}}(g_{2}g_{4}N) \overline{\psi}(N)  }{N^{u_{2}}}. 
\end{align}

Next, we apply a second set of M\"obius inversions  
\begin{align*}
\mathbf{1}_{(H,\, \ell)}=\sum_{d_{2}\mid \ell,\,  d_{2} \mid H} \mu(d_{2}) \hspace{20pt} \text{and} \hspace{20pt} \mathbf{1}_{(K, \,\ell)}=\sum_{d_{3} \mid \ell, \, d_{3}\mid K} \mu(d_{3})   
\end{align*}
 to (\ref{nowprepsecMob}) with the substitutions $H \to d_{2}H$ and $K\to d_{3}K$ in (\ref{nowprepsecMob}).  The new coprimality conditions introduced are $(d_{2}, g_{3}\ma )=(d_{3}, g_{4}\ma )=1$ for the sums over $d_{2}$, $d_{3}$, as well as $(M, d_{3})=(N, d_{2})=1$ for the sums over $M, N$.  Hence, the sum $\mathcal{U}^{\pm}(c; a, e, \ell; \psi)_{\circ}$ becomes:
\begin{align}\label{nowprepthirMob}
    \    &\psi(\mp 1)\,  \iiiint \, \eta(t)\, \widetilde{V}_{\alpha,\beta}(z;it)\widetilde{\mathcal{V}}^{^\pm}_{\mcX}(s_{1},s_{2}; z)  \Big( \frac{c}{Q\ell}\Big)^{-s_{1}-s_{2}} (\mfq Q^\mfr)^z\nonumber\\
	& \hspace{5pt}  \cdot \sum_{\substack{g_{1}, g_{2}, g_{3}, g_{4}: \\ (\ref{g-coprdiv})\  \text{holds}}  } \, \sum_{(d_{1},\, g_{3}g_{4}\ml)=1} \sum_{\substack{d_{2}\mid \ell \\ (d_{2}, \, g_{3}\ma)=1}} \sum_{\substack{d_{3}\mid \ell \\ (d_{3}, \, g_{4}\ma)=1}} \frac{\mu(d_{1})(\mu\psi)(d_{2})(\mu\overline{\psi})(d_{3})}{g_{1} g_{2} g_{3} g_{4}  (g_{2}g_{3})^{\alpha+z} (g_{2}g_{4})^{\beta+z} d_{1}^{1+s_{1}+s_{2}} d_{2}^{1/2+s_{1}+it} d_{3}^{1/2+s_{2}-it}} \nonumber\\
    &\hspace{60pt} \cdot \sum_{(H, \, g_{3}\ma)=1} \sum_{(K, \, g_{4}\ma)=1}\, \frac{\lambda_{g_{1}g_{4}d_{1}d_{2}H}\overline{\lambda}_{g_{1}g_{3}d_{1}d_{3}K} \psi(H) \overline{\psi}(K)}{H^{\frac{1}{2}+s_{1}+it}K^{\frac{1}{2}+s_{2}-it}}   \nonumber\\
	&\hspace{100pt} \cdot \mathop{\sum_{\substack{(M, \,g_{4}d_{1}d_{3}\ml)=1 \\(M,K)=1 }} \sum_{\substack{(N, \,g_{3}d_{1}d_{2}\ml)=1 \\ (N,H)=1}}}_{(M, N)=1}\ \frac{\lambda_{\Pi_{1}}(g_{2}g_{3}M) \psi(M)}{M^{u_{1}}}   \  \frac{\lambda_{\Pi_{2}}(g_{2}g_{4}N) \overline{\psi}(N)  }{N^{u_{2}}}. 
\end{align}

A third and final set of M\"obius inversion 
\begin{align*}
    \mathbf{1}_{(M,K)=1} \, = \,  \sum_{r_{1}\mid M, \, r_{1} \mid K} \, \mu(r_{1}) \hspace{20pt} \text{ and } \hspace{20pt} \mathbf{1}_{(N,H)=1} \, = \,  \sum_{r_{2}\mid N, \, r_{2} \mid H} \, \mu(r_{2})
\end{align*}
with the changes $M\to r_{1}M$, $K\to r_{1}K$; $N\to r_{2}N$, $H\to r_{2}H$ allow us to interchange the order of the double sum over $(M, N)$ and that over $(H, K)$. This introduces restrictions $(r_{2}, g_{3}\ma)=(r_{1}, g_{4}\ma)=1$ (already included), $(r_{1}, g_{4}d_{1}d_{3}\ml)=(r_{2}, g_{3}d_{1}d_{2}\ml)=1$, $(r_{1}, r_{2})=1$ to the $r_{1}, r_{2}$-sums; and $(M, r_{2})=(N,r_{1})=1$ to the $M, N$-sums. Hence, we turn the sum $\mathcal{U}^{\pm}(c; a, e, \ell; \psi)_{\circ}$ into: 
\begin{align}\label{UafterMobmn}
    \    &\psi(\mp 1)\,  \iiiint \, \eta(t)	\, \widetilde{V}_{\alpha,\beta}(z;it)\widetilde{\mathcal{V}}^{^\pm}_{\mcX}(s_{1},s_{2}; z)   \Big( \frac{c}{Q\ell}\Big)^{-s_{1}-s_{2}} (\mfq Q^\mfr)^z  \sum_{\substack{g_{1}, g_{2}, g_{3}, g_{4}: \\ (\ref{g-coprdiv})\  \text{holds}}  } \, \sum_{(d_{1}, \,g_{3}g_{4}\ml)=1} \sum_{\substack{d_{2}\mid \ell \\ (d_{2}, \, g_{3}\ma)=1}} \sum_{\substack{d_{3}\mid \ell \\ (d_{3}, \, g_{4}\ma)=1}} \nonumber\\
	& \hspace{5pt} \cdot \,  \mathop{\sum_{(r_{1}, \, g_{4}d_{1}d_{3}\ml)=1} \, \sum_{(r_{2}, \, g_{3}d_{1}d_{2}\ml)=1}}_{(r_{1}, r_{2})=1} \, \frac{\mu(d_{1})(\mu\psi)(d_{2})(\mu\overline{\psi})(d_{3}) \mu(r_{1}) \mu(r_{2})}{g_{1} g_{2} g_{3} g_{4}  (g_{2}g_{3})^{\alpha+z} (g_{2}g_{4})^{\beta+z} d_{1}^{1+s_{1}+s_{2}} d_{2}^{1/2+s_{1}+it} d_{3}^{1/2+s_{2}-it} r_{1}^{\mathfrak{w}+\alpha} r_{2}^{\mathfrak{w}+\beta}}  \nonumber\\
	&\hspace{35pt} \cdot \Big( \mathop{\sum_{(M, \, r_{2}g_{4}d_{1}d_{3}\ml)=1 } \sum_{(N, \, r_{1}g_{3}d_{1}d_{2}\ml)=1}}_{(M, N)=1}\ \frac{\lambda_{\Pi_{1}}(g_{2}g_{3}r_{1}M) \psi(M)}{M^{u_{1}}}  \frac{\lambda_{\Pi_{2}}(g_{2}g_{4}r_{2}N) \overline{\psi}(N)  }{N^{u_{2}}} \Big) \nonumber\\
    & \hspace{75pt} \cdot \, \Big(  \sum_{(H, \, g_{3}\ma)=1} \sum_{(K, \, g_{4}\ma)=1}\, \frac{\lambda_{g_{1}g_{4}d_{1}d_{2}r_{2}H}\overline{\lambda}_{g_{1}g_{3}d_{1}d_{3}r_{1}K} \psi(H) \overline{\psi}(K)}{H^{\frac{1}{2}+s_{1}+it}K^{\frac{1}{2}+s_{2}-it}} \Big).
\end{align}

\begin{rem}
    The sums over $H,K$, $g_{1}, g_{3}, g_{4}$, $d_{i}, r_{j}$'s are finite by the support of the arbitrary coefficients. The $d_{i}$'s can be considered fixed throughout as they will be treated trivially, whereas $r_{j}$'s are essential as they occur inside the Dirichlet coefficients.
\end{rem}


\subsection{A triple Dirichlet series and its Euler product}\label{sect: TDSandEP}

On the region of absolute convergence, we arrange the sums of (\ref{UafterMobmn}) as follows:
\begin{align}\label{rigaftsket}
 &\psi(\mp 1)\, \iiiint \, \eta(t)\widetilde{V}_{\alpha,\beta}(z;it)\widetilde{\mathcal{V}}^{^\pm}_{\mcX}(s_{1},s_{2}; z)  \left( \frac{c}{Q\ell}\right)^{-s_{1}-s_{2}}(\mfq Q^\mfr)^z\sum_{(g_{1}, \, ec)=1} \,  \sum_{(d_{1}, \, \ml)=1} \,  \sum_{\substack{d_{2}\mid \ell, \, d_{3} \mid \ell \\ (d_{2}d_{3}, \, \mathfrak{a})=1}} \,   \nonumber\\
    &\hspace{30pt}   \frac{\mu(d_{1})(\mu\psi)(d_{2})(\mu\overline{\psi})(d_{3}) }{g_{1}  d_{1}^{1+s_{1}+s_{2}} d_{2}^{1/2+s_{1}+it} d_{3}^{1/2+s_{2}-it} } \mathop{\sum_{(g_{3}, \, d_{1}d_{2}ec)=1} \,  \sum_{(g_{4}, \, d_{1}d_{3}ec)=1}}_{(g_{3}, g_{4})=1} \, \mathop{\sum_{(r_{1}, \, g_{4}d_{1}d_{3}\ml)=1} \, \sum_{(r_{2}, \,  g_{3}d_{1}d_{2}\ml)=1}}_{(r_{1}, r_{2})=1} \,  \frac{ \mu(r_{1}) \mu(r_{2})  }{g_{3}^{1+\alpha+z} g_{4}^{1+\beta+z}   r_{1}^{\mathfrak{w}+\alpha} r_{2}^{\mathfrak{w}+\beta}}   \nonumber\\
    &\hspace{60pt} \,  \cdot \, \mathcal{L}_{\textbf{d,r,g}}(u_{1}, u_{2}, \mz) \, \sum_{(H, \, g_{3}\mathfrak{a})=1} \sum_{(K, \, g_{4}\mathfrak{a})=1}\, \frac{\lambda_{g_{1}g_{4}d_{1}d_{2}r_{2}H}\overline{\lambda}_{g_{1}g_{3}d_{1}d_{3}r_{1}K} \psi(H) \overline{\psi}(K)}{H^{\frac{1}{2}+s_{1}+it}K^{\frac{1}{2}+s_{2}-it}}, 
\end{align}
where 
\begin{align}\label{tripDSuponglu} 
  \mathcal{L}_{\textbf{d,r,g}}(u_{1}, u_{2}, \mz) \, := \,  \sum_{\substack{(g_{2}, \, ec)=1 \\ a\mid g_{1}g_{2}g_{3}g_{4}}}  \mathop{\sum_{(M, \, r_{2}g_{4}d_{1}d_{3}\ml)=1 } \sum_{(N, \, r_{1}g_{3}d_{1}d_{2}\ml)=1}}_{(M, N)=1}\ \frac{1}{(g_{2})^{\mz}}\frac{\lambda_{\Pi_{1}}(g_{2}g_{3}r_{1}M) \psi(M)}{M^{u_{1}}} & \frac{\lambda_{\Pi_{2}}(g_{2}g_{4}r_{2}N) \overline{\psi}(N)  }{N^{u_{2}}},  
\end{align}
and the dependence on $c,e,a,\ell, \Pi_{1}, \Pi_{2}$ are suppressed. Glueing variables by  
\begin{align}\label{defofabstcoef}
    \mathfrak{k}_{1}\ = \ g_{3}r_{1} \hspace{15pt} \text{ and } \hspace{15pt}\mathfrak{k}_{2} \ = \ g_{4}r_{2}. 
\end{align}
Then  $(\mk_{1}, \mk_{2})=1$ $\iff$ $(g_{3},g_{4})=(g_{3}, r_{2})=(r_{1}, g_{4})=(r_{1}, r_{2})=1$. Then  $\mathcal{U}^{\pm}(c; a, e, \ell; \psi)_{\circ}$ becomes
\begin{align}
 &\psi(\mp 1)\mu^2(a)\, \iiiint \, \eta(t)\widetilde{V}_{\alpha,\beta}(z;it)\widetilde{\mathcal{V}}^{^\pm}_{\mcX}(s_{1},s_{2}; z) \left( \frac{c}{Q\ell}\right)^{-s_{1}-s_{2}}(\mfq Q^\mfr)^z   \sum_{(g_{1}, \, ec)=1} \,  \sum_{(d_{1}, \, \ml)=1} \,  \sum_{\substack{d_{2}\mid \ell, \, d_{3} \mid \ell \\ (d_{2}d_{3}, \, \mathfrak{a})=1}} \,  \nonumber\\
    &\hspace{30pt}  \,  \frac{\mu(d_{1})(\mu\psi)(d_{2})(\mu\overline{\psi})(d_{3}) }{g_{1}  d_{1}^{1+s_{1}+s_{2}} d_{2}^{1/2+s_{1}+it} d_{3}^{1/2+s_{2}-it} }  \mathop{\sum_{\mk_{1}} \, \sum_{\mathfrak{k}_{2}}}_{(\mathfrak{k}_{1}, \mathfrak{k}_{2})=1} \,   \mathop{\sum \sum}_{\substack{(g_{3}, \, d_{1}d_{2}ec)=1 \\ (r_{1}, \, d_{1}d_{3}\ml)=1 \\ g_{3}r_{1}=\mk_{1}}} \,  \frac{ \mu(r_{1}) }{g_{3}^{1+\alpha+z}   r_{1}^{\mathfrak{w}+\alpha}}  \mathop{\sum \sum}_{\substack{(g_{4}, \, d_{1}d_{3}ec)=1 \\ (r_{2}, \, d_{1}d_{2}\ml)=1\\ g_{4}r_{2}=\mk_{2}}} \, \frac{ \mu(r_{2})   }{ g_{4}^{1+\beta+z}  r_{2}^{\mathfrak{w}+\beta}}\,   \nonumber\\
    &\hspace{120pt} \cdot \,  \mathcal{L}_{\textbf{d,r,g}}(u_{1}, u_{2},\mathfrak{z}) \, \sum_{(H, \, g_{3}\mathfrak{a})=1} \sum_{(K, \, g_{4}\mathfrak{a})=1}\, \frac{ \lambda_{g_{1}d_{1}d_{2}\mk_{2}H}\overline{\lambda}_{g_{1}d_{1}d_{3}\mk_{1}K} \psi(H) \overline{\psi}(K)}{H^{\frac{1}{2}+s_{1}+it}K^{\frac{1}{2}+s_{2}-it}},  \label{grouptripDS} 
\end{align}
and
\begin{align}
    \mathcal{L}_{\textbf{d,r,g}}(u_{1}, u_{2}, \mz) \, =\,  \sum_{\substack{(g_{2},\, ec)=1 \\ a\mid g_{1}g_{2}g_{3}g_{4}}}  \mathop{\sum_{(M, \, \mk_{2} d_{1}d_{3}\ml)=1 } \sum_{(N, \, \mk_{1} d_{1}d_{2}\ml)=1}}_{(M, N)=1}\ \frac{1}{(g_{2})^{\mz}}\frac{\lambda_{\Pi_{1}}(g_{2}\mk_{1} M) \psi(M)}{M^{u_{1}}}  \frac{\lambda_{\Pi_{2}}(g_{2}\mk_{2} N) \overline{\psi}(N)  }{N^{u_{2}}}.  \label{triplecontDS} 
\end{align}

\begin{rem}\label{almostgrouperf}
The condition $a\mid g_{1}g_{2}g_{3}g_{4}$ causes inconvenience. Were it not present, we could freely pull the triple Dirichlet series $ \mathcal{L}_{\textbf{d,r,g}}(u_{1}, u_{2}, \mz) $ out of the sums over $g_{3}, r_{1}, g_{4}, r_{2}$. However, this condition turns out to be crucial for the factor of $1/a$ in Lemmas \ref{Trianclean} and \ref{lem: optimi}. 
\end{rem}

By Lemma \ref{lem: coprimEPdoub}, the $M, N$-sums of (\ref{triplecontDS}) can be expressed as a product of three Euler products. That is,  the triple Dirichlet series $  \mathcal{L}_{\textbf{d,r,g}}(u_{1}, u_{2}, \mz)$ is equal to
\begin{align}
   \sum_{g_{2}\ge 1}  \,  \prod_{p} \, \frac{1_{o_{p}(a) \le  \, o_{p}(g_{1}g_{2}g_{3}g_{4})}}{(p^{o_{p}(g_{2})})^{\mz}} (F_{0})_{p}(g_{2}) \, \prod_{\substack{p \,\nmid\, \mk_{2} d_{1}d_{3}\ml \\ p \,\mid \, \mk_{1} d_{1}d_{2}\ml }} \, (F_{1})_{p}(g_{2}) \,   \prod_{\substack{p \,\mid\, \mk_{2} d_{1}d_{3}\ml \\ p \,\nmid \, \mk_{1} d_{1}d_{2}\ml }} \, (F_{2})_{p}(g_{2}) \prod_{p\, \nmid\, \mk_{1}\mk_{2} d_{1}d_{2}d_{3}\ml } \,  (F_{3})_{p}(g_{2}), \label{compensapolar}
\end{align}
and the local components are described as follows. When $p\nmid ec$, define $(F_{0})_{p}(g_{2}):= 1$; and when  $p\mid ec$, define  $(F_{0})_{p}(g_{2}):= 1_{o_{p}(g_{2})=0}$. Also, we set 
\begin{align}
 &\hspace{50pt} (F_{i})_{p}(g_{2}) \, := \,\sum_{M=0}^{\infty} \, \lambda_{\Pi_{i}}(p^{M+o_{p}(\mk_{i}g_{2})})(\frac{\psi_{i}(p)}{p^{u_{i}}})^M, \label{rawshiftHec} \\[.5em]
    (F_{3})_{p}(g_{2}) \, &:= \, \lambda_{\Pi_{1}}(p^{o_{p}(g_{2})})(F_{2})_{p}(g_{2})  +  \lambda_{\Pi_{2}}(p^{o_{p}(g_{2})})(F_{1})_{p}(g_{2}) - \lambda_{\Pi_{1}}(p^{o_{p}(g_{2})})\lambda_{\Pi_{2}}(p^{o_{p}(g_{2})}).  \label{rawshiftRS} 
\end{align}
The dependence of $\mk_{i}$, $u_{i}$, $\psi_{i}$, $\Pi_{i}$ are suppressed in the notation. Observe that  $p\nmid \mk_{2} d_{1}d_{3}\ml $ (resp. $p\nmid \mk_{1} d_{1}d_{2}\ml$) implies that $p\nmid ec$, i.e., $(F_{0})_{p}(g_{2})=1$. We evaluate the $g_{2}$-sum of (\ref{compensapolar}) with Lemma \ref{genEPexp}, and thus, the Dirichlet series $  \mathcal{L}_{\textbf{d,r,g}}(u_{1}, u_{2}, \mz)$ is the product of five Euler products:
\begin{align}
      & \prod_{\substack{p \,\nmid\, \mk_{2} d_{1}d_{3}\ml \\ p \,\mid \, \mk_{1} d_{1}d_{2}\ml }} \, \sum_{\substack{g_{2}\,\ge\, g_{p}(a)}} \, \frac{(F_{1})_{p}(p^{g_{2}})}{p^{\mz g_{2}}}  \prod_{\substack{p \,\mid\, \mk_{2} d_{1}d_{3}\ml \\ p \,\nmid \, \mk_{1} d_{1}d_{2}\ml }} \, \sum_{\substack{g_{2}\,\ge\, g_{p}(a)}} \, \frac{(F_{2})_{p}(p^{g_{2}})}{p^{\mz g_{2}}} \prod_{\substack{p \,\nmid\, \mk_{1} \mk_{2} d_{1}d_{2}d_{3}\ml  }} \, \sum_{\substack{g_{2}\,\ge\, g_{p}(a)}} \, \frac{(F_{3})_{p}(p^{g_{2}})}{p^{\mz g_{2}}}  \nonumber\\
      &\hspace{30pt} \cdot \prod_{\substack{p \,\mid\, \mk_{1}\mk_{2} d_{1}d_{2}d_{3}\ml \\  p \, \nmid \,  ec }} \, \sum_{g_{2}\ge g_{p}(a)} \, \frac{1}{p^{ \mz g_{2}}}\,  \cdot \, \prod_{\substack{ p \, \mid \,  ec }} \, 1_{o_{p}(a)\le o_{p}(g_{1}g_{3}g_{4})},  \label{fiveEPs}
\end{align}
where $g_{p}(a) \, := \,  \max\{0, o_{p}(a)-o_{p}(g_{1}g_{3}g_{4})\} \in \{0,1\}$ as in (\ref{gp0: divindic}). Notice that $a$ is square-free here.


\subsection{Simplifications}\label{simplfffying}

We observe the following:
\begin{enumerate}
    \item For Euler products involving $F_{i}$'s,  we have  $g_{p}(a)=0$, as $p\nmid \ml:=cea\ell$ (implying $o_{p}(a)=0$). 

    \item Suppose that $p\mid ec$. Recall that $(a, ec)=1$ from (\ref{g-coprdiv}) and (\ref{grouptripDS}). Then $o_{p}(a)=0$ and  $ \prod_{\substack{ p \, \mid \,  ec }} \, \mathbf{1}_{o_{p}(a)\le o_{p}(g_{1}g_{3}g_{4})} = 1$.

    \item \label{finEPg30} Suppose $p \nmid \mk_{2} d_{1}d_{3}\ml$ and  $p \mid \mk_{1} d_{1}d_{2}\ml$.  We immediately see that $p\mid \mk_{1} d_{2}$. From the conditions of the sums in (\ref{grouptripDS}), we readily deduce that:
\begin{itemize}
    \item $p\mid \mk_{1}$ (as $d_{2}\mid \ell \mid \ml$), which implies that:

    \item  $p\nmid \mk_{2} d_{1}$ (as $(\mk_{1}, \mk_{2})=1$ and  $(\mk_{1}, d_{1})=1$).
\end{itemize}
Let 
\begin{align}
     \mk_{1}  :=   g_{3}r_{1}  =   g_{30}(g_{3}'r_{1})  =:  g_{30}\mk_{1}',
\end{align}
where
\begin{align}
      (r_{1}, d_{3}\ml)=1, \hspace{10pt}\hspace{10pt} g_{30} \mid (d_{3}\ml)^{\infty}, \hspace{10pt}  \text{ and } \hspace{10pt} (g_{3}', d_{3}\ml) \ = \ 1. 
\end{align}
As $\psi$ has modulus $rea\ell$, we have $\psi(p)\neq 0$ for $p\mid \mk_{1}'$. This Euler product simplifies to:
\begin{align}\label{finUrdualEP1}
    \mathcal{E}_{\mk_{1}'}(u_{1}, \mz; \Pi_{1}, \psi_{1}) \, := \,  \prod_{\substack{p \,\mid \, \mk_{1}' }} \, \sum_{\substack{g_{2}= 0}}^{\infty} \, \frac{(F_{1})_{p}(p^{g_{2}})}{p^{\mz g_{2}}}. 
\end{align}

\item \label{finEPg40} Suppose $p \mid \mk_{2} d_{1}d_{3}\ml$ and  $p \nmid \mk_{1} d_{1}d_{2}\ml$. Similar to above,  we write 
\begin{align}
    \mk_{2} \, = \,  g_{40}(g_{4}'r_{2})=g_{40}\mk_{2}'
\end{align}
with 
\begin{align}
    g_{40}\mid (d_{2}\ml)^{\infty} \hspace{10pt} \text{ and }  \hspace{10pt} (g_{4}', d_{2}\ml) \ = \ (r_{2}, d_{2}\ml)\ = \ 1.
\end{align}
The Euler product here is:
\begin{align}\label{finUrdualEP2.0}
    \mathcal{E}_{\mk_{2}'}(u_{2}, \mz; \Pi_{2}, \psi_{2}) \ := \ \prod_{\substack{p \,\mid \, \mk_{2}' }} \, \sum_{\substack{g_{2}= 0}}^{\infty} \, \frac{(F_{2})_{p}(p^{g_{2}})}{p^{\mz g_{2}}}. 
\end{align}

\item Let  $\md :=d_{1}d_{2}d_{3}$. Suppose $p \mid \mk_{1}\mk_{2}\md\ml$ and $p \nmid ec$, i.e., $p \mid \mk_{1}\mk_{2} \md a\ell$.
The Euler product here is 
\begin{align}
  \hspace{20pt}  \zeta_{ \mk_{1}\mk_{2} \md a\ell}(\mz)\cdot \prod_{p \,\mid\, \mk_{1}\mk_{2} \md a\ell} \, p^{-\mz g_{p}^{0}} \ = \   \zeta_{ \mk_{1}\mk_{2} \md a\ell}(\mz)\cdot \prod_{\substack{p\mid a\\ p\nmid g_{1}g_{3}g_{4} }} \, p^{-\mz} \ = \  \zeta_{ \mk_{1}\mk_{2} \md a\ell}(\mz)\Big(\frac{a}{(a, g_{1}g_{3}g_{4})}\Big)^{-\mz}.
\end{align}

\end{enumerate}

\noindent In summary, the above observations allow us to conclude from (\ref{grouptripDS}) that

\begin{lem}\label{sum: MobcoprimEupr}
    The sum $\mathcal{U}^{\pm}(c; a, e, \ell; \psi)_{\circ}$ is equal to 
\begin{align}
 &\psi(\mp 1)\mu^2(a)\, \iiiint \, \eta(t) \widetilde{V}_{\alpha,\beta}(z;it)\widetilde{\mathcal{V}}^{^\pm}_{\mcX}(s_{1},s_{2}; z)  \Big( \frac{c}{Q\ell}\Big)^{-s_{1}-s_{2}} (\mfq Q^\mfr)^z  \sum_{(d_{1}, \, \ml)=1}   \sum_{\substack{d_{2}\mid \ell, \, d_{3} \mid \ell \\ (d_{2}d_{3}, \, \mathfrak{a})=1}} \sum_{(g_{1}, \, ec)=1} \,  \,   \nonumber\\
 & \hspace{15pt}\cdot  \mathop{\sum_{\substack{g_{30} \mid (d_{3}\ml)^{\infty} \\ (g_{30},\, d_{1}d_{2}ec)=1 }} \sum_{\substack{g_{40} \mid (d_{2}\ml)^{\infty} \\ (g_{40}, \, d_{1}d_{3}ec )=1}}}_{(g_{30}, g_{40})=1} \mathcal{G}_{\textbf{d}, g_{1}, g_{30}, g_{40}; \psi}(s_{1}, s_{2},z; it )   \mathop{\sum_{(\mk_{1}',\,  g_{40})=1} \, \sum_{(\mathfrak{k}_{2}', \, g_{30})=1}}_{(\mathfrak{k}_{1}', \mathfrak{k}_{2}')=1}  \, \prod_{i=1}^2 \, \mathcal{E}_{\mk_{i}'}(u_{i}, \mz; \Pi_{i}, \psi_{i}) \nonumber\\
 & \hspace{50pt} \cdot \prod_{\substack{p \,\nmid\, g_{30}g_{40}\mk_{1}' \mk_{2}' \md\ml  }} \, \sum_{\substack{g_{2}= 0}}^{\infty} \, \frac{(F_{3})_{p}(p^{g_{2}})}{p^{\mz g_{2}}}     \zeta_{g_{30}g_{40}\mk_{1}'\mk_{2}' \md a\ell}(\mz) \, \mathop{\sum \sum}_{\substack{(g_{3}', \, \md\ml)=1 \\ (r_{1}, \, d_{1}d_{3}\ml)=1 \\ g_{3}'r_{1}=\mk_{1}'}} \,  \frac{ \mu(r_{1}) }{(g_{3}')^{1+\alpha+z}   r_{1}^{\mathfrak{w}+\alpha}}  \mathop{\sum \sum}_{\substack{(g_{4}', \, \md\ml)=1 \\ (r_{2}, \, d_{1}d_{2}\ml)=1\\ g_{4}'r_{2}=\mk_{2}'}} \, \frac{ \mu(r_{2})   }{ (g_{4}')^{1+\beta+z}  r_{2}^{\mathfrak{w}+\beta}} \nonumber\\
 &\hspace{90pt} \cdot \, \sum_{(H, \, g_{30}g_{3}'\mathfrak{a})=1} \sum_{(K, \, g_{40}g_{4}'\mathfrak{a})=1} \frac{ \lambda_{g_{1}d_{1}d_{2}g_{40}\mk_{2}'H}\overline{\lambda}_{g_{1}d_{1}d_{3}g_{30}\mk_{1}'K}\psi(H) \overline{\psi}(K)}{H^{\frac{1}{2}+s_{1}+it}K^{\frac{1}{2}+s_{2}-it}}, \label{pulloutg30g40}
\end{align}
where we recall the variables from (\ref{mamlu12}), $\md :=d_{1}d_{2}d_{3}$, the function $\widetilde{\mathcal{V}}^{\pm}_{\mcX}$ defined in (\ref{2-varMelint}), the Euler products $\mathcal{E}_{\mk_{i}'}$ defined in (\ref{finUrdualEP1}) and (\ref{finUrdualEP2.0}), the local arithmetic functions $(F_{i})_{p}$ defined in (\ref{rawshiftHec})-- (\ref{rawshiftRS}). The factor $ \mathcal{G}_{\textbf{d}, g_{1}, g_{30}, g_{40}; \psi}(s_{1}, s_{2},z; it )$ in (\ref{pulloutg30g40}) is defined by
\begin{align}\label{Greicporsum}
     \frac{ \mu(d_{1})(\mu\psi)(d_{2})(\mu\overline{\psi})(d_{3})  }{g_{1}  d_{1}^{1+s_{1}+s_{2}} d_{2}^{1/2+s_{1}+it} d_{3}^{1/2+s_{2}-it}(g_{30})^{1+\alpha+z} (g_{40})^{1+\beta+z} }\Big(\frac{a}{(a, g_{1}g_{30}g_{40})}\Big)^{-\mz}.
\end{align}
\end{lem}


\subsection{Euler product over $p\mid \mk_{i}'$}\label{colHec1} 
 The computations of this section are implemented in the accompanying \texttt{Mathematica} notebook \texttt{UrFinExact.nb} (Appendix \ref{sect: UrHeckeExact}).

According to our sketch in Section \ref{sect: sketUrsum}, we expect the finite product $\mathcal{E}_{\mk_{i}'}(u_{i}, \mz; \Pi_{i}, \psi_{i})$ to be roughly $\lambda_{\Pi_{i}}(\mk_{i}')$. We manage to obtain an exact evaluation. To ease notation, we temporarily drop the subscript $i$ in $\Pi_{i}$, $\psi_{i}$, $u_{i}$, $\mk_{i}'$.

\begin{prop}\label{finiteHecprop}
  The product $\mathcal{E}_{\mk'}(u, \mz; \Pi, \psi)$ admits an entire continuation and is given by
    \begin{align}\label{findcorefactfinUr}
    \Big(\sum_{\delta\beta= \mk'} \ f_{\Pi,\, \psi}(\delta; u, \mz)\lambda_{\Pi}(\beta)\Big)    \prod_{\substack{p \,\mid \, \mk'   }} \,   L_{p}^{\mathrm{ur}}(u, \Pi\times \psi) L_{p}(\mz, \Pi), 
    \end{align}
    where $f_{\Pi,\, \psi}(\delta; u, \mz)$ is a multiplicative in $\delta\in \N$ such that for $\re u \ge  1/2$ and $\re \mz \ge  1$, we have
    \begin{align}\label{boundmulti}
        |f_{\Pi,\, \psi}(\delta; u, \mz)| \ \ll_{\epsilon} \ 
        \delta^{\epsilon} \, \begin{cases}
             \hspace{10pt}\delta^{-1/2} \hspace{56pt} \text{if} \hspace{10pt} \mfr \, = \, 2\\
            \delta^{-1/2+\vartheta}(\delta_{0})^{\vartheta} \hspace{30pt} \text{if} \hspace{10pt} \mfr \, = \, 3.
        \end{cases}
    \end{align}
As before, we write $\delta=\delta_{0}\delta'$ with $\delta_{0} \mid \mfq^{\infty}$ and $(\delta', \mfq)=1$.
\end{prop}

\begin{proof}
Let $U:= \psi(p)p^{-u}$, $Z:= \omega(p)p^{-\mz}$, and $k=o_{p}(\mk')$. Suppose that $\re u, \, \re \mz \gg 1$, and 
\begin{align}\label{definesomefinUr}
  \mathcal{E}_{\mk'}(u, \mz; \Pi, \psi) \ = \   \prod_{\substack{p \,\mid \, \mk'   }} \,   L_{p}^{\mathrm{ur}}(u, \Pi\times \psi) L_{p}(\mz, \Pi)   \mathfrak{E}_{p}(U,Z; \Pi; k).
\end{align}
The main task below is to compute $\mathfrak{E}_{p}(U,Z; \Pi; k)$ explicitly.

We begin with the case $\Pi\in \mathcal{A}(\mathrm{G}_{2})$. The required calculation can be done by hand using the Hecke recurrence relation (\ref{Hecrecur}). Indeed, applying Lemma \ref{Heckasymp} twice, we obtain
\begin{align}
\mathcal{E}_{\mk'}(u, \mz; \Pi, \psi) \ = \ & \prod_{\substack{p \,\mid \, \mk'   }} \,   \sum_{\substack{g= 0}}^{\infty} \, \frac{1}{p^{\mz g}}  \sum_{M=0}^{\infty} \, \lambda_{\Pi}(p^{k+g+M}) U^{M}\nonumber\\
 \, = \,  &\prod_{\substack{p \,\mid \, \mk'   }} \,   L_{p}^{\mathrm{ur}}(u, \Pi\times \psi) \, \sum_{\substack{g= 0}}^{\infty} \, \frac{1}{p^{\mz g}} \Big\{ \lambda_{\Pi}(p^{k+g}) -  \lambda_{\Pi}(p^{k+g-1})\omega(p)U \Big\} \nonumber\\
 \, = \,  &\prod_{\substack{p \,\mid \, \mk'   }} \,  L_{p}^{\mathrm{ur}}(u, \Pi\times \psi) \,  \Big\{ \big(1-  UZ\big)\sum_{\substack{g= 0}}^{\infty} \, \frac{\lambda_{\Pi}(p^{k+g})}{p^{\mz g}}   - \omega(p)U \lambda_{\Pi}(p^{k-1}) \Big\} \nonumber\\
 \, = \, & \prod_{\substack{p \,\mid \, \mk'   }} \,  L_{p}^{\mathrm{ur}}(u, \Pi\times \psi) L_{p}(\mz, \Pi)   \mathfrak{E}_{p}(U,Z; \Pi; k),\nonumber
\end{align}
where
\begin{align}
   \mathfrak{E}_{p}(U,Z; \Pi; k) \ = \  (1-  UZ) \big(\lambda_{\Pi}(p^{k}) -  \lambda_{\Pi}(p^{k-1})Z\big)   - U\omega(p) \lambda_{\Pi}(p^{k-1}) L_{p}(\mz, \Pi)^{-1}. \nonumber
\end{align}
Applying (\ref{GL2EP}) and (\ref{Hecrecur}), we find that
\begin{align}
  \mathfrak{E}_{p}(U,Z; \Pi; k) \, = \ &\lambda_{\Pi}(p^{k}) -\lambda_{\Pi}(p^{k}) UZ  - \lambda_{\Pi}(p^{k-1})Z + \lambda_{\Pi}(p^{k-1})UZ^2\nonumber\\
    &\hspace{60pt}  -  \lambda_{\Pi}(p^{k-1})\omega(p)U +    \lambda_{\Pi}(p^{k-1})\lambda_{\Pi}(p) UZ- \lambda_{\Pi}(p^{k-1}) UZ^2 \nonumber\\
    \, = \, &\big(\lambda_{\Pi}(p^{k}) -  U\lambda_{\Pi}(p^{k-1})\omega(p)\big)  - UZ \big(\lambda_{\Pi}(p^{k}) -  \lambda_{\Pi}(p^{k-1})\lambda_{\Pi}(p) \big)   -Z \lambda_{\Pi}(p^{k-1}) \nonumber\\
    \, = \, &\big(\lambda_{\Pi}(p^{k}) -  U \lambda_{\Pi}(p^{k-1})\omega(p)\big)  - Z\big( \lambda_{\Pi}(p^{k-1}) \, - \, U \lambda_{\Pi}(p^{k-2})\omega(p)\big). \nonumber
\end{align}
From this, we readily observe that (\ref{findcorefactfinUr}) and (\ref{boundmulti}) hold with the choice
\begin{align}
    f_{\Pi,\, \psi}(\delta; u, \mz) \ :=  \ \sum_{\alpha\gamma =\delta} \, \frac{(\mu \omega\psi)(\alpha)(\mu\omega)(\gamma) }{\alpha^u \gamma^{\mz}}.
\end{align}


We turn to the case $\Pi\in \mathcal{A}_{0}(\mathrm{G}_{3})$.  The determination of $\mathfrak{E}_{p}$ requires the pattern-matching function of \texttt{Mathematica} (see \texttt{UrFinExact.nb} or Appendix \ref{sect: UrHeckeExact}). The result is surprisingly clean:
\begin{align}\label{fullramunrdivUr}
     \mathfrak{E}_{p}(U,Z; \Pi; k) \ = \ \lambda_{\Pi}(p^{k}) & -(Z+U) \lambda_{\Pi}(p,p^{k-1}) + U Z  \lambda_{\Pi}(p^2, p^{k-2})  + (U^2\omega(p)+UZ+Z^2) \lambda_{\Pi}(p^{k-1}) \nonumber\\
      &+ U^2 Z^2 \lambda_{\Pi}(p^{k-2})  -UZ(U+Z) \lambda_{\Pi}(p, p^{k-2})
\end{align}
for $k\ge 2$. The same formula holds for $k \in \{0,1\}$ with the convention of Remark \ref{conv: Hecke}. 

Suppose that $p\nmid \mfq$. Then from the relation (\ref{FourtoHeckgl3}), we determine, for $k\ge 3$, that
\begin{align}
       \mathfrak{E}_{p}(U,Z; \Pi; k)
      \ &= \ \sum_{j+i=k} \ f_{\Pi,\, \psi}(p^{j}; u, \mz)\lambda_{\Pi}(p^i),
\end{align}
where
\begin{align}
    f_{\Pi,\, \psi}(p^{j};\,  u, \mz) \ := \  
    \begin{cases}
      \hspace{90pt}  1  \hspace{173pt}\text{ if } j \, = \, 0\\
        \hspace{20pt}U^2+UZ+Z^2-(Z+U)\lambda_{\widetilde{\Pi}}(p) \hspace{90pt} \text{ if } j \, = \, 1\\
        Z+U+U^2Z^2+UZ\lambda_{\widetilde{\Pi}}(p^2)-UZ(U+Z)\lambda_{\widetilde{\Pi}}(p) \hspace{25pt}\text{ if } j \, = \, 2\\
       \hspace{40pt} UZ(U+Z) - U Z  \lambda_{\widetilde{\Pi}}(p) \hspace{108pt}\text{ if } j \, = \, 3\\
        \hspace{93pt} 0  \hspace{171pt} \text{ if } j  \ge  \, 4.
    \end{cases}
\end{align}
The validity for $k\in \{0,1,2\}$ can again be checked by \texttt{Mathematica}. 

Next, suppose that $p\mid \mfq$. Then (\ref{fullramunrdivUr}) simplifies to $ \mathfrak{E}_{p}(U,Z; \Pi; k)=\lambda_{\Pi}(p^k)-U \lambda_{\Pi}(p, p^{k-1})$. By the Jacobi--Trudi identity (\ref{JacobiTrud}) followed by the Hecke recurrence relation (\ref{gl3: Hecrecurrel}), we have
\begin{align}
 \mathfrak{E}_{p}(U,Z; \Pi; k) \, = \, \lambda_{\Pi}(p^k)-U \mathfrak{e}_{\Pi}(p)\lambda_{\Pi}(p^{k-1}).
\end{align}
In this case, we set $  f_{\Pi,\, \psi}(p^{j};\,  u, \mz) \ = \  \mu(p^j)U \mathfrak{e}_{\Pi}(p)$.

We define $f_{\Pi,\, \psi}(\delta; u, \mz)$ as a multiplicative function in $\delta$ whose values on prime powers are given as above. This proves (\ref{findcorefactfinUr}), and the bound (\ref{boundmulti}) follows from (\ref{uniformRC}) and (\ref{epiboundgl3}).
\end{proof}


\subsection{Euler product over $p \,\nmid\, g_{30}g_{40}\mk_{1}' \mk_{2}' \md\ml$ }\label{colHectwistRS}  

 The computations of this section are implemented in the accompanying \texttt{Mathematica} file \texttt{UrInf.nb} (Appendix \ref{sect: UnrInfExact}). Recall that $o_{p}(\mk_{1})=o_{p}(\mk_{2})=0$, and 
\begin{align}\label{F3def}
    (F_{3})_{p}(p^{g}) \, &:= \, \lambda_{\Pi_{1}}(p^{g})(F_{2})_{p}(p^{g})  +  \lambda_{\Pi_{2}}(p^{g})(F_{1})_{p}(p^{g}) - \lambda_{\Pi_{1}}(p^{g})\lambda_{\Pi_{2}}(p^{g}).   
\end{align}

\begin{prop}\label{HecktwistRSnuden}
 On the half-planes 
 \begin{align}\label{Urconthalf}
     \re u_{1}, \ \re \, u_{2} >1/2, \hspace{10pt} \text{ and } \hspace{10pt} \re \mz \, > \, 1,
 \end{align}
 we have
    \begin{align}\label{firsttripEP}
 \prod_{p \,\nmid\, g_{30}g_{40}\mk_{1}' \mk_{2}' \md\ml} \Big(L_{p}^{\mathrm{ur}}(\mz, \Pi_{1}\otimes \Pi_{2})\prod_{i=1}^{2} \, L_{p}^{\mathrm{ur}}(u_{i}, \Pi_{i}\times \psi_{i}) \Big)^{-1} \sum_{g=0}^{\infty} \, \frac{(F_{3})_{p}(p^{g})}{(p^{\mz})^g}  \ \preceq \ 1
\end{align}
\end{prop}

\begin{proof}
    Suppose that $\re u_{1},\, \re u_{2}, \,  \re \mz \gg 1$.The calculation is routine and proceeds as before, using Lemma \ref{Heckasymp}, Lemma \ref{lem: diffby1RSn}, and the Hecke recurrence relations. Again, we carry out this with the aid of \texttt{Mathematica} (see \texttt{UrInf.nb} or Appendix \ref{sect: UnrInfExact}). Let $U_{i}:= \psi_{i}(p)p^{-u_{i}}$ ($i=1,2$), $Z:= p^{-\mz}$, $\omega=\omega(p)$. 
    
    We first consider the case when $\Pi_{1},   \Pi_{2} \in \mathcal{A}_{0}(\mathrm{G}_{2})$. The local factor in (\ref{firsttripEP}) can be evaluated as
   \begin{align}
       1 -  &\omega\left\{ (U_{1}U_{2})^2 + Z^2 -(U_{1}U_{2}Z)^2 \right\} 
       + \omega(Z-UV) \left\{ (U_{1} Z-U_{2}) \lambda_{\Pi_{1}}(p)  + (U_{2}Z-U_{1})\lambda_{\Pi_{2}}(p)\right\} \nonumber\\
 & -U_{1} U_{2} (1 + Z^2 \omega) \lambda_{\Pi_{1}}(p)\lambda_{\Pi_{2}}(p) +U_{1} U_{2} Z \omega \left(\lambda_{\Pi_{1}}(p)^2 +\lambda_{\Pi_{2}}(p)^2\right). \label{gl2:infinprodcheck}
\end{align}
For $\re u_{1}, \, \re \, u_{2} >1/2$ and $\re \mz >1$, it is easy to observe, with $\vartheta<1/2$, that (\ref{gl2:infinprodcheck}) is   
\begin{align}
     1 - \lambda_{\Pi_{1}}(p)\lambda_{\Pi_{2}}(p) U_{1}U_{2}\, + \, O_{\epsilon}(p^{-1-\epsilon}),
\end{align}
and the bound in (\ref{firsttripEP}) follows from Lemma \ref{lem: diffEUproLi}. 

Next, we consider $\Pi_{1}, \Pi_{2}\in \mathcal{A}_{0}(\mathrm{G}_{3})$. The corresponding polynomial consists of $158$ terms (upon a judicious combination of 753 terms) and is too lengthy to reproduce here. It has degree $3,3,6$ in $U_{1}$, $U_{2}$, $Z$.  Our \texttt{Mathematica} code shows (using $\vartheta\le 5/14$) that the local factor in (\ref{firsttripEP}) is 
\begin{align}
        1- &\lambda_{\Pi_{1}}(p)\lambda_{\Pi_{2}}(p) U_{1}U_{2} + \big(\lambda_{\widetilde{\Pi_{1}}}(p)\lambda_{\Pi_{2}}(p) U_{1}+\lambda_{\Pi_{1}}(p)\lambda_{\widetilde{\Pi_{2}}}(p)U_{2}\big)(U_{1}U_{2} -Z) \nonumber\\
      &\hspace{50pt}+U_{1}U_{2}Z\big(\lambda_{\widetilde{\Pi_{1}}}(p)\lambda_{\Pi_{2}}(p)^2+ \lambda_{\Pi_{1}}(p)^2\lambda_{\widetilde{\Pi_{2}}}(p)\big) \, + \, O_{\epsilon}(p^{-1-\epsilon})  \hspace{70pt} (p\nmid \mfq); \nonumber
\end{align}
 and is $1+O_{\epsilon}(p^{-\epsilon})$ if $p\mid \mfq$. The result follows from the same argument as in Section \ref{twistedprinesund: GL3}.
\end{proof}

By throwing in a finite number of Euler factors that are $\ll_{\epsilon} (\mfq g_{30}g_{40}\mk_{1}' \mk_{2}' \md\ml)^{\epsilon}\preceq 1$, we have
\begin{align}\label{giganundivEUr}
    \prod_{p \,\nmid\, g_{30}g_{40}\mk_{1}' \mk_{2}' \md\ml} \, \sum_{g= 0}^{\infty} \, \frac{(F_{3})_{p}(p^{g})}{(p^{\mz})^g} = L(\mz, \Pi_{1}\otimes \Pi_{2}) \prod_{i=1}^{2} \, L(u_{i}, \Pi_{i}\times \psi_{i})  \mathfrak{A}^{(g_{30}g_{40}\mk_{1}' \mk_{2}' \md\ml)}\left(u_{1}, u_{2},\mz; \,  (\Pi_{i},\psi_{i})_{i=1}^{2}\right),
\end{align}
where $\mathfrak{A}$ denotes an infinite product that is $\preceq 1$ on (\ref{Urconthalf}). With the holomorphy of the $\hbox{GL}_{\mfr}\times \hbox{GL}_{1}$ and $\hbox{GL}_{\mfr}\times \hbox{GL}_{\mfr}$ $L$-functions, the left side of (\ref{giganundivEUr}) admits a holomorphic continuation to  (\ref{Urconthalf}).


\subsection{Contour-shifting and cleaning}\label{sect: contshclean}

 Applying Propositions \ref{finiteHecprop} and \ref{HecktwistRSnuden} to (\ref{pulloutg30g40}), we may now shift the lines of integration to $\re s_{1}=\re s_{2}=\re z=\epsilon$, followed by rearrangement of summation. Then the sum $\mathcal{U}^{\pm}(c; a, e, \ell; \psi)_{\circ}$ is equal to
\begin{align}
 &\psi(\mp 1)\mu^2(a)\, \iiiint_{(\epsilon)} \, \eta(t)	 \widetilde{V}_{\alpha,\beta}(z;it)\widetilde{\mathcal{V}}^{^\pm}_{\mcX}(s_{1},s_{2}; z)  \Big( \frac{c}{Q\ell}\Big)^{-s_{1}-s_{2}} (\mfq Q^\mfr)^z  L(\mz, \Pi_{1}\otimes \Pi_{2}) \prod_{i=1}^{2} \, L(u_{i}, \Pi_{i}\times \psi_{i}) \nonumber\\[0.5em]
 &\hspace{10pt} \cdot \,  \sum_{(d_{1}, \, \ml)=1}   \sum_{\substack{d_{2}\mid \ell, \, d_{3} \mid \ell \\ (d_{2}d_{3}, \, \mathfrak{a})=1}} \sum_{(g_{1}, \, ec)=1} \, \mathop{\sum_{\substack{g_{30} \mid (d_{3}\ml)^{\infty} \\ (g_{30},\, d_{1}d_{2}ec)=1 }} \sum_{\substack{g_{40} \mid (d_{2}\ml)^{\infty} \\ (g_{40}, \, d_{1}d_{3}ec )=1}}}_{(g_{30}, g_{40})=1}  \,  \mathcal{G}_{\textbf{d}, g_{1}, g_{30}, g_{40}; \psi}(s_{1}, s_{2},z; it )    \nonumber\\
 & \hspace{15pt}\cdot \mathop{\sum_{(\mk_{1}',\,  g_{40})=1} \, \sum_{(\mathfrak{k}_{2}', \, g_{30})=1}}_{(\mathfrak{k}_{1}', \mathfrak{k}_{2}')=1}  \, \zeta_{g_{30}g_{40}\mk_{1}'\mk_{2}' \md a\ell}(\mz)\mathfrak{A}^{(g_{30}g_{40}\mk_{1}' \mk_{2}' \md\ml)}\left(u_{1}, u_{2},\mz; \,  (\Pi_{i},\psi_{i})_{i=1}^{2}\right) \prod_{i=1}^2 \,   \Big(\sum_{\delta\beta= \mk_{i}'} \ f_{\Pi_{i},\, \psi_{i}}(\delta; u_{i}, \mz)\lambda_{\Pi_{i}}(\beta)\Big)     \nonumber\\
 &\hspace{50pt} \cdot \,  L_{ \mk_{i}'}^{\mathrm{ur}}(u_{i}, \Pi_{1}\times \psi_{i}) L_{ \mk_{i}'}(\mz, \Pi_{i})      \, \mathop{\sum \sum}_{\substack{(g_{3}', \, \md\ml)=1 \\ (r_{1}, \, d_{1}d_{3}\ml)=1 \\ g_{3}'r_{1}=\mk_{1}'}} \,  \frac{ \mu(r_{1}) }{(g_{3}')^{1+\alpha+z}   r_{1}^{\mathfrak{w}+\alpha}}  \mathop{\sum \sum}_{\substack{(g_{4}', \, \md\ml)=1 \\ (r_{2}, \, d_{1}d_{2}\ml)=1\\ g_{4}'r_{2}=\mk_{2}'}} \, \frac{ \mu(r_{2})   }{ (g_{4}')^{1+\beta+z}  r_{2}^{\mathfrak{w}+\beta}} \nonumber\\
 &\hspace{90pt} \cdot \, \sum_{(H, \, g_{30}g_{3}'\mathfrak{a})=1} \sum_{(K, \, g_{40}g_{4}'\mathfrak{a})=1} \frac{ \lambda_{g_{1}d_{1}d_{2}g_{40}\mk_{2}'H}\overline{\lambda}_{g_{1}d_{1}d_{3}g_{30}\mk_{1}'K}\psi(H) \overline{\psi}(K)}{H^{\frac{1}{2}+s_{1}+it}K^{\frac{1}{2}+s_{2}-it}}. \label{finalequaUrclea}
\end{align}
The notation above have been explained in (\ref{mamlu12}), (\ref{2-varMelint}), (\ref{Greicporsum}), (\ref{findcorefactfinUr}), (\ref{giganundivEUr}).


\begin{lem}\label{Trianclean}
The following inequality holds:
    \begin{align}\label{cleanupexprUr}
 |\, \mathcal{U}^{\pm}(c; a, e, \ell; \psi)_{\circ}|\, \preceq \,  & \frac{1}{a} \iiiint_{(\epsilon)} \eta(t) \big|\widetilde{V}_{\alpha,\beta}(z;it)\widetilde{\mathcal{V}}^{^\pm}_{\mcX}(s_{1},s_{2}; z)\big| |L(u_{1}, \Pi_{1}\times \psi) L(u_{2}, \Pi_{1}\times \overline{\psi})|  \nonumber\\
 &\hspace{40pt}  \, \cdot \,\max_{\substack{v_{1}, v_{2}, g_{3}, g_{4}\le (TQ)^{\theta}} } \, \Big|\sum_{(H, \, g_{3}\mathfrak{a})=1}\, \sum_{ (K, \, g_{4}\mathfrak{a})=1}\frac{ \lambda_{v_{1}H}\overline{\lambda_{v_{2}K}}\psi(H) \overline{\psi}(K)}{H^{\frac{1}{2}+s_{1}+it}K^{\frac{1}{2}+s_{2}-it}} \Big|.
\end{align}
\end{lem}

\begin{proof}
Apply the triangle inequality to (\ref{finalequaUrclea}) in the following way:
\begin{align}
   |\,\mathcal{U}^{\pm}(c; a, e, \ell; \psi)_{\circ}| \, & \preceq \  \mu^2(a)\iiiint_{(\epsilon)} \, \eta(t)	 \big|\widetilde{V}_{\alpha,\beta}(z;it)\widetilde{\mathcal{V}}^{^\pm}_{\mcX}(s_{1},s_{2}; z)\big| \big| \prod_{i=1}^{2} \, L(u_{i}, \Pi_{i}\times \psi_{i})L(\mz, \Pi_{1}\otimes \Pi_{2})\Big|  \nonumber\\[0.5em]
   &\hspace{90pt} \cdot \mathop{\sum \ \cdots  \ \sum}_{d_{1}, d_{2}, d_{3}, g_{1}, g_{30}, g_{40}, \mk_{1}', \mk_{2}', g_{3}', g_{4}', r_{1}, r_{2}} |(\cdots)|\nonumber\\[0.5em]
     &\hspace{70pt} \cdot\Big|\sum_{(H, \, g_{30}g_{3}'\mathfrak{a})=1} \sum_{(K, \, g_{40}g_{4}'\mathfrak{a})=1} \frac{\lambda_{g_{1}d_{1}d_{2}g_{40}\mk_{2}'H}\overline{\lambda}_{g_{1}d_{1}d_{3}g_{30}\mk_{1}'K}\psi(H) \overline{\psi}(K)}{H^{\frac{1}{2}+s_{1}+it}K^{\frac{1}{2}+s_{2}-it}} \Big|\nonumber\\[0.5em]
     \, &\hspace{-70pt} \preceq \   \mu^2(a)\iiiint_{(\epsilon)} \, \eta(t)	 \big|\widetilde{V}_{\alpha,\beta}(z;it)\widetilde{\mathcal{V}}^{^\pm}_{\mcX}(s_{1},s_{2}; z)\big| \, \big| \prod_{i=1}^{2} \, L(u_{i}, \Pi_{i}\times \psi_{i})L(\mz, \Pi_{1}\otimes \Pi_{2})\Big|  \nonumber\\[0.5em]
      &\hspace{15pt} \cdot \max_{\substack{d_{1}, d_{2}, d_{3} \\  g_{1}, g_{30}, g_{40}, g_{3}', g_{4}'\\ r_{1},r_{2}} } \, \Big|\sum_{(H, \, g_{30}g_{3}'\mathfrak{a})=1} \sum_{(K, \, g_{40}g_{4}'\mathfrak{a})=1} \frac{ \lambda_{g_{1}d_{1}d_{2}g_{40}g_{4}'r_{2}H}\overline{\lambda}_{g_{1}d_{1}d_{3}g_{30}g_{3}'r_{1}K}\psi(H) \overline{\psi}(K)}{H^{\frac{1}{2}+s_{1}+it}K^{\frac{1}{2}+s_{2}-it}} \Big| \nonumber\\[0.5em]
      &\hspace{40pt} \cdot \mathop{\sum \ \cdots  \ \sum}_{d_{1}, d_{2}, d_{3}, g_{1}, g_{30}, g_{40}, \mk_{1}', \mk_{2}', g_{3}', g_{4}', r_{1}, r_{2}} |(\cdots)|. \label{mainTRIAN}
\end{align}
All variables appearing under the ``$\max$'' occur within the arbitrary coefficients $\lambda_{(\cdots)}$, and therefore have size at most $(TQ)^{\theta}$. Recall that $\re\, u_{1}, \, \re \, u_{2} >1/2$ and $\re \mz >1$ (see (\ref{mamlu12})). Observe the elementary bounds for the expressions in $(\cdots)$:
\begin{align}
    |\mathcal{G}_{\textbf{d}, g_{1}, g_{30}, g_{40}; \psi}(s_{1}, s_{2},z; it )| \, \ll \,  \frac{1}{g_{1}d_{1} (d_{2}d_{3})^{1/2} g_{30}g_{40}}\Big(\frac{a}{(a, g_{1}g_{30}g_{40})}\Big)^{-\mz}; \nonumber
\end{align}
\begin{align}
\zeta_{g_{30}g_{40}\mk_{1}'\mk_{2}' \md a\ell}(\mz),  \hspace{5pt} \mathfrak{A}^{(g_{30}g_{40}\mk_{1}' \mk_{2}' \md\ml)}\left(u_{1}, u_{2},\mz; \,  (\Pi_{i},\psi_{i})_{i=1}^{2}\right), \hspace{5pt}   L_{\mk_{i}'}^{\mathrm{ur}}(u_{i}, \Pi_{1}\times \psi_{i}), \hspace{5pt} L_{ \mk_{i}'}(\mz, \Pi_{i}),  \hspace{5pt} |L(\mz, \Pi_{1}\otimes \Pi_{2})|  \,  \preceq \,  1; \nonumber
\end{align}
\begin{align}
   \Big| \mathop{\sum \sum}_{\substack{(g_{3}', \, \md\ml)=1 \\ (r_{1}, \, d_{1}d_{3}\ml)=1 \\ g_{3}'r_{1}=\mk_{1}'}}  \frac{ \mu(r_{1}) }{(g_{3}')^{1+\alpha+z}   r_{1}^{\mathfrak{w}+\alpha}} \Big| \, \le \, \sum_{g_{3}'r_{1}=\mk_{1}'} \, \frac{1}{g_{3}' r_{1}} \, \preceq \, \frac{1}{\mk_{1}'};\nonumber
\end{align}
and the following uses $\vartheta<1/2$, (\ref{boundmulti}), and Lemma \ref{lem: ROA}: 
\begin{align}
  \sum_{\mk_{i}'\le (TQ)^{\theta}} \,  \frac{1}{\mk_{i}'}   \sum_{\delta\beta= \mk_{i}'} \ |f_{\Pi_{i}, \psi_{i}}(\delta; u_{i}, \mz)||\lambda_{\Pi_{i}}(\beta)| 
    \, &\le \,  \sum_{\substack{\beta\le (TQ)^{\theta}}} \, \frac{|\lambda_{\Pi_{i}}(\beta)|}{\beta} \sum_{\delta_{0}\mid \mfq^{\infty}} \frac{1}{(\delta_{0})^{3/2-2\vartheta}} \sum_{(\delta', \mfq)=1} \, \frac{1}{(\delta')^{3/2-\vartheta}}\nonumber\\
    \,  &\preceq \, \sum_{\beta\le (TQ)^{\theta}} \, \frac{|\lambda_{\Pi_{i}}(\beta)|^2}{\beta} \, \preceq \, 1.\nonumber
\end{align}

Plugging the bounds above into (\ref{mainTRIAN}) and dropping the coprimality constraints of the summation by positivity, we have
\begin{align}
 |\,\mathcal{U}^{\pm}(c; a, e, \ell; \psi)_{\circ}|  \   & \preceq \ \mu^2(a) \iiiint_{(\epsilon)} \eta(t) \big|\widetilde{V}_{\alpha,\beta}(z;it)\widetilde{\mathcal{V}}^{^\pm}_{\mcX}(s_{1},s_{2}; z)\big|  \prod_{i=1}^{2}\, | L(u_{i}, \Pi_{i}\times \psi_{i})|  \\
 & \hspace{5pt} \, \cdot \, \, \max_{\substack{d_{1}, d_{2}, d_{3} \\  g_{1}, g_{30}, g_{40}, g_{3}', g_{4}'\\ r_{1},r_{2}} } \, \Big|\sum_{(H, \, g_{30}g_{3}'\mathfrak{a})=1} \sum_{(K, \, g_{40}g_{4}'\mathfrak{a})=1} \frac{ \lambda_{g_{1}d_{1}d_{2}g_{40}g_{4}'r_{2}H}\overline{\lambda}_{g_{1}d_{1}d_{3}g_{30}g_{3}'r_{1}K}\psi(H) \overline{\psi}(K)}{H^{\frac{1}{2}+s_{1}+it}K^{\frac{1}{2}+s_{2}-it}} \Big| \\[0.5em]
 & \hspace{10pt} \cdot  \frac{1}{a}\sum_{d_{1}\le (TQ)^{\theta}}   \sum_{\substack{d_{2}\mid \ell }} \,  \sum_{ d_{3} \mid \ell} \, \sum_{g_{1}\le (TQ)^{\theta}} \, \mathop{\sum_{\substack{g_{30} \le (TQ)^{\theta} }} \sum_{\substack{g_{40} \le (TQ)^{\theta} }}}  \,  \frac{(a, g_{1}g_{30}g_{40})}{g_{1}g_{30}g_{40}d_{1}\sqrt{d_{2}d_{3}} }.  \label{gluefora}
\end{align}
The sums over $d_{i}$'s are clearly $\preceq 1$. Making the change of variables $g=g_{1}g_{30}g_{40}$, Display (\ref{gluefora}) is $\preceq a^{-1}$ because of Lemma \ref{lem: simpleGCD}. This completes the proof of Lemma \ref{Trianclean}.
\end{proof}

\begin{prop}\label{afteriniclearnUr}
Let $\mathcal{X}$ be chosen as in (\ref{keyXvalue}) and  $\ma := rcea$. We have the estimate
    \begin{align}
     \bigg| \,  \sum_{h, k \ge 1} \ \frac{\lambda_{h}\overline{\lambda_{k}}}{\sqrt{hk}}  \ \mathcal{U}^{(r)}(h, k)\bigg|    \, &\preceq \,  Q  \max_{\substack{AL 
	 \ll   \frac{C}{Q}(TQ)^{\mfr/2+\theta}\mathfrak{C} \\ AEL \ll \mathfrak{C} (TQ)^{1000}  }} \,  \frac{1 }{A^2E^2L}  \, \sideset{}{^d}{\sum}_{S\ge 1} \, \min_{k_{1}\ge 1} \frac{\min\{T, S\} \mcX^{k_{1}-1} }{S^{k_{1}}} \nonumber\\
    &\hspace{40pt}\cdot \, \sum_{f\ll \mfq AEL} \, \sum_{\mnl'\ll \, \mfq AEL/f} \, \   \sideset{}{^*}{\sum}_{\substack{\psi \, (\mnl')}}   \, \int_{\substack{\re s =\epsilon \\ |\im s| \ll \max\{T,S\} }} \,  \max_{  AE \le \ma \le 4\mfq AEC}\, \max_{g_{3}, v_{1}\le (TQ)^{\theta} }\,      \nonumber\\
    &\hspace{80pt} 
      \Big| L(1/2+s,\,  \Pi_{1}\times \psi) \sum_{\substack{(H, \, g_{3}\ma)=1 }} \, \frac{ \lambda_{v_{1}H}\psi(H) }{H^{\frac{1}{2}+s}} \Big|^2	\, |ds|.
    \end{align}
\end{prop}

\begin{proof}
Recall (\ref{2-varMelint}) and (\ref{stdStirAFEcut}). By $|XY|\ll |X|^2+|Y|^2$, we have 
\begin{align}\label{mainredestUr}
|\,\mathcal{U}^{\pm}(c; a, e, \ell; \psi)_{\circ}\, | \  \preceq_{A}  \  \frac{1}{a} \, \int_{\R} \, \eta(t)  &\iiint\limits_{(\re s_{1}, \re s_{2}, \re z)=(\epsilon, \epsilon, \epsilon)} \,   |z|^{-A} |\widetilde{\mathcal{V}}_{\mcX}^{^\pm}(s_{1},s_{2};z) |\nonumber\\
& \hspace{20pt}\cdot \ \max_{\substack{v_{1}, \, g_{3}\le (TQ)^{\theta}} } \,  \Big| L(u_{1}, \Pi_{1}\times \psi) \sum_{\substack{(H, \, g_{3}\mathfrak{a})=1 }} \, \frac{ \lambda_{v_{1}H}\psi(H) }{H^{\frac{1}{2}+s_{1}+it}} \Big|^2
\end{align}
plus the dual piece, which is obtained by $\Pi_{1}\to \Pi_{2}$, $\psi\to \bar{\psi}$, $\lambda_{v_{1}H} \to \overline{\lambda_{v_{2}K}}$, and  $(u_{1}, g_{3}, s_{1}+it)$ $ \to (u_{2}, g_{4}, s_{2}-it)$. We shall suppress the dual piece when stating the inequality for $\mathcal{U}^{\pm}(c; a, e, \ell; \psi)_{\circ}$. Applying Lemma \ref{SOVmainlem} to (\ref{mainredestUr}), we are left to bound the expression
\begin{align}\label{preHLSquant}
    & \frac{\mcX^{k_{1}-1}}{a}     \,     \int_{\R} \, \eta(t) \,  \iiint_{\substack{\re s_{1}=\re s_{2}=\re z=\epsilon }} \,    	\frac{1 }{ \max\{ 1+|s_{1}|, 1+|s_{2}|\}^{k_{1}} |s_{1}+s_{2}|^{k_{2}}|z|^{A} } 
 \nonumber\\
 &\hspace{180pt}  \cdot \,  \max_{\substack{v_{1}, g_{3}\le (TQ)^{\theta}} } \, \Big| L(u_{1}, \Pi_{1}\times \psi) \sum_{\substack{(H, \, g_{3}\ma)=1 }} \, \frac{ \lambda_{v_{1}H}\psi(H) }{H^{\frac{1}{2}+s_{1}+it}} \Big|^2, \, 
\end{align}
which is also equal to
\begin{align}
    &  \frac{1 }{a}  \, \sideset{}{^d}{\sum}_{S\ge 1} \, \frac{\mcX^{k_{1}-1}}{S^{k_{1}}}  \,   \int_{\R}  \,  \iiint_{\substack{\re s_{1}= \re s_{2} = \re z=\epsilon \\ \max\{ 1+|s_{1}|, 1+|s_{2}|\} \sim S }} \,  	\frac{\eta(t)}{|z|^{A} |s_{1}+s_{2}|^{k_{2}}}  \max_{\substack{v_{1}, \, g_{3}\le (TQ)^{\theta}} }\,    \Big| L(u_{1}, \Pi_{1}\times \psi) \sum_{\substack{(H, \, g_{3}\ma)=1 }} \, \frac{ \lambda_{v_{1}H}\psi(H) }{H^{\frac{1}{2}+s_{1}+it}} \Big|^2.\nonumber
\end{align}
Here, the parameters $k_{1}, k_{2}, A\ge 1$ are at our disposal.  The $s_{2}$-integral can be truncated to $|s_{1}+s_{2}|\le (TQ)^{\epsilon}$ with an arbitrarily small error, and its size can be seen to be $\preceq \, \mathbf{1}_{|s_{1}|\le 2S}$. The $z$-integral can also be truncated to $|z|\le (TQ)^{\epsilon}$.  Putting $s=s_{1}+it$, we have
\begin{align}
\mathop{\int_{|s_{1}|\le 2S} \int_{t\asymp T}}_{s= s_{1}+it} \,  |ds_{1}| \, dt \, \ll \, \min\{T, S\}, \nonumber
\end{align}
and thus, (\ref{preHLSquant}) is bounded by 
\begin{align}
      \frac{1  }{a} \, \sideset{}{^d}{\sum}_{S\ge 1} \, \frac{\min\{T, S\} \mcX^{k_{1}-1}}{S^{k_{1}}} 
    \int\limits_{\substack{\re s =\epsilon \\ |\im s| \le 2 \max\{T,S\} }} \, \max_{\substack{v_{1}, g_{3}\le (TQ)^{\theta}} } \,   \Big| L\Big(\frac{1}{2}+s,\,  \Pi_{1}\times \psi\Big) \sum_{\substack{(H, \, g_{3}\ma)=1 }} \, \frac{ \lambda_{v_{1}H}\psi(H) }{H^{\frac{1}{2}+s}} \Big|^2.\label{cleanbddUrone}
\end{align}
Recall (\ref{dyadicaelsum}). By replacing $\psi \to \psi_{0} \psi$ with $\psi \, (\bmod\, \mnl')$  primitive and $\psi_{0} \, (\bmod\, \mnl)$ trivial, we have
\begin{align}
  \bigg| \,  \sum_{h, k \ge 1} \ \frac{\lambda_{h}\overline{\lambda_{k}}}{\sqrt{hk}}  \ \mathcal{U}^{(r)}(h, k)\bigg| \preceq \,  Q  \max_{\substack{AL 
	\ll  \frac{C}{Q}(TQ)^{\mfr/2+\theta}\mathfrak{C} \\ AEL \ll \mathfrak{C} (TQ)^{1000}  }} \, \frac{1}{AE^2L}
 \sum_{c\le C }    \,  \frac{1}{c} \, \sum_{\mnl\ll \mfq AEL}  \, \sum_{\substack{rea\ell=\mnl \\ a\sim A \\ e\sim E\\ r\mid \mfq}} \, \sum_{\substack{\mnl' \mid \mnl \\ \mnl'>1} } \sideset{}{^*}{\sum}_{\substack{\psi \, (\mnl')}}  \, \left|\, \mathcal{U}^{\pm}(c; a, e, \ell; \psi_{0} \psi)\right|.  \nonumber
\end{align}
Taking the maximum over $AE \le \ma \le 4\mfq AEC$ and letting $\mnl=\mnl'f$, the desired result follows from (\ref{cleanbddUrone}) and the fact that the sums over $e, a, c, r$ are $\preceq 1$. 
\end{proof}


\subsection{End game: a second application of large sieve}\label{Sec: finauseLS}

Throughout this section, we suppose that one of the following holds:
\begin{enumerate}
    \item $\lambda_{h}\ll_{\epsilon} h^{\epsilon}$; or

    \item both Condition $(\mathbf{\Lambda})$ (see (\ref{bdd: 24normabstract}))  and Hypothesis $(\mathbf{\Pi}^4)$ (see (\ref{fourthpowerestHec})) hold.
\end{enumerate}

\begin{lem}\label{Usumpredeterbdd}
We have
   \begin{align}
       \bigg| \,  \sum_{h, k \ge 1} \ \frac{\lambda_{h}\overline{\lambda_{k}}}{\sqrt{hk}}  \ \mathcal{U}^{(r)}(h, k)\bigg| 
    & \preceq \, TQ\mathfrak{C}^3 \mfq^2   \,  \max_{\substack{AL 
	\ll  \frac{C}{Q}(TQ)^{\mfr/2+\theta}\mathfrak{C} \\ AEL \ll\mathfrak{C} (TQ)^{1000}  }} \ \frac{1 }{A^2E^2L}\, \mathcal{E}_{[1, \infty)}, 
  \end{align}
where
  \begin{align}\label{optimisum}
           \mathcal{E}_{\mathcal{I}} \, := \, \frac{1}{T}\, \sideset{}{^d}{\sum}_{S\in \mathcal{I}} \, \min_{k\ge 1} \frac{\min\{T, S\} \mcX^{k-1}}{S^{k}} \Big(\max\{T,S\}(AEL)^2 + (TQ)^{\theta}\big(AEL\max\{T,S\}\big)^{\mfr/2}\Big). 
    \end{align}

\end{lem}

\begin{proof}
We use the same argument as Proposition \ref{LsuHLSter}. Let 
\begin{align}
    \mathcal{N} \, := \, \sqrt{\mfC(\Pi_{1})}(TQ)^{\theta}(\mfq AEL\max\{T,\,S\}/f)^{\mfr/2}.
\end{align}
By  (\ref{singAFE}) and (\ref{singAFEtrunctail}), up to an error term of size $\preceq (\mfC(\Pi_{1})^{1/4}(TQ)^{-A}$, we have
\begin{align}
    \Big|L(1/2+s,\,  \Pi_{1}\times \psi) &\sum_{\substack{(H, \, g_{3}\ma)=1 }} \, \frac{ \lambda_{v_{1}H}\psi(H) }{H^{\frac{1}{2}+s}} \Big| \preceq  \, \sum_{M_{0}\mid \mfq^{\infty}} \, \frac{|\lambda_{\Pi\times \psi}(M_{0})|}{\sqrt{M_{0}}} \, \Big|\sum_{N \ll \mathcal{N}} \, \frac{\psi(N)}{N^{1/2+s}}\sum_{\substack{M'H=N\\ (H,\, g_{3}cea)=1\\ (M',\mfq)=1}} \lambda_{\Pi_{1}}(M') \lambda_{v_{1}H} \Big|. \nonumber
\end{align}
 By $\vartheta<1/2$, the $M_{0}$-sum is $\ll_{\epsilon} \mfq^{\epsilon}$. Let $g_{3}, v_{1}, \ma$ be given. It follows from the Hybrid Large Sieve (Lemma \ref{Gallager}), Lemma \ref{lem: Galsum}, and Lemma \ref{lem: ROA} that
\begin{align}
     \sum_{f\ll \mfq AEL} \, & \sum_{\mnl'\ll \mfq AEL/f} \,   \sideset{}{^*}{\sum}_{\substack{\psi \, (\mnl')}} \,  \int_{\substack{\re s =\epsilon \\ |\im s| \ll \max\{T,\,S\} }} \,   \Big|L(1/2+s,\,  \Pi_{1}\times \psi) \sum_{\substack{(H, \, g_{3}\ma)=1 }} \, \frac{ \lambda_{v_{1}H}\psi(H) }{H^{\frac{1}{2}+s}} \Big|^2  \, |ds|\nonumber\\[5pt]
    & \hspace{30pt} \preceq \mfC(\Pi_{1})^{3/4} \sum_{f\ll \mfq AEL} \, \Big(\max\{T,S\}(\mfq AEL/f)^2 + (TQ)^{\theta}\big(\mfq AEL\max\{T,S\}/f\big)^{\mfr/2}\Big) \nonumber\\
     &\hspace{120pt}  \cdot \, \sum_{N\ll \mathcal{N} } \, \frac{1}{N}\Big|\sum_{\substack{M'H=N\\ (H,\, g_{3}cea)=1\\ (M',\mfq)=1}} \lambda_{\Pi_{1}}(M') \lambda_{v_{1}H}\Big|^2 \nonumber\\[5pt]
     \, &\hspace{30pt} \preceq  \, \mfC(\Pi_{1})^{3/4} \mfq^2 \big\{\max\{T,S\}(AEL)^2 + (TQ)^{\theta}\big(AEL\max\{T,S\}\big)^{\mfr/2}\big\}.  \nonumber
\end{align}
The last step uses the assumptions made at the beginning of this section. The desired result follows from Proposition \ref{afteriniclearnUr}.
\end{proof}

\begin{lem}\label{lem: optimi}
We have the following estimates:
    \begin{align}\label{balancestim}
        \mathcal{E}_{[1, \infty)} \,  \preceq \,   (AEL)^2 +(TQ)^{\theta} (AEL)^{\mfr/2} \max\{\mcX, T\}^{\mfr/2-1};  
    \end{align}
\begin{align}
     \bigg| \,  \sum_{h, k \ge 1} \ \frac{\lambda_{h}\overline{\lambda_{k}}}{\sqrt{hk}}  \ \mathcal{U}^{(r)}(h, k)\bigg| 
      \, \preceq \,  TQ\mathfrak{C}^4 \mfq^2   \max_{\substack{L 
	\ll  \frac{C}{Q}(TQ)^{\mfr/2+\theta}}}  \big\{ L +(TQ)^{\theta} \, (L \max\{ \mcX, T\})^{\mfr/2-1} \big\}.   \label{optimi}
\end{align}
\end{lem}

\begin{proof}
The argument for (\ref{balancestim}) is as follows.
    
\noindent \textbf{Case A:} $\mcX\ge T$. Firstly, because $\mcX\ge (TQ)^{\epsilon}$ and $\min_{k\ge 1} \mcX^{k-1}=1$, we obtain
\begin{align}
  \mathcal{E}_{[1, T]} \, &\ll \,  \frac{1}{T} \big(T(AEL)^2 + (TQ)^{\theta}\big(AELT\big)^{\mfr/2}\big)\, \sideset{}{^d}{\sum}_{1\le S \le T} \, \min_{k\ge 1} \frac{\mcX^{k-1}}{S^{k-1}}  \nonumber\\
  \,  &\preceq \,  (AEL)^{2} + (TQ)^{\theta}(AEL)^{\mfr/2}T^{\mfr/2-1}. \nonumber
\end{align}
Secondly, we have
\begin{align}
    \mathcal{E}_{[T, \mcX]}  \, &= \,  \sideset{}{^d}{\sum}_{T \le S\le \mcX} \, \min_{k\ge 1} \, \mcX^{k-1} \Big( \frac{(AEL)^2}{S^{k-1}} + \frac{(TQ)^{\theta}(AEL)^{\mfr/2}}{S^{k-\mfr/2}}\Big)\nonumber\\
    \, &\preceq \, \min_{k\ge 2}\left\{ (AEL)^2 + (TQ)^{\theta}(AEL)^{\mfr/2} \mcX^{\mfr/2-1}, \, \Big(\frac{\mcX}{T}\Big)^{k-1} \big( (AEL)^2 + T^{-1} (TQ)^{\theta}(AELT)^{\mfr/2} \big)\right\}\nonumber\\
    \, &\preceq \, (AEL)^2 + (TQ)^{\theta}(AEL)^{\mfr/2} \mcX^{\mfr/2-1}. \nonumber
\end{align}
Thirdly, for any $k\ge 2$, we have
\begin{align}
    \mathcal{E}_{(\mcX, \infty)} \,\le  \,  \sideset{}{^d}{\sum}_{S>\mcX} \,   \mcX^{k-1} \Big( \frac{(AEL)^2}{S^{k-1}} + \frac{(TQ)^{\theta}(AEL)^{\mfr/2}}{S^{k-\mfr/2}}\Big) 
    \ \ll \, (AEL)^2 + (TQ)^{\theta}(AEL)^{\mfr/2} \mcX^{\mfr/2-1}. \nonumber
\end{align}
Hence, we arrive at (\ref{balancestim}).

\noindent \textbf{Case B:} $\mcX<T$. For $\mathcal{I}=[1, \mcX]$ or $[\mcX, T]$, we have
\begin{align}
    \mathcal{E}_{\mathcal{I}} \, := \, \frac{1}{T}\, \sideset{}{^d}{\sum}_{S\in \mathcal{I}} \, \min_{k\ge 1} \frac{\mcX^{k-1}}{S^{k-1}} \big(T(AEL)^2 + (TQ)^{\theta}\big(AELT\big)^{\mfr/2}\big) 
    \, \preceq (AEL)^{2} + (TQ)^{\theta}(AEL)^{\mfr/2} T^{\mfr/2-1}. \nonumber 
\end{align}
In addition, for any $k\ge 2$, we have
\begin{align}
    \mathcal{E}_{(T,\infty)} \, &\ll \,  \sideset{}{^d}{\sum}_{S>T} \,  \,  \mcX^{k-1} \Big( \frac{(AEL)^2}{S^{k-1}} + \frac{(TQ)^{\theta}(AEL)^{\mfr/2}}{S^{k-\mfr/2}}\Big) \nonumber\\
    \, &\ll \,  \,  \Big(\frac{\mcX}{T}\Big)^{k-1}\big((AEL)^2  +  (TQ)^{\theta}(AEL)^{\mfr/2}T^{\mfr/2-1}\big) \, \ll \, (AEL)^2  +  (TQ)^{\theta}(AEL)^{\mfr/2}T^{\mfr/2-1}. \nonumber
\end{align}
Again, we arrive at (\ref{balancestim}). 

Now, applying the established (\ref{balancestim}) to Lemma \ref{Usumpredeterbdd}, it follows that
\begin{align}
     \bigg| \,  \sum_{h, k \ge 1} \ \frac{\lambda_{h}\overline{\lambda_{k}}}{\sqrt{hk}}  \ \mathcal{U}^{(r)}(h, k)\bigg| \ \preceq \ &TQ\mathfrak{C}^3 \mfq^2    \max_{\substack{AL 
	\ll  \frac{C}{Q}(TQ)^{\mfr/2+\theta}\mathfrak{C} \\ AEL \ll \mathfrak{C} (TQ)^{1000}  }}   \, \bigg\{ L + (TQ)^{\theta} \, \frac{(L \max\{\mcX, T\})^{\mfr/2-1}}{(AE)^{2-\mfr/2}}\bigg\},   \nonumber 
\end{align}
and (\ref{optimi}) is immediate. This completes the proof of Lemma \ref{lem: optimi}. 
\end{proof}

\begin{prop}\label{cleanupUbd}
Suppose that $C\ge 1$ and $T\gg Q^{100\epsilon}$. If either $\mfr=2$ and $\theta\ge 0$; or if $\mfr=3$ and  $\theta\in [0,\, 1/2]$, then the following bound holds:
    \begin{align}\label{Usumfinalbddd}
     \bigg| \,  \sum_{h, k \ge 1} \ \frac{\lambda_{h}\overline{\lambda_{k}}}{\sqrt{hk}}  \ \mathcal{U}^{(r)}(h, k)\bigg| 
    & \, \preceq \, CT(TQ)^{\mfr/2+\theta}\mathfrak{C}^4\mfq^2
\end{align}
\end{prop}

\begin{proof}
When $\mfr=2$, from Lemma \ref{lem: optimi},  we immediately have
\begin{align}
     \bigg| \,  \sum_{h, k \ge 1} \ \frac{\lambda_{h}\overline{\lambda_{k}}}{\sqrt{hk}}  \ \mathcal{U}^{(r)}(h, k)\bigg| 
      \  &\preceq \,TQ\mathfrak{C}^4 \mfq^2   \max_{\substack{L \ll  \frac{C}{Q}(TQ)^{1+\theta}}}  \big\{ L +(TQ)^{\theta} \,\big\} 
     \  \preceq \,  CT(TQ)^{1+\theta}\mathfrak{C}^4 \mfq^2.\nonumber
\end{align}

Next, we consider the case when $\mfr=3$. Recall the choice of $\mcX$ from (\ref{keyXvalue}). Since $T\gg Q^{100\epsilon}$, we can pick $\mcY$ in the range $(TQ)^{10\epsilon}\ll \mcY < T$. Then 
\begin{align}
\Big|\sum_{h, k \ge 1} \, \frac{\lambda_{h}\overline{\lambda_{k}}}{\sqrt{hk}} \, \mathcal{U}^{(r)}(h, k) \Big|
  \, \preceq \, TQ\mathfrak{C}^4 \mfq^2    \Bigg\{ &\max_{\substack{L 
	\ll  \frac{C}{Q}(TQ)^{3/2+\theta}  \\  
    L> \mcY C(TQ)^{1/2+\theta}}}     \, ( L +(TQ)^{\theta} \sqrt{LT}\, ) \nonumber\\
    &\hspace{50pt} \, + \,  \max_{L  <   C (TQ)^{1/2+\theta} \mcY  } \big( L + (TQ)^{\frac{1}{4}+\frac{3\theta}{2}} \sqrt{ \mcY C T } \  \big)\bigg\}. \nonumber
\end{align}
Since $\theta\le 1/2$, we have
\begin{align}
   \max_{\substack{L 
	\ll  \frac{C}{Q}(TQ)^{3/2+\theta}  \\  
    L> \mcY C(TQ)^{1/2+\theta}}}  \, ( L +(TQ)^{\theta}\sqrt{LT}) \, &\preceq  \, CT(TQ)^{\frac{1}{2}+\theta} \left\{1+ \frac{1}{\sqrt{C}(TQ)^{\frac{1}{4}-\frac{\theta}{2}}}\right\}\, \preceq \,   CT(TQ)^{1/2+\theta}; \nonumber
\end{align}
and
\begin{align}
     \max_{L  \le   C (TQ)^{1/2+\theta} \mcY}\big( L + (TQ)^{\frac{1}{4}+\frac{3\theta}{2}} \sqrt{ \mcY C T } \  \big) \, \preceq \,  CT(TQ)^{1/2+\theta}\Big( \frac{\mcY}{T}+ \sqrt{\frac{\mcY}{TC}} (TQ)^{-\frac{1}{4}+ \frac{\theta}{2}}\Big)  \, \preceq \,  CT(TQ)^{1/2+\theta}.\nonumber
\end{align}
This completes the proof of Lemma \ref{cleanupUbd}. 
\end{proof}

\begin{cor}\label{optimUrLr}
Suppose that one of the following holds:
\begin{enumerate}
    \item $\mfr=2$, $\theta \in [0,1)$, and $Q^{\epsilon} \ll T < Q^{(1-\theta)/(1+\theta)} $; 

    \item $\mfr=3$, $\theta \in [0,1/2)$, and $Q^{\epsilon} \ll T < Q^{(1-2\theta)/(3+2\theta)} $.
\end{enumerate}
Then 
    \begin{align}\label{avLr}
  \Big|  \sum_{h, k \ge 1} \, \frac{\lambda_{h}\overline{\lambda_{k}}}{\sqrt{hk}} \,  \big(\mathcal{U}^{(r)}+\mathcal{L}^{(r)}\big)(h,k) \Big|\, \preceq  \, (TQ)^{1+\mfr/4+\theta/2}\mathfrak{C}^4\mfq^2. 
\end{align}
\end{cor}

\begin{proof}
Recall from Proposition \ref{LsuHLSter} that
\begin{align}\label{Recall: Lsum}
   \Big| \, \sum_{h, k \ge 1} \, \frac{\lambda_{h}\overline{\lambda_{k}}}{\sqrt{hk}} \,  \mathcal{L}^{(r)}(h,k) \, \Big|\, \preceq \,   \mathfrak{C}^3\Big\{  \frac{TQ^2}{C} \, + \, \frac{(TQ)^{\mfr/2+\theta}}{C^{\mfr/2-1}}\Big\}.
\end{align}
 The second term on the right-hand side of (\ref{Recall: Lsum}) is clearly smaller than $CT(TQ)^{\mfr/2+\theta}$. The result follows from Propositions \ref{LsuHLSter} and \ref{cleanupUbd} upon taking
 \begin{align}
     C \ = \  
     \begin{cases}
     Q^{(1-\theta)/2}/ T^{(1+\theta)/2}  \hspace{15pt} \text{ if } \hspace{15pt}  \mfr \, = \, 2 \\
         Q^{1/4-\theta/2}/T^{3/4+\theta/2} \hspace{13pt} \text{ if } \hspace{15pt}  \mfr \, = \, 3. 
     \end{cases}
 \end{align}
 The restriction on the range of $T$ originates from the requirement $1\le C \ll Q^{1-\epsilon}$. 
\end{proof}

\begin{proof}[Proof of Theorem \ref{thm: gl3, dualcase} and \ref{thm: GL2general}]
These follow from putting Corollary \ref{optimUrLr}, Propositions \ref{+vedrop}--\ref{Umindneg}, Proposition \ref{tersm}, Propositions \ref{U2dropeq}--\ref{L0ddropterms}, Proposition \ref{prop: auxcanU1L0}, Proposition \ref{LsuHLSter} together. (Recall also the set-up in Section \ref{sect: DLUsumdef}). The rest of our theorems of Section \ref{sect: MVTtheoemre} are special cases of Theorem \ref{thm: gl3, dualcase}.
\end{proof}



\section{Levinson and Ratios without GRC: Proof of Corollaries \ref{cor: fullresultpgl3per} and \ref{cor: simplpergl2}}\label{sect: CSLevGL3}


In this section, we deduce the consequences stated in Section \ref{sect: introLevin} using our mean value theorem (Theorem \ref{generalL4version them}) and Levinson's method. The computations are implemented in the accompanying \texttt{Mathematica} notebook \texttt{GL3Ratio.nb} (Appendix \ref{sect: ratiosMath}).

Thanks to \cite{Yo10} and \cite{CS07}, this step is largely reduced to a \emph{ratios} calculation, making use of the arithmetic of the diagonal  (\ref{mainthm: momentexp}) and the mollifier (\ref{ourexamollif}). We focus on the harder case  $\Pi\in \mathcal{A}_{0}(\mathrm{G}_{3})$. Recall, for $\re s\gg 1$, that
\begin{align}
    \frac{1}{L(s, \Pi)} \, = \,  \prod_{p} \, \sum_{n=0}^{\infty} \frac{\mu_{\Pi}(p^n)}{p^{ns}} \ = \  \prod_{p} \, ( 1- \lambda_{\Pi}(p)p^{-s} + \mathfrak{e}_{\Pi}(p)p^{-2s} - \omega(p)p^{-3s}).
\end{align}
Our goal is to establish an equality of the form
\begin{align}\label{intro: ratio}
      \sum_{\substack{mh=nk \\ (mnhk,q)=1}} \, \frac{\lambda_{\Pi}(m) \overline{\lambda_{\Pi}}(n) \mu_{\Pi}(h) \overline{\mu_{\Pi}}(k)}{m^{1/2+\alpha} n^{1/2+\beta} h^{1/2+\gamma} k^{1/2+\delta}} \, = \,  \prod_{p\nmid q} \, \frac{\mathbf{L}_{p}(1+\alpha+\beta)\mathbf{L}_{p}(1+\gamma+\delta)}{\mathbf{L}_{p}(1+\alpha+\delta)\mathbf{L}_{p}(1+\beta+\gamma)}\mathbf{A}_{p}(\alpha, \beta, \gamma, \delta)
\end{align}
for $\re \alpha, \, \re \beta,\, \re \gamma,\, \re \delta \gg 1$. \footnote{ In practice (\cite[Section 6]{Yo10}), the parameters $\alpha, \beta$ in (\ref{ratioEP}) can be interpreted as $\alpha+s,\, \beta+s$, where $s$ is initially taken with large real part and is the integration variable in the cut-off function (\ref{newcutoff}); likewise, the parameters $\gamma, \delta$ are initially taken with large real parts, since they occur as the integration variables in Perron's formula.} Here $\mathbf{L}_{p}$ denotes a suitable $L$-factor such that the infinite product of $\mathbf{A}_{p}(\alpha, \beta, \gamma, \delta)$ converges absolutely for $\re \alpha, \re \beta, \re \gamma, \re \delta> -\delta_{0}$ for some small absolute $\delta_{0}>0$.

Following the procedure of \cite{CS07}, and guided by Corollary \ref{cor: untwismoSpect} and (\ref{intro: naiveRs}), the suitable $L$-functions initially appear to be $\mathcal{L}_{p}(\,\cdot\, ,\Pi\otimes \widetilde{\Pi})$. However, this runs into (at least) two issues when $\Pi\in \mathcal{A}_{0}(\mathrm{G}_{3})$: first, the infinite product of the arithmetic factors converges absolutely in the desired region only if one assumes unproven progress towards \textbf{GRC}; second, the local factor $ \mathcal{L}_{p}(\mbs, \Pi\otimes \widetilde{\Pi})$ can vanish on $\re \mbs >1-\delta$ for small $p$'s; see Corollary \ref{Hlocalnonvan}. 

It turns out that one should instead use the Rankin--Selberg local factor $L_{p}(\,\cdot\, ,\Pi\otimes \widetilde{\Pi})$ from automorphic representations. Nevertheless, the computation of the corresponding arithmetic factor is by no means straightforward, and we carry it out using \texttt{Mathematica}.\footnote{ whereas the counterpart for $\mathrm{GL}_{2}$ can be reasonably done by hand; see \cite[Section 3.2.1]{Ber15}. }  It is also \emph{a priori} unclear whether the resulting arithmetic factors satisfy the required analytic properties. We show that they do, unconditionally, using bounds of  Kim--Sarnak \cite{KS03} and Li \cite{Li10}.

To simplify our upcoming discussions, we use the shorthands: 
\begin{align*}
    x=p^{-1-\alpha-\beta}, \hspace{10pt} y=p^{-1-\gamma-\delta},  \hspace{10pt} z_{1}=p^{-1-\alpha-\delta},  \hspace{10pt} z_{2}:=p^{-1-\beta-\gamma};
\end{align*}
\begin{align}\label{ratio: polypart}
  \mathfrak{U}(y; \Pi, d) \ := \   \sum_{h=0}^{3-d} \, \mu_{\Pi}(p^{d+h})\mu_{\widetilde{\Pi}}(p^h) y^h \hspace{15pt} \text{and} \hspace{15pt}
  \mathcal{L}(x; \Pi,d) \ := \ \sum_{n\ge 0} \, \lambda_{\Pi}(p^{n+d}) \lambda_{\widetilde{\Pi}}(p^n)x^n; 
\end{align}
and $L_{p}(x):= L_{p}(1+\alpha+\beta, \Pi\otimes\widetilde{\Pi})$ (similarly for the remaining Rankin--Selberg local factors). Upon rearranging the summation, observe that the local Euler factor of (\ref{intro: ratio}) at $p$, i.e.,
\begin{align}\label{ratioEP}
\mathbf{E}_{p}(\alpha,\beta,\gamma, \delta) \ := \     \sum_{\substack{m,n,h,k\ge 0\\ m+h=n+k}}\, \frac{\lambda_{\Pi}(p^m) \overline{\lambda_{\Pi}}(p^n) \mu_{\Pi}(p^h) \overline{\mu_{\Pi}}(p^k)}{(p^m)^{1/2+\alpha} (p^n)^{1/2+\beta} (p^h)^{1/2+\gamma} (p^k)^{1/2+\delta}}
\end{align}
is equal to
\begin{align}\label{ratio: decompose}
\hspace{-5pt} \mathbf{E}(x,y;  z_{1},z_{2})=    \mathfrak{U}(y; \widetilde{\Pi},0)  \mathcal{L}(x;\Pi,0) +    \sum_{d=1}^{3}  z_{1}^d\, \mathfrak{U}(y; \widetilde{\Pi},d)  \mathcal{L}(x;\Pi,d)  +  \sum_{d=1}^{3}   z_{2}^{d}\,   \mathfrak{U}(y; \Pi,d)  \mathcal{L}(x; \widetilde{\Pi},d),
\end{align}
where $\mathcal{L}(x, \Pi,d):=\mathcal{L}_{p}(1+\alpha+\beta;\,  \Pi,   \widetilde{\Pi}; \, d,0) $ has been evaluated in Lemma \ref{lem: diffby1RSn}, and is divisible by $L_{p}(x)$. With the help of \texttt{Mathematica (\texttt{GL3Ratio.nb})}, we show that
\begin{prop}\label{compratioprop}
Let $\delta_{0}:=10^{-6}$. For any $p\nmid q$, we define
\begin{align}\label{ratio: arithfact}
    \mathbf{A}(x,y,z_{1},z_{2}) \ := \  \frac{L_{p}(z_{1})L_{p}(z_{2})}{L_{p}(x)L_{p}(y)} \, \mathbf{E}(x,y;  z_{1},z_{2}). 
\end{align}
The Euler product $\prod_{p\,\nmid\, q} \,  \mathbf{A}(p^{-u},p^{-v}, p^{-w_{1}},p^{-w_{2}})$ converges absolutely on $\re u$, $\re v$, $\re w_{i}> 1-\delta_{0}$.
\end{prop}

\begin{proof}
In what follows, the terms $O(p^{-1-\epsilon})$ are understood after one makes the substitution $(x,y;  z_{1},z_{2})$ $\to (p^{-u},p^{-v}, p^{-w_{1}},p^{-w_{2}})$ and for $\re u$, $\re v$, $\re w_{i}> 1- \delta_{0}$.

We first handle the difficult case when $p\nmid\mfq$. Since $\prod\, \alpha_{i}=1$ and $\Pi$ is unitary, the coefficients of $  \mathbf{E}$ can be expressed in $\lambda:= \lambda_{\Pi}(p)=\sum \alpha_{i}$ and $\overline{\lambda}=\sum \alpha_{i}^{-1}$. Moreover, it suffices to consider the first two summands of (\ref{ratio: decompose}); the third follows by symmetry: $\Pi \to \widetilde{\Pi}$ and $z_{1}\to z_{2}$. We will use $|\lambda| \le 3p^{\vartheta}$ with $\vartheta\le 1/2-1/7$ below. We have the following evaluations:
\begin{align}
   \mathfrak{U}(y; \widetilde{\Pi},0)  \mathcal{L}(x;\Pi,0) \ &= \  L_{p}(x)\big( (1-x^3)^2 - |\lambda|^2 \big(x(1-x)\big)^2\big) \big(1 +|\lambda|^2y + |\lambda|^2y^2 +y^3\big) \nonumber\\
   \ &= \ L_{p}(x) ( 1 +  O(p^{2\vartheta-2}))(1 + |\lambda|^2y  +  O(p^{2\vartheta-2})) \nonumber\\
   \ &= \ L_{p}(x)(1+|\lambda|^2y + O(p^{-1-\epsilon}));
\end{align}
\begin{align}
 z_{1}\mathfrak{U}(y; \widetilde{\Pi},1)  \mathcal{L}(x;\Pi,1) \ &= \   - z_{1} (1-x)^2(1+x)\big(\lambda(1+x+x^2) -\overline{\lambda}^2x\big)\big((1+y^2)\overline{\lambda}+y \lambda^2\big) L_{p}(x) \nonumber\\
 \ &= \  \big(\overline{\lambda}^3 x z_{1}  -|\lambda|^2 z_{1}- \lambda^3yz_{1} \, + \, O(p^{-1-\epsilon})\big)L_{p}(x);
\end{align}
\begin{align}
 z_{1}^2\, \mathfrak{U}(y; \widetilde{\Pi},2)  \mathcal{L}(x;\Pi,2) \ &= \    (x-1)^2 (y+1) z_{1}^2 \lambda \big((1-x+x^2)(\lambda^2- \overline{\lambda})+ x\lambda ( \overline{\lambda}^2-\lambda)  \big) L_{p}(x)\nonumber\\
 \ &= \  (z_{1}^2 \lambda^3 + O(p^{-1-\epsilon}))L_{p}(x);
\end{align}
and the following quantity is $O(p^{-1-\epsilon})$ because:
\begin{align}
 z_{1}^3\, \mathfrak{U}(y; \widetilde{\Pi},3)  \mathcal{L}(x;\Pi,3) \ = \    -(x-1)^2 z_{1}^3\Big(1-&2 |\lambda|^2 
   +\lambda^3+x (2 \lambda^3+\overline{\lambda}^3-|\lambda|^4-3|\lambda|^2  +2)\nonumber\\
   &\hspace{-50pt}+x^2 (\lambda^3+\overline{\lambda}^3-4 |\lambda|^2 +3)+x^3 (2-|\lambda|^2 )+x^4\Big) L_{p}(x).
\end{align}
Denote by $\mathbf{P}(x,y; z_{1}, z_{2})$ following quadratic polynomial:
\begin{align}\label{ratio: quadrapoly}
  1 \, + \, |\lambda|^2y -|\lambda|^2 z_{1} -|\lambda|^2 z_{2} \, + \, \overline{\lambda}^3 x z_{1} + \lambda^3 x z_{2} \   - \lambda^3yz_{1}  - \overline{\lambda}^3yz_{2}         + \lambda^3 z_{1}^2 + \overline{\lambda}^3 z_{2}^2.  
\end{align}
Collecting the calculations above, we have
\begin{align}
      \mathbf{E}(x,y;  z_{1},z_{2}) \ = \ L_{p}(x)(\mathbf{P}(x,y; z_{1}, z_{2}) + O(p^{-1-\epsilon})).
\end{align}

Next, using $\vartheta\le 1/2-1/7$ and the definition of the Rankin--Selberg local factors, we obtain
\begin{align}
    L_{p}(y)^{-1} \ &= \  1- y|\lambda|^2 +y^2(\lambda^3-2|\lambda|^2+ \overline{\lambda}^3) \ + \ O(p^{-1+\epsilon}), \\
    L_{p}(z) \ &= \ 1+ z|\lambda|^2 +z^2 (-\lambda^3 + 2 |\lambda|^2 + |\lambda|^4 - \overline{\lambda}^3)\nonumber\\
    &\hspace{40pt}+ z^3 (3 + \lambda^3 - 6 |\lambda|^2 - 2 \lambda^4 \overline{\lambda} + 5  |\lambda|^4 + \overline{\lambda}^3 +|\lambda|^6 - 2 \lambda \overline{\lambda}^4)  \ + \ O(p^{-1+\epsilon}).
 \end{align}
It follows that
\begin{align}
    \mathbf{P}(x,y; z_{1}, z_{2}) L_{p}(y)^{-1} \prod_{i=1}^{2} \, L_{p}(z_{i}) \ = \  \mathbf{Q}(x,y; z_{1}, z_{2}) \, + \, O(p^{-1-\epsilon}),\nonumber
\end{align}
where $ \mathbf{Q}(x,y; z_{1}, z_{2})$  is given by the following cubic polynomial:
\begin{align}\label{ratio: cubic}
  \hspace{-5pt} 1+ &\lambda^3 y^2 -|\lambda|^4 y^2 +\overline{\lambda}^3 y^2 +\overline{\lambda}^3 x z_{1} -\lambda^3 y z_{1} + |\lambda|^4 y z_{1} - |\lambda|^6 y^2 z_1 - \overline{\lambda}^3 z_1^2 +|\lambda|^6 y z_1^2 + \lambda^3 x z_2  \nonumber\\
    &\hspace{-15pt}+ |\lambda|^4 y z_2 - \overline{\lambda}^3 y z_2-|\lambda|^6 y^2 z_2-|\lambda|^4 z_1 z_2 + 2 |\lambda|^6 y z_1 z_2 -|\lambda|^6 z_1^2 z_2 - \lambda^3 z_2^2 + |\lambda|^6 y z_2^2 - |\lambda|^6 z_1 z_2^2.
\end{align}
In other words, we take $ \mathbf{A}(x,y; z_{1}, z_{2}) = \mathbf{Q}(x,y; z_{1}, z_{2}) \, + \, O(p^{-1-\epsilon}),$ and the absolute convergence of the Euler product $\prod_{p\nmid q\mfq} \,  \mathbf{A}(p^{-u},p^{-v}, p^{-w_{1}},p^{-w_{2}})$ follows from that of the infinite series:
\begin{align}
    \sum_{p\nmid \mfq}\,  \frac{|\lambda(p)|^3}{p^{2-\delta_{0}}} \ll   \sum_{p\nmid \mfq}\,  \frac{|\lambda(p)|^2}{p^{2-\vartheta-\delta_{0}}}; \hspace{15pt} \sum_{p\nmid \mfq}\,  \frac{|\lambda(p)|^4}{p^{2-\delta_{0}}}  \ll \sum_{p\nmid \mfq}\,  \frac{|\lambda(p)|^2}{p^{2-2\vartheta-\delta_{0}}} ; \hspace{15pt}
   \sum_{p\nmid \mfq}\,  \frac{|\lambda(p)|^6}{p^{3-\delta_{0}}}\ll \sum_{p\nmid \mfq}\,  \frac{|\lambda(p)|^2}{p^{3-4\vartheta-\delta_{0}}},\nonumber
\end{align}
which are consequences of $\vartheta\le 1/2-1/7$ and (\ref{trivrankcoefbdd}). 

Next, we treat the ramified primes $p\mid \mfq$, which is quite straightforward. Note that since $(q,\mfq)=1$, we must have $p\nmid q$. Returning to (\ref{ratio: decompose}), with $\vartheta\le 1/2-1/7$ and Lemma \ref{lem: diffby1RSn}, we have
\begin{align}
    \sum_{d=1}^{3} \, z_{1}^d\, \mathfrak{U}(y; \widetilde{\Pi},d)  \mathcal{L}(x;\Pi, d) \ &= \  z_{1}(\mu_{\widetilde{\Pi}}(p)+\mu_{\widetilde{\Pi}}(p^2)\mu_{\Pi}(p)y)\mathcal{L}(x;\Pi, 1) + z_{1}^2 \mu_{\widetilde{\Pi}}(p^2)\mathcal{L}(x;\Pi, 2) \nonumber\\
    \ &\ll \ p^{10\delta}(p^{-1}\cdot p^{\vartheta}\cdot p^{\vartheta} + p^{-2}\cdot  p^{2\vartheta}\cdot p^{2\vartheta}) \ \ll \ p^{-1+2\vartheta+10\delta}
\end{align}
Similarly, $ \mathfrak{U}(y; \widetilde{\Pi}, 0) \mathcal{L}(x;\Pi, 0) = 1+ O(p^{-1+2\vartheta+10\delta})$. Hence, 
\begin{align}
 \prod_{p\mid\mfq} \,  \mathbf{E}(p^{-u},p^{-v}, p^{-w_{1}},p^{-w_{2}}) \ = \  \prod_{p\mid \mfq} \, ( 1+ O(p^{-1+2\vartheta+10\delta})) \, \ll_{\epsilon} \, \mfq^{\epsilon}.
\end{align}
With (\ref{genRSbdd}), the factors $L_{p}(z_{1})$, $L_{p}(z_{2})$, $L_{p}(x)$, $L_{p}(y)$ do not vanish and are $1+O(p^{-1+2\vartheta+\delta}))$. Therefore, we also have
\begin{align*}
    \prod_{p\mid \mfq} \,  \frac{L_{p}(z_{1})L_{p}(z_{2})}{L_{p}(x)L_{p}(y)} \, \ll_{\epsilon} \, \mfq^{\epsilon}.
\end{align*}
This completes the proof of Proposition \ref{compratioprop}.
\end{proof}

\begin{proof}[Proof of Corollary \ref{cor: fullresultpgl3per}]
We follow the argument of \cite{Yo10, Ber15} closely. Let $\Pi\in \mathcal{A}_{0}(\mathrm{G}_{3})$ be fixed. Let $\varpi,\, \theta>0$, $T= Q^{\varpi}$, $X:=(TQ)^{\theta}$, and $L:=\log TQ$. Let $\alpha, \beta \ll L^{-1}$ and $|\alpha+\beta|\gg L^{-1}$. Take $H_{\alpha, \beta}$ to be the cut-off function defined in (\ref{doublecut}) with $G(s)  := e^{s^2} \, \frac{(\alpha+\beta)^2-(2s)^2}{(\alpha+\beta)^2}.$
   
By the same exact argument as \cite[Lemmata 6-7]{Yo10} or \cite[Lemmata 9-10]{Ber15}, we have
\begin{align}\label{followYou}
   \sum_{\substack{h,k \le X}}\frac{\mu_{\Pi}(h) \mu_{\widetilde{\Pi}}(k)}{\sqrt{hk}} \mcP\Big( \frac{\log X/h}{\log X}\Big) &  \mcP\Big( \frac{\log X/k}{\log X}\Big)  \mathop{\sum \sum}_{\substack{mh = nk  \\ (mnhk, \, q)=1}} \, \frac{\lambda_{\Pi}(m)\lambda_{\widetilde{\Pi}}(n)}
        {m^{\frac12 + \alpha} n^{\frac12 + \beta}} H_{\alpha, \beta}\Big(1, \frac{ mn}{\mfq q^\mfr}\Big) \nonumber\\
 \ & = \  \Big(\int_{\R} \, \eta(t) \, dt\Big)\, \frac{\phi^*(q)}{2}  \mathcal{S}(\alpha, \beta) \, + \, O(TQL^{-1})
\end{align}
for any $\theta<1$, where
\begin{align}
   \mathcal{S}(\alpha, \beta) \, := \,  \frac{1}{(\alpha+\beta)\log X} \, \frac{d^2}{dx\,dy} \, X^{\alpha x+\beta y}\int_{0}^{1} \, \mathcal{P}(x+u)\mathcal{P}(y+u) \, du \Big|_{x=y=0}. 
\end{align}
The argument uses Proposition \ref{compratioprop} with a standard prime-number-theorem-type argument.  The standard zero-free region for $L(s, \Pi \otimes\widetilde{\Pi})$ and the associated logarithmic lower bound (see \cite[Theorem 1.9]{HB19} and \cite[Theorem 2.1]{HT22}) are essential. We also need a simple fact:
\begin{align}\label{constbeingone}
    \mathbf{A}(0,0,0,0) \ = \ 1.
\end{align}
This can be seen as follows. For $\re \alpha \gg 1$, we have $ \mathbf{A}(\alpha,\alpha, \alpha, \alpha )$ given by
\begin{align}
     \sum_{\substack{m,n,h,k\ge 1 \\ mh=nk}} \, \frac{\lambda_{\Pi}(m) \overline{\lambda_{\Pi}}(n) \mu_{\Pi}(h) \overline{\mu_{\Pi}}(k)}{(mnhk)^{1/2+\alpha} }   \sum_{r\ge 1} \, \frac{1}{r^{1+2\alpha}} \sum_{mh=r} \, \lambda_{\Pi}(m)\mu_{\Pi}(h) \, \sum_{nk=r} \, \overline{\lambda_{\Pi}(n)\mu_{\Pi}(k)} \ = \  1. 
\end{align}
By Proposition \ref{compratioprop} again, (\ref{constbeingone}) follows by analytic continuation.  

To apply Theorem \ref{generalL4version them}, let's check that Condition  $\mathbf{(\Lambda)}$ is  satisfied for the coefficients
\begin{align*}
    \lambda_{h} \, := \, \mu_{\Pi}(h) \mcP\Big( \frac{\log X/h}{\log X}\Big).
\end{align*}
Indeed, under Hypothesis $(\mathbf{\Pi}^4)$ (which is known to hold for $\Pi$ self-dual), observe that 
\begin{align}
    \sum_{h\le X} \, \frac{|\lambda_{h}|^4}{h} \ &\ll_{\epsilon} \ X^{\epsilon} \,  \sum_{h=1}^{\infty} \, \frac{|\mu_{\Pi}(h)|^4}{h^{1+\epsilon}} \nonumber\\
    \ &= \  X^{\epsilon}\,  \prod_{p\nmid\mfq} \, \Big(1+ \frac{|\lambda_{\Pi}(p)|^4}{p^{1+\epsilon}}+ \frac{|\lambda_{\Pi}(p)|^4}{p^{2(1+\epsilon)}} + \frac{1}{p^{3(1+\epsilon)}}\Big) \prod_{p\mid \mfq} \, \Big(1+ \frac{|\lambda_{\Pi}(p)|^4}{p^{1+\epsilon}} \, + \, \frac{|\mathfrak{e}_{\Pi}(p)|^4}{p^{2(1+\epsilon)}}\Big) \nonumber\\
    \, &\ll_{\Pi, \, \epsilon}  \, X^{\epsilon} \exp\Big(2\sum_{n=1}^{\infty} \, \frac{|\lambda_{\Pi}(n)|^4}{n^{1+\epsilon}}\Big) \ \ll_{\Pi, \, \epsilon}\ X^{\epsilon}. 
\end{align}
The two remaining bounds in Condition  $\mathbf{(\Lambda)}$ are obvious to check.

Let  $a=1/2-R/L$, $\theta= (1-3\varpi)/(2(1-\varpi))-O(\epsilon)$, and recall our choice of mollifier  (\ref{ourexamollif}). Applying Theorem \ref{generalL4version them} and (\ref{followYou}), we obtain 
\begin{align}
 &\sum_{\substack{(q, \, \mfq)=1}} \,   W\big(\frac{q}{Q}\big)  \,
    \sideset{}{^\flat}{\sum}_{\chi(q)} \,  \int_{\mathbb{R}}    \eta(t)L\big(\frac{1}{2}  + \alpha+it, \, \Pi \times \chi\big)  L\big(\frac{1}{2}+ \beta-it, \, \widetilde{\Pi}
    \times \overline{\chi}\big)|M(a+it, \Pi\times \chi)|^2  \ dt \nonumber\\
    \, &\hspace{10pt}= \, \Big(\int_{\R} \, \eta(t) \, dt\Big)\, \sum_{\substack{(q, \, \mfq)=1}} \,   W\big(\frac{q}{Q}\big)\frac{\phi^*(q)}{2} \,  \big(\mathcal{S}(\alpha, \beta)+ (qT)^{-3(\alpha+\beta)}\mathcal{S}(-\beta,-\alpha)\big) \, + \, O(TQ^2L^{-1}).
 \end{align}
 Let $c_{\theta}(\mathcal{P}, \mathcal{Q},R)$ be the constant introduced in (\ref{Conreyformconst}). Following \cite[Lemma 3]{Yo10} or \cite[Theorem 5]{Ber15}, we readily conclude the evaluation of our mollified second moment:
\begin{align}
    \sum_{\substack{(q, \, \mfq)=1}} \,   W\big(\frac{q}{Q}\big) \,
    \sideset{}{^\flat}{\sum}_{\chi(q)} \int_{0}^{T} \, &\Big| \mcQ\Big(-\frac{1}{3L} \frac{d}{ds}\Big) (L M)(a+it, \Pi \times \chi) \Big|^2 \, dt \nonumber\\
    \  &= \   c_{\theta/3}(\mcP,\mcQ, 3R)\, T\sum_{\substack{(q, \, \mfq)=1}} \,   W\big(\frac{q}{Q}\big)\, \frac{\phi^*(q)}{2} \, + \, O(TQ^2L^{-1}).
\end{align}
Corollary \ref{cor: fullresultpgl3per} follows immediately from \cite[Corollary A]{CIS13}. When $\mfr=3$, The proportion (\ref{ourproportionlev}) is obtained by taking  $\mcP(x)=x$, $\mcQ(y)$ constructed in our prior work \cite{ConShortMollif}, and $R=(1/\sqrt{15})\theta^{-1}$. 
\end{proof}

\begin{proof}[Proof of Corollary \ref{cor: simplpergl2}]
This goes back to an observation of Heath-Brown \cite{HB79} and Selberg that Levinson’s method detects simple zeros, provided that $\mathcal{Q}(y)$ is a \emph{linear} polynomial with $\mathcal{Q}(0)=1$. For $T=Q^{\epsilon}$ and $\mfr=2$, we take $\mcP(x)=x$, $\mathcal{Q}(x)=1-x$, $R=1.3$ (i.e., the original choice of \cite{Lev74}), we obtain the proportion $1/3-O(\epsilon)$. The second part of Corollary \ref{cor: simplpergl2} follows from the computation of \cite[p. 215]{Far94}. For Remark \ref{gl3simplezeors}, see  \cite[Remark (p. 179)]{CIS13}.    
\end{proof}


\appendix

\section{Glossary of notation}\label{sect: gloss}

\subsection{Sections \ref{sect: collectprelim}--\ref{LsumHLSbdd}}

\begin{enumerate}
\item $\mathfrak{e}_{\Pi}(p)$; see (\ref{GL3secelem}).

\item $\mathcal{L}_{p}(s;\,  \Pi_{1},   \Pi_{2}; \, d_{1}, d_{2})$; see (\ref{dshiftnaiveDsdef}).

\item $ \mathfrak{f}_{s,\Pi_{1}, \Pi_{2}}(n)$; see (\ref{convshiftRS}).

\item $V_{\alpha,\beta}(y; it)=V_{\alpha,\beta}(y; it; \, \Pi_{1}, \Pi_{2})$, $H_{\alpha, \beta}(v,y)=H_{\alpha, \beta}(v,y;\,  \Pi_{1}, \Pi_{2})$, $\widetilde{V}_{\alpha,\beta}(s;it)=\widetilde{V}_{\alpha,\beta}(s;it;\,  \Pi_{1}, \Pi_{2})$; see (\ref{newcutoff}), (\ref{doublecut}), (\ref{Mellincutprod}).

\item  The parameter $C$ is specified in (\ref{parameterC}).

\item See Section \ref{sect: DLUsumdef} for the list of sums considered in this article.

\item  $\mbs= 1+2s+\alpha+\beta$ and $\mbw:=u-\mfr s$; see (\ref{conveshiftdiag}) and (\ref{Arithmaintermprep}).

\item $\mbw_{1} := s+it+\frac{1}{2}+\alpha$ and $ \mbw_{2} :=  s-it+\frac{1}{2}+\beta$; see (\ref{openupLs}).

\end{enumerate}

\subsection{Section \ref{absDiv}}

\begin{enumerate}
\item $g:= (mh, nk)$, $M:=mh/g$, $N:= nk/g$;  see (\ref{divswiauxpara}).

\item The variables $r, e,a, \ell$; see (\ref{ddivswi}).

\item $\mnl:=rea\ell$; see (\ref{key: conduct}).

\item $\mathrm{U}^{(*)}(m,n; h,k)$; see  (\ref{absU0sum}).

\item $\mathcal{W}(x)$, $\widetilde{\mathcal{W}}^{\pm}(w)$; see (\ref{Ffunc})--(\ref{simMell}).

\end{enumerate}

\subsection{Sections \ref{U0sumsect}--\ref{compdoubsumGCD}}

\begin{enumerate}
    \item $m_{0}, m', n_{0}, n', h_{0}, h', k_{0}, k', g_{0}, g', M_{0}, M', N_{0}, N'$; see (\ref{splitmnaccortoq})--(\ref{MNdecompo0pr}).

    \item $H:=h/(h,k)$, $K:=k/(h,k)$, $H_{0}, H', K_{0}, K', \ell_{0},\ell'$; see (\ref{diagspecasesplit}).

    \item $R_{1}(w; u, v)$, $R_{M_{0}N_{0}}(w)$, $R_{\mfq}^{(M_{0}N_{0})}(w)$, 
   $\mathfrak{R}(w; g, MN, \mfq)$; see Lemma \ref{lem: U0anapar} and (\ref{collapseeal}).

   \item $\mathrm{U}^{\mathbf{1}}(m,n; h,k)$, $\mathrm{U}^{\mathbf{2}}(m,n; h,k)$, $ \mathrm{L}^{\mathbf{0}}(m,n;h,k)$; see (\ref{split}) and (\ref{innerL0}).

   \item $g_{\delta}(z)$, $\mathcal{H}(z, -w)$, $\widetilde{\mathcal{W}_{\delta}^{^\pm}}(w)$, $
 \mathcal{U}_{\delta}^{\mathbf{2}}(h,k)$; see (\ref{CIStrickdel}), (\ref{eq: smallsmthwei}), (\ref{U2arrane}). 
\end{enumerate}

\subsection{Section \ref{sect: NoODU2}}
\begin{enumerate}
    \item $\mathcal{K}^{(\mfq)}_{it}(s; z,w)$, 
  $\mathcal{E}_{\mfq}(s; z,w; it )$; see (\ref{keydoubDS})--(\ref{ram: finiteEPforU0}).

  \item $\mbu, \mbv_{0}, \mbs$; see (\ref{U2varchoi}). 

  \item $ \mathbf{L}_{p}(\mbu, \mbs; \Pi_{1}, \Pi_{2} )$, 
 $\mathbf{D}_{p}(\mbu, \mbv_{0}; \, \Pi_{1}, \Pi_{2})$, $\mathcal{A}_{p}(\mbu, \mbs; \Pi_{1}, \Pi_{2})$; see (\ref{mathbddLuspi}), (\ref{U2Dfactdef}), (\ref{remainbyKSgl3}).
\end{enumerate}

\subsection{Section \ref{sect: UrHLS}}

\begin{enumerate}
 \item  $M,N, H,K$, $g_{1}, g_{2}, g_{3}, g_{4}$ as the summation variables in $ \mathcal{U}^{(\mathbf{r})}(h,k)$; see Section \ref{sepmovar}. The notation $M,N,H,K$ used here is independent of its earlier use in the preceding sections.

 \item $\mathcal{U}^{\pm}(c; a, e, \ell; \psi)$, and $a\sim A$, $e\sim E$, $\ell\sim L$; see (\ref{sepsum}).

   \item The parameters $\mcX, \, \mcY,\, U$; see (\ref{keyXvalue})--(\ref{func: redund}). 

   \item  $\mathfrak{W}_{\mcX}^{\pm}(x,y; v,u)$, $\widetilde{\mathcal{V}}_{\mcX}^{^\pm}(s_{1},s_{2}; z)$; see (\ref{newtrans}), (\ref{2-varMelint}).

      \item $\mathfrak{a}:=rcea$, $\mathfrak{w}:=1+s_{1}+s_{2}+z$, $u_{1}:= \frac{1}{2}+\alpha+it+s_{1}+z$, $ u_{2} :=   \frac{1}{2}+\beta-it+s_{2}+z$; $\mz:=1+\alpha+\beta+2z$; see (\ref{mamlu12}).  

      \item  $\mathcal{L}_{a, g_{1}}^{(ec)}(\mz; \, \Pi_{1}, \Pi_{2};\, g_{3}, g_{4})$; see (\ref{dropdiagg2sum}).

    \item $g_{p}(a):= \max\{0, o_{p}(a)-o_{p}(g_{1}g_{3}g_{4})\}$; see (\ref{gp0: divindic}) and (\ref{fiveEPs}).

\end{enumerate}

\subsection{Section \ref{sect: boundUrsum}}

\begin{enumerate}
    \item  $d_{1},d_{2}, d_{3}, r_{1}, r_{2}$; see Section \ref{sect: MobinTwist}.

  \item $\md := d_{1}d_{2}d_{3}$; see (\ref{pulloutg30g40}).

  \item
  $\mathcal{L}_{\textbf{d,r,g}}(u_{1}, u_{2}, \mz)$; see  (\ref{tripDSuponglu}). 

    \item  $\mathfrak{k}_{1}=g_{3}r_{1}$ and $\mathfrak{k}_{2}=g_{4}r_{2}$; see (\ref{defofabstcoef}) and (\ref{grouptripDS}).

    \item $(F_{i})_{p}(g_{2})$; see (\ref{rawshiftHec})--(\ref{rawshiftRS}). 

 \item  $\mk_{1}', \mk_{2}', g_{30}, g_{40}, g_{3}', g_{4}'$; see Section \ref{simplfffying}.

  \item $\mathcal{G}_{\textbf{d}, g_{1}, g_{30}, g_{40}; \psi}(s_{1}, s_{2},z; it )$; see (\ref{Greicporsum}). 

  \item $ \mathcal{E}_{\mk_{i}'}(u_{i}, \mz; \,\Pi_{i}, \psi_{i})$, $f_{\Pi,\, \psi}(\delta; u, \mz)$, $ \mathfrak{E}_{p}(U,Z; \Pi; k)$; see Proposition \ref{finiteHecprop} and (\ref{definesomefinUr}).

\item $\mathfrak{A}^{(g_{30}g_{40}\mk_{1}' \mk_{2}' \md\ml)}\left(u_{1}, u_{2},\mz; \,  (\Pi_{i},\psi_{i})_{i=1}^{2}\right)$; see (\ref{giganundivEUr}).

\item $\mathcal{E}_{\mathcal{I}}$; see (\ref{optimisum}).

\end{enumerate}

\subsection{Section \ref{sect: CSLevGL3}}

\begin{enumerate}
    \item $\mathbf{E}_{p}(\alpha,\beta,\gamma, \delta)$, $\mathbf{E}(x,y;  z_{1},z_{2})$, $\mathbf{A}(x,y; z_{1},z_{2})$; see (\ref{ratioEP}), (\ref{ratio: decompose}), (\ref{ratio: arithfact}).

    \item $ \mathfrak{U}(y; \Pi,d)$, $\mathcal{L}(x; \Pi,d)$; see (\ref{ratio: polypart}).

    \item $\mathbf{P}(x,y; z_{1},z_{2})$, $\mathbf{Q}(x,y; z_{1},z_{2})$; see (\ref{ratio: quadrapoly}), (\ref{ratio: cubic}).
\end{enumerate}


\section{Mathematica code}

We make extensive use of \texttt{Mathematica}, together with the computer package \emph{gln.m} written by Kevin A. Broughan (and is included as part of the book \cite{Go15}). Our session starts with
\begin{lstlisting}
     SetDirectory[NotebookDirectory[]]; 
           << gln.m
\end{lstlisting}
The Schur polynomials described in Section \ref{Hprelim} of this article can be generated by \texttt{SchurPolynomial[]}. In particular, the Dirichlet coefficient $\lambda_{\Pi}(p^n)$ for $\Pi$ on $\mathrm{PGL}_{d}(\A_{\Q})$ can be created by:
\begin{lstlisting}
Schu[n_, d_, nota_ : a] := SchurPolynomial[Table[nota[i], {i, 1, d}], Join[ConstantArray[0, d - 2], {n}]]
\end{lstlisting}

The following commands will be used often in extracting coefficients and common factors:
\begin{lstlisting}
coeffsAndMonoms[p_, vars_] := Module[{rules, exps, coeffs, monoms}, rules = CoefficientRules[p, vars]; exps = Keys[rules];
  coeffs = Values[rules]; monoms = Times @@ MapThread[Power, {vars, #}] & /@ exps; {coeffs, monoms}]

commonMonomialFactor[expr_, vars_, ass_ : True] := Module[{terms, exps, mins}, terms = List @@ Expand[expr]; exps = Table[Exponent[t, v], {t, terms}, {v, vars}]; mins = Simplify[Min /@ Transpose[exps], ass]; Times @@ MapThread[#1^#2 &, {vars, mins}]]

pullCommonMonomial[expr_, vars_, ass_ : True] :=  Module[{fac}, fac = commonMonomialFactor[expr, vars, ass]; fac FullSimplify[expr/fac, ass]]

OmegaRule2[\[Omega]_, ass_ : True] := HoldPattern[c_. a[1]^p_. a[2]^q_.] :> With[{s = FullSimplify[Min[p, q], ass]}, c \[Omega]^s a[1]^FullSimplify[p - s, ass] a[2]^ FullSimplify[q - s, ass]];

OmegaRule3[\[Omega]_, ass_ : True] := HoldPattern[c_. a[1]^p_. a[2]^q_. a[3]^r_.] :> With[{s = FullSimplify[Min[p, q, r], ass]}, c \[Omega]^s a[1]^FullSimplify[p - s, ass] a[2]^ FullSimplify[q - s, ass] a[3]^FullSimplify[r - s, ass]];
\end{lstlisting}

The commands below generate useful Euler factors for normalization:
\begin{lstlisting}
InvL[X_, d_, nota_ : a] := Product[(1 - nota[i] X), {i, 1, d}]
InvRSL[X_, d_, nota1_ : a, nota2_ : b] := Product[Product[(1-nota1[i]*nota2[j]X), {i, 1, d}], {j, 1, d}]
Syma2 = a[1] - a[2];
Symb2 = b[1] - b[2];
Syma3 = (a[1] - a[2]) (a[1] - a[3]) (a[2] - a[3]);
Symb3 = (b[1] - b[2]) (b[1] - b[3]) (b[2] - b[3]);
\end{lstlisting}

We use \texttt{Fou[m,n]}, \texttt{Hec[m]}, \texttt{dHec[n]}  refer to $\lambda_{\Pi}(p^m, p^n)$, $\lambda_{\Pi}(p^m)$, $\lambda_{\widetilde{\Pi}}(n)$. Also, \texttt{A[1]}, \texttt{A[2]} refer to the elementary polynomials of degree $1$ and $2$. The following is commonly used in pattern-matching with the Schur polynomials:
\begin{lstlisting}
Rulpl={a[1] :> Hec[1], a[1]^k_. - a[2]^k_. :> Hec[k- 1], a[1]^(k + n_.)a[2]^r_. - a[1]^(k + n_.)a[3]^r_. + a[2]^(k + n_.)a[3]^r_. - a[2]^(k + n_.)a[1]^r_. + a[3]^(k + n_.)a[1]^r_. - a[3]^(k + n_.)a[2]^r_. :> Fou[r-1, k +n-r-1]};

Rulmin = {a[1] :> Hec[1], -a[1]^k_. + a[2]^k_. :> -Hec[k - 1], -a[1]^(k + n_.) a[2]^r_. + a[1]^(k + n_.) a[3]^r_. - a[2]^(k + n_.) a[3]^r_. + a[2]^(k + n_.) a[1]^r_. - a[3]^(k + n_.) a[1]^r_. + a[3]^(k + n_.) a[2]^r_. :> -Fou[r - 1, k + n - r - 1]};

Char[\[Omega]_] := a[1]^k_. a[2]^k_. a[3]^k_. :> \[Omega];

UnrRepla = Fou[m_., n_.] :> dHec[m]*Hec[n] - dHec[m - 1]*Hec[n - 1];
RamRepla = Fou[m_., n_.] :> Hec[m + n]*Hec[m] - Hec[m + n + 1]*Hec[m - 1];
HecRecur[k_] := Hec[k] -> Hec[1]*Hec[k - 1] - A[2]*Hec[k - 2] + \[Omega]*Hec[k - 3];
AnyHecRec = Hec[k_.] :> Hec[1]*Hec[k - 1] - A[2]*Hec[k - 2] + \[Omega]*Hec[k - 3];
Replsch = Fou[m_., n_.] :> SchurPolynomial[{a[1], a[2], a[3]}, {m, n}];
\end{lstlisting}


\section{Lemmata \ref{lem: naiveRSexpl} and \ref{lem: diffby1RSn}  {\normalfont\ttfamily (ShiftNaRS.nb)}}\label{sect:shiftedRS}

\begin{lstlisting}
R[d_, k_, l_] := Cancel[InvRSL[x, d, a, b]* Sum[Schu[m + k, d, a] *Schu[m + l, d, b] x^m, {m, 0, Infinity}]]
Syma = Join[{a[1]}, Table[Denominator[Schu[m, d, a]], {d, 2, 3}]];
assu = Element[k, Integers] && k >= 2;
GL3HecVar = Table[Hec[k - i], {i, 0, 4}];

ListCoef[d_, k_, l_] := coeffsAndMonoms[Cancel[R[d, k, l]*Syma[[d]]], {x}]
AfterSym[d_, k_, l_] := (SymmetricReduction[#, Table[b[i], {i, 1, d}], Join[Table[B[i], {i, 1, d - 1}], {\[Omega]}]][[1]] & /@ ListCoef[d, k, l][[1]]) /. \[Omega]^r_. :> \[Omega]
PartA[d_, k_, l_] := coeffsAndMonoms[#, Join[Table[B[i], {i, 1, d - 1}], {\[Omega]}]] & /@ AfterSym[d, k, l]
PartACollect[d_, k_, l_, \[Omega]_] := PartA[d, k, l][[All, 1]] /. 
If[d == 2, OmegaRule2[\[Omega], assu], OmegaRule3[\[Omega], assu]]

FinalPoly[d_, k_, l_, sgn_] := Total[(PartA[d, k, l][[All, 2]] /. \[Omega] -> sgn)*((PartACollect[d, k, l, \[Omega]] /. \[Omega] -> sgn) //. {Rulpl[[d]], Rulmin[[d]]})*ListCoef[d, k, l][[2]] //Flatten]

(*The H-function, PGL3*)
Collect[SymmetricReduction[SymmetricReduction[
R[3, 0, 0], {a[1], a[2], a[3]}, {A[1], A[2], \[Omega]}][[1]], {b[1], b[2], b[3]}, {B[1], B[2], \[Omega]}][[1]], x]

(*Naive RS the first  argument shifted by k*)

(*Unramified, PGL2*)
Collect[FinalPoly[2, k, 0, 1], x]
(*Ramified, PGL(2)*)
Collect[FinalPoly[2, k, 0, 0], x]

(*Unramified, PGL3*)
FinUnrGL3 = Collect[FinalPoly[3, k, 0, 1], x]
FinUnrGL3InHec = Collect[(FinUnrGL3 /. UnrRepla) /. {Hec[0] -> 1, dHec[0] -> 1, dHec[-1] -> 0}, GL3HecVar]

(*Bounding Unr*)
coeffsAndMonoms[FinUnrGL3InHec, GL3HecVar][[1]]
coeffsAndMonoms[FinUnrGL3InHec, GL3HecVar][[2]]
UnrGL3Expo =  Exponent[#, p] & /@ (List @@ Expand[#] & /@ coeffsAndMonoms[FinUnrGL3InHec, GL3HecVar][[1]] /. {(Hec | dHec)[n_] :> p^(n*Th), B[n_] :> p^Th, x -> p^(-1)})
UnrGL3RespScaleMax = Assuming[0 < Th < 1/2, Max /@ (MapIndexed[#1/#2[[1]] &, Rest[UnrGL3Expo]]) // FullSimplify]
Assuming[5/14 < Th < 1/2, Max[UnrGL3RespScaleMax] // FullSimplify]

(*Ramified, PGL3*)
FinalPoly[3, k, 0, 0]
(FinalPoly[3, k, 0, 0] - Fou[0, k]) /. RamRepla
FinRamGL3InHec=(FinalPoly[3, k, 0, 0] - Fou[0, k])/. RamRepla /. {Hec[0] -> 1, Hec[-1] -> 0}
FinRamGL3InHecRec = Hec[k] + (FinRamGL3InHec /. HecRecur[k + 1] /. HecRecur[k] /. \[Omega] -> 0)
coeffsAndMonoms[FinRamGL3InHecRec, GL3HecVar][[1]]
coeffsAndMonoms[FinRamGL3InHecRec, GL3HecVar][[2]]

(*Bounding Ram*)
RamGL3Expo = Exponent[#, p]&/@ ((List @@ Expand[#]& /@coeffsAndMonoms[FinRamGL3InHecRec, GL3HecVar][[1]])/.{(Hec|A|B)[n_] :> p^(n Th) x ->p^(-1)})
Assuming[0<Th<1/2,Max/@(MapIndexed[#1/#2[[1]]&,Rest[RamGL3Expo]])//FullSimplify]

(*Naive RS both arguments shifted by 1*)
(*PGL2 Unr*)
Collect[FinalPoly[2, k, 1, 1], x] 
(*PGL2 Ram*)
Collect[FinalPoly[2, k, 1, 0], x]

(*PGL3 Unr*)
GL3UnrBoth1Gen = Collect[FinalPoly[3, k, 1, 1], x]
GL3UnrBoth1 = Collect[GL3UnrBoth1Gen /. k -> 1, x]
ExceReplsch = (Fou[m_., n_.] /; TrueQ[m < 0 || n < 0]) :> SchurPolynomial[{a[1], a[2], a[3]}, {m, n}];
GL3Unrsh1Repl = Collect[SymmetricReduction[GL3UnrBoth1 /. {ExceReplsch, Fou[0, 0] -> 1}, {a[1], a[2], a[3]}, {Fou[0, 1], Fou[1, 0], 1}][[1]], x]

(*Bounding Unr*)
List @@ Expand[GL3Unrsh1Repl]
BdListGL3UnrShf1[Th_] := Exponent[#, p] & /@ ((List @@ Expand[GL3Unrsh1Repl]) /. {x -> p^(-1), Fou[m_., n_.] :> p^((m + n)*Th), B[n_] :> p^Th})

(*PGL3 Ram*)
Collect[FinalPoly[3, k, 1, 0], x]
GL3Ramsh1Repl = ((Collect[FinalPoly[3, k, 1, 0], x]) /. k -> 1) /. ExceReplsch

(*Bounding Ram*)
BdListGL3RamShf1[Th_] := Exponent[#, p] & /@ (List @@ Expand[GL3Ramsh1Repl] /. {x -> p^(-1), Fou[m_., n_.] :> p^((2 m + n)*Th), B[n_] :> p^(n*Th)})
Assuming[5/14 < Th < 1/2, Max[BdListGL3RamShf1[Th]] // FullSimplify]
\end{lstlisting}

\section{Proposition \ref{lem: ACKit} {\normalfont\ttfamily (U2pnmid.nb)}}\label{sect:U2Lfuncmatchgl3}

\begin{lstlisting}
(*Set x=1/p^u[1], y=1/p^s and; real parts of u[1] and s are 1/2+ and 1+ *)

(*PGL2*)
(*Computing L_{p}(u[1];s, Pi[1], Pi[2])*)
L2[x_, y_] := Sum[Schu[r, 2, a] x^r, {r, 1, Infinity}] + Sum[Schu[n, 2, b] y^n*Sum[Schu[n + r, 2, a] x^r, {r, 1, Infinity}], {n, 1, Infinity}]
P2[x_, y_] := Cancel[InvL[x, 2, a]*InvRSL[y, 2, a, b]*L2[x, y]]
P20[x_, y_] := Assuming[a[1] a[2] == b[1] b[2] == 1, FullSimplify[P2[x, y]]]
M20 = MonomialList[P20[x, y], {x, y, a[1], a[2], b[1], b[2]}];

(*The following outputs non-negligible terms in the Euler product*)

Select[M20, Exponent[#, x]/2 + Exponent[#, y] - (Exponent[#, a[1]] + Exponent[#, a[2]] + Exponent[#, b[1]] +Exponent[#, b[2]])*7/64 <= 1.00001 &]

(*Considering also the factor (1-1/p^(2-w)); Re(2-w)=1-\epsilon*)

Select[M20, 0.9999 + Exponent[#, x]/2 + Exponent[#, y] - (Exponent[#, a[1]] + Exponent[#, a[2]] + Exponent[#, b[1]] +Exponent[#, b[2]])*7/64 <= 1.00001 &]

(*PGL3*)
(*Computing L_{p}(u[1];s, Pi[1], Pi[2]) ---I*)

L3Pol1[x_, y_] := Cancel[InvRSL[y, 3, a, b]*x*Sum[Schu[n, 3, b] Schu[n + 1, 3, a] y^n, {n, 1, Infinity}]]
Mono1 = MonomialList[L3Pol1[x, y], {x, y, a[1], a[2], a[3], b[1], b[2], b[3]}];
M1u = Assuming[a[1] a[2] a[3] == 1 && b[1] b[2] b[3] == 1, FullSimplify[Mono1]];

(*Checking:fully simplified already?*)
Select[M1u, And @@ (Not[FreeQ[#, #2]] & @@@ Thread[{#, {a[1], a[2], a[3]}}]) &]
Select[M1u, And @@ (Not[FreeQ[#, #2]] & @@@ Thread[{#, {b[1], b[2], b[3]}}]) &]

NP1 = Total[Select[M1u, 0.000001 + Exponent[#, x]/2 + Exponent[#, y] - (Exponent[#, a[1]] + Exponent[#, a[2]] + Exponent[#, a[3]] + Exponent[#, b[1]] + Exponent[#, b[2]] + Exponent[#, b[3]])*(5/14) <= 1 &]];

(*Coeff of xy*)
SymmetricReduction[SymmetricReduction[Coefficient[Collect[Cancel[NP1/x], y], y, 1], {b[1], b[2], b[3]}, {B[1], B[2], 1}][[1]], {a[1], a[2], a[3]}, {A[1], A[2], 1}][[1]]

(*Coeff of xy^2*)
SymmetricReduction[ SymmetricReduction[Coefficient[Collect[Cancel[NP1/x], y], y, 2], {b[1], b[2], b[3]}, {B[1], B[2], 1}][[1]], {a[1], a[2], a[3]}, {A[1], A[2], 1}][[1]]

(*Computing L_{p}(u[1];s, Pi[1], Pi[2]) ---II*)

L3Pol2[x_, y_] := Cancel[InvRSL[y, 3, a, b]*x^2*
Sum[Schu[n, 3, b] Schu[n - 1, 3, a] y^n, {n, 1, Infinity}]]
Mono2 = MonomialList[L3Pol2[x, y], {x, y, a[1], a[2], a[3], b[1], b[2], b[3]}];
M2u = Assuming[a[1] a[2] a[3] == 1 && b[1] b[2] b[3] == 1, FullSimplify[Mono2]];

Select[M2u, And @@ (Not[FreeQ[#, #2]] & @@@ Thread[{#, {a[1], a[2], a[3]}}]) &]
Select[M2u, And @@ (Not[FreeQ[#, #2]] & @@@ Thread[{#, {b[1], b[2], b[3]}}]) &]

NP2 = Total[Select[M2u, 0.000001 + Exponent[#, x]/2 + Exponent[#, y] - (Exponent[#, a[1]] + Exponent[#, a[2]] + Exponent[#, a[3]] + Exponent[#, b[1]] + Exponent[#, b[2]] + Exponent[#, b[3]])*(5/14) <= 1 &]]
\end{lstlisting}

\section{Proposition \ref{finiteHecprop} {\normalfont\ttfamily (UrFinExact.nb)}}\label{sect: UrHeckeExact}
\begin{lstlisting}
assu = Element[k, Integers] && k >= 2;

(*The case of PGL(2)*)  
Sortsum2 = Collect[InvL[U, 2]*InvL[Z,2]*Syma2*Sum[Sum[Schu[M + g + k, 2, a]*U^M, {M, 0, Infinity}]*Z^g, {g, 0,Infinity}]//FullSimplify//Expand, {U, Z, U*Z}]

GL2CoefinHec = pullCommonMonomial[#, {a[1], a[2]}, assu]&/@ coeffsAndMonoms[Sortsum2, {U, Z}][[1]]/.{a[1]^k_. - a[2]^k_. :> Hec[k - 1], -a[1]^k_. + a[2]^k_. :> -Hec[k - 1]}

Total[GL2CoefinHec*coeffsAndMonoms[Sortsum2, {U, Z}][[2]]]/.{a[1]^k_. a[2]^k_. :> \[Omega]}

(*The case of PGL(3)*)  
Sortsum3 = Collect[InvL[U,3]*InvL[Z,3]*Syma3*Sum[Sum[Schu[M + g + k, 3, a]*U^M, {M, 0, Infinity}]*Z^g, {g, 0,Infinity}] // FullSimplify // Expand, {U, U^2, Z, Z^2, U*Z, U^2*Z, U*Z^2, U^2*Z^2}];

GL3CoefinHec = pullCommonMonomial[#, {a[1], a[2], a[3]}, assu] & /@ 
coeffsAndMonoms[Sortsum3, {U, Z}][[1]]/.{a[1]^k_. a[2]^k_. a[3]^k_. :>\[Omega]}

OmePos = Flatten@Position[GL3CoefinHec, p_ /; ! FreeQ[p, \[Omega]], {1}]
coeffsAndMonoms[Sortsum3, {U, Z}][[2]][[Flatten@OmePos]]
GL3inSchur = MapAt[\[Omega] # &, Expand[GL3CoefinHec /. \[Omega] -> 1] /. {Rulpl[[3]], Rulmin[[3]] }, List /@ OmePos]

(*Ramified*)

GL3inSchurRam =  Assuming[a[1] a[2] a[3] == 0, FullSimplify[GL3inSchur /. {\[Omega] -> 0}]]
GL3FouMonoRam = GL3inSchurRam*coeffsAndMonoms[Sortsum3, {U, Z}][[2]] /. Z -> 0
Total[GL3FouMonoRam]

(*Unramified*)

GL3inSchurUnr = Assuming[a[1] a[2] a[3] == 1 && assu, Expand[FullSimplify[Expand[GL3CoefinHec /. \[Omega] -> 1] /. {Rulpl[[3]], Rulmin[[3]]}]]/. {Rulpl[[3]], Rulmin[[3]]}]
GL3FouMonoUnr = GL3inSchurUnr*coeffsAndMonoms[Sortsum3, {U, Z}][[2]]
Total[GL3FouMonoUnr]

(*In terms of Hecke eigenvalues*)
CoefUnrInHec=Coefficient[Total[GL3FouMonoUnr]/.UnrRepla,Table[Hec[k-i],{i,0,3}]]
farith=SymmetricReduction[#, {a[1], a[2], a[3]}, {A[1], A[2], 1}][[1]] & /@ (CoefUnrInHec /. {dHec[m_.] :> SchurPolynomial[{a[1], a[2], a[3]}, {m, 0}]})

(*Bounding*)
ExpoList = Exponent[#, p] & /@ (List @@ # & /@ farith /. {Z -> p^(-1), U -> p^(-1/2), A[1] -> p^Th, A[2] -> p^Th})
RespScaleMax = Assuming[0 < Th < 1/2, Max /@ (MapIndexed[#1/#2[[1]] &, Rest[ExpoList]]) //FullSimplify]
Assuming[5/14 < Th < 1/2, Max[RespScaleMax]//FullSimplify]

(*Checking: we got the right polynomial*)
See the notebook online.
\end{lstlisting}


\section{Proposition \ref{HecktwistRSnuden}  {\normalfont\ttfamily (UrInf.nb)} }\label{sect: UnrInfExact}

\begin{lstlisting}
FF2[d_, g_] := Sum[Schu[M + g, d, b]*V^M, {M, 0, Infinity}]
FF1[d_, g_] := Sum[Schu[M + g, d, a]*U^M, {M, 0, Infinity}]
FF3[d_, g_] := Schu[g, d, a]*FF2[d, g] + Schu[g, d, b]*FF1[d, g] -Schu[g, d, a]*Schu[g, d, b]
NormaSum[d_] := Cancel[Sum[FF3[d, g]*Z^g, {g, 0, Infinity}]*InvL[U, d, a]*InvL[V, d, b]*InvRSL[Z, d, a, b]]

(*The case of PGL(2)*)  
GL2inSym = (SymmetricReduction[#, {a[1], a[2]}, {A[1], \[Omega]}][[1]] & /@ (SymmetricReduction[#, {b[1], b[2]}, {B[1], \[Omega]}][[1]] & /@ NormaSum[2])) /. \[Omega]^k_. :> \[Omega]
Collect[GL2inSym, DeleteCases[Times @@ ({A[1], B[1]}^#) & /@ Tuples[Range[0, 2], 2], 1]]

(*PGL(3): the calculation is divided into 3 parts*)

Part1Norm = Cancel[Sum[ Schu[g, 3, a]*Cancel[FF2[3, g]*Symb3*InvL[V, 3, b]]*Z^g, {g, 0, Infinity}]*InvRSL[Z, 3, a, b]];
Part1Symm = (SymmetricReduction[#, {a[1], a[2], a[3]}, {A[1], P[1], \[Omega]}][[1]] & /@ (SymmetricReduction[#, {b[1], b[2], b[3]}, {B[1], Q[1], \[Omega]}][[1]] & /@  Cancel[Part1Norm/Symb3])) /. \[Omega]^k_. :> \[Omega]

Part2Norm = Cancel[Sum[Schu[g, 3, b]*Cancel[FF1[3, g]*Syma3*InvL[U, 3, a]]*Z^g, {g, 0, Infinity}]*InvRSL[Z, 3, a, b]];
Part2Symm = (SymmetricReduction[#, {a[1], a[2], a[3]}, {A[1], P[1], \[Omega]}][[1]] & /@ (SymmetricReduction[#, {b[1], b[2], b[3]}, {B[1], Q[1], \[Omega]}][[1]] & /@ Cancel[Part2Norm/Syma3])) /. \[Omega]^k_. :> \[Omega]

Part3Norm = Cancel[Sum[Schu[g, 3, a]*Schu[g, 3, b]*Z^g, {g, 0, Infinity}]*InvRSL[Z, 3, a, b]];
Part3Symm = (SymmetricReduction[#, {a[1], a[2], a[3]}, {A[1], P[1], \[Omega]}][[1]] & /@ (SymmetricReduction[#, {b[1], b[2], b[3]}, {B[1], Q[1], \[Omega]}][[1]] & /@ Part3Norm)) /. \[Omega]^k_. :> \[Omega]

InvLSymm3a = SymmetricReduction[InvL[U, 3, a], {a[1], a[2], a[3]}, {A[1], P[1], \[Omega]}][[1]]
InvLSymm3b = SymmetricReduction[InvL[V, 3, b], {b[1], b[2], b[3]}, {B[1], Q[1], \[Omega]}][[1]]

GL3inSym = Expand[FullSimplify[(Part1Symm*InvLSymm3a + Part2Symm*InvLSymm3b - Part3Symm*InvLSymm3a*InvLSymm3b)]] /.{\[Omega]^k_. :> \[Omega]}

(*List of monomials for the unramified case*)
UnrListMono3= List @@ (Expand[GL3inSym /. \[Omega] -> 1]);

(*Bounding the polynomial with vartheta*)
UnrListexpo3 = Exponent[#, p] & /@ (List @@ UnrListMono3 /. {A[1] -> p^Th, B[1] -> p^Th, P[1] -> p^Th, Q[1] -> p^Th, U -> p^(-1/2), V -> p^(-1/2), Z -> p^(-1)} /. Th -> 5/14) // N;
Total[UnrListMono3[[#]] & /@ Flatten@Position[# < -1 & /@ UnrListexpo3, False]]

(*Ramified primes*)
RamListMono3 = List @@ (GL3inSym /. \[Omega] -> 0);
RamListexpo3 = Exponent[#, p] & /@ (List @@ RamListMono3 /. {A[1] -> p^Th, B[1] -> p^Th, P[1] -> p^(2*Th), Q[1] -> p^(2*Th), U -> p^(-1/2), V -> p^(-1/2), Z -> p^(-1)} /. Th -> 5/14) // N;
\end{lstlisting}


\section{Ratio calculation {\normalfont\ttfamily (GL3Ratio.nb)}}\label{sect: ratiosMath}

\begin{lstlisting}
del=1/1000000;
SubRules = Join[Thread[{A1, B1}->p^Th],Thread[{x, y, z, z1, z2}->p^(-1+del)]];
A[n_] := SchurPolynomial[{a[1], a[2], a[3]}, {0, n}]
B[n_] := SchurPolynomial[{1/a[1], 1/a[2], 1/a[3]}, {0, n}]
muA = {-A[1], B[1], -1}
muB = {-B[1], A[1], -1}
RS[x_] := Product[1 - (a[i]/a[j]) x, {i, 1, 3}, {j, 1, 3}]
L[d_, x_]:= Cancel[RS[x]*Sum[A[n + d]*B[n]*x^n, {n, 0, Infinity}]]
MuPoly[d_, y_] := muB[[d]] + Sum[muA[[h]]*muB[[d + h]]*y^h, {h, 1, 3 - d}]

Polybefcan = Table[Assuming[a[1]*a[2]*a[3] == 1, z^n*MuPoly[n, y]*L[n, x] // FullSimplify], {n,1,3}];
FacCan = {-z (1 - x)^2 (1 + x), (-1 + x)^2 (1 + y) z^2, -(-1 + x)^2 z^3};
Polycancrawd = Table[Cancel[Polybefcan[[n]]/FacCan[[n]]], {n, 1, 3}];
Polyd = Table[SymmetricReduction[Polycancrawd[[n]], {a[1], a[2], a[3]}, {A1, B1, 1}][[1]] // Factor, {n, 1, 3}]
MonoList = Table[FacCan[[n]]*Polyd[[n]] // ExpandAll // MonomialList,{n, 1, 3}];
MonoListExpo =  Table[Exponent[#,p] & /@ (MonoList[[n]] /. SubRules /. Th -> 5/14) // N, {n, 1, 3}];
FirstKS = Table[MonoList[[n]][[Flatten@Position[# < -1 & /@ MonoListExpo[[n]], False]]] /. {z ->z1}, {n, 1, 3}] // Flatten
DualFirstKS = FirstKS /. {A1 -> B1, B1 -> A1, z1 -> z2}

(*The following is the polynomial \mathbf{P}*)
Storemypoly = 1 + A1 B1 y + Total[FirstKS] + Total[DualFirstKS]

(*Comparison with other 3 Rankin-Selberg factors*)

RSRatio3 = RS[y]^(-1)*RS[z1]*RS[z2];
RSExpMono = Table[SymmetricReduction[Assuming[a[1] a[2] a[3] == 1, Coefficient[RS[y], y, n] // Together // FullSimplify], {a[1], a[2], a[3]}, {A1, B1, 1}][[1]], {n, 0, 3}]; 
RSinvpol[y_] := Sum[RSExpMono[[n]]*y^(n - 1), {n, 1, 3}]

RScanList = Storemypoly*RSinvpol[y] // ExpandAll // MonomialList;
Listexpocance = Exponent[#,p] & /@ (RScanList /. SubRules/. Th -> 5/14) // N;
Aftercleany = Total[RScanList[[#]] & /@ Flatten@Position[# < -1 & /@ Listexpocance, False]]

(*Next, let's clean with 1/(RS) over z1 z2*)
RSusuMono = Table[SymmetricReduction[Assuming[a[1] a[2] a[3] == 1, Together[Coefficient[Series[RS[z]^(-1), {z, 0, 4}], z, n]] // FullSimplify], {a[1], a[2], a[3]}, {A1, B1, 1}][[1]], {n, 0, 4}];
RSusual[z_] := Sum[RSusuMono[[n]]*z^(n - 1), {n, 1, 4}]

FinalList = Aftercleany*RSusual[z1]*RSusual[z2] // ExpandAll // MonomialList;
FinalListexpocance = Exponent[#, p] &/@ (FinalList /. SubRules /.Th -> 5/14)//N;
Total[FinalList[[#]]&/@Flatten@Position[# < -1 & /@FinalListexpocance,False]]
\end{lstlisting}


\printbibliography

\end{document}